\documentclass{article}[11pt]
\usepackage{amssymb}
\usepackage{amsmath}
\usepackage{amscd}
\usepackage{amsthm}
\usepackage{multirow}
\usepackage{sectsty}

\usepackage{url}
\usepackage{hyperref}
\usepackage{cancel}
\usepackage{stackrel}
\usepackage{bigints}
\usepackage{array}
\usepackage{enumerate}
\usepackage{enumitem}
\usepackage{changepage} % allows to shift chunks of text to left/right

\usepackage{MnSymbol}
\usepackage{verbatim} % for verbatim text AND multi-line comments
\usepackage{epsfig}
\usepackage{mdwlist}
\usepackage{float}
\usepackage[ruled,lines numbered,noend]{algorithm2e} % ruled gives lines, numbered=line numbers, noend=hide ends
\usepackage[hang,flushmargin,bottom]{footmisc}% no indention for footnotes

\usepackage{framed}
\usepackage{color}
\usepackage{xcolor}
\usepackage{graphics}
\usepackage{graphicx}
\usepackage{rotating} % rotate text 90 degrees
\allsectionsfont{\sc\centering}
\newcolumntype{L}{>{\centering\arraybackslash}m{\dimexpr.15\linewidth-2\tabcolsep}}
\newcolumntype{C}{>{\centering\arraybackslash}m{\dimexpr.42\linewidth-2\tabcolsep}}

\definecolor{LightGreen}{HTML}{D1FFD1}
\definecolor{LightOrange}{HTML}{FFEBCC}
\definecolor{LightLightBlue}{HTML}{E0F5FF}
\definecolor{LightBlue}{HTML}{A3E0FF}
\definecolor{tan}{HTML}{FFFFCC}
\definecolor{DarkBlue}{HTML}{000099}
\definecolor{DarkGreen}{HTML}{009900}
\definecolor{DarkRed}{HTML}{990000}
\definecolor{BrightOrange}{HTML}{FF9900}
\definecolor{PurpleIsh}{HTML}{FF59FF}

\def\barint_#1{\mathchoice
{\mathop{\vrule width 6pt height 3 pt depth -2.5pt
\kern -8.8pt \intop}\nolimits_{#1}}%
{\mathop{\vrule width 5pt height 3 pt depth -2.6pt
\kern -6.5pt \intop}\nolimits_{#1}}%
{\mathop{\vrule width 5pt height 3 pt depth -2.6pt
\kern -6pt \intop}\nolimits_{#1}}%
{\mathop{\vrule width 5pt height 3 pt depth -2.6pt
\kern -6pt \intop}\nolimits_{#1}}}

\usepackage{pgf}
\usepackage{tikz}
\usepackage{pgfcore}
\usetikzlibrary{shadows}
\usetikzlibrary{arrows}
\usetikzlibrary{decorations.markings}
\usetikzlibrary{patterns,plotmarks,shapes,snakes,3d,automata,backgrounds,topaths,trees,petri,mindmap}
\numberwithin{equation}{section}

\newcommand*{\CopyCounter}[2]{
  \expandafter\def\csname c@#2\endcsname{\csname c@#1\endcsname}
  \expandafter\def\csname p@#2\endcsname{\csname p@#1\endcsname}
  \expandafter\def\csname the#2\endcsname{\csname the#1\endcsname}}
\newcounter{Theorem}
\numberwithin{Theorem}{section}
\CopyCounter{Theorem}{Proposition}
\CopyCounter{Theorem}{ProposedProblem}
\CopyCounter{Theorem}{Property}
\CopyCounter{Theorem}{Claim}
\CopyCounter{Theorem}{Lemma}
\CopyCounter{Theorem}{Corollary}
\CopyCounter{Theorem}{Conjecture}
\CopyCounter{Theorem}{Definition}
\CopyCounter{Theorem}{Example}
\CopyCounter{Theorem}{Remark}
\CopyCounter{Theorem}{Question}
\CopyCounter{Theorem}{Condition}
\CopyCounter{Theorem}{Criterion}
\CopyCounter{Theorem}{Observation}
\theoremstyle{plain}%% needs amsthm.sty
\newtheorem{thm}[Theorem]{Theorem}

\newtheorem{lemma}[Lemma]{Lemma}

\theoremstyle{definition}%% needs amsthm.sty

\newtheorem{remark}[Remark]{Remark}

\newtheorem{conjecture}[Conjecture]{Conjecture}

\newtheorem{observ}[Observation]{Observation}

\newif\ifnever\neverfalse

\newcommand{\bm}[1]{\mbox{\boldmath${#1}$}}

\newcommand{\N}{\mathbb{N}}

\newcommand{\R}{\mathbb{R}}

\newcommand{\bs}{\text{$\bm{s}$}}
\newcommand{\bt}{\text{$\bm{t}$}}

\newcommand{\bx}{\text{$\bm{x}$}}

\newcommand{\ba}{\mathbf{a}}

\newcommand{\by}{\text{$\bm{y}$}}

\newcommand{\abs}[1]{\left|{#1}\right|}
\newcommand{\braces}[1]{\left[{#1}\right]}
\newcommand{\cbraces}[1]{\left\{{#1}\right\}}

\newcommand{\norm}[1]{\left\|{#1}\right\|}
\newcommand{\parens}[1]{\left({#1}\right)}
\newcommand{\fillmeup}{\vspace{.1 in} \\%
}
\newcommand{\Forall}{\hspace{.03in} \forall \hspace{.03in}}

\newcommand{\FAR}{$\mbox{\textsc{\color{DarkRed}far}}$}

\newcommand{\CON}{\textsc{\color{DarkRed}considered}}
\newcommand{\ACC}{\textsc{\color{DarkRed}accepted}}
\newcommand{\ACCing}{\textsc{\color{DarkRed}accepting}}
\newcommand{\AC}{\textsc{\color{DarkRed}accept}}
\newcommand{\ACs}{\textsc{\color{DarkRed}accepts}}

\newcommand{\UplusVFigScale}{0.475}

\newcommand{\sinPicScale}{0.235}
\newcommand{\statsScale}{0.42}
\newcommand{\discErr}{\mathcal{E}^d}
\newcommand{\astarErr}{\mathcal{E}^*}
\newcommand{\astarErrOnly}{\mathcal{E}^*_N}
\newcommand{\Upper}{\Psi}
\newcommand{\exitset}{\mathcal{Q}}

\newcommand{\figstart}{\begin{figure}[H]}
\newcommand{\zachEdit}[1]{{#1}}
\newcommand{\alexEdit}[1]{{#1}}

\newcommand{\marginfix}{
\setlength{\parskip}{0.01cm}
\setlength{\textwidth}{6.0in}
\setlength{\oddsidemargin}{-0.0 in}
\setlength{\evensidemargin}{0.0 in}
\setlength{\topmargin}{-0.5in}
\setlength{\textheight}{9.0 in}
}

\marginfix

  \renewenvironment{thebibliography}[1]{%
    \begin{oldthebibliography}{#1}%
      \setlength{\parskip}{.3ex}%
      \setlength{\itemsep}{.3ex}%
  }%
  {%
    \end{oldthebibliography}%
  }

\usepackage{etoolbox} 
\newtoggle{usecolor}
\toggletrue{usecolor}
\begin{document}

% COMMENT REMOVED
% COMMENT REMOVED
% COMMENT REMOVED
% COMMENT REMOVED
% COMMENT REMOVED

\centerline{\Large\textbf{Causal Domain Restriction for Eikonal Equations}}

\vspace*{.1in}

% COMMENT REMOVED
\renewcommand*{\thefootnote}{\fnsymbol{footnote}}

{\Large
\centerline{Z. Clawson$\footnotemark[2]{}^{,}\footnotemark[8]$, A. Chacon$\footnotemark[3]{}^{,}\footnotemark[8]$, A. Vladimirsky\footnotemark[8]
\footnotetext[2]{\sc Supported in part by the NSF through Graduate Research Fellowship and the 2010 REU at Cornell.}
\footnotetext[3]{\sc Supported in part by Alfred P. Sloan Foundation Graduate Fellowship.}
\footnotetext[8]{\sc Supported in part by the National Science Foundation grant DMS-1016150.}
}}

% COMMENT REMOVED
% COMMENT REMOVED
\renewcommand*{\thefootnote}{\arabic{footnote}}

% COMMENT REMOVED
{\large
\vspace*{.1in}
\centerline{Center for Applied Mathematics and Department of Mathematics}
\vspace{-.05cm}
\centerline{Cornell University, Ithaca, NY 14853}
}
% COMMENT REMOVED

\vspace*{.1in}
\begin{abstract}
\noindent
Many applications require efficient methods for solving continuous shortest path problems. Such paths can be viewed as characteristics of static Hamilton-Jacobi equations. Several fast numerical algorithms have been developed to solve such equations on the whole domain. In this paper we consider a somewhat different problem, where the solution is needed at one specific point, so we restrict the computations to a neighborhood of the characteristic. We explain how heuristic under/over-estimate functions can be used to obtain a {\em causal} domain restriction, significantly decreasing the computational work without sacrificing convergence under mesh refinement. The discussed techniques are inspired by an alternative version of the classical A* algorithm on graphs. We illustrate the advantages of our approach on continuous isotropic examples in 2D and 3D. We compare its efficiency and accuracy to previous domain restriction techniques. We also analyze the behavior of errors under the grid refinement and show how Lagrangian (Pontryagin's Maximum Principle-based) computations can be used to enhance our method.
\end{abstract}

\vspace*{.1in}

% COMMENT REMOVED
% COMMENT REMOVED
% COMMENT REMOVED
% COMMENT REMOVED

% COMMENT REMOVED
\section{Introduction}
The Eikonal equation
\begin{equation}\label{Eikonal intro}
\left\{
\begin{array}{rcll}
\abs{\nabla u(\bx)}f(\bx) & = & 1 & \qquad \forall \ \bx \in \Omega \\
u(\bx) & = & q(\bx) & \qquad \forall \ \bx \in \exitset \subseteq \partial \Omega,
\end{array}
\right.
\end{equation}
arises naturally in many application\zachEdit{s} including
continuous optimal path planning, computational geometry, photolithography, optics, shape from shading, and image processing \cite{SethBook}.
One natural interpretation for the solution of \eqref{Eikonal intro} comes from
isotropic time-optimal control problems. For a vehicle traveling through $\bar{\Omega}$, $f$ describes the speed of travel
and $q$ gives the exit time-penalty charged on $\exitset$.  In this framework, $u(\bx)$ is the \emph{value function}; i.e., the minimum time to exit $\bar{\Omega}$ through $\exitset$ if we start from a point $\bx \in \Omega$. The characteristic curves of the PDE \eqref{Eikonal intro} define the optimal trajectories for the vehicle motion.
\fillmeup
The value function is Lipschitz continuous, but generally is not smooth on $\Omega$.  (The gradient of $u$ is undefined at all points for which an optimal trajectory is not unique.)
Correspondingly, the PDE \eqref{Eikonal intro} typically does not have a smooth solution and admits infinitely many Lipschitz continuous weak solutions.  Additional conditions introduced in \cite{Crandall_ViscOriginal} are used to restore the uniqueness: the {\em viscosity solution} is unique and coincides with the value function of the above control problem.
\fillmeup
In the last 20 years, many fast numerical methods have been developed  to solve \eqref{Eikonal intro} \emph{on the entire domain} $\bar{\Omega}$; e.g., see \cite{Adam, Seth_FMM, Tsitsiklis, Zhao_FSM}. Many of these fast methods were inspired by classical label-correcting and label-setting algorithms on graphs; e.g., Sethian's Fast Marching Method \cite{Seth_FMM} mirrors the logic of the classical Dijkstra's algorithm \cite{Dijkstra}, which finds the minimum time to a target-node from every other node in the graph.
\fillmeup
Our focus here is on a somewhat different situation, with the solution needed \emph{for one specific starting position only}. On graphs, an A* modification of Dijkstra's method \cite{Hart_Astar} is widely used for similar {\em single source / single target} shortest path problems.
% COMMENT REMOVED
% COMMENT REMOVED
% COMMENT REMOVED
% COMMENT REMOVED
% COMMENT REMOVED
% COMMENT REMOVED
% COMMENT REMOVED
% COMMENT REMOVED
% COMMENT REMOVED
% COMMENT REMOVED
% COMMENT REMOVED
There have been several prior attempts to extend A* techniques to algorithms for continuous optimal trajectory problems, but all of them have significant drawbacks:
 these methods either produce additional errors that do not vanish under numerical grid refinement \cite{Petres_thesis, Peyre_coarse, Peyre_landmark, Peyre_Geodesic}, or provide much more limited computational savings \cite{FergusonStentz, Yershov_1, Yershov_2}.
We believe that these disadvantages stem from an overly faithful mirroring of the ``standard" A* on graphs.
Our own approach is based on an alternative version of the A* algorithm %on graphs
\cite{Bertsekas_DPbook} that has clear advantages in continuous optimal control problems.  Numerical testing confirms that our method is both efficient (in terms of the percentage of domain restriction) and convergent under grid refinement.
% COMMENT REMOVED
% COMMENT REMOVED
% COMMENT REMOVED
% COMMENT REMOVED
% COMMENT REMOVED
% COMMENT REMOVED
% COMMENT REMOVED
% COMMENT REMOVED
% COMMENT REMOVED
% COMMENT REMOVED
% COMMENT REMOVED
% COMMENT REMOVED
% COMMENT REMOVED
% COMMENT REMOVED
% COMMENT REMOVED
% COMMENT REMOVED
% COMMENT REMOVED
% COMMENT REMOVED
% COMMENT REMOVED
\fillmeup
We begin by reviewing two flavors of A* techniques on graphs in \S\ref{s:graphs}.
We then describe the standard Fast Marching Method and its various A*-type modifications in \S\ref{s:continuous}.
The numerical tests in \S\ref{s:examples} are used to compare the efficiency and accuracy of competing domain restriction techniques.  We discuss the limitations of our approach and directions of future work in \S\ref{s:conclusions}.
The Appendix (\S\ref{s:it_works}) contains convergence analysis of the alternative A* under grid refinement, exploiting the probabilistic interpretation of the discretized equations.

% COMMENT REMOVED
\section{Domain restriction techniques on graphs}
\label{s:graphs}
We start by defining the shortest path problem on a graph:

\begin{itemize}[leftmargin=6mm]\itemsep-2pt
\item A graph $\mathcal{G}$ is defined by a set of nodes (vertices) $X = \cbraces{\bx_1,\bx_2,\ldots,\bx_{M+1} = \bt}$ and a set of directed arcs between these nodes.
\item Along each arc we prescribe a transition time penalty $C(\bx_i,\bx_j)=C_{ij} > 0$, and assume $C_{ij} = +\infty$ if there is no transition from $\bx_i$ to $\bx_j$.
\item The sets of \emph{in-neighbors} and \emph{out-neighbors} of a node $\bx_j$ are respectively defined by
\[
N^-_j \ = \ N^-(\bx_j) \ \triangleq \ \cbraces{\bx_i \mid C_{ij} < +\infty}, \qquad
N^+_j \ = \ N^+(\bx_j) \ \triangleq \ \cbraces{\bx_k \mid C_{jk} < +\infty }.
\]
% COMMENT REMOVED
\item Assume the graph is \emph{sparsely connected}, i.e. $\abs{N^{\pm}(\bx_i)} \leq \kappa \ll M \ \forall \ \bx_i \in X$ for some fixed $\kappa \in \N$.
\item The {\bf goal} is to find the {\em ``value function''} $U:X\to [0,+\infty)$, defined as
\[
U_i \ = \ U(\bx_i) \ \triangleq \ \text{the minimum total time to travel from $\bx_i$ to $\bt=\bx_{M+1}$.}
\]
\end{itemize}

% COMMENT REMOVED

Naturally, $U_{M+1} = 0$.  On the rest of the graph, {\em Bellman's Optimality Principle} \cite{Bellman_Opt} yields
a coupled system of $M$ nonlinear equations:
\begin{equation}\label{Bellmans}
U_i \; = \; \min_{\bx_j \in N^+_i} \left\{ U_j \ + \ C_{ij} \right\}, \qquad \forall \ i =1,2,\ldots, M.
\end{equation}
Once the value function is known, an optimal path from any node to the target $\bx_{M+1}$ can be quickly recovered by recursively transitioning to the minimizing neighbor.
% COMMENT REMOVED
A straight-forward iterative method for solving the system \eqref{Bellmans} would result in $O(M^2)$ computational cost.
Fortunately, this system is {\em monotone causal}: $U_i$ cannot depend on $U_j$ unless $U_i > U_j$.
This observation is the basis of the classical Dijkstra's method, which recovers the value function on the entire graph in  $O(M \log M)$ operations \cite{Dijkstra}.
% COMMENT REMOVED
% COMMENT REMOVED
% COMMENT REMOVED
% COMMENT REMOVED
In Dijkstra's method, all nodes are split into three classes: \FAR{} (no value yet assigned), \CON{} (assigned a tentative value), or \ACC{} (assigned a permanent value).
% COMMENT REMOVED
\begin{center}
\IncMargin{1em}
\begin{algorithm}[H]
% COMMENT REMOVED
\SetKwData{Left}{left}
\SetKwData{Up}{up}
\SetKwFunction{FindCompress}{FindCompress}
\SetKwInOut{Initialization}{Initialization}
\SetKwInOut{Algorithm}{Algorithm}
\Indm
\Initialization{}
\Indp
% COMMENT REMOVED
$U_i \gets +\infty$ and mark $\bx_i$ as \FAR{} for $i=1,2,\ldots,M$ \\
$U(\bt) \gets 0$ and mark $\bt$ as \ACC. \\
For all $\bx_i \in N^-(\bt)$, mark as \CON{} and $U_i \gets C(\bx_i,\bt)$
\BlankLine
\Indm
\Algorithm{}
\Indp
\While{$\exists$ a \CON{} node}{
Find the \CON{} node $\bx_j$ with minimal $U$-value and mark as \ACC{} \label{DijkstrasAlg:accept node} \\
\For{$\bx_i \in N^-_j$ such that $U_i > U_j$ and $\bx_i$ is \FAR{} or \CON{}}{
$\tilde{U} \gets U_j + C_{ij}$ \\
\If{$\tilde{U} < U_i$}{
$U_i \gets \tilde{U}$ \\
Mark $\bx_i$ as \CON{} \label{dijkstra:considered mark}
} % if end
} % for end
} % while end
{\color{DarkBlue}\caption{\em Dijkstra's Algorithm}}
% COMMENT REMOVED
\end{algorithm}
\DecMargin{1em}
\end{center}
% COMMENT REMOVED
Efficient implementations usually maintain the \CON{} nodes as a binary heap, resulting in the $\log M$ term in the computational complexity.
% COMMENT REMOVED

\subsection{Estimates for ``single-source / single-target" problems.}
If we are only interested in an optimal path from a single starting location $\bs \in X$, Dijkstra's method can be terminated as soon $\bs$ becomes \ACC.  (This changes the stopping criterion on line 4 of the pseudocode.)
Other modifications of the algorithm can be introduced to further reduce the computational cost on this narrower problem.  Consider a function
\[
V_i \ = \ V(\bx_i) \ \triangleq \ \text{minimum total time to travel from $\bs$ to $\bx_i$.}
\]
Any node $\bx_i$ lying on an optimal path from $\bs$ to $\bt$ must satisfy $U_i + V_i = U(\bs) = V(\bt).$  This provides an obvious relevance criterion, since for any $\bx_i$ that is not on an optimal path, $U_i + V_i > U(\bs)$. But since $V$ is generally unknown, all techniques for focusing computations on a neighborhood of this optimal path must instead rely on some ``heuristic underestimate''
% COMMENT REMOVED
\begin{equation}\label{discrete under}
\varphi_i \ = \ \varphi(\bx_i) \ \leq \ V_i.
\end{equation}
A stronger \textsl{``consistency''} requirement is often imposed instead:
\begin{equation}\label{monotone condition}
\varphi_j \ \leq \ C_{ij} + \varphi_i; \qquad \Forall i,j.
\end{equation}
(Note that $\varphi\equiv V$ is the maximum among all consistent heuristics that also satisfy $\varphi(\bs)=0$.)
\fillmeup
Such consistent underestimates are readily available for geometrically embedded graphs.
Suppose $X \subset \R^n$ and $d_{ij} = \norm{\bx_i - \bx_j}_2$.  If the ``maximum speed" $F_2>0$ is such that
$C_{ij} \geq d_{ij} / F_2$ for all $i$ and $j$, then $\varphi_i \ = \norm{\bx_i - \bs}_2 / F_2 \ \leq \ V_i.$
On a Cartesian grid-type graph, the Manhattan distance provides a better (tighter) underestimate
$\varphi_i \ = \norm{\bx_i - \bs}_1 / F_2.$
For more general embedded graphs, a much better underestimate $\varphi$ can be produced by
{\em ``Landmark sampling''} \cite{Goldberg_Landmarks}, but this requires additional precomputation and
increases the memory footprint of the algorithm.
\fillmeup
Some algorithms for this problem also rely on ``heuristic overestimates"
\[
\psi_i \ = \ \psi(\bx_i) \ \geq \ V_i.
\]
An overestimate can be obtained as a total cost of {\em any} path from $\bx_i$ to $\bs$ or can also be found using landmark precomputations \cite{Goldberg_Landmarks}. For structured geometrically embedded graphs, an analytic expression might also be available. E.g., on a Cartesian grid, if the ``minimum speed" $F_1>0$ is such that $C_{ij} \leq d_{ij} / F_1$, we can use
$\psi_i \ = \norm{\bx_i - \bs}_1 / F_1.$

\subsection{A* Restriction Techniques}
A* techniques restrict computations to potentially relevant nodes %to the $\bs-\bt$ optimal path
by limiting the number of nodes that become \CON.
A more accurate $\varphi$ restricts a larger number of nodes from becoming \CON,
and if $\varphi=V$ then only those nodes actually on the $\bs\to\bt$ optimal path are ever \ACC.
% COMMENT REMOVED
% COMMENT REMOVED
\fillmeup
{\bf Standard A* (SA*).} This version of A* is the one most often described in the literature \cite{Hart_Astar}.
Unlike in Dijkstra's algorithm, the \CON{} nodes are sorted and \ACC{} based on $(U_i + \varphi_i)$ values.
This change affects line \ref{DijkstrasAlg:accept node} in our pseudocode.
% COMMENT REMOVED
The resulting algorithm typically \ACs{} far fewer nodes before terminating: irrelevant nodes with large $\varphi$ values might still become \CON{} (if their neighbors are accepted) but will have lower priority and most of them will never become \ACC{} themselves.  Moreover, the consistency of $\varphi$ ensures that \ACC{} nodes receive exactly the same values as would have been produced by the original Dijkstra's method.  If $\bx_i$ actually depends on $\bx_j \in N^+_i$, then
\[
U_i \ = \ C_{ij} \ + \ U_j \qquad \Longrightarrow \qquad
U_i \ \geq \ \parens{\varphi_j - \varphi_i} \ + \ U_j \qquad \iff \qquad
U_i \ + \ \varphi_i \ \geq \ U_j \ + \ \varphi_j,
\]
guaranteeing that under SA* $\bx_i$ will not be \ACC{} before $\bx_j$. %has received its correct value.
\fillmeup
% COMMENT REMOVED
{\bf Alternative A* (AA*).} A less common variant of A* is described in \cite{Bertsekas_DPbook}.
Instead of favoring nodes with small $\varphi$, AA* simply ignores nodes that are clearly irrelevant. AA* relies on an underestimate $\varphi$ (no longer required to satisfy \eqref{monotone condition}) and an additional upper bound $\Upper \geq U(\bs).$  (If an analytic or precomputed $\psi$ is available, we can take $\Upper = \psi(\bt)$.
But it is also possible to use the total cost of any feasible path from $\bs$ to $\bt$.)
\fillmeup
During Dijkstra's algorithm, a node $\bx_i$ with $U_i + \varphi_i > \Upper$ (hence $U_i + V_i > \Upper$) is surely not a part of the optimal path. Thus, to speed up Dijkstra's algorithm, in AA* we still sort \CON{} nodes based on $U$ values, but on line \ref{dijkstra:considered mark} we only mark $\bx_i$ \CON{} if $U_i + \varphi_i \leq \Upper$.
Since the order of acceptance is the same, it is clear that AA* produces the same values as Dijkstra's, but the efficiency of this technique is clearly influenced by the quality of $\Upper$ (the smaller it is, the smaller is the number of \CON{} nodes). This reliance on $\Upper$ is a downside (since SA* only needs $\varphi$), but has the advantage of making AA* also applicable to the label-correcting methods \cite{Bertsekas_DPbook}.  In section \S\ref{s:continuous} we argue that AA* is also more suitable for continuous optimal control problems, in which an $\Upper$ is often readily available.
\fillmeup
{\bf AA* with Branch \& Bound (B\&B).} In AA* $\Upper$ remains static throughout the algorithm. The idea of {\em Branch \& Bound} (B\&B) is to dynamically decrease $\Upper$ as we gain more information about the graph, making use of an overestimate function $\psi.$
When \ACCing{} a node $\bx_i$, we can also set
\[
\Upper \ \ \gets \ \ \min\cbraces{\Upper, \ U_i + \psi_i}.
\]
% COMMENT REMOVED
% COMMENT REMOVED
{\bf Exact estimates.} Using ``exact estimates" with A* would result in the maximal domain restriction.
% COMMENT REMOVED
For both A* techniques, if $\varphi \equiv V$ and $\Upper \equiv U(\bs)$, the algorithm would only \AC{} the nodes lying on an optimal path.
% COMMENT REMOVED

% COMMENT REMOVED
\section{Domain restriction in a continuous setting}
\label{s:continuous}
The continuous time-optimal isotropic control problem deals with
% COMMENT REMOVED
minimizing the time-to-exit
% COMMENT REMOVED
for a vehicle, whose dynamics is governed by
\begin{equation}\label{dynamics}
\left\{
\begin{array}{rcl}
\dot{\by}(t) & = & f\parens{\by(t)} \ba(t), %\quad \mbox{ where } f:\bar{\Omega} \to \braces{F_1,F_2}
\\
\by(0) & = & \bx \in \Omega \subset \R^n,
\end{array}
\right.
\end{equation}
Here $\bx$ is the starting position, $\ba(t) \in S^{n-1}$ is the control (i.e., the direction of motion) chosen at the time $t$, $\by(t)$ is the vehicle's time-dependent position, and $f$ is the spatially-dependent speed of motion.
We will further assume the existence of two constants $F_1$ and $F_2$ such that
$0 < F_1 \leq f(\bx) \leq F_2$ holds$\Forall \bx \in \bar{\Omega}.$
% COMMENT REMOVED
% COMMENT REMOVED
% COMMENT REMOVED
% COMMENT REMOVED
For every time-dependent control $\ba(\cdot)$ we define the total time to the exit set $\exitset\subseteq\partial\Omega$ as
$T_{\mathbf{x}, \mathbf{a}} = \min\cbraces{t \geq 0 \mid \by(t) \in \exitset}.$
The {\em value function} $u:\Omega \to [0,+\infty)$ is then naturally defined as
\[
u(\bx) \ = \ \inf\limits_{\ba(\cdot)} \cbraces{T_{\mathbf{x}, \mathbf{a}} \
+ \ q\braces{\by\parens{T_{\mathbf{x}, \mathbf{a}}}}},
\]
where $q:\exitset\to [0,+\infty)$ is the exit-time penalty.
% COMMENT REMOVED
Bellman's optimality principle can be used to show that, if $u$ is a smooth function,
it must satisfy a static Hamilton-Jacobi-Bellman PDE
\[
\min\limits_{\ba \in A}\cbraces{\parens{\nabla u(\bx) \cdot \ba}f(\bx) \ + \ 1} \ \ = \ \ 0,
\]
with the natural boundary condition $u=q$ on $\exitset$.
Using the isotropic nature of the dynamics, it is clear that the minimizer (i.e., the optimal initial direction of motion starting from $\bx$) is $\ba_* = - \nabla u(\bx) / \|\nabla u(\bx)\|$ and the equation is equivalent to the Eikonal PDE \eqref{Eikonal intro}.  If the value function $u$ is not smooth, it can still be interpreted as a unique
{\em viscosity solution} of this PDE \cite{Crandall_ViscOriginal}.
\fillmeup
% COMMENT REMOVED
\alexEdit{Solving this PDE to recover the value function is the key idea of the {\em dynamic programming}.}
An analytic solution is usually unavailable, so numerical methods are needed to approximate $u$.
We use a first-order upwind discretization, whose monotonicity and consistency yield convergence to the viscosity solution \cite{BarlesSouganidis}.  To simplify the notation, we describe everything on a cartesian grid in $\R^2$, though higher dimensional generalizations are straightforward and similar discretizations are also available on simplicial meshes (e.g., \cite{KimmelSethian_tri, SethVlad_trimesh}; see also Figure \ref{fig:stencils}).  We will assume
\begin{itemize}[leftmargin=6mm]\itemsep0.1pt
\item $\bar{\Omega} = [0,1]\times [0,1]$ is discretized on a $m\times m$ uniform grid $X$ with spacing $h = 1/(m-1)$.
% COMMENT REMOVED
% COMMENT REMOVED
\item A \emph{gridpiont} or \emph{node} is denoted by $\bx_{ij}$ with corresponding \emph{value} $U_{ij} = U(\bx_{ij}) \approx u(\bx_{ij})$, \emph{speed} $f_{ij} = f(\bx_{ij})$, and \emph{neighbors}
\[
N_{ij} \ = \ N(\bx_{ij}) \ \triangleq \ \cbraces{\bx_{i-1, j}, \, \bx_{i+1,j}, \, \bx_{i,j- 1}, \, \bx_{i,j+1}}.
\]
This notation will be slightly abused (e.g., a gridpoint $\bx_i$ with a corresponding value $U_i$, speed $f_i$, etc...)
whenever we emphasize the ordering of gridpoints rather than their geometric position.
\item We will assume that $\exitset\subseteq\partial\Omega$ is {\em well-discretized} on the grid, and we define the \emph{discretized exit-set} $Q = \exitset \bigcap X.$
% COMMENT REMOVED
In particular, the focus of our computational experiments will be on the case $\exitset = Q = \{\bt\}$ with the exit time-penalty $q(\bt)=0$.
We note that $\bt$ does not have to be on the boundary of the square $\bar{\Omega}$:
if $\bt \in (0,1)^2$, then
$\Omega = (0,1)^2 \backslash \{ \bt \}$, and the border of the square is treated as an essentially outflow boundary.
This corresponds to solving a $\bar{\Omega}-${\em constrained} optimal control problem, with $u$ interpreted as
a constrained viscosity solution \cite{Bardi_ViscBook}.
\end{itemize}
We use the upwind finite differences \cite{RouyTour} to approximate the derivatives of \eqref{Eikonal intro}, resulting in a system of discretized equations. Using the standard four-point nearest-neighbors stencil at each $\bx_{ij} \in X$,
this results in:
% COMMENT REMOVED
\begin{equation}\label{Eikonal approx}
\parens{\max\cbraces{D^{-x}U_{ij}, -D^{+x}U_{ij}, 0}}^2 \ + \ \parens{\max\cbraces{D^{-y}U_{ij}, -D^{+y}U_{ij},0}}^2 \ \ = \ \ \frac{1}{f_{ij}^2},
\end{equation}
\vspace{-0.2cm}
\[
\mbox{where} \quad u_x(x_i,y_j) \approx D^{\pm x}U_{ij} = \frac{U_{i\pm 1,j} - U_{ij}}{\pm h}, \quad \mbox{and} \quad u_y(x_i,y_j) \approx D^{\pm y} U_{ij} = \frac{U_{i,j\pm1} - U_{ij}}{\pm h}.
\]
If all the neighboring values are known, this is really a ``quadratic equation in disguise" for $U_{ij}$.
Letting $U^{}_H = \min\cbraces{U^{}_{i-1,j},U^{}_{i+1,j}}$ and $U^{}_V = \min\cbraces{U^{}_{i,j-1},U^{}_{i,j+1}}$ reduces \eqref{Eikonal approx} to
\begin{equation}\label{Eikonal discrete formula}
\parens{U_{ij} - U^{}_H}^2 \ + \ \parens{U _{ij}- U^{}_V}^2 \ \ = \ \ \frac{h^2}{f_{ij}^2},
\end{equation}
provided the solution %$U_{ij}$ 
satisfies $U_{ij} \geq \max\cbraces{U^{}_H,U^{}_V}$; otherwise we perform a \textsl{one-sided update}:
\begin{equation}\label{Eikonal one-sided}
U_{ij} \ \ = \ \ \min\cbraces{U^{}_H,U^{}_V} \ + \ \frac{h}{f_{ij}}.
\end{equation}
The system of discretized equations (\eqref{Eikonal discrete formula} and \eqref{Eikonal one-sided} for all $(i,j)$)
are {\em monotone causal} since $U_{ij}$ needs %(at most two of)
only its {\em smaller} neighboring values to produce an update.
\fillmeup
Sethian's \textbf{Fast Marching Method (FMM)} \cite{Seth_FMM} and another Dijkstra-like algorithm \cite{Tsitsiklis} due to Tsitsiklis take advantage of this monotone causality. FMM can be obtained from Dijkstra's Method %(Algorithm 1)
by changing the lines 3 and 7 to instead use the continuous update procedure (equations \eqref{Eikonal discrete formula} and \eqref{Eikonal one-sided}).
\alexEdit{Similarly to Dijkstra's method, FMM computes the value function on the entire grid in $O(M \log M)$ operations, where $M = m^2$ is the number of gridpoints.  The key question is whether a significant reduction of computational cost is possible
if we are only interested in an optimal trajectory starting from a single (pre-specified) source gridpoint $\bs$.
}%
% COMMENT REMOVED
% COMMENT REMOVED
% COMMENT REMOVED
% COMMENT REMOVED
\begin{remark}{\label{rem:PMP}
\alexEdit{%
Restricting FMM to a smaller (relevant) subset of $\Omega$ via A*-techniques is precisely the focus of this paper.
But a legitimate related question is whether the dynamic programming approach is at all necessary when a single trajectory is all that we desire?}
\zachEdit{In contrast to path planning on graphs, in the continuous control community, optimal trajectories for single source problems are typically recovered via \textbf{Pontryagin Maximum Principle (PMP)} \cite{Pontryagin_original}.  This involves solving a two point boundary value problem for a state-costate system of ODEs, which in our context could be also derived as characteristic ODEs of the Eikonal PDE \eqref{Eikonal intro}.
One advantage of using PMP is that, unlike the dynamic programming, it does not suffer from the \emph{curse of dimensionality}. In higher dimensions, solving a two-point boundary value problem is much more efficient than solving a PDE on the whole domain.  Unfortunately, PMP is harder to apply if the speed function $f$ is not smooth.  Even more unpleasantly, depending on the initial guess used to solve the two-point boundary value problem,
that method often converges to {\em locally} optimal trajectories. In contrast, the dynamic programming always yields a globally optimal trajectory, and our approach can be used to lower its computational cost in higher dimensions.  In fact, we show that both techniques can be used together, with a prior use of PMP improving the efficiency of A*, and A* verifying the global optimality of a PMP-produced trajectory.}
}
\end{remark}

\subsection{Domain restriction without heuristic underestimates.}
\label{ss:apriori_restrict}
% COMMENT REMOVED
% COMMENT REMOVED
% COMMENT REMOVED
% COMMENT REMOVED
% COMMENT REMOVED
\alexEdit{Our algorithmic goal is to restrict FMM to a dynamically defined subset of the grid using underestimates of
the cost-to-go and the previously computed values.  This is the essence of several A*-type techniques compared in sections \ref{ss:dynamic_DR}-\ref{ss:AA-FMM}.  But to motivate the discussion, we start by considering several simpler domain restriction techniques that do not involve the run-time use of underestimates.
}
% COMMENT REMOVED
% COMMENT REMOVED
% COMMENT REMOVED
% COMMENT REMOVED
% COMMENT REMOVED
% COMMENT REMOVED
% COMMENT REMOVED
\fillmeup
First, we note that FMM can be terminated immediately after the gridpoint $\bs$ is \ACC.  In practice, this is unlikely to yield  significant computational savings
% COMMENT REMOVED
unless the set
$$
L \ = \ \{ \bx \in \bar{\Omega} \, \mid \, u(\bx) \leq u(\bs) \}
$$
is much smaller than the entire $\bar{\Omega}$.
(E.g., see the bolded level set $\partial L$ in Figure \ref{fig:ellipse}A.)
\fillmeup
Second, it is possible to use a ``bi-directional FMM" (similar to the bi-directional Dijkstra's \cite{bidir dijkstra}) by
expanding two \ACC $\, $ clouds from the source and the target and stopping the process when they meet.
The first gridpoint accepted in both clouds is guaranteed to lie on an $O(h)$-suboptimal trajectory from $\bs$ to $\bt$.
% COMMENT REMOVED
% COMMENT REMOVED
This approach is potentially much more efficient than the above.
E.g., for a constant speed function $f=1$, it cuts the $n$-dimensional volume of the $\ACC$ set by the factor of $2^{n-1}$; see Figure \ref{fig:BiFMM}.
% COMMENT REMOVED
\figstart
\begin{center}
\begin{tikzpicture}[point/.style={draw,shape=circle,inner sep = 1mm,minimum size = .1cm},scale = .5,->,>=stealth',shorten >=1pt,shorten <=1pt,auto,node distance = 1.5cm,semithick]
	\begin{scope}
	\draw[fill=LightLightBlue] (0,0) circle (5cm);
	
    	\node[point,fill=gray!20] (T) at (0,0) {$\bt$};
    	\node[point,fill=gray!20] (S) at (5,0) {$\bs$};
	
	\draw[orange,->, ultra thick] (S) -- (T);
	\end{scope}
	
	\begin{scope}[shift={(13.5,0)}]
	\draw[fill=LightLightBlue] (0,0) circle (5cm);
	\draw[fill=green!50, fill opacity = .5] (0,0) circle (2.5cm);
	\draw[fill=green!50, fill opacity = .5] (5,0) circle (2.5cm);
	
    	\node[point,fill=gray!20] (T) at (0,0) {$\bt$};
    	\node[point,fill=gray!20] (S) at (5,0) {$\bs$};
    	\node[point,fill=gray!20] (M) at (2.5,0) {$\bx$};
	%\node[draw, circle,inner sep=2pt,fill=red] (M) at (2.5,0) {};
	
	\draw[orange,->, ultra thick] (M) -- (T);
	\draw[orange,->, ultra thick] (M) -- (S);
	\end{scope}
\end{tikzpicture}
\caption{\footnotesize FMM expands computations outwards from $\bt$, shown by the large circle. The two smaller circles each expand from $\bt$ and $\bs$ and represent the computations performed during BiFMM. In this simple situation BiFMM considers 50\% of the domain that FMM considers. To recover the global optimal trajectory from $\bs$ to $\bt$ using BiFMM one must recover the optimal trajectories from $\bx$ to $\bt$ and $\bx$ to $\bs$ and join them together.}
% COMMENT REMOVED
\label{fig:BiFMM}
\end{center}
\end{figure}

Third, a different ``elliptical restriction" approach is also applicable (and can be combined with the above bidirectional technique) whenever an overestimate for the minimal time from $\bs$ to $\bt$ is available.
\begin{lemma}\label{ellipse lemma} Suppose the exit-set is given by a single target point $\bt$, $\, d = \abs{\bs-\bt},$ and $\Upper$ is a known constant such that $\Upper \geq u(\bs)$.  Then the optimal trajectory $\by(\cdot)$ satisfying \eqref{dynamics} from $\by(0)=\bs$ is contained within the prolate spheroid $E(\bs,\bt)$ (an ellipse in 2D) satisfying
\begin{equation}\label{ellipse info}
\mbox{Foci = } \cbraces{\bs,\bt} \qquad \mbox{and} \qquad \left\{
\renewcommand*{\arraystretch}{2}
\begin{array}{r c l c >{\displaystyle}l}
\mbox{Major semi-axis} & = & a & = & \frac{F_2\Upper}{2}, \\
\mbox{Minor semi-axis} & = & b & = & \frac{1}{2}\sqrt{F_2^2\Upper^2 - d^2}.
% COMMENT REMOVED
% COMMENT REMOVED
\end{array}
\renewcommand*{\arraystretch}{1}
\right.
\end{equation}
% COMMENT REMOVED
\end{lemma}
\begin{proof}
Let $d^*$ and $T^*$ be the distance and time along the optimal trajectory from $\bs$ to $\bt$. Then
\begin{equation}\label{ellipse bound}
\frac{d^*}{F_2} \ \leq \ T^* \ \leq \ \Upper,
\end{equation}
For any $\bx$ along the optimal trajectory we have
\[
\abs{\bx - \bs} + \abs{\bt-\bx} \ \leq \ d^*
\ \leq \ F_2 \Upper.
\]
This inequality defines
% COMMENT REMOVED
a prolate spheroid in $\R^n$ and \eqref{ellipse info} immediately follows.
\end{proof}

Even if we are interested in an unconstrained problem (find the quickest $(\bs,\bt)$ trajectory in $\R^2$),
finite computer memory forces us to solve a {\em state-constrained} problem instead (find the quickest $(\bs,\bt)$ trajectory contained in $\bar{\Omega}$).
The above Lemma is thus also useful to answer a related question: for which starting points $\bs$ does the $\bar{\Omega}$-constrained
% COMMENT REMOVED
problem have the same value function as the unconstrained?
Clearly, for any point $\bs$ such that $E(\bs,\bt) \subset \bar{\Omega}$, enlarging the domain would not decrease $u(\bs)$.
\fillmeup
{\bf Higher dimensional savings.} Restricting computations to $E(\bs,\bt)$
% COMMENT REMOVED
has an increasing effect in higher dimensions.
The fraction $\mathcal{P}$ of the volume of $E(\bs,\bt)$ to the volume of the smallest bounding rectangular box $B$
is given by $\mathcal{P} = \frac{\pi^{n/2}}{2^n \Gamma(n/2 + 1)}$, which quickly approaches zero as $n$ grows.
(E.g., in $\R^2$ this fraction is $(\pi/4) \approx 78.5\%$, while in $\R^6$ it is already $\approx 8\%.$)
If $\bar{\Omega}=B$, the restriction to $E(\bs,\bt)$ yields the computational savings of $(1-\mathcal{P})$;
the savings are even higher if $\bar{\Omega}$ is any other box-rectangular domain fully containing $E(\bs,\bt)$.
\fillmeup
{\bf Formulas for $\Upper$}
can be naturally obtained by computing (or bounding from above) the time along any feasible path from $\bs$ to $\bt.$
On a convex domain $\Omega$,
the most obvious choice  is $\Upper_1 = d / F_1$ (i.e., follow the straight line from $\bs$ to $\bt$ at the minimum speed $F_1$).
For problems with the unit speed of motion, $f(\bx) = 1 = F_1 = F_2, \, \Upper_1 = d$, and the ellipse collapses to a straight line segment.
\fillmeup
A more accurate overestimate can be obtained by computing the exact time needed to traverse that straight line trajectory:
$$
\Upper_2 \ \ = \ \ \bigintsss_0^{\abs{\bt-\bs}} \frac{dr}{f\parens{\bs + \frac{\bt-\bs}{\abs{\bt-\bs}}r}} \ \ = \ \ \bigintsss_0^1 \frac{\abs{\bt-\bs}}{f\parens{\bs + (\bt-\bs)r}}dr \ \ \leq \ \ \Upper_1.
$$
For non-convex domains, a similar upper bound can be obtained by integrating the slowness $1/f$ along any feasible trajectory (e.g.,
the shortest
$\bar{\Omega}$-constrained path
from $\bs$ to $\bt$).
% COMMENT REMOVED
\fillmeup
Finally, we will also consider the third (``ideal") option, with $\Upper_3 = u(\bs) \leq \Upper_2$.
% COMMENT REMOVED
While practically unattainable, $\Upper_3$ is useful to illustrate the upper bound on efficiency of various domain restriction techniques.  In practice, it can be approximated by using $U(\bs)$ precomputed on a coarser grid or using the output of Pontryagin-Maximum-Principle-based computations (see the example in section \ref{ss:observers}).
In the latter case, the techniques discussed in this paper can be viewed as a method for verifying the global optimality of a known locally-optimal trajectory.
Figure \ref{fig:ellipse}A shows the $(\bs, \bt)$-focused ellipses
% COMMENT REMOVED
for a specific example with a highly oscillatory speed function.
\figstart
% COMMENT REMOVED
\begin{center}
% COMMENT REMOVED
% COMMENT REMOVED
\hspace*{-8mm}
\begin{tabular}{c c}
A & B \\
\iftoggle{usecolor}{%
\includegraphics[scale=0.6]{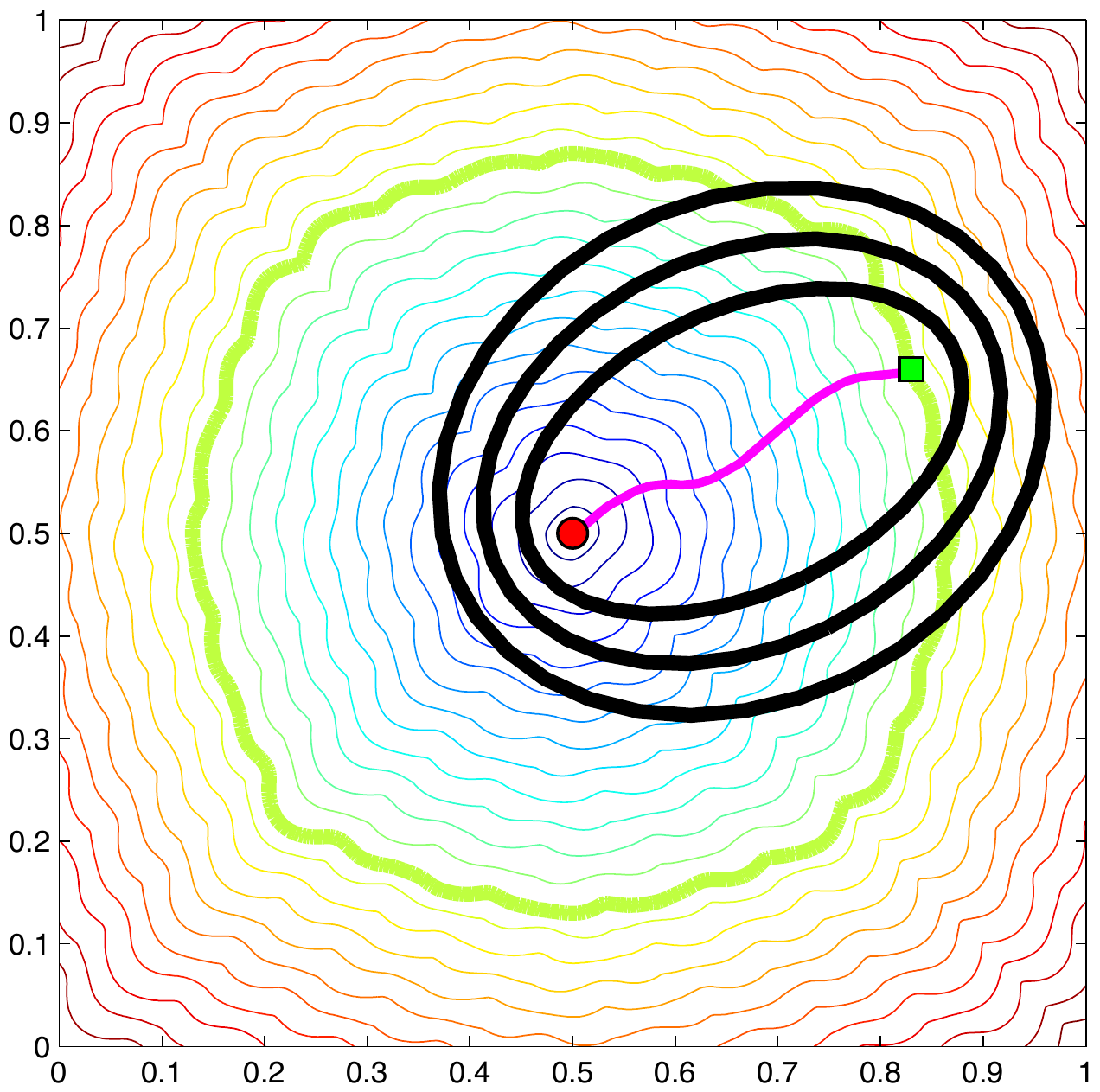} &
\includegraphics[scale=0.6]{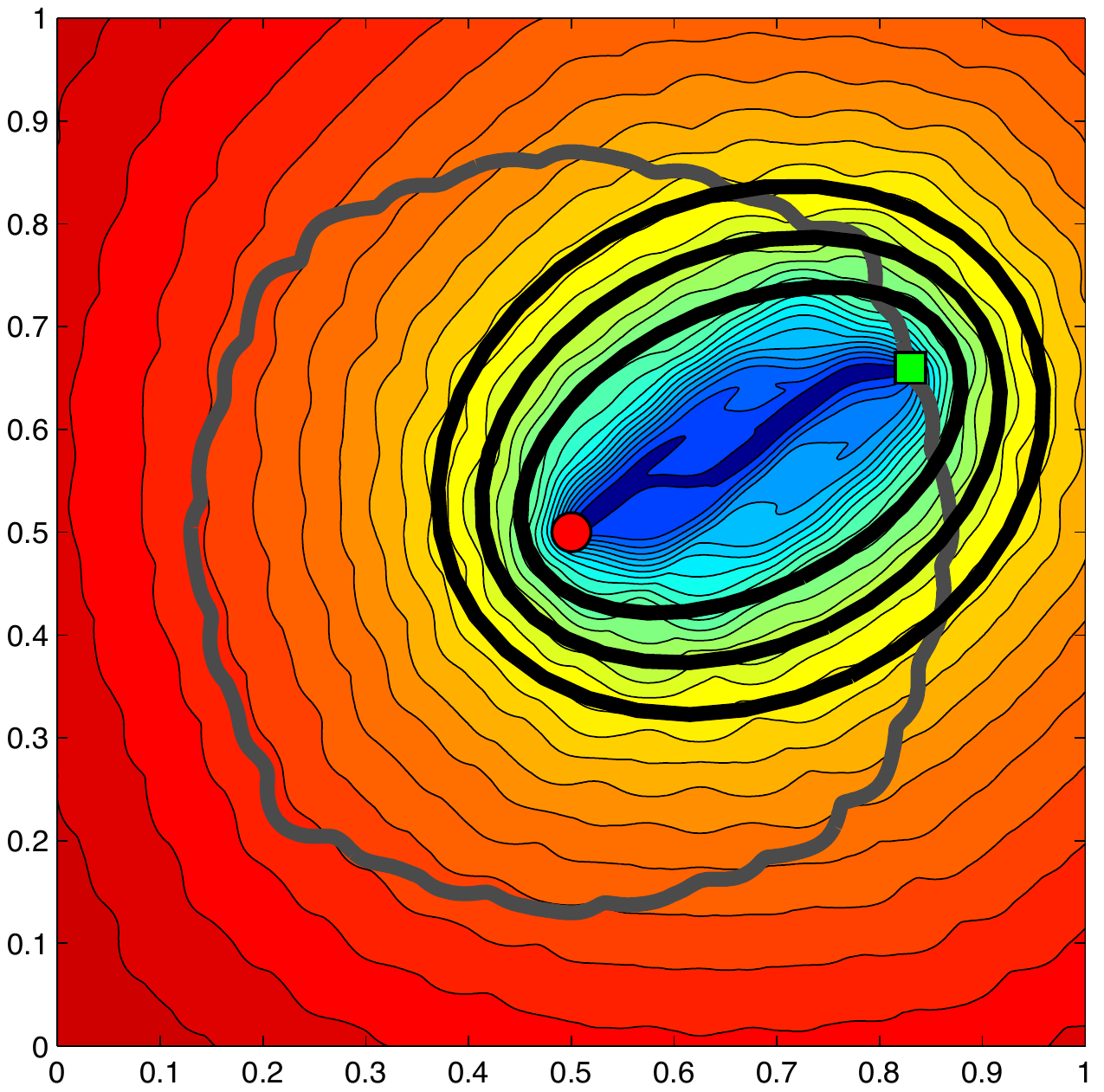}
}{%
\includegraphics[scale=0.6]{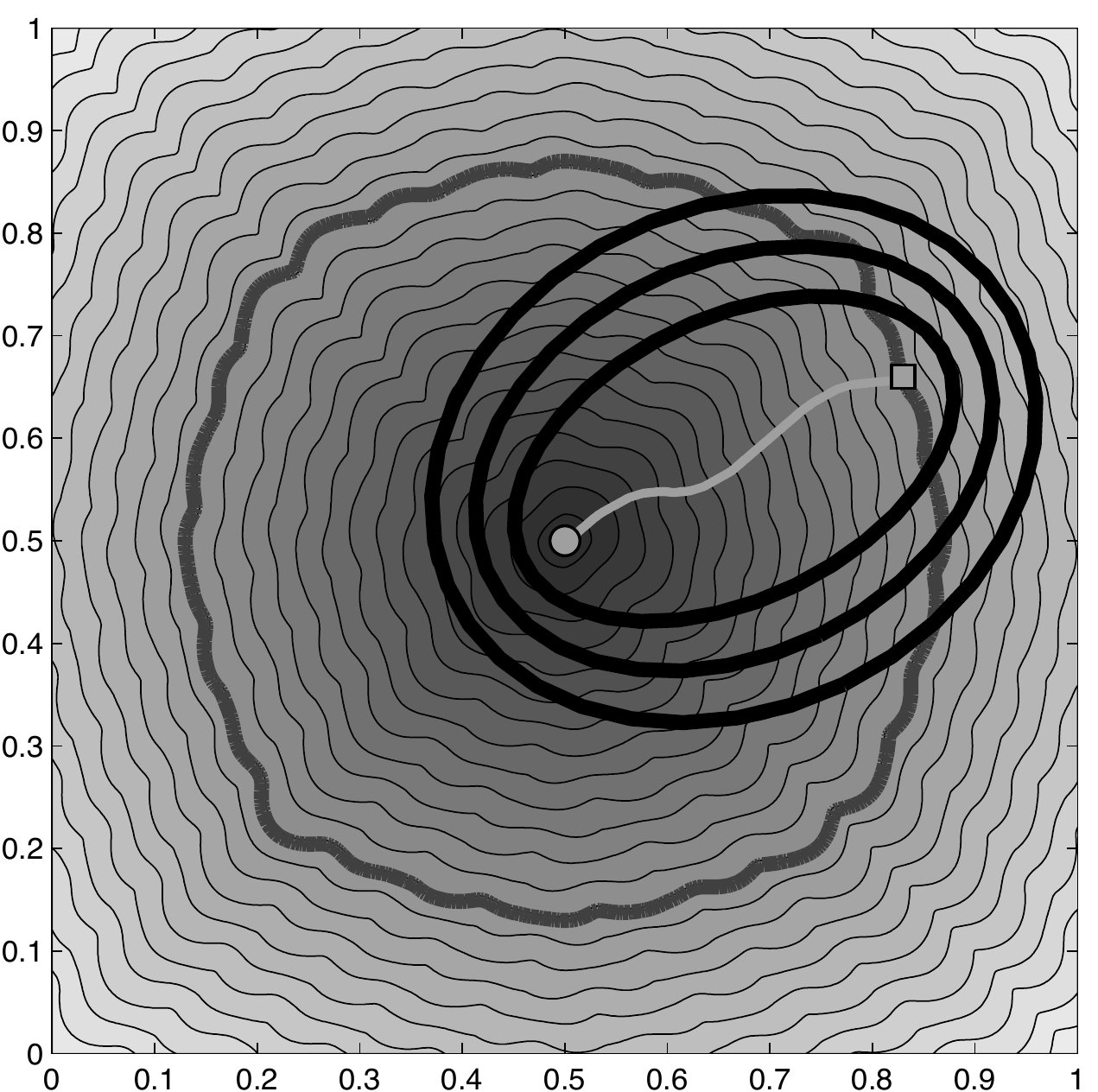} &
\includegraphics[scale=0.6]{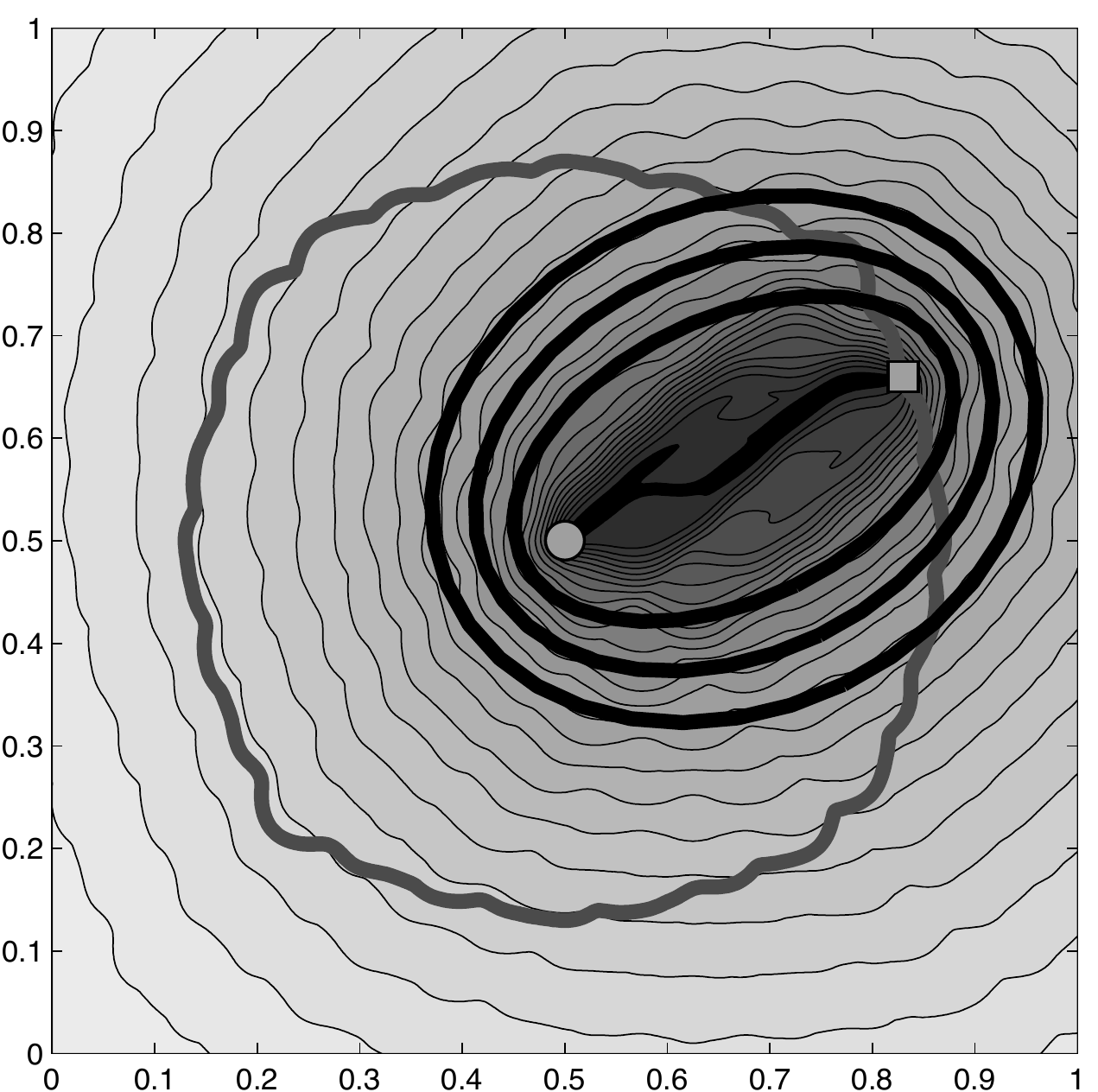}
}
\end{tabular}
\caption{\footnotesize \textbf{A.} Level sets of $U$ computed with a highly oscillatory speed $f(x,y) = 2 + 0.5 \sin(20\pi x) (20\pi y)$. The curve $\partial L$ is indicated by a thicker contour line. Three ellipses corresponding to $\Upper_1, \Upper_2,$ and $\Upper_3$ are shown in black. \textbf{B.} the level sets of
$\log_{10}\left[U(\bx)+V(\bx) - U(\bs) + 0.01\right]$ for the same problem.}
\label{fig:ellipse}
\end{center}
\end{figure}

\subsection{Dynamic domain restriction: underestimates and A*-techniques.}
\label{ss:dynamic_DR}
The previous subsection described a priori domain restriction techniques.  Here, our goal is to further restrict the computations {\em dynamically} by using the solution already computed on parts of $\bar{\Omega}$.
The actual viscosity solution $u(\bs)$ depends only on values along a characteristic (i.e., an optimal $(\bs,\bt)$ trajectory). Ideally, we would like to compute the numerical solution $U$ only for the gridpoints within an immediate neighborhood of that trajectory, potentially yielding a much greater %computational savings
speedup than the techniques described above.
\fillmeup
Consider a function $v(\bx)$ specifying the min-time from $\bs$ to $\bx$.  (It is easy to see that $v$ is also a viscosity solution of the Eikonal PDE \eqref{Eikonal intro}, but with the different boundary condition $v(\bs)=0$.)
We note that
$u(\bx)+v(\bx) \geq u(\bs) = v(\bt)$ for all $\bx \in \Omega$, and this becomes an equality if and only if $\bx$ lies on an optimal $(\bs,\bt)$ trajectory.
(See the level sets of $u+v$ in Figure \ref{fig:ellipse}B.)
% COMMENT REMOVED
% COMMENT REMOVED
Since $v$ is generally unknown, any practical restriction of computational domain will have to rely on an ``admissible underestimate heuristic" $\varphi$, satisfying $\varphi(\bx) \leq v(\bx), \Forall \bx \in \bar{\Omega}.$
As we will see, tighter underestimates result in more efficient domain restrictions.
Here we enumerate several natural heuristic underestimates:
\begin{enumerate}[leftmargin=6mm]\itemsep-2pt

\item \emph{Na\"{i}ve heuristic} is obtained by assuming the maximum speed of travel along the straight line:
\begin{equation}\label{naive under heuristic}
\varphi^0(\bx) \ = \ \abs{\bs-\bx} \ / \ F_2.
\end{equation}
Several papers on SA* versions of FMM
\cite{FergusonStentz, Petres_thesis, Yershov_1, Yershov_2} have relied on its scaled-down version
$\varphi_{\lambda}^0 = \lambda \varphi^0(\bx)$ with $\lambda \in [0,1]$.

\item \emph{Coarse grid heuristic} \cite{Peyre_coarse, Peyre_Geodesic} is based on precomputing $V$
on a coarser $(Rm) \times (Rm)$ grid with $R\in (0,1)$.
% COMMENT REMOVED
If we use $V^R$ to denote the interpolation
of that solution on $X$, the heuristic is then defined as
$\varphi^C_{\lambda, R} = \lambda V^R$, where $\lambda \in [0,1]$ is chosen to ensure that the result is a true underestimate.

\item \emph{Landmarking-based heuristic} \cite{Peyre_landmark, Peyre_Geodesic} is a continuous version of the landmarking technique on graphs \cite{Goldberg_Landmarks}. This relies on pre-computing/storing the minimum time from every node to a number of ``landmarks''; the triangle inequality is then used to obtain the lower bound
    $\varphi^L(\bx) \leq v(\bx)$. The high computational cost and memory footprint make this approach useful for repeated queries only.  (I.e., only if the optimal
    trajectory problem has to be solved for many different $(\bs,\bt)$ pairs.)

\item \emph{Higher-speed heuristic}
can be obtained by starting with a special speed-overestimate $f_0(\bx) \geq f(\bx)$,
such that the corresponding value function $v_0(\bx) \leq v(\bx)$ is known analytically, and then setting
$\varphi(\bx) = v_0(\bx)$. (Note that \eqref{naive under heuristic} can be also derived this way by taking
$f_0(\bx) \equiv F_2$.)
If the $(\bs,\bt)$ path-planning has to be performed for {\em many different} speed functions $f$, the above approach
can be useful even if $v_0$ has to be approximated numerically.  One such example is included in section \ref{ss:observers}.

\item \emph{Scaled ``Oracle'' heuristic} \cite{Peyre_Geodesic} is
defined as $\bar{\varphi}^{}_{\lambda}(\bx) = \lambda v(\bx)$ with $\lambda \in [0,1]$.
This is clearly not a practical underestimate, but a theoretical device useful in studying
the accuracy/efficiency tradeoffs of various domain restriction techniques.
Since $v$ is generally unavailable, our benchmarking relies on a numerical approximation; i.e.,
$\bar{\varphi}^{}_{\lambda}(\bx) = \lambda V(\bx)$,  where $V$ is (pre-)computed on the same grid $X$.

\end{enumerate}

\zachEdit{The first of these (the Na\"{i}ve heuristic) is a conservative underestimate that is cheaply available for all problems -- including the situations with discontinuous speed functions and/or non-convex domains.  The other underestimates are more expensive to produce, but usually result in a more significant domain restriction.  Thus, their use is particularly justified when the same speed function is used repeatedly to solve numerous (single source / single target) problems.}
\fillmeup
\alexEdit{We emphasize that this paper is in a sense ``underestimate-neutral.''
A good underestimate is obviously important, but
our focus is on how it should be used rather than on how to build it.
}
\fillmeup
{\bf Continuous A* Techniques.}
Both SA* and AA* algorithms on graphs may be easily adapted to the continuous setting using any of the above heuristics.
Just like on graphs, SA*-FMM
increases the ``processing-priority" of nodes with low $\varphi$ values,
% COMMENT REMOVED
% COMMENT REMOVED
whereas our AA*-FMM avoids considering nodes guaranteed not to be a part of any $(\bs,\bt)$-optimal trajectory. Each of these methods successfully restricts the computations, but with different trade-offs between the execution time, memory footprint, amount of restriction, and computational error.

\subsection{Prior work on SA*-FMM.}
% COMMENT REMOVED
From the implementation standpoint, SA*-FMM is fairly straightforward.  It requires
modifying a single line \ref{DijkstrasAlg:accept node} of FMM (see Dijkstra's algorithm): \AC{} the node with minimal $U + \varphi$. (Since $u+v$ is minimal along the $(\bs,\bt)$-optimal trajectory,
% COMMENT REMOVED
$U(\bx) + \varphi(\bx)$ is used to indicate how close $\bx$ is to that trajectory.)  However, the analysis of this method's
output is more subtle.
\fillmeup
On graphs, the consistency of the heuristic underestimate (i.e., the condition \eqref{monotone condition}) guarantees that all SA*-accepted nodes receive the same values as would have been produced by Dijkstra's.
% COMMENT REMOVED
% COMMENT REMOVED
In contrast, SA*-FMM exhibits a performance tradeoff based on whether $\varphi$ satisfies a more restrictive and stencil-dependent consistency condition (defined below).
If $\varphi$ is inconsistent, some of the gridpoints may be \ACC{} prematurely, resulting in additional numerical errors.
% COMMENT REMOVED
% COMMENT REMOVED
% COMMENT REMOVED
On the other hand, if $\varphi$ is consistent, the efficiency of the domain restriction is significantly decreased,
and this restricted domain does not shrink to zero volume as $h \to 0.$
\fillmeup
The presence of additional errors might seem counterintuitive.  After all, if an $(\bx_i,\bt)$-optimal trajectory passes through some $\bx_j$, then,
for $\varphi$ defined by formula \eqref{naive under heuristic},
$$
u(\bx_i) \; = \; (\text{Time from $\bx_i$ to $\bx_j$}) \, + \, u(\bx_j)
\; \geq \; \frac{\abs{\bx_j-\bx_i}}{F_2}  \, + \, u(\bx_j)
\; \geq \; \varphi(\bx_j)-\varphi(\bx_i)  \, + \, u(\bx_j),
$$
guaranteeing that $u(\bx_i) + \varphi(\bx_i) \geq u(\bx_j) + \varphi(\bx_j)$.
% COMMENT REMOVED
% COMMENT REMOVED
% COMMENT REMOVED
% COMMENT REMOVED
% COMMENT REMOVED
% COMMENT REMOVED
% COMMENT REMOVED
Turning to numerical solutions,
we would hope for the same argument to work for $U_i$ and $U_j$, and
indeed it does  if $U_i$ is computed by a one-sided update formula \eqref{Eikonal one-sided}.
% COMMENT REMOVED
% COMMENT REMOVED
But for a first-order upwind discretization in $\R^2$,
a generic gridpoint $\bx_i$ depends on $2$ other gridpoints that straddle $\bx_i$'s characteristic.
To produce the same numerical values under SA*-FMM and FMM, we would need to know that
$U_i + \varphi(\bx_i) \geq U_j + \varphi(\bx_j)$ whenever
$\bx_i$ directly depends on $\bx_j$.
\fillmeup
Suppose there exists a constant $\lambda > 0$ such that
\begin{equation}
\label{eq:dependency_slack}
U_i \text{ directly depends on } U_j
\qquad \Longrightarrow \qquad
U_i \; > \; U_j \, + \, \lambda |\bx_i - \bx_j|,  \qquad \qquad \Forall i,j.
\end{equation}
The proper ordering is then guaranteed provided the underestimate $\varphi$ satisfies the
consistency condition
\begin{equation}
\label{eq:contin_consistent}
|\varphi(\bx_i) - \varphi(\bx_j)| \ \leq \ \lambda |\bx_i - \bx_j|, \qquad \qquad \Forall i,j,
\end{equation}
which is easy to ensure by using the underestimate
\[
\varphi^0_{\lambda}(\bx) \ = \ \lambda \varphi^0(\bx).
\]
Unfortunately, the condition \eqref{eq:dependency_slack} is stencil-dependent and
in this section we explore its implications both on grids and triangular meshes; see Figure \ref{fig:stencils}.
% COMMENT REMOVED
\figstart
\begin{center}
\begin{tikzpicture}[inner sep=0pt,thick, dot/.style={fill=black,circle,minimum size=6pt},scale=1.1]

\def\Xshift{3.8};
\def\Ashift{0.4};
\def\Px{1.5};
\def\Py{-0.5};
\def\SQRTTWO{1.414213562373095};
\def\SQRTTHREE{1.732050807568877};
\def\Athick{semithick};
\def\Acolor{blue};
\def\Lthick{ultra thick};
\def\Lcolor{red};
\def\Dthick{very thick};
\def\Dcolor{black!50};

% COMMENT REMOVED
\begin{scope}[shift={(0,0)}]
\node [] at (0, 1.4) {A};
% COMMENT REMOVED
\draw[-, \Athick, \Acolor] (\Ashift, 0) [out=-90, in=0] to (0, -\Ashift);
\draw[dashed, \Dcolor, \Dthick] (1,0) to (0,-1);
\draw[->,\Lthick, \Lcolor, line width= 0.75mm] (0,0) -- (\Px, \Py);
\draw [-,thick,black] (-1,0) -- (1,0);
\draw [-,thick,black] (0,-1) -- (0,1);
\node[dot,black] at (-1,0) {};
\node[dot,black] (N1) at (1,0) {};
\node[dot,black] (N2) at (0,-1) {};
\node[dot,black] at (0,1) {};
\node[dot,black] (O) at (0,0) {};
% COMMENT REMOVED
% COMMENT REMOVED
\node [above=2mm, left=1.35mm] at (O) {$\bx_i$};
\node [above=1.5mm, right=1mm] at (N1) {$\bx_j$};
\node [below=1mm, left=1mm] at (N2) {$\bx_k$};
\node [blue, rotate=20] at (0.7, 0.3) {\footnotesize$\theta = 90^{\circ}$};
\end{scope}

% COMMENT REMOVED
\begin{scope}[shift={(\Xshift,0)}]
\node [] at (0, 1.4) {B};
% COMMENT REMOVED
\draw[-, \Athick, \Acolor] (\Ashift, 0) [out=-90, in=45] to (\Ashift*\SQRTTWO/2, -\Ashift*\SQRTTWO/2);
\draw[dashed, \Dcolor, \Dthick] (1,0) to (1,-1);
\draw[->,\Lthick, \Lcolor, line width= 0.75mm] (0,0) -- (\Px, \Py);
\draw [-,thick,black] (-1,0) -- (1,0);
\draw [-,thick,black] (0,-1) -- (0,1);
\draw [-,thick,black] (-1,-1) -- (1,1);
\draw [-,thick,black] (-1,1) -- (1,-1);
\node[dot,black] at (-1,0) {};
\node[dot,black] (N1) at (1,0) {};
\node[dot,black] at (0,-1) {};
\node[dot,black] at (0,1) {};
\node[dot,black] at (-1,-1) {};
\node[dot,black] at (-1,1) {};
\node[dot,black] (N2) at (1,-1) {};
\node[dot,black] at (1,1) {};
\node[dot,black] (O) at (0,0) {};
% COMMENT REMOVED
% COMMENT REMOVED
\node [blue, rotate=20] at (0.7, 0.3) {\footnotesize$\theta = 45^{\circ}$};
\node [above=1.5mm, left=2.5mm] at (O) {$\bx_i$};
\node [above=1.5mm, right=1mm] at (N1) {$\bx_j$};
\node [below=1mm, left=1mm] at (N2) {$\bx_k$};
\end{scope}

% COMMENT REMOVED
% COMMENT REMOVED
\begin{scope}[shift={(2*\Xshift,0)}]
\node [] at (0, 1.4) {C};
\draw[-, \Athick, \Acolor] (\Ashift, 0) [out=-90, in=30] to (\Ashift*0.5, -\Ashift*\SQRTTHREE/2);
\draw[->,\Lthick, \Lcolor, line width= 0.75mm] (0,0) -- (\Px, \Py);
\draw[dashed, \Dcolor, \Dthick] (1.2,0) to (1.2*0.500000,-1.2*\SQRTTHREE/2);
\begin{scope}[scale=1.2]
\node[dot,black] (O) at (0,0) {};
\draw[-,thick,black] (0,0) -- (1.000000,0.000000);
\node[dot,black] (N1) at (1.000000,0.000000) {};
\draw[-,thick,black] (0,0) -- (0.500000,\SQRTTHREE/2);
\node[dot,black] at (0.500000,\SQRTTHREE/2) {};
\draw[-,thick,black] (0,0) -- (-0.500000,\SQRTTHREE/2);
\node[dot,black] at (-0.500000,\SQRTTHREE/2) {};
\draw[-,thick,black] (0,0) -- (-1.000000,-0.000000);
\node[dot,black] at (-1.000000,-0.000000) {};
\draw[-,thick,black] (0,0) -- (-0.500000,-\SQRTTHREE/2);
\node[dot,black] at (-0.500000,-\SQRTTHREE/2) {};
\draw[-,thick,black] (0,0) -- (0.500000,-\SQRTTHREE/2);
\node[dot,black] (N2) at (0.500000,-\SQRTTHREE/2) {};
\node [above=1.5mm, left=1.5mm] at (O) {$\bx_i$};
\node [above=1.5mm, right=1mm] at (N1) {$\bx_j$};
\node [below=1mm, left=1mm] at (N2) {$\bx_k$};
\node [blue, rotate=20] at (0.6, 0.25) {\footnotesize$\theta = 60^{\circ}$};
\end{scope}
\end{scope}

% COMMENT REMOVED
\begin{scope}[shift={(3*\Xshift, 0)}]
\node [] at (0, 1.4) {D};
\draw[-, \Athick, \Acolor] (\Ashift*0.923076923076923, \Ashift*0.384615384615385) [out=-67.380135, in=27.45707] to (\Ashift*0.461083967620704, -\Ashift*0.887356509416114);
\draw[->,\Lthick, \Lcolor, line width= 0.75mm] (0,0) -- (\Px, \Py);
\draw[dashed, \Dcolor, \Dthick] (1.2*0.6,1.2*0.25) to (1.2*0.4500000,-1.2*\SQRTTHREE/2);
\begin{scope}[scale=1.2]
\node[dot,black] (O) at (0,0) {};
\draw[-,thick,black] (0,0) -- (0.600000,0.25000); % east point
\node[dot,black] (N1) at (0.600000,0.250000) {};
\draw[-,thick,black] (0,0) -- (-0.1500000,0.45);
\node[dot,black] at (-0.150000,0.45) {};
\draw[-,thick,black] (0,0) -- (-0.850000,0.100000);
\node[dot,black] at (-0.850000,0.100000) {};
\draw[-,thick,black] (0,0) -- (-0.200000,-0.4);
\node[dot,black] at (-0.20000,-0.4) {};
\draw[-,thick,black] (0,0) -- (0.4500000,-\SQRTTHREE/2);
\node[dot,black] (N2) at (0.4500000,-\SQRTTHREE/2) {};
\node [above=2mm, left=1mm] at (O) {$\bx_i$};
\node [above=1.5mm, right=1mm] at (N1) {$\bx_j$};
\node [below=1mm, left=1mm] at (N2) {$\bx_k$};
\node [blue, rotate=23] at (0.45, 0.55) {\footnotesize$\theta \approx 85^{\circ}$};
\end{scope}
\end{scope}
% COMMENT REMOVED
\end{tikzpicture}
\end{center}
\caption{\footnotesize Four computational stencils for a node $\bx_{i}$: four-point and eight-point stencils
on a Cartesian grid (A and B), a six-point stencil on a regular triangular mesh (C),
and a five-point stencil on a unstructured triangular mesh (D).}
\label{fig:stencils}
\end{figure}
Suppose that $\bx_i$'s characteristic is straddled by $\bx_j$ and $\bx_k$,
where $\theta$ is the angle $\angle \bx_j \bx_i \bx_k$ and $\gamma$ is the angle between
the characteristic and $\bx_i \bx_j$.
Since for the Eikonal equation the characteristics coincide with gradient lines
and our numerical approximation is piecewise linear,
it is easy to show that\footnote{
We note that this observation was previously used in \cite{Vlad_MSSP}
to find the conditions for applicability of Dial-like algorithms.}
$$
(U_i-U_j) \ = \ \cos\gamma \, |\bx_i - \bx_j| \, / \, f(\bx_i) \ \geq \ \cos(\theta) h  \, / \, f(\bx_i).$$
The latter lower bound is actually sharp when the characteristic is parallel to $\bx_i \bx_k$
and $h = |\bx_i - \bx_j|$.
% COMMENT REMOVED
This means that
\begin{itemize}[leftmargin=6mm]\itemsep-5pt
\item $\varphi^0_0\equiv0$ is the only consistent underestimate for
% COMMENT REMOVED
% COMMENT REMOVED
stencil %in Figure
\ref{fig:stencils}A
(i.e., $\lambda=\frac{1}{F_2} \cos\frac{\pi}{2}=0$).  Thus, SA*-FMM will usually result in additional errors.
\item
% COMMENT REMOVED
for stencil
% COMMENT REMOVED
\ref{fig:stencils}B,
$\varphi^0_{\lambda}$ becomes consistent for
 $\lambda \leq \frac{1}{F_2 \sqrt{2}} = \frac{\cos\frac{\pi}{4}}{F_2}$.
 \item for a local stencil used on a general triangular mesh (e.g., Figure \ref{fig:stencils}D),
 if $\bar{\theta} < \frac{\pi}{2}$ is an upper bound on
 angles $\theta$ present in the mesh, then $\varphi^0_{\lambda}$ becomes consistent for
 $\lambda \leq \cos(\bar{\theta}) / F_2$.
 \end{itemize}
Interestingly, the importance of consistency conditions for SA*-FMM was only recently recognized in \cite{Yershov_1, Yershov_2},
while all the prior versions treated this in an ad-hoc fashion.  To summarize:
\begin{itemize}[leftmargin=6mm]\itemsep-1pt
% COMMENT REMOVED
\item $\mbox{[2005]}$ Ferguson \& Stentz \cite{FergusonStentz} adapt D* algorithms to continuous optimal trajectory problems discretized on stencil \ref{fig:stencils}B.
% COMMENT REMOVED
    They also introduce an SA*-type technique within D* to further improve the performance.
    The method relies on $\varphi_{\lambda}^0$ to ensure the right order of gridpoint processing, but the choice of $\lambda$ is never explained explicitly.
% COMMENT REMOVED
\item $\mbox{[2005, 2006, 2008]}$ Peyr\'{e} \& Cohen \cite{Peyre_coarse, Peyre_landmark, Peyre_Geodesic} adapt SA* for FMM on stencil \ref{fig:stencils}A using underestimates $\varphi^C_{\lambda, R}$ and $\varphi^L$.  The authors acknowledge that their version of SA*-FMM produces additional errors
and experimentally study the dependence of these errors on
% COMMENT REMOVED
the tightness of underestimates.
However, they do not analyze the behavior of errors under grid refinement.
% COMMENT REMOVED
\item $\mbox{[2007]}$ P\^{e}tr\`{e}s \cite{Petres_thesis} defines an SA*-FMM variant
% COMMENT REMOVED
on a stencil \ref{fig:stencils}A
% COMMENT REMOVED
with $\varphi^0_{\lambda}$. A brief description of a bi-directional version of SA*-FMM is also included.
P\^{e}tr\`{e}s acknowledges that, for large $\lambda$, the additional (SA*-induced) errors can be larger than discretization errors, but does not analyze how that ratio changes under grid refinement.
\item $\mbox{[2011, 2012]}$ Yershov \& LaValle \cite{Yershov_1, Yershov_2} use FMM with acute triangular meshes  as in \cite{SethVlad_trimesh} in $\R^2$, $\R^3$, and on two dimensional manifolds. Their problems of interest use $f \equiv 1$ on a domain with obstacles.
% COMMENT REMOVED
The authors use SA*-FMM with $\varphi^0_{\lambda}$ and prove that $\lambda = \cos(\bar{\theta})$ % / F_2$
guarantees absence of additional errors.
The authors state that in their experiments SA*-FMM processed only 50\% of the gridpoints processed by FMM.
\end{itemize}

\begin{remark}
\label{rem:store_phi}
Every implementation of SA*-FMM also involves an ``efficiency versus memory footprint'' tradeoff.
Since the binary heap of \CON{} nodes is sorted based on $U+\varphi$, every heap-maintenance operation
relies on availability of $\varphi(\bx)$ for many nodes on the heap.  This happens whenever a \FAR{} node becomes \CON,
or a \CON{} node receives a smaller value or becomes \ACC.  If $\varphi$ is re-computed each time it is needed (e.g., by \eqref{naive under heuristic}), this introduces a noticeable overhead to each heap operation.  An alternative (to cache $\varphi$ the first time it is computed for each \CON{} node) is certainly more efficient, but significantly increases the memory footprint, particularly on larger grids and in higher-dimensional problems.  In Section \ref{s:examples}
we include the performance data for both of these approaches.
\end{remark}

\subsection{Accuracy or efficiency?}
\label{ss:accuracy_efficiency}
The errors introduced by any A*-type restriction techniques
% COMMENT REMOVED
are not very surprising
once we recall that the numerical viscosity of the discretization results in a large
domain of computational dependency for $U(\bs).$  To formalize this argument,
we will consider a {\em dependency digraph} $G$ built on the nodes of $X$.
For $\bx_i$ and $\bx_j \in N_i$, $G$ includes an arc $(\bx_i, \bx_j)$ if $\bx_i$ {\em directly depends}
on $\bx_j$; i.e., if $U_j$ is needed to compute $U_i$.
% COMMENT REMOVED
We will say that $\bx_i$ {\em depends} on $\bx_j$ if there exists a path in $G$ from
$\bx_i$ to $\bx_j$.
Due to the monotone causality of the upwinding discretization, this dependence implies $U_i > U_j$;
thus, $G$ is acyclic and every path on it leads to $\bt$.
We will also use $G(\bs)$ to denote the subset of $G$ reachable from $\bs$.
\fillmeup
Consider {\em any} domain restriction technique that results in accepting only nodes from some $\hat{X} \subset X$
and produces some numerical approximation of the value function $U^*(\bx)$ for each $\bx \in \hat{X}$.
If $G(\bs) \not \subset \hat{X}$, we cannot expect $U^*(\bs)$ to be the same as $U(\bs)$
produced by FMM on the full $X$.
In other words, if we insist on avoiding {\em any} additional (restriction-induced) errors,
this typically results in severe constraints on the efficiency of the domain restriction.
\vspace*{-3mm}
\input{./tikz/dependency_graph.tex}
To illustrate this point, we consider a very simple problem with $\bt$ and $\bs$ in opposite corners of
$\bar{\Omega}$; see Figure \ref{fig:dgraph}A.  With $f \equiv 1$ the optimal trajectory from every starting position
$\bx$ is just a straight line to $\bt$.  But it is easy to see that $G(\bs)$ includes all nodes in $X$; thus, {\bf any} restriction
will result in $U^*(\bs) > U(\bs)$.  We emphasize that this phenomenon has nothing to do with the
non-existence of consistent $\varphi$ for the 4-point stencil discretization on a cartesian grid.
Figure \ref{fig:dgraph}C shows an equivalent example on a regular triangular mesh.
As explained in \cite{Yershov_1, Yershov_2}, taking
$\lambda = 1/2$ will ensure that $\varphi^0_{\lambda}$ is consistent for this problem and stencil.  As a result,
SA*-FMM will produce $U^*(\bs) = U(\bs)$, but at the cost of accepting exactly the same set of nodes\footnote{
\zachEdit{
The computational savings of $50\%$ were reported in \cite{Yershov_1, Yershov_2} for $f\equiv1$ on the domain with obstacles.  Based on the above discussion, such savings are in fact highly dependent on the size of $G(\bs)$ relative to the total number of meshpoints.  This percentage is, in turn, defined by the type of the mesh and the positions of $\bs$ and $\bt$ relative to the obstacles.
}
% COMMENT REMOVED
% COMMENT REMOVED
}
as FMM
(i.e., $\hat{X}=X$).
% COMMENT REMOVED
\fillmeup
For these reasons, we believe that asking for $U^*(\bs) = U(\bs)$ on every fixed grid is unrealistic
and makes the domain restriction much less efficient.  A more attractive strategy is to
ensure that $|U^*(\bs)-U(\bs)|$ is small relative to discretization errors and $U^*(\bs) \to u(\bs)$
as $h \to 0.$    This can be ensured provided $\hat{X}$ covers a neighborhood of
the $(\bs,\bt)$-optimal trajectory {\bf and}  $U^* = \hat{U}$, the solution that FMM would have produced on $\hat{X}$.
This is precisely what AA*-FMM does when used with an inconsistent $\varphi$;
on the other hand, the different order of acceptance under SA*-FMM typically results in $U^* \neq \hat{U}$
and a lack of convergence (or a very slow convergence  -- see Section \ref{ss:example_constant}) under grid refinement.

\subsection{
% COMMENT REMOVED
\zachEdit{The new method:}
AA*-FMM}
\label{ss:AA-FMM}

The AA* technique is also quite easy to use in the continuous setting as a modification of the standard FMM.
Our current implementation is based on the upwind discretization \eqref{Eikonal approx} on a standard 4-point stencil,
but the required FMM-changes would be the same for any other monotone-causal stencil (either on a grid or on a simplicial mesh).
% COMMENT REMOVED
Similarly to a version of AA* for graphs:
\begin{itemize}[leftmargin=6mm]\itemsep-2pt
\item we rely on an overestimate $\Upper$ of the time along the $(\bs,\bt)$-optimal trajectory (see section \ref{ss:apriori_restrict});
\item
\zachEdit{the \CON{} nodes are still sorted by $U$-values, and thus}
the underestimate $\varphi$ does not have to satisfy any consistency conditions;
\item we only mark a node $\bx$ as \CON{} if it satisfies the \textsl{``A* condition''}
\begin{equation}\label{cont A* condition}
U(\bx) \ + \ \varphi(\bx) \ \ \leq \ \ \Upper.
\end{equation}
\end{itemize}
% COMMENT REMOVED
% COMMENT REMOVED
% COMMENT REMOVED
% COMMENT REMOVED
% COMMENT REMOVED
% COMMENT REMOVED
% COMMENT REMOVED
% COMMENT REMOVED
% COMMENT REMOVED
% COMMENT REMOVED
% COMMENT REMOVED
% COMMENT REMOVED
% COMMENT REMOVED
This simple criterion allows for AA* to be adapted to {\em both} label-setting and label-correcting methods.
If $\varphi$ satisfies the consistency condition \eqref{eq:contin_consistent}, the values produced by AA*-FMM are also the same as those resulting from FMM, but on a smaller (\ACC) subset of the grid $\hat{X}$.  However, AA*-FMM can be also used even if $\varphi$ does not satisfy \eqref{eq:contin_consistent}, which results in additional errors but does not prevent the convergence to viscosity solution of the PDE under grid refinement.
\fillmeup
To illustrate the efficiency of the AA*-type domain restrictions, we consider the boundaries of 3 sets:
\begin{equation}\label{C sets}
\begin{array}{rcl}
C_1 & = & \cbraces{\bx \, \mid \, \abs{\bx - \bs} + \abs{\bx - \bt} \leq F_2 \Upper}, \\
C_2 & = & \cbraces{\bx \, \mid \, u(\bx) + \varphi(\bx) \leq \Upper} \quad \subseteq \ C_1, \\
C_3 & = & \cbraces{\bx \, \mid \, u(\bx) + v(\bx) \leq \Upper} \quad \subseteq \ C_2.
\end{array}
\end{equation}
All three are shown in Figure \ref{fig:UplusV} for the example introduced in section \ref{ss:apriori_restrict}.
Both $u$ and $v$ are numerically approximated by FMM on the entire domain $\bar{\Omega}$.
The boundaries $\partial C_i$ are shown by bold lines for $\Upper_1, \Upper_2,$ and $\Upper_3$.
The set $C_1$ corresponds to the ellipse defined for each specific $\Upper$.
The set $C_2 \cap L$ is roughly the set accepted by AA*-FMM with the specified $\Upper$ {\bf and}
the underestimate $\varphi^0$.
The set $C_3 \cap L$ is the minimum part of the domain that AA*-FMM would have to accept
with that $\Upper$ even if we were to use the perfect $\varphi = \bar{\varphi}_1.$
If $\Upper=\Upper_3$, then $C_3$ collapses to the optimal trajectory.\fillmeup
% COMMENT REMOVED
\figstart
\begin{center}
% COMMENT REMOVED
\begin{adjustwidth}{-1.55cm}{}
\setlength{\tabcolsep}{1pt}
\begin{tabular}{c c c}
A. {\em Na\"{i}ve $\Upper_1 = \abs{\bs-\bt}/F_1$} &
B. {\em Integral-based $\Upper_2$} &
C. {\em Exact $\Upper_3 = U(\bs)$} \\
% COMMENT REMOVED
\iftoggle{usecolor}{%
\includegraphics[scale=\UplusVFigScale]{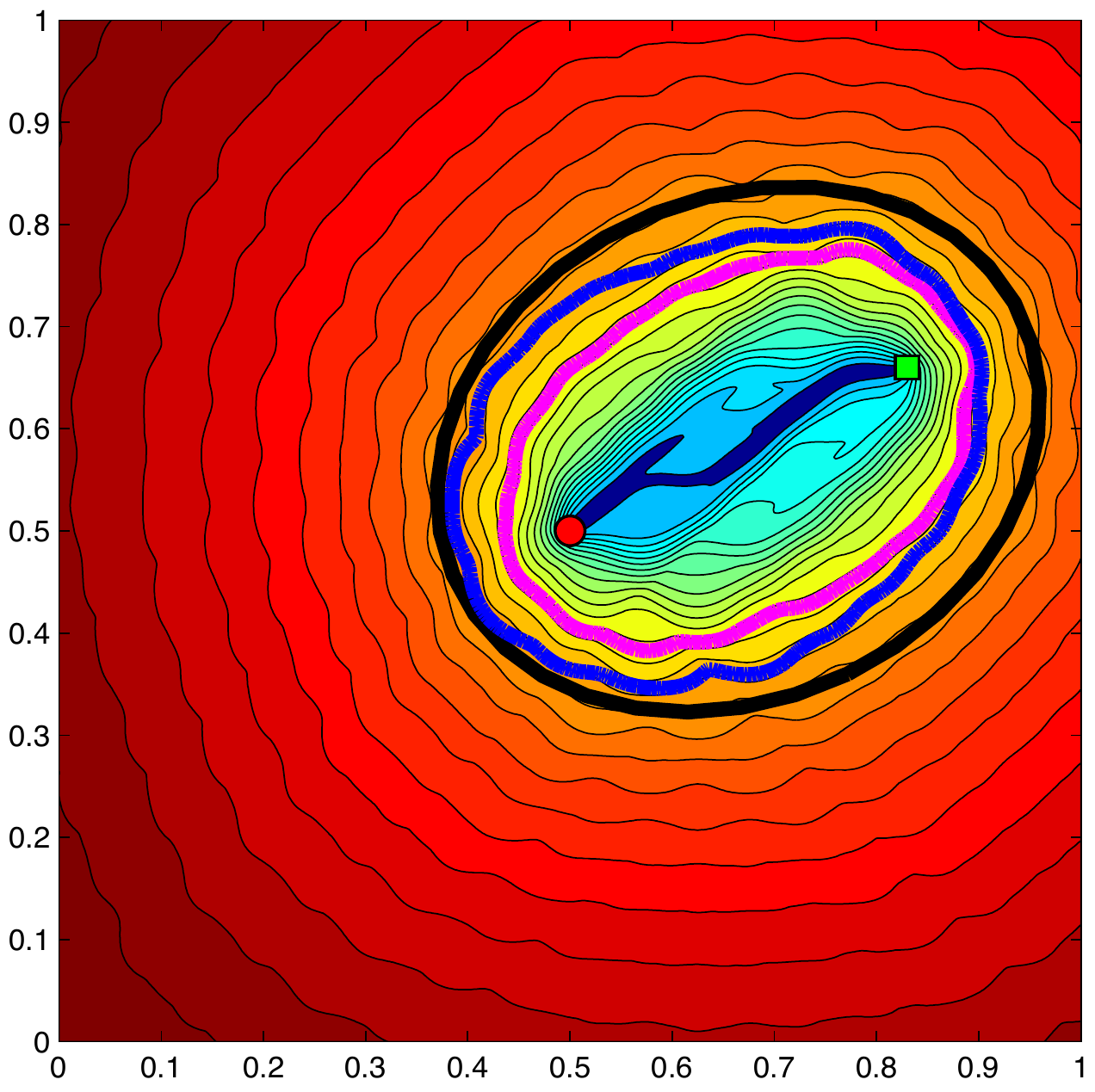} &
\includegraphics[scale=\UplusVFigScale]{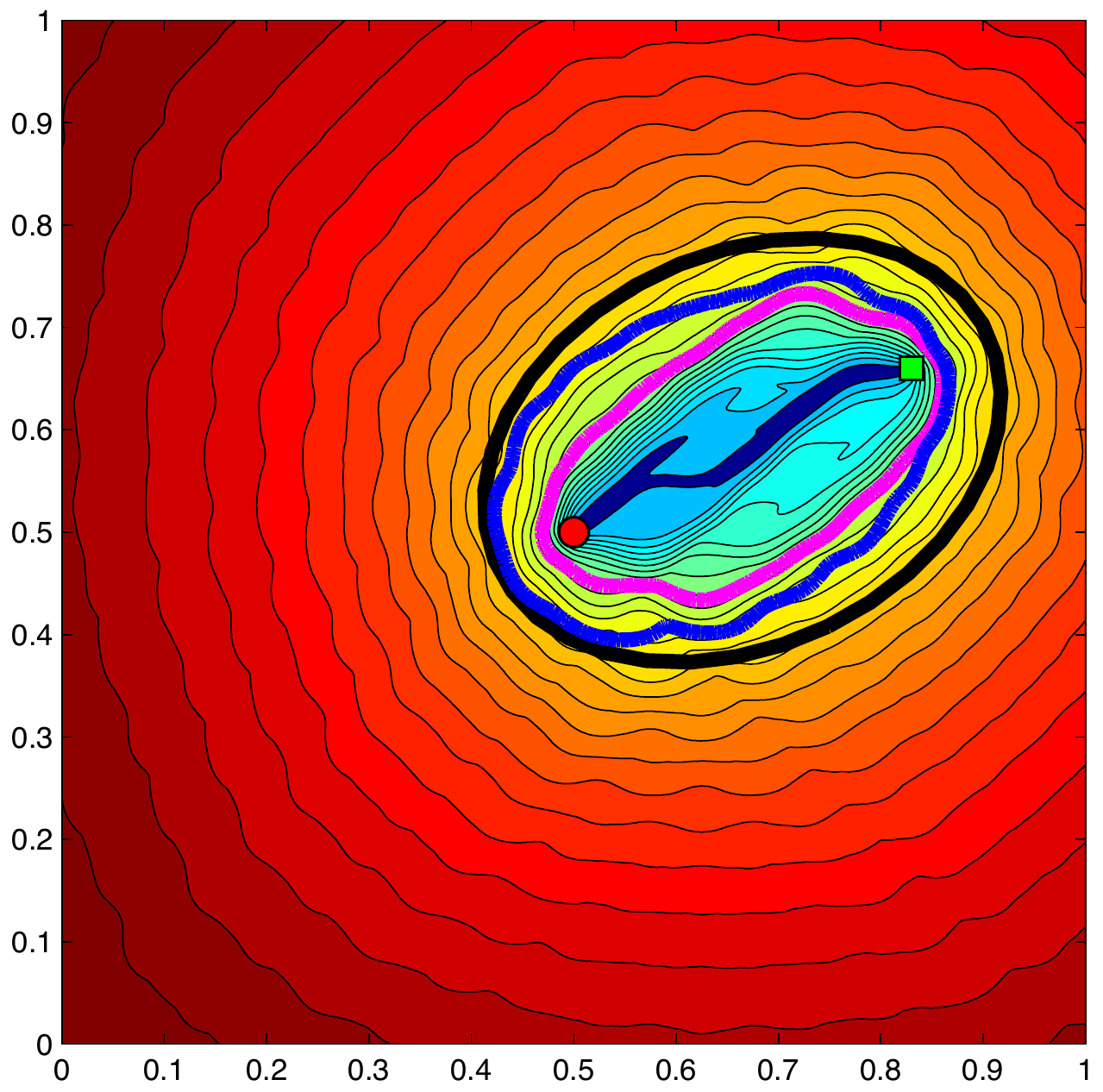} &
\includegraphics[scale=\UplusVFigScale]{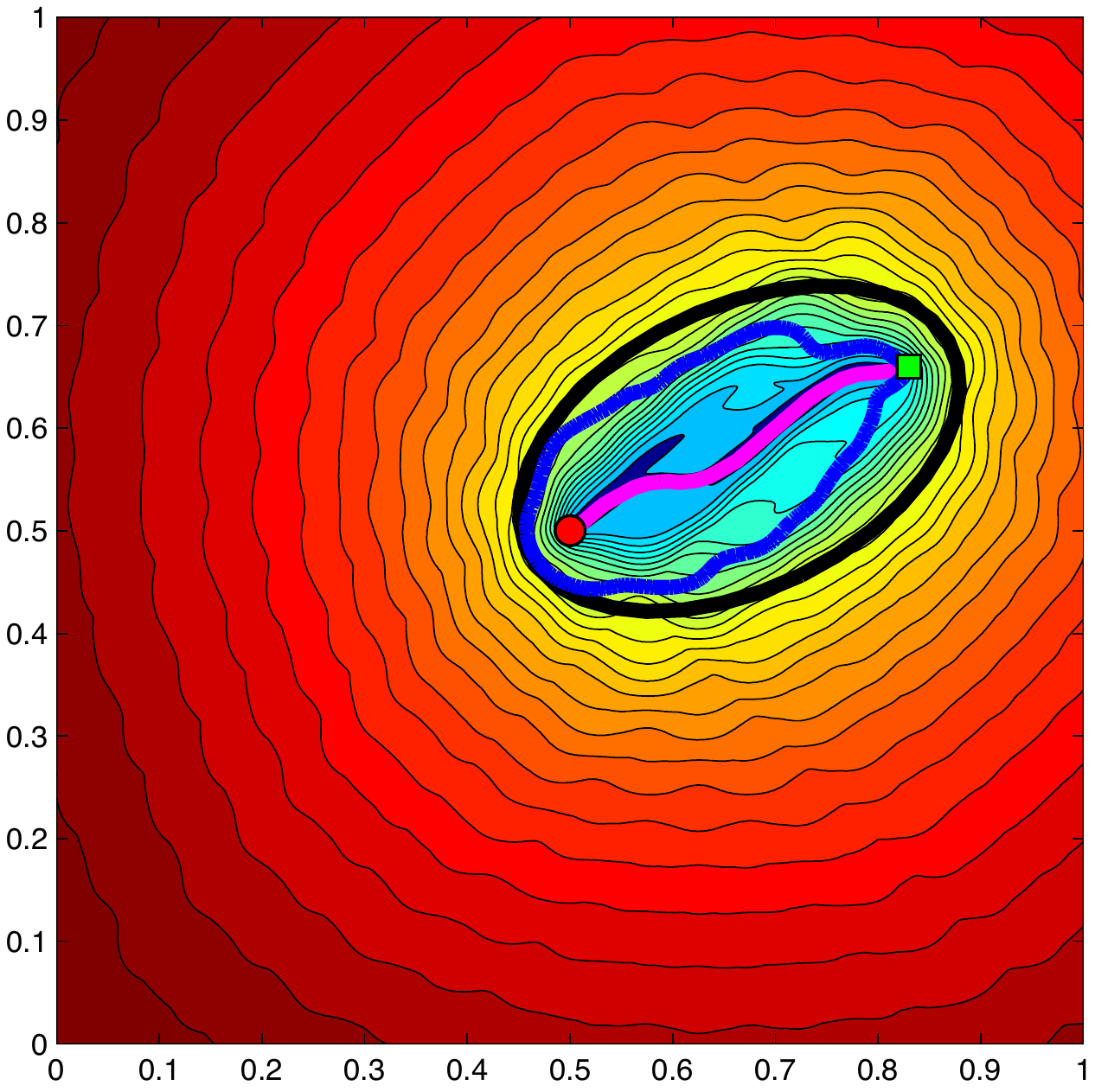}
}{%
\includegraphics[scale=\UplusVFigScale]{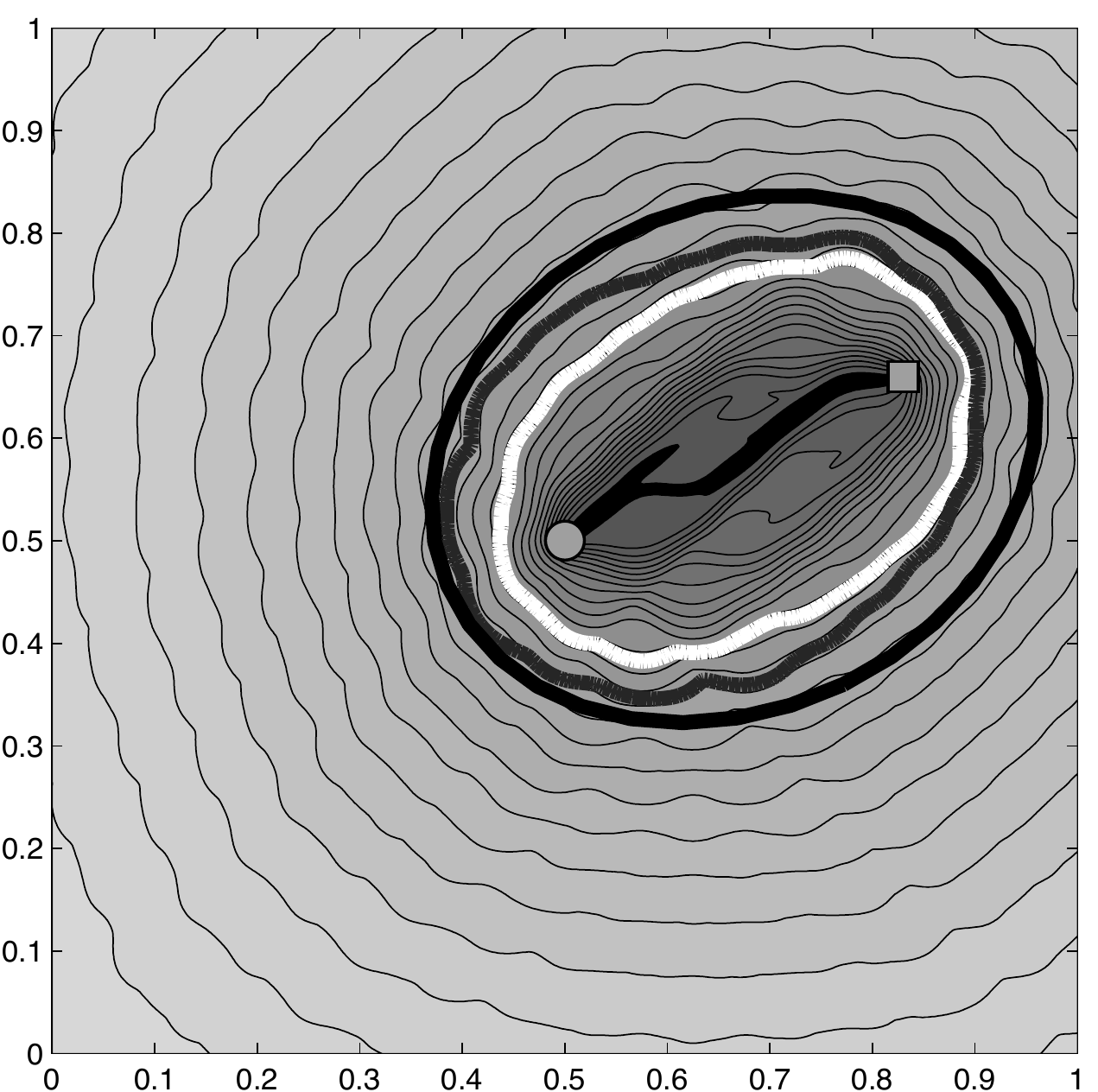} &
\includegraphics[scale=\UplusVFigScale]{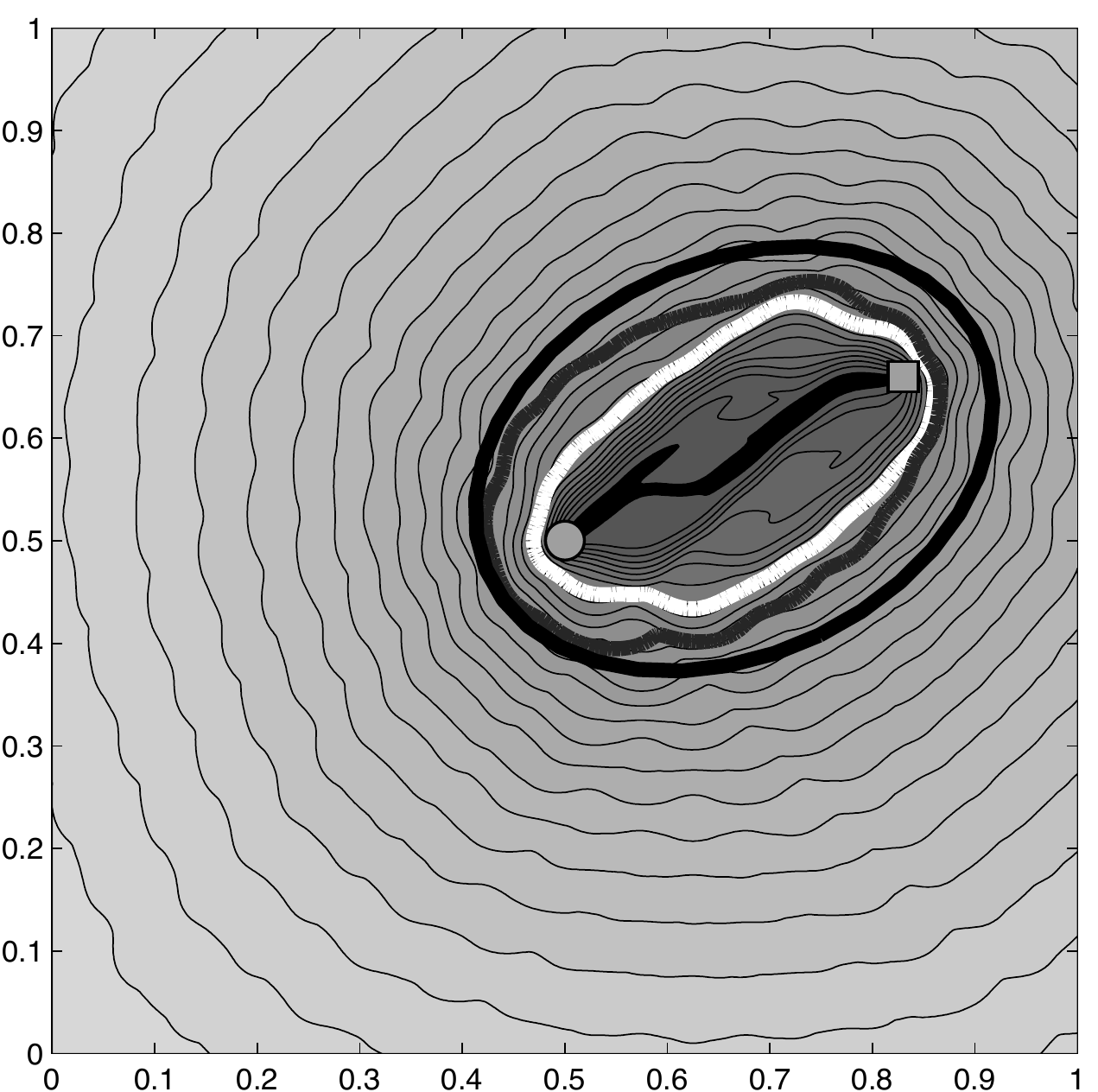} &
\includegraphics[scale=\UplusVFigScale]{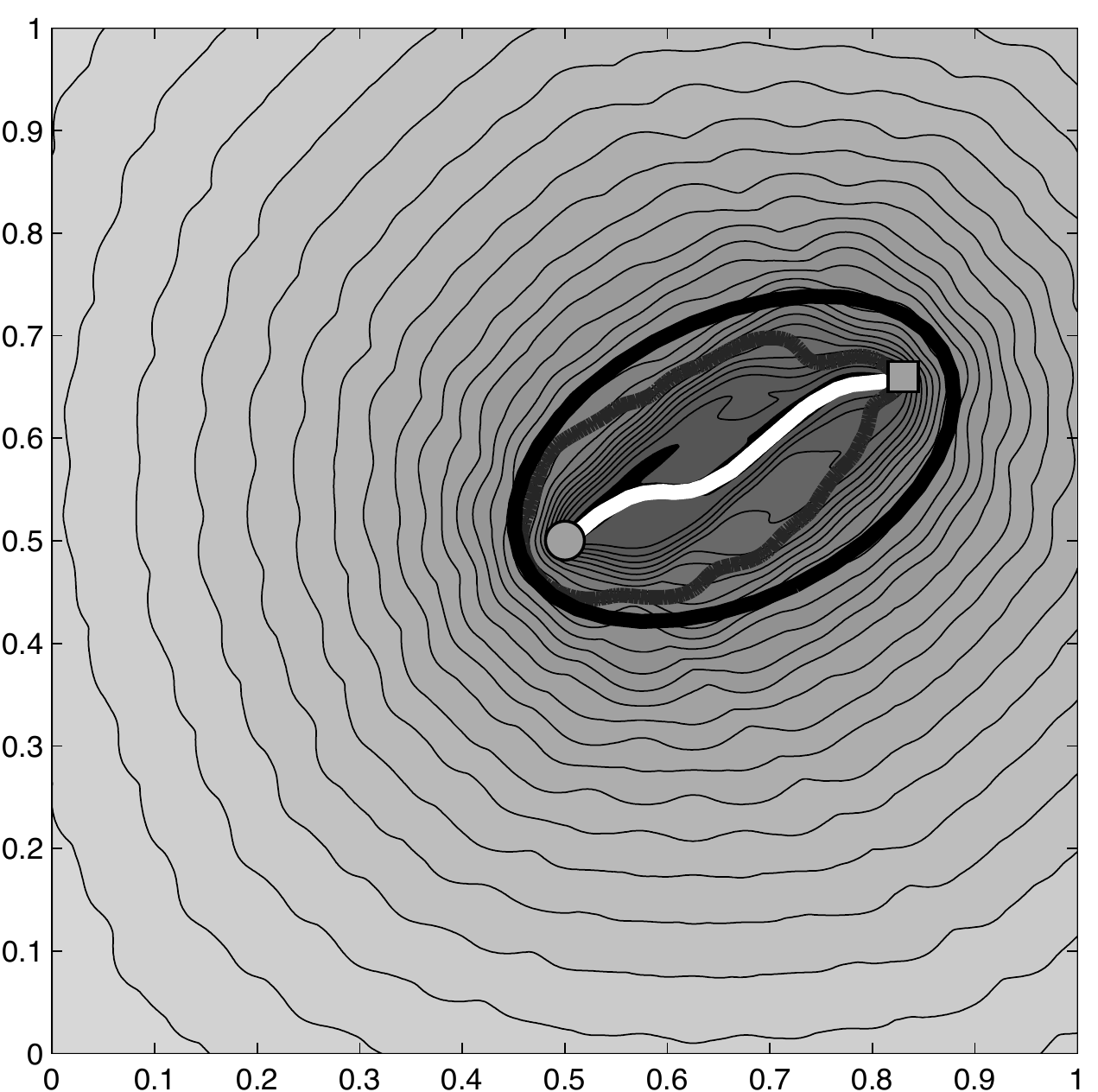}
}
\end{tabular}
\end{adjustwidth}
% COMMENT REMOVED
\caption{\footnotesize Level sets for $u+v$ computed by FMM on a $401^2$ grid. In each subfigure,
the bold lines show the boundaries of $C_1$, $C_2$, and $C_3$ (from out-to-in) for the specified $\Upper$.
}
\label{fig:UplusV}
\end{center}
\end{figure}
This Figure also clearly demonstrates the importance of an accurate $\Upper$ for the efficiency of the domain restriction in AA*-FMM. If the initial $\Upper$ is not particularly tight, the performance can be further improved by decreasing $\Upper$ {\em dynamically} in a Branch\&Bound fashion.  This approach relies on availability of a \textsl{``heuristic overestimate''}
$\psi(\bx) \geq v(\bx), \Forall \bx \in \bar{\Omega}.$  For a convex domain, our implementation uses an obvious (and cheaply computed) overestimate
\begin{equation}\label{overestimate function}
\psi(\bx) \ = \ \abs{\bx-\bs} \, / \, F_1,
\end{equation}
which is consistent with our definition of $\Upper_1$ in \S\ref{ss:apriori_restrict}.
Each time a gridpoint $\bx$ is \ACC{}, this AA*BB-FMM algorithm attempts to decrease $\Upper$ as follows:
\begin{equation}\label{cont_update_upper}
\Upper \ \ \gets \ \ \min\cbraces{\Upper, \ U(\bx) + \psi(\bx)}.
\end{equation}
A better $\psi$ can be obtained by numerically integrating the slowness along any feasible $(\bs,\bt)$ trajectory,
or even using PMP-based techniques.
Performing such computations for every \ACC{} gridpoint would be clearly prohibitive, but using it every so often
(in addition to the systematic use of formula \eqref{overestimate function}) could be a useful technique to investigate in the future.

\begin{remark}
\label{rem:terminate_before_S}
On graphs, using the exact ``underestimate'' $\varphi = V$ simply resulted in accepting only those nodes that lie on the optimal path. In the continuous case, the optimal trajectory does not pass through every node it directly depends on. Even for a node $\bx$ immediately next to the optimal $(\bs,\bt)$ trajectory, if the underestimate $\varphi$ is very accurate, this may cause the A* condition \eqref{cont A* condition} to fail (resulting in $\bx$ never becoming \CON). This situation rarely arises in practice -- e.g., with $\varphi=\varphi^0$ this can happen only if $\Upper$ is exact and the speed $f(\bx)=F_2$ on some neighborhood of $\bs$.\\
We have used two different approaches to address this issue:
\begin{itemize}[leftmargin=6mm]\itemsep-3pt
\item Introduce a numerical tolerance factor; i.e., use $ (1+ \epsilon_{tol}h^{\mu}) \Upper$ instead of $\Upper$. Our analysis of restriction-caused errors in Section \ref{s:it_works} applies as long as $\epsilon_{tol} > 0$ and $\mu \in [0, 1/2)$.
All the numerical tests in Section \ref{s:examples} rely on this approach and confirm the convergence
even with $\mu = 1/2$.
\item Alternatively, if $\bs$ has not been accepted by the end of AA*-FMM, one can simply take $U(\bs) = \Upper$. Since $\Upper$ was obtained as a cost of some known $(\bs,\bt)$ trajectory, that trajectory is then declared optimal (at least for the current grid resolution).
\end{itemize}
\end{remark}

% COMMENT REMOVED
% COMMENT REMOVED

% COMMENT REMOVED
\section{Numerical results}
\label{s:examples}
% COMMENT REMOVED
% COMMENT REMOVED
% COMMENT REMOVED
% COMMENT REMOVED
% COMMENT REMOVED

All algorithms were implemented in C++  and compiled with {\tt g++} version 4.2 on a Macbook Pro %early-2011 model
(4 GB RAM and an Intel Core i7 processor -- four 2 GHz cores). To make the benchmarking results as compiler/platform-independent as possible, we have turned off all
compiler optimizations (option \texttt{-O0}).
For all of the 2D and 3D examples, $\bar{\Omega} = [0,1]^n$ is discretized by a uniform cartesian grid with $m^n$ gridpoints.
% COMMENT REMOVED
To test the numerical approximation errors in distance computations (Section \ref{ss:example_constant}), we have used an analytical solution $u(\bx) = \abs{\bx - \bt}$. In all other cases, the `ground truth' $u$ was computed numerically by FMM on the full domain using the `highly' refined grid:
% COMMENT REMOVED
\begin{center}
\tabcolsep=0.52cm
\begin{tabular}{| c | c | c |}
\hline
\bf Dimension &
\bf Ground truth &
\bf Resolutions considered \\
\hline
$ n = 2 $ &
$m = 6401$ &
$m = 101, \ 201, \ 401, \ 801, \ 1601 \mbox{ and } 3201$ \\
\hline
$ n = 3 $ &
$m = 401$ &
$m = 26, \ 51, \ 101, \ 201 \ \mbox{ (and 401 when $f\equiv 1$)}$ \\
\hline
\end{tabular}
\end{center}
% COMMENT REMOVED
\textbf{Accuracy metrics.} Since we are interested in single-source / single-target problems, all accuracy metrics are based on comparing various numerical approximations and the true solution at a single point $\bs$.
As before, we use $U$ to denote the solution produced by FMM on the entire $X$ while $U^*$ denotes the solutions produced by the respective
% COMMENT REMOVED
A*-modifications of FMM.  We base our comparison on the following ``relative errors'' for each example:
\[
\begin{array}{r c l c l}
\discErr & = &
\mbox{relative discretization error (DE) at $\bs$ using FMM} & = & \abs{U(\bs) - u(\bs)} \ / \ u(\bs) \\
\astarErr & = &
\mbox{relative error at $\bs$ when using A*} & = & \abs{U^*(\bs) - u(\bs)} \ / \ u(\bs) \\
\astarErrOnly & = &
\mbox{relative error at $\bs$ explicitly due to A*} & = & \braces{U^*(\bs) - U(\bs)} \ / \ U(\bs) \ \geq \ 0.
\end{array}
\]
Since the upwind discretization is convergent, $U \to u$ and thus $\discErr \to 0$ as $h\to 0$.
Correspondingly, a successful domain restriction should have $\astarErr \to 0$ (and thus $\astarErrOnly \to 0$) as $h\to 0$.
\fillmeup
To measure the efficiency of the domain restriction, we also define
\[
\mathcal{P} \ = \ \mbox{fraction of domain computed} \ = \ (\mbox{\# of gridpoints \ACC{} or \CON}) \ / \ m^n.
\]
{\bf Underestimate functions.} In all examples except for section \ref{ss:observers}, we rely on  na\"{i}ve and scaled-oracle heuristics (i.e., $\varphi^0_{\lambda}$ and $\bar{\varphi}_{\lambda}$).
We consider this sufficient since the accuracy of the AA* approach is really underestimate-neutral (though the efficiency is clearly dependent on both $\varphi$ and $\Upper$).  We expect that the results based on any other heuristics (including those in \cite{Peyre_coarse, Peyre_landmark, Peyre_Geodesic}) will be qualitatively similar.

% COMMENT REMOVED
% COMMENT REMOVED
% COMMENT REMOVED
% COMMENT REMOVED
% COMMENT REMOVED
\subsection{Constant speed $f\equiv 1$ in 2D and 3D}
\label{ss:example_constant}
% COMMENT REMOVED
In the constant speed case, all characteristics are straight lines and the the na\"{i}ve heuristic coincides with the actual time-to-go (i.e., $\varphi^0=v$).  In this subsection we use the underestimate $\varphi=\varphi^0_{\lambda}$ and
place $\bs$ and $\bt$ at opposite corners of $\bar{\Omega}$.  Our goal is to test the effect of
$\lambda \in [0,1]$ on the accuracy and efficiency  for different grid resolutions $h=1/(m-1)$.
In testing AA*-FMM, we use $\Upper=(1+ \epsilon_{tol}h^{\mu})|\bs-\bt|$, where
$\mu = 1/2$ with $\epsilon_{tol} = 1/4$  in 2D and $\epsilon_{tol} = 1/3$ in 3D.
This ensures that AA*-FMM does not terminate before $\bs$ is \ACC \ and also results in the set $C_3=C_2$ shrinking to a straight line as $h\to0$.
\fillmeup
Figure \ref{fig:constCont} shows the level sets of $U^*$ computed by SA*-FMM and AA*-FMM on a 2D grid with
$m=351$ and $\lambda \in \{0.25, \, 0.5, \, 0.75, 1 \}$.  The non-smoothness of the level-sets produced by SA*-FMM
is due to the additional errors introduced by that method.  For $\lambda = 1$ these errors also result in
a larger $\mathcal{P}$ -- despite the fact that our AA*-FMM has a built-in ``restriction slackness'' (since $\Upper > u(\bs)$).
Figure \ref{fig:constConv} shows $\log_{10}(\astarErrOnly)$ as $m$ and $\lambda$ vary. For $\lambda \geq 0.55$, the errors produced by SA*-FMM are not only relatively large, but also do not decrease much under grid refinement. In contrast, the errors in AA*-FMM decrease quite rapidly even though the set $C_3$ is also shrinking as $h\to0$; see also the convergence analysis in Section \ref{s:it_works}.
\fillmeup
\alexEdit{Since in this example $G(\bs)=X$, additional errors should result from {\em any} domain restriction.  However, the finite-precision of the floating point arithmetic results in ``zero domain restriction errors'' (white spaces in Figure \ref{fig:constConv}) for AA*-FMM even for many test runs where $G(\bs)$ is partly truncated.  E.g., see the case $(\lambda=0.75, \, m=351)$ in Figures \ref{fig:constCont} and \ref{fig:constConv}.}
\fillmeup
Figures \ref{fig:constNaive} and \ref{fig:const3dNaive} show the full accuracy/efficiency data holding $\lambda = 1$ and varying $m$.
% COMMENT REMOVED
\iftoggle{usecolor}{%
% COMMENT REMOVED
\figstart
\begin{center}
% COMMENT REMOVED
\textbf{Contours of $u$ produced by A*-FMM using $\varphi_{\lambda}^0$}%(scaled) na\"{i}ve heuristic}
\begin{adjustwidth}{-0.9cm}{}
\tabcolsep=1pt
\begin{tabular}{c c c c c}
&
\small \em $\mathit{\lambda = 0.25}$ &
\small \em $\mathit{\lambda = 0.50}$ &
\small \em $\mathit{\lambda = 0.75}$ &
\small \em $\mathit{\lambda = 1.00}$ \\
\begin{sideways}\textbf{\Large\hspace{1.5cm}SA*}\end{sideways}&
\includegraphics[scale=0.3]{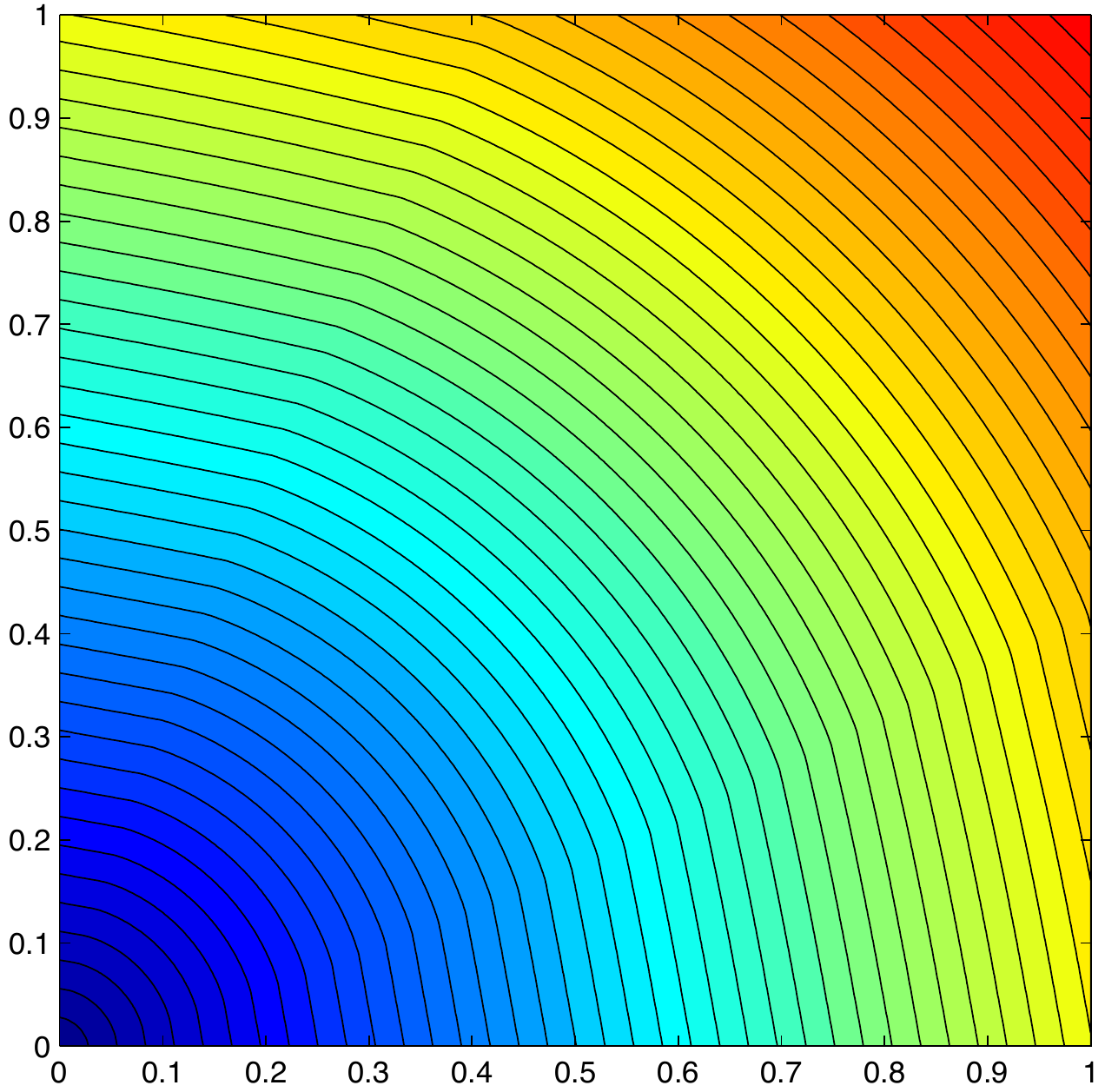} &
\includegraphics[scale=0.3]{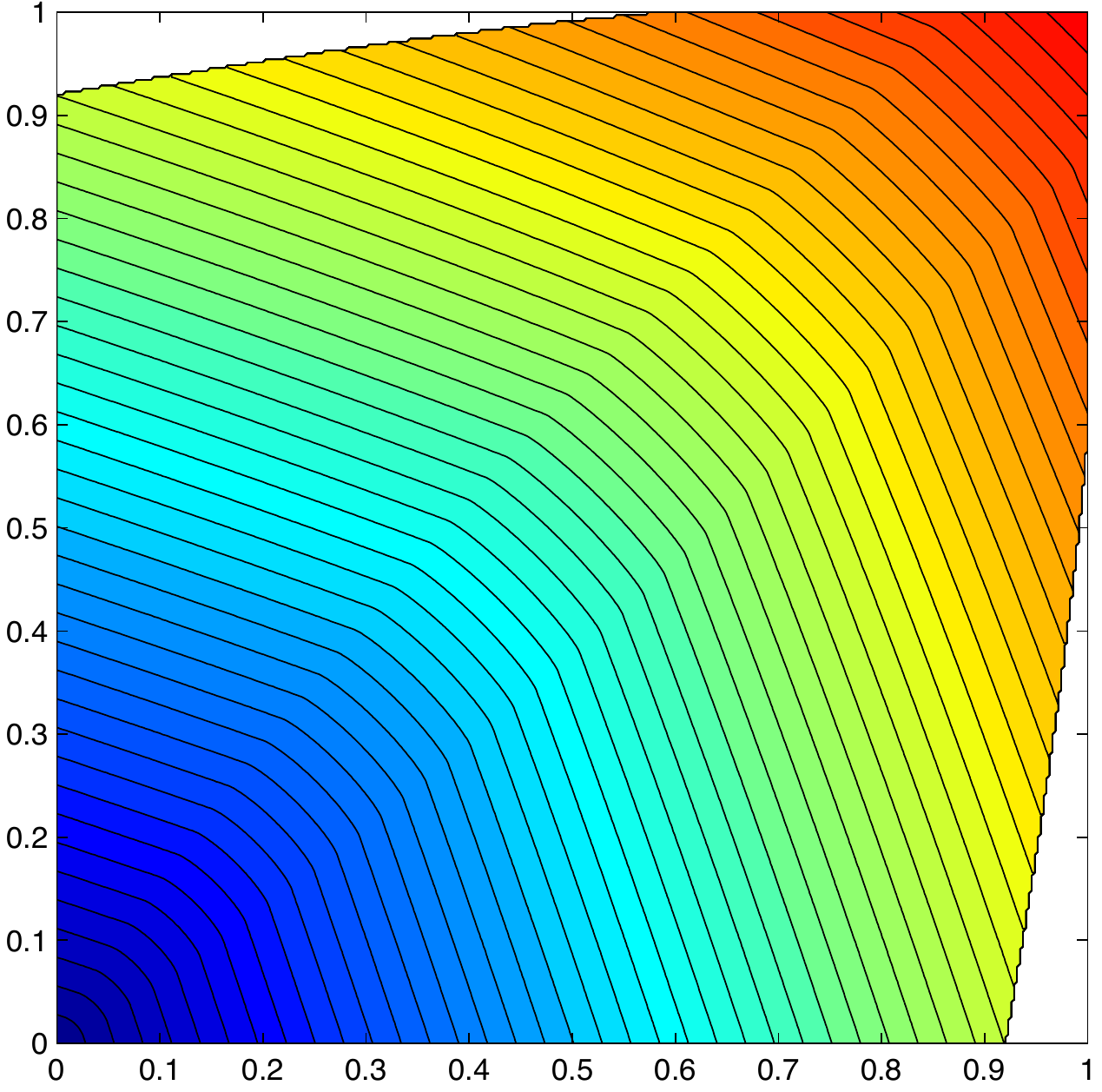} &
\includegraphics[scale=0.3]{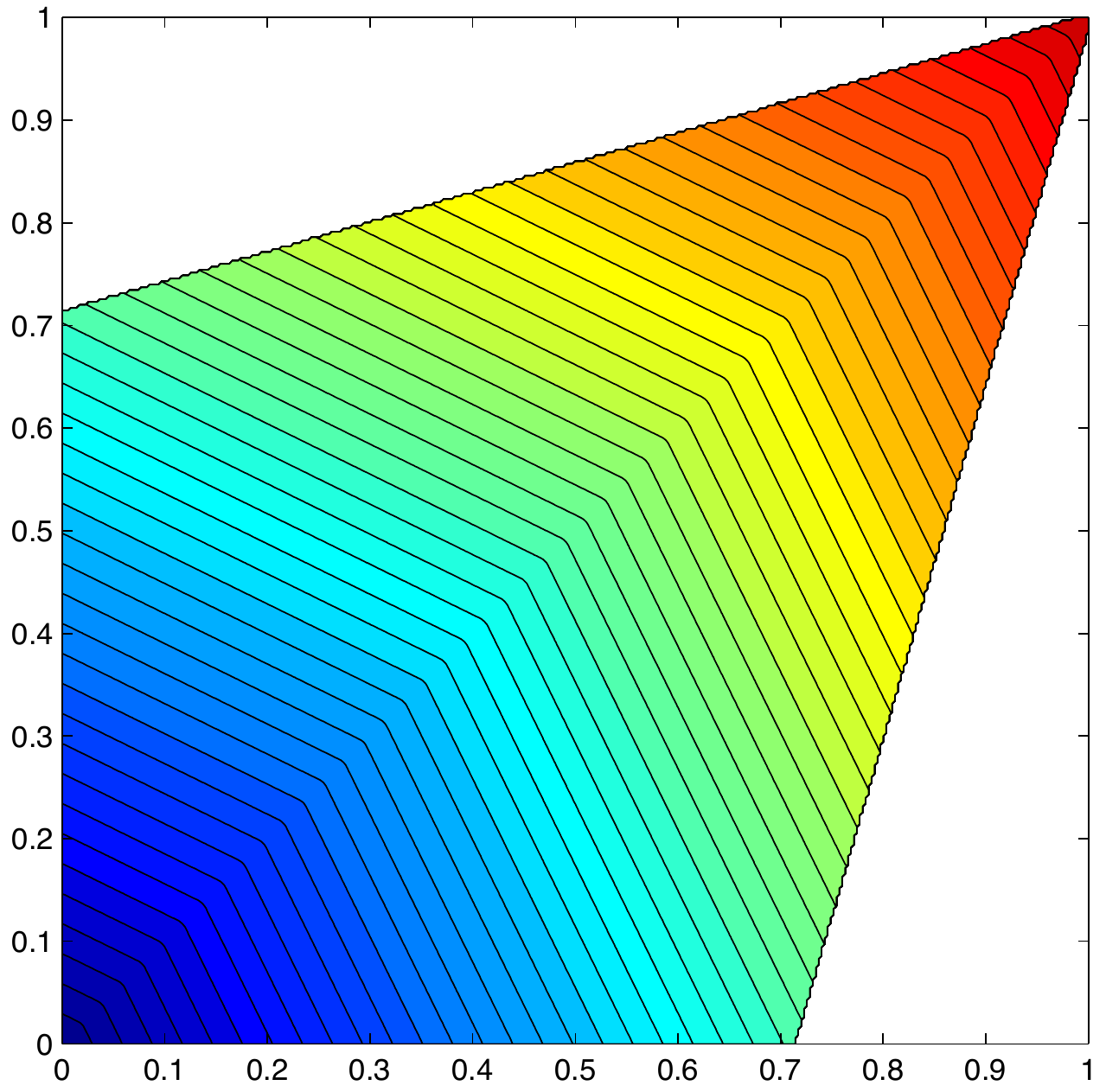} &
\includegraphics[scale=0.3]{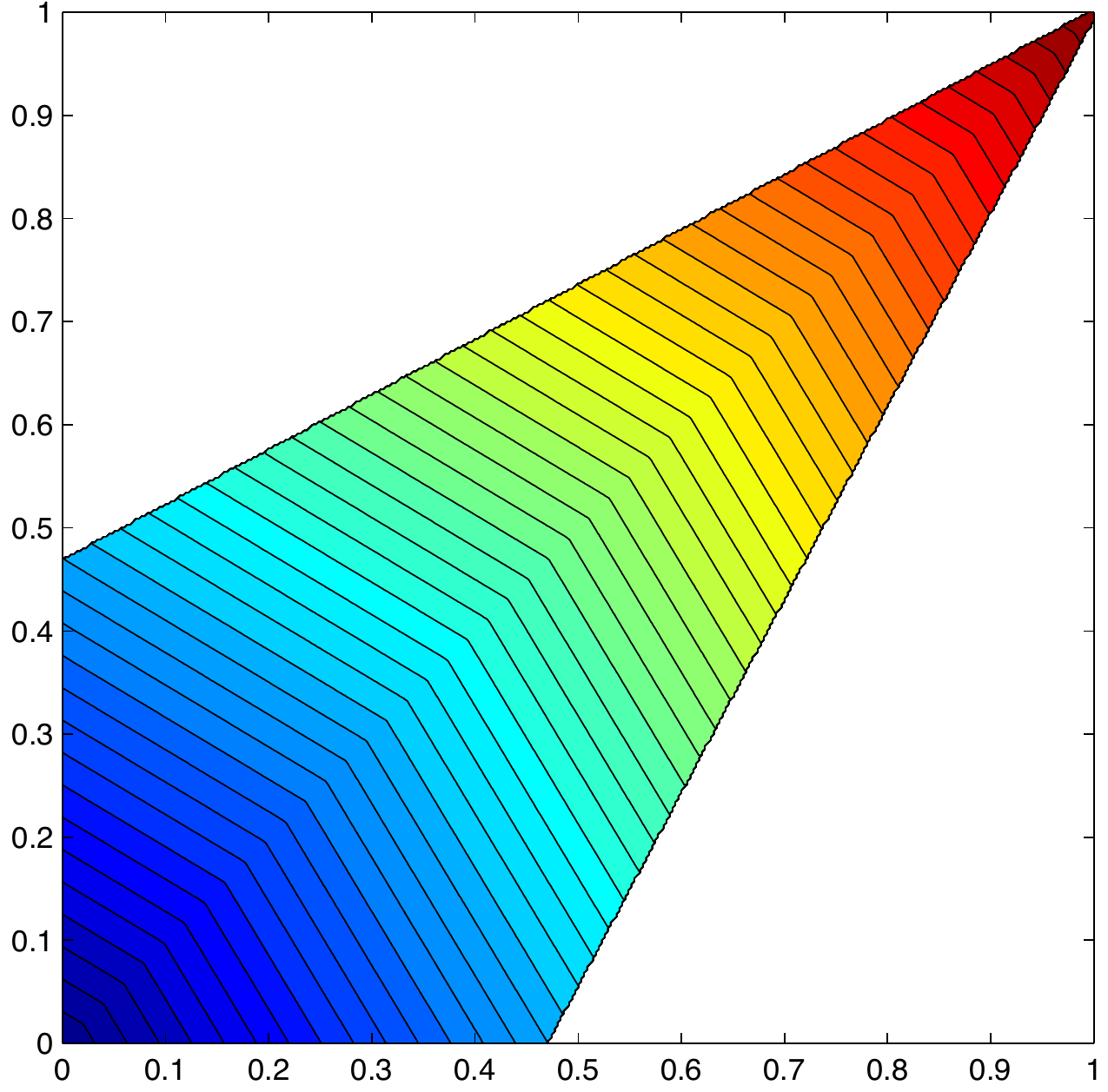} \vspace{-0.2cm} \\
&
\footnotesize $\mathcal{P} = 1$, $\mathcal{E}^*_N \approx 9\times 10^{-7} $ &
\footnotesize $\mathcal{P} = 0.96$, $\mathcal{E}^*_N \approx 2 \times 10^{-4}$ &
\footnotesize $\mathcal{P} = 0.72$, $\mathcal{E}^*_N = 0.051$ &
\footnotesize $\mathcal{P} = 0.47$, $\mathcal{E}^*_N = 0.127$ \\
&
\multicolumn{4}{c}{\includegraphics[scale=0.5]{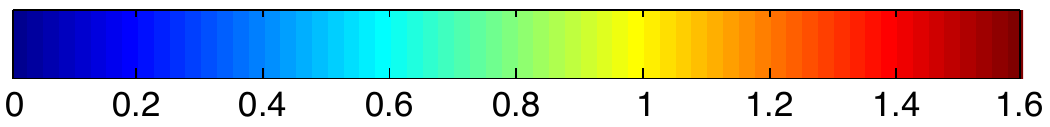}} \\
\begin{sideways}\textbf{\Large\hspace{1.5cm}AA*}\end{sideways}&
\includegraphics[scale=0.3]{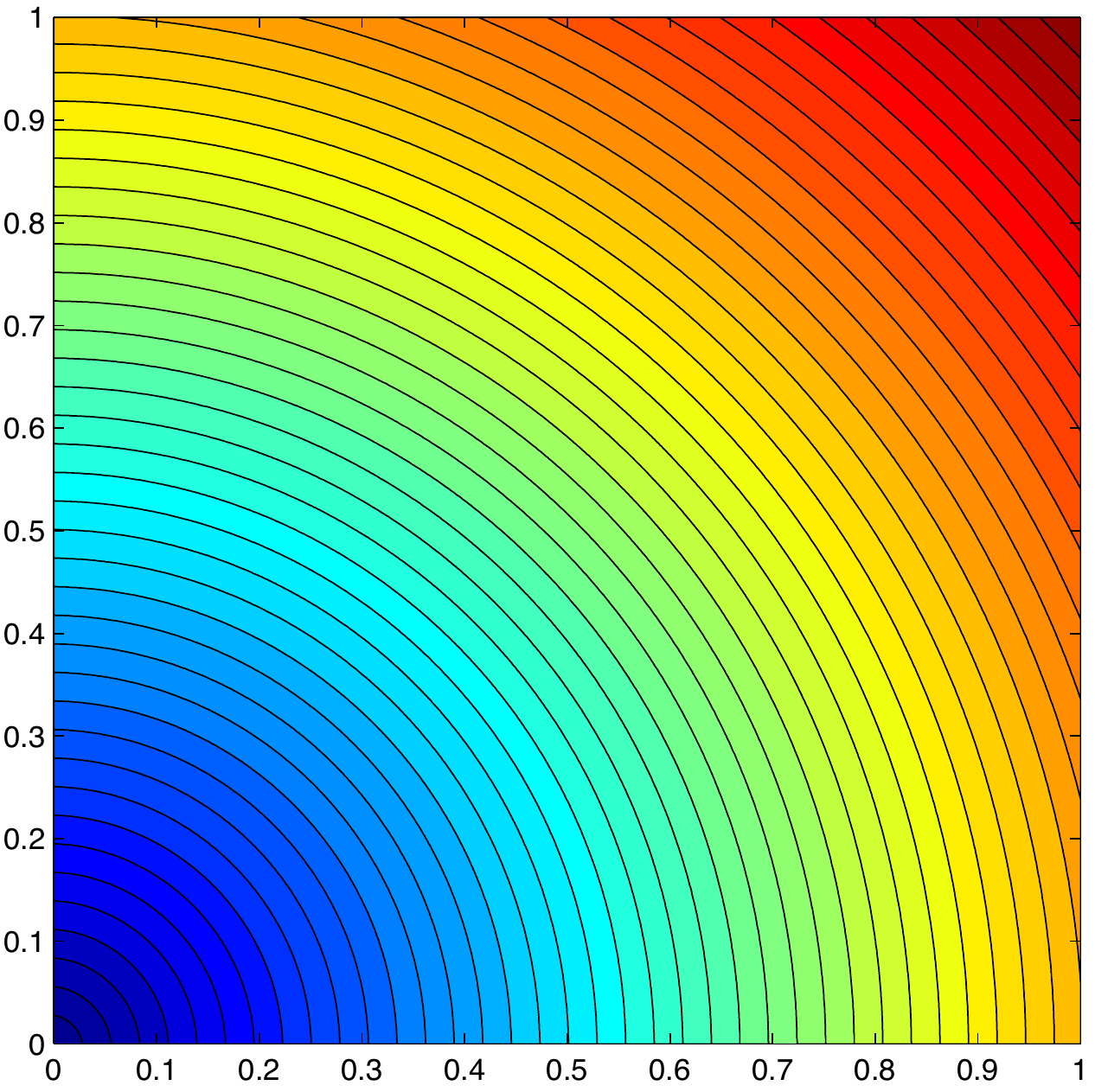} &
\includegraphics[scale=0.3]{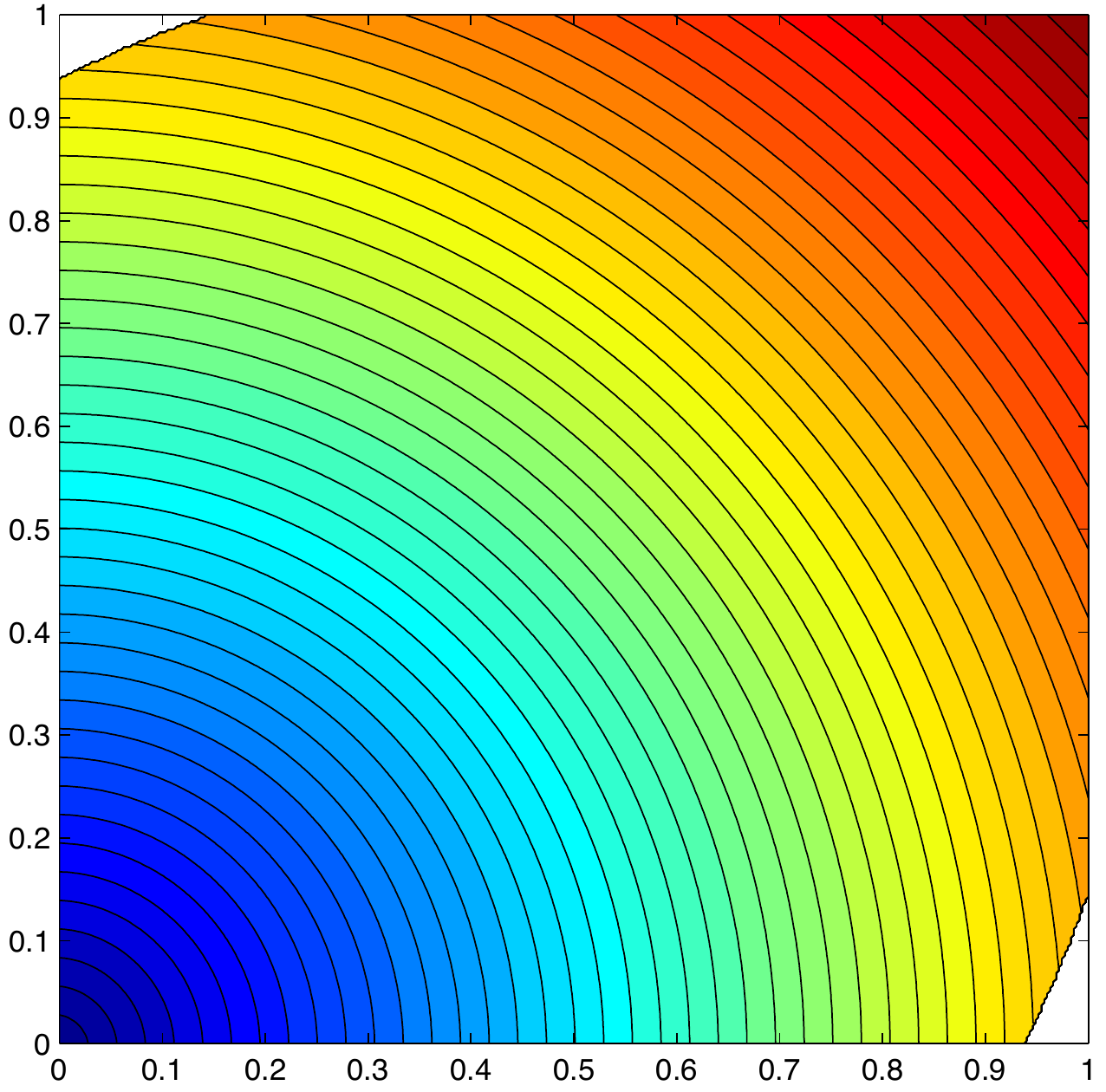} &
\includegraphics[scale=0.3]{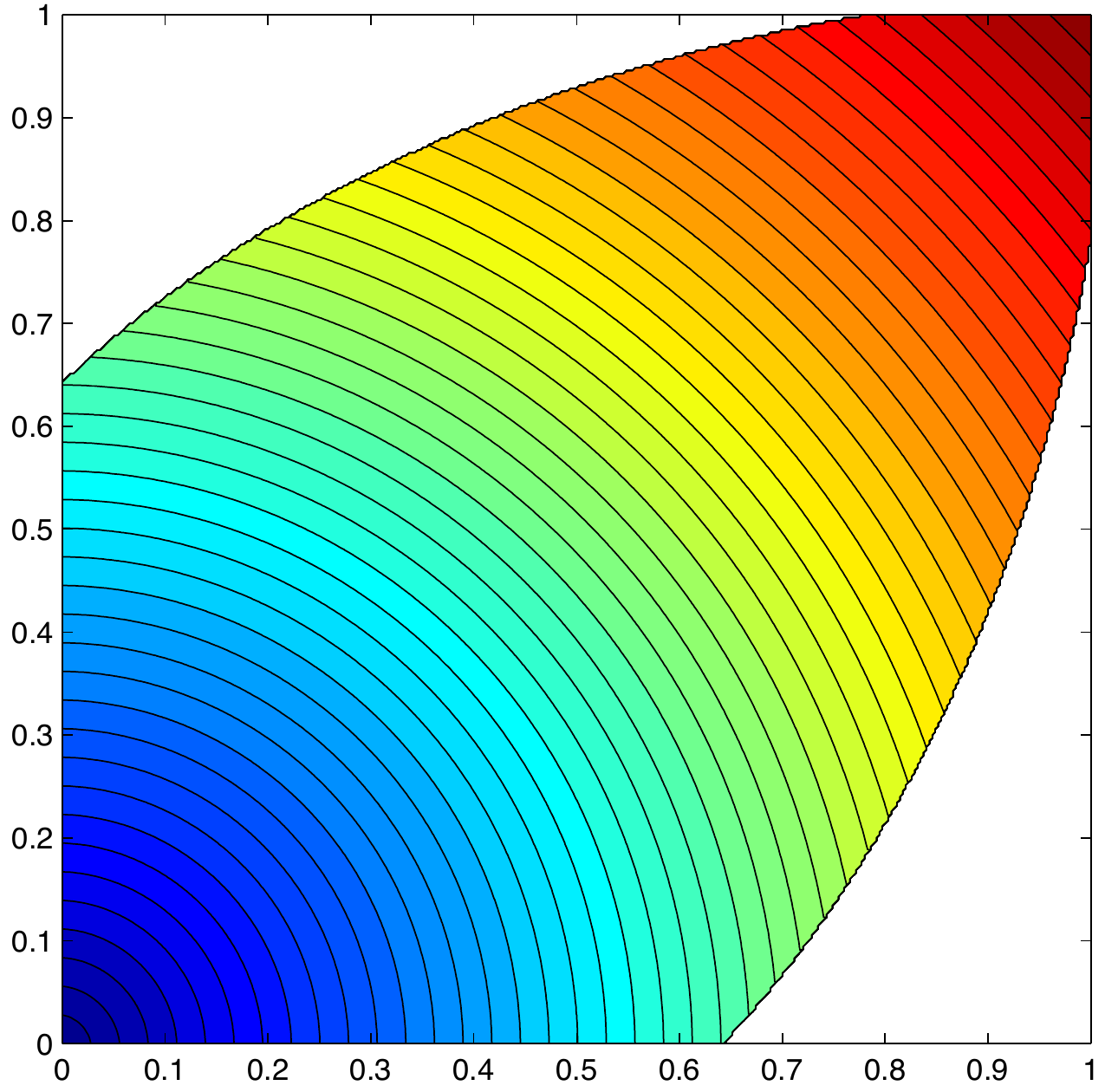} &
\includegraphics[scale=0.3]{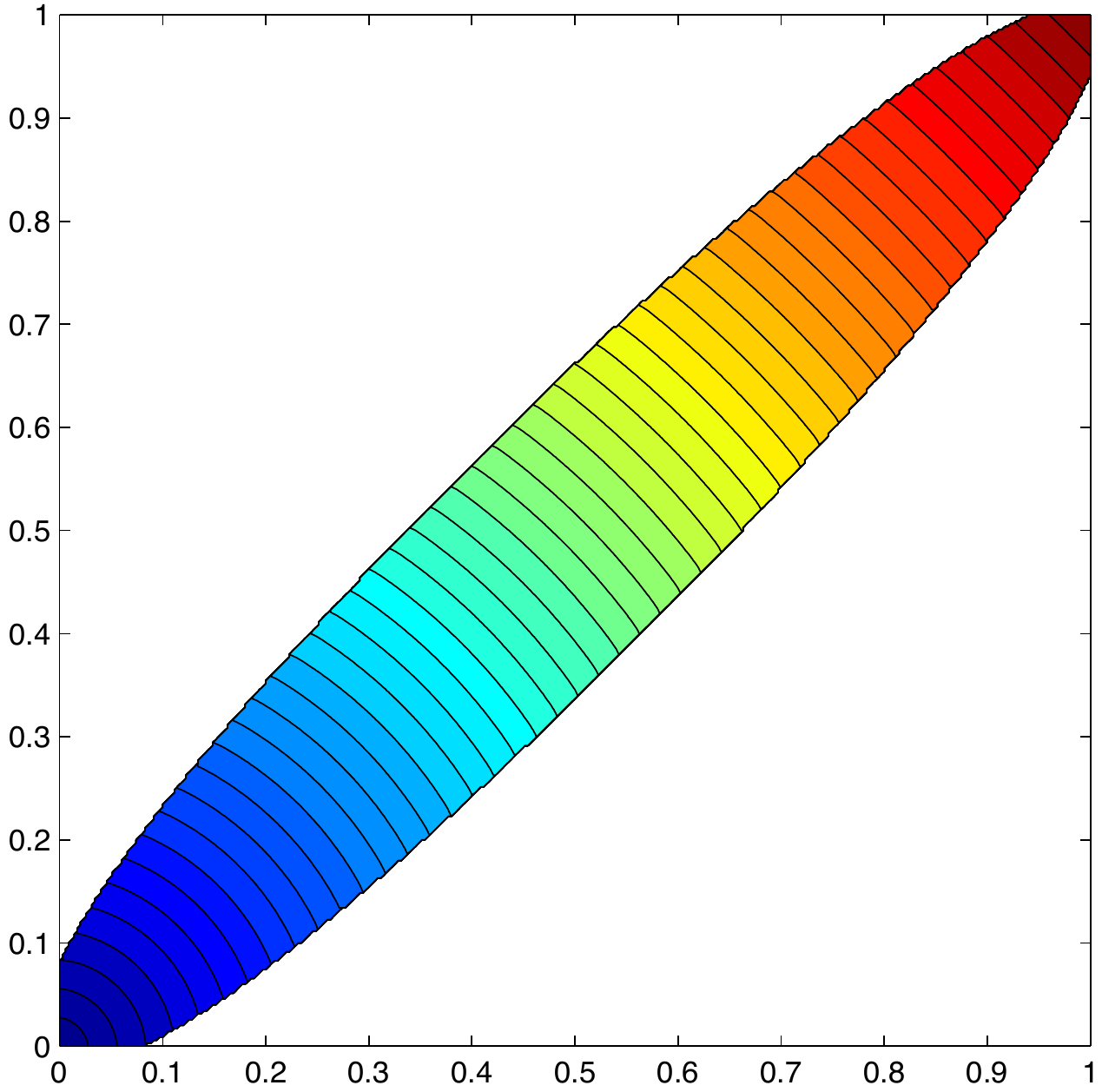} \vspace{-0.2cm} \\
&
\footnotesize $\mathcal{P} =1$, $\mathcal{E}^*_N = 0$ &
\footnotesize $\mathcal{P} =0.99$, $\mathcal{E}^*_N = 0$ &
\footnotesize $\mathcal{P} =0.79$, $\mathcal{E}^*_N = 0$ &
\footnotesize $\mathcal{P} =0.26$, $\mathcal{E}^*_N 4 \times \approx 10^{-7}$ \\
&
\multicolumn{4}{c}{\includegraphics[scale=0.5]{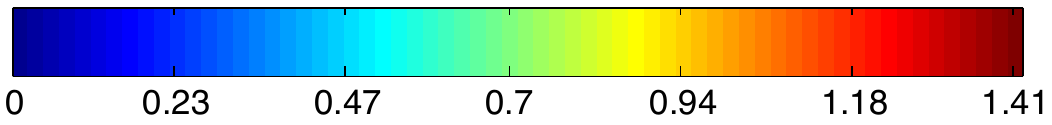}}
\end{tabular}
\end{adjustwidth}
\caption{\footnotesize The top row was produced with SA*-FMM and the error can be seen in two ways: (1) the deformation of the level sets and (2) the value at the source is $\approx 1.61$. The bottom row shows the results of AA*-FMM.
We hold $m = 351$ while $\lambda$ values increase from left to right.}
\label{fig:constCont}
\end{center}
\end{figure}
% COMMENT REMOVED
% COMMENT REMOVED

% COMMENT REMOVED
\figstart
\begin{center}
{\bf 2D constant speed: Error $ \ = \ {\log_{10}}(\astarErrOnly)$.} \\
% COMMENT REMOVED
% COMMENT REMOVED
% COMMENT REMOVED
% COMMENT REMOVED
\begin{tabular}{c c}
\em SA* & \em AA* \\
\includegraphics[scale=.5]{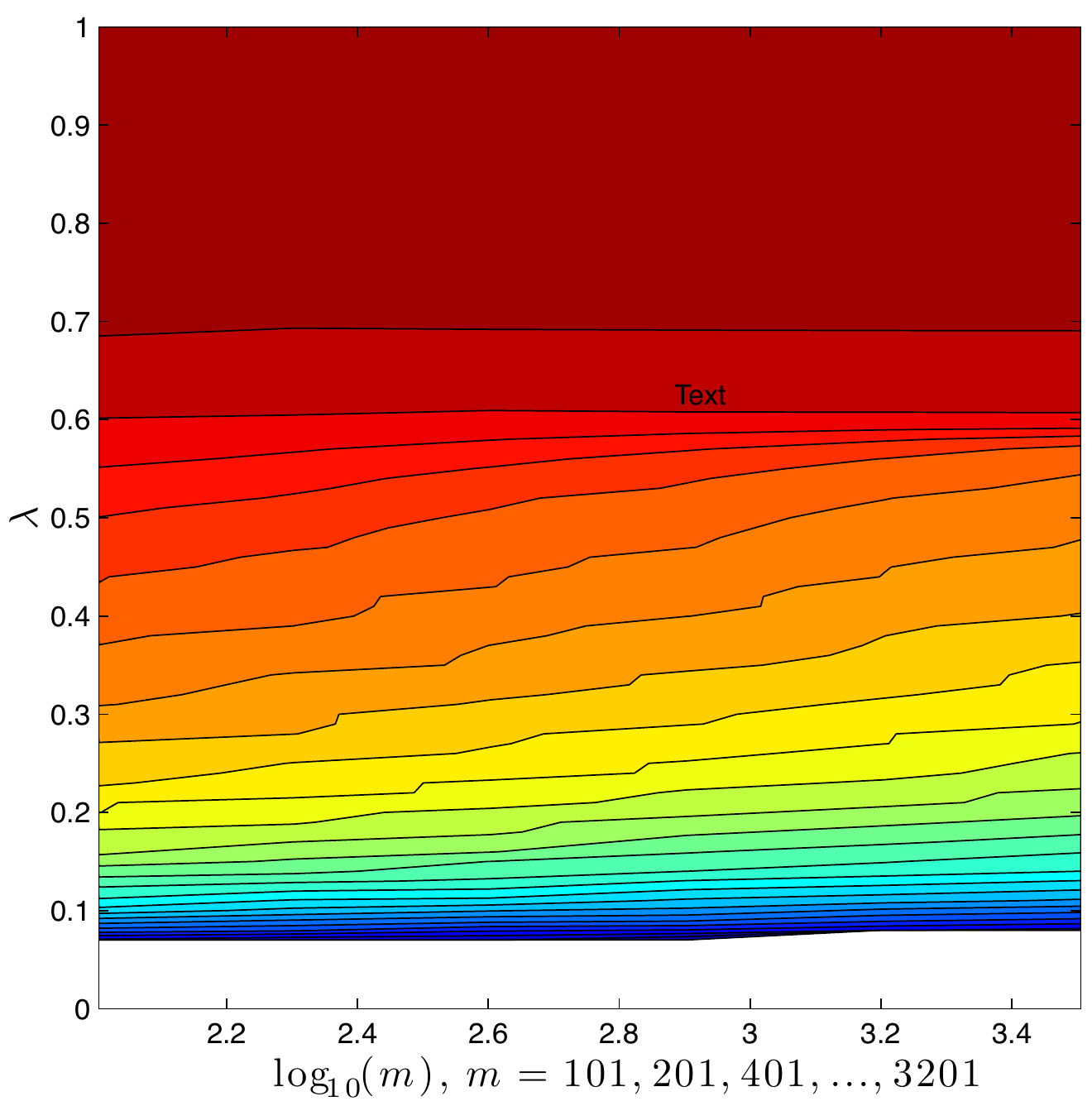} &
\includegraphics[scale=.5]{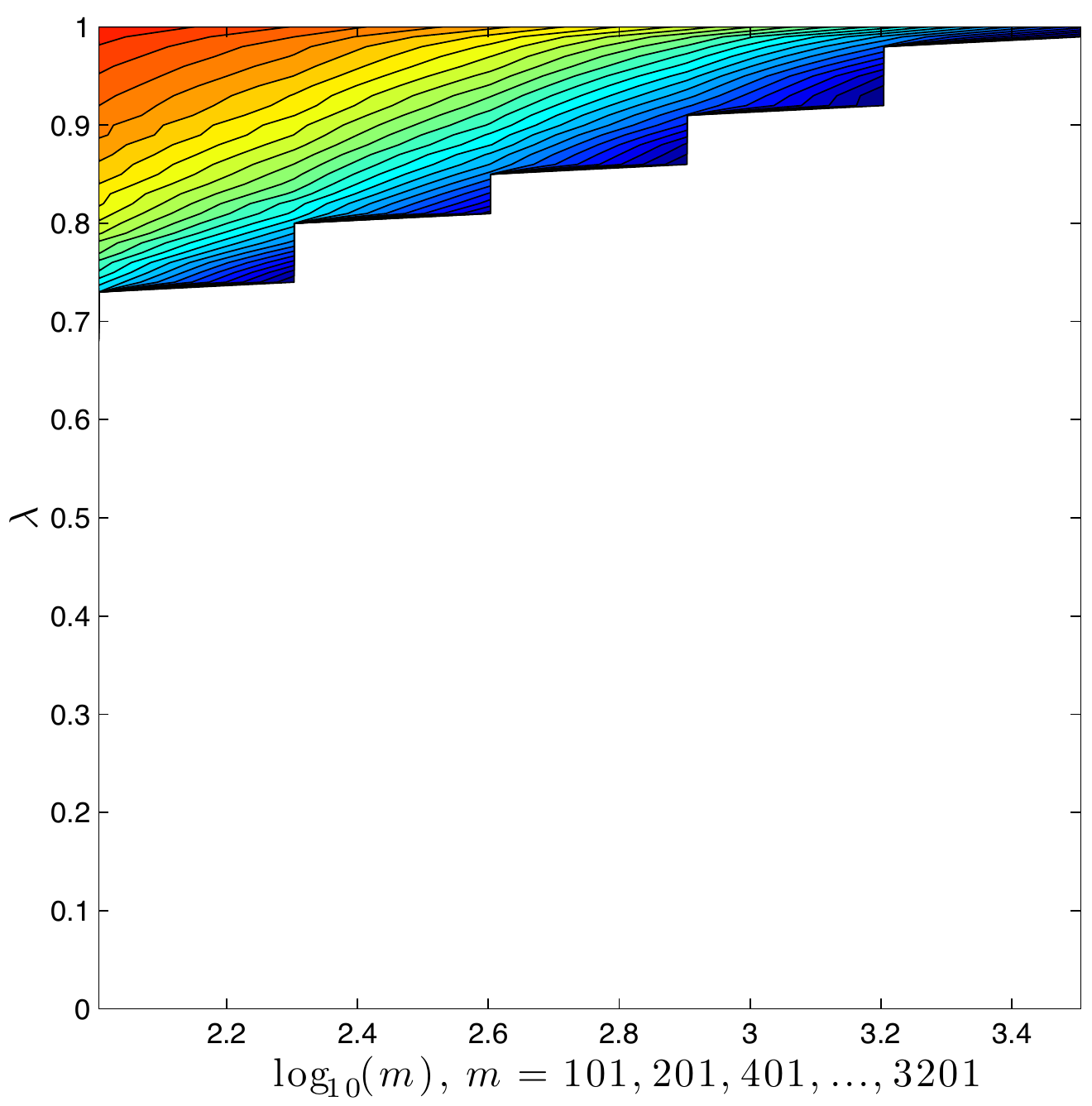} \\
\multicolumn{2}{c}{\includegraphics[scale=.7]{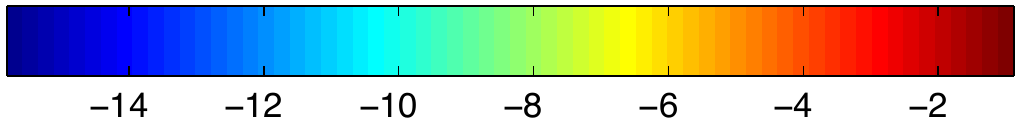}}
% COMMENT REMOVED
\end{tabular}
% COMMENT REMOVED
% COMMENT REMOVED
\vspace{-.35cm}
\caption{\footnotesize SA* versus AA* comparison based on $\astarErrOnly$ errors. The horizontal axis
% COMMENT REMOVED
shows the grid resolution $m$, and the vertical axis corresponds to the heuristic strength $\lambda$.
% COMMENT REMOVED
White corresponds to errors smaller than the machine $\varepsilon$.
% COMMENT REMOVED
% COMMENT REMOVED
% COMMENT REMOVED
% COMMENT REMOVED
}
\label{fig:constConv}
\end{center}
\end{figure}

}{%
% COMMENT REMOVED
\figstart
\begin{center}
% COMMENT REMOVED
\textbf{Contours of $u$ produced by A*-FMM using $\varphi_{\lambda}^0$}%(scaled) na\"{i}ve heuristic}
\begin{adjustwidth}{-0.9cm}{}
\tabcolsep=1pt
\begin{tabular}{c c c c c}
&
\small \em $\mathit{\lambda = 0.25}$ &
\small \em $\mathit{\lambda = 0.50}$ &
\small \em $\mathit{\lambda = 0.75}$ &
\small \em $\mathit{\lambda = 1.00}$ \\
\begin{sideways}\textbf{\Large\hspace{1.5cm}SA*}\end{sideways}&
\includegraphics[scale=0.3]{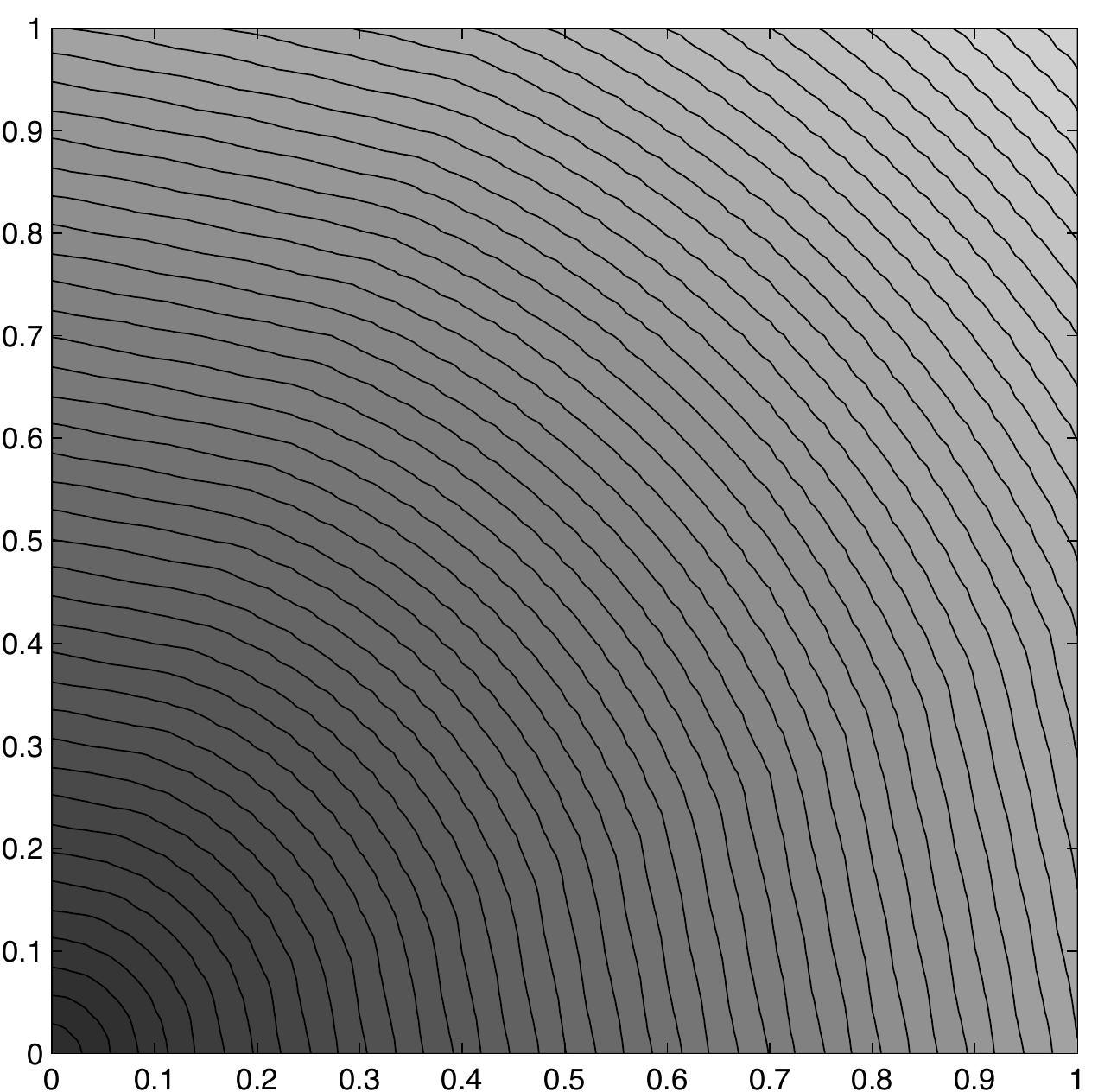} &
\includegraphics[scale=0.3]{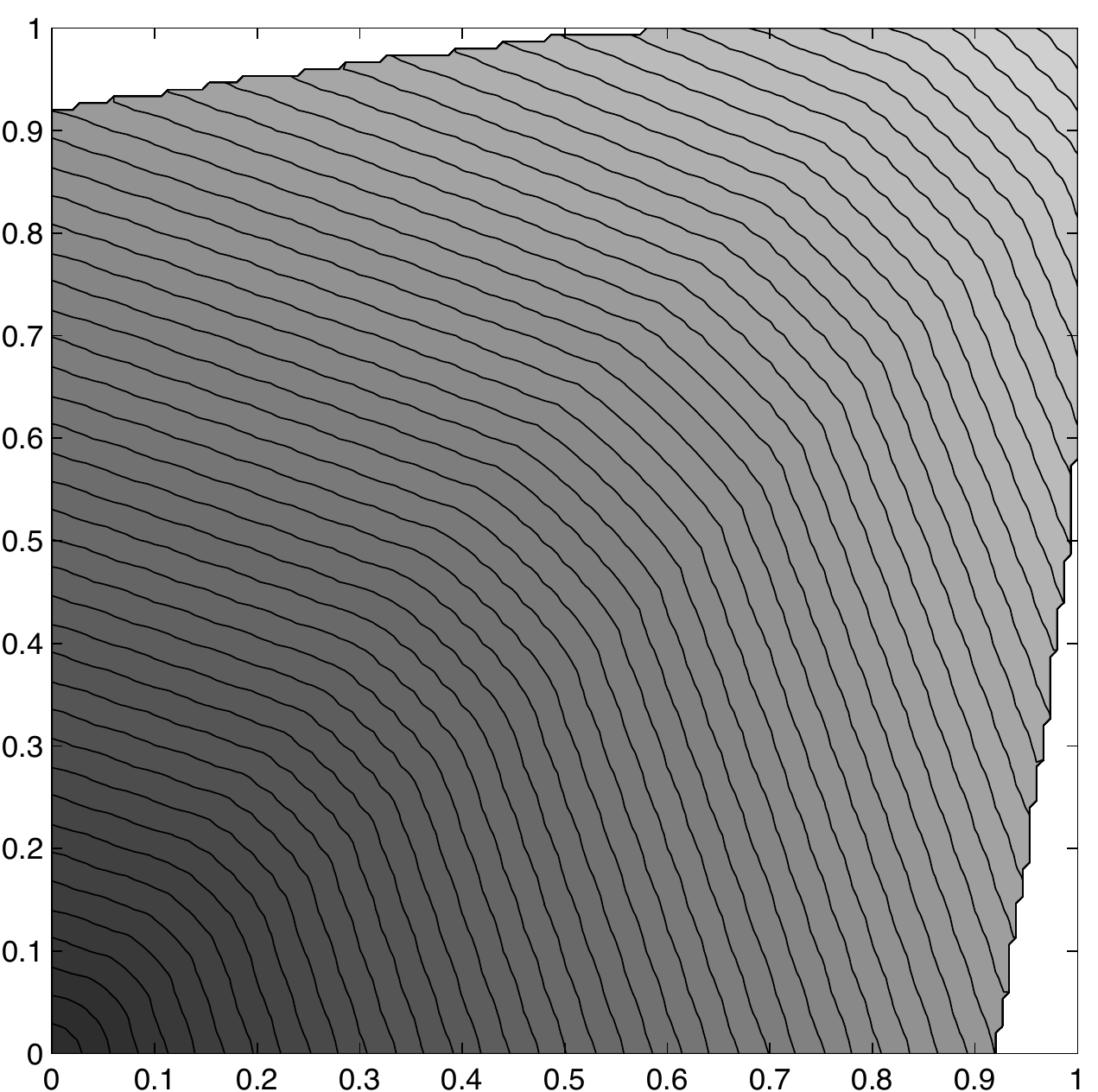} &
\includegraphics[scale=0.3]{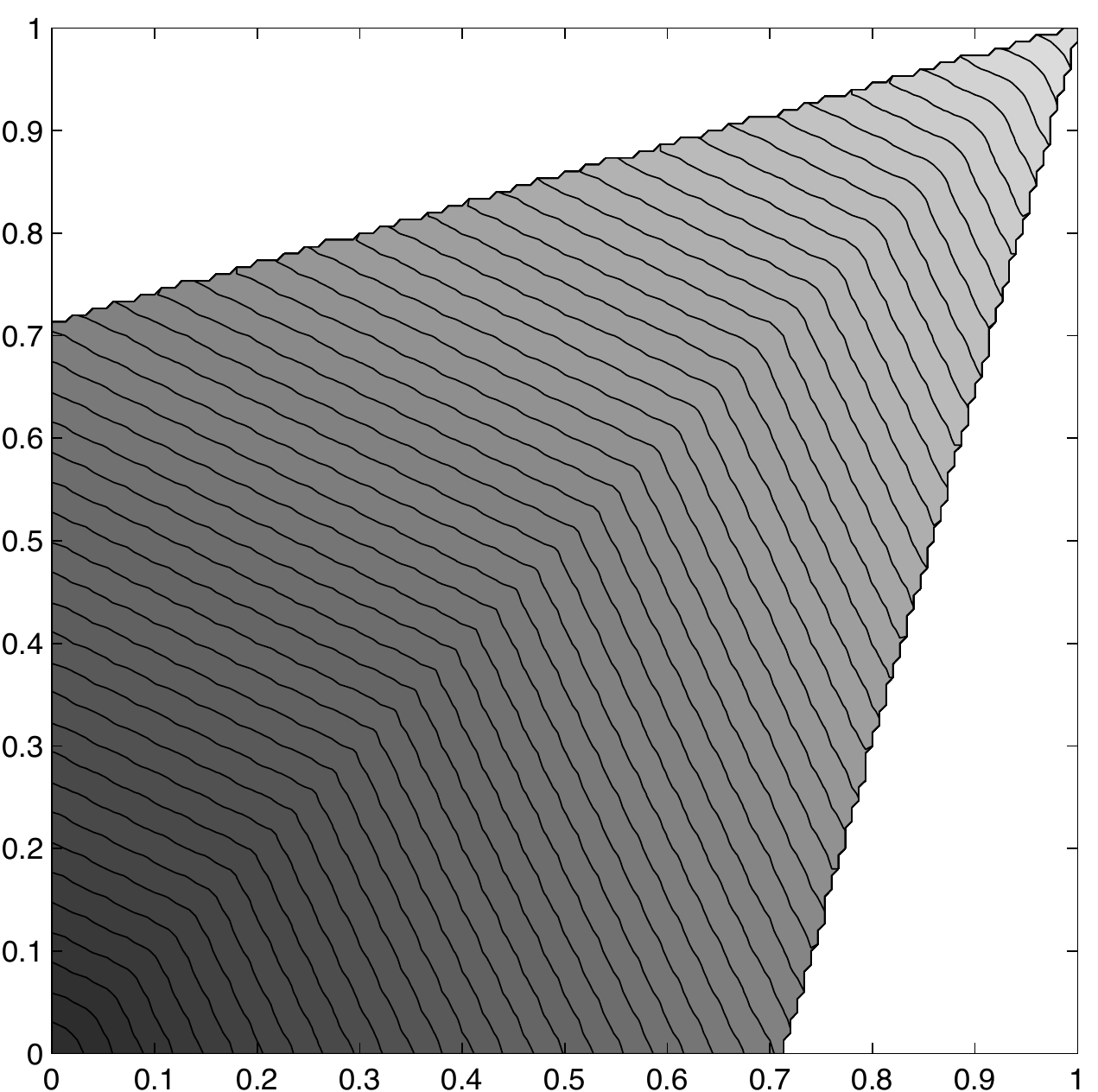} &
\includegraphics[scale=0.3]{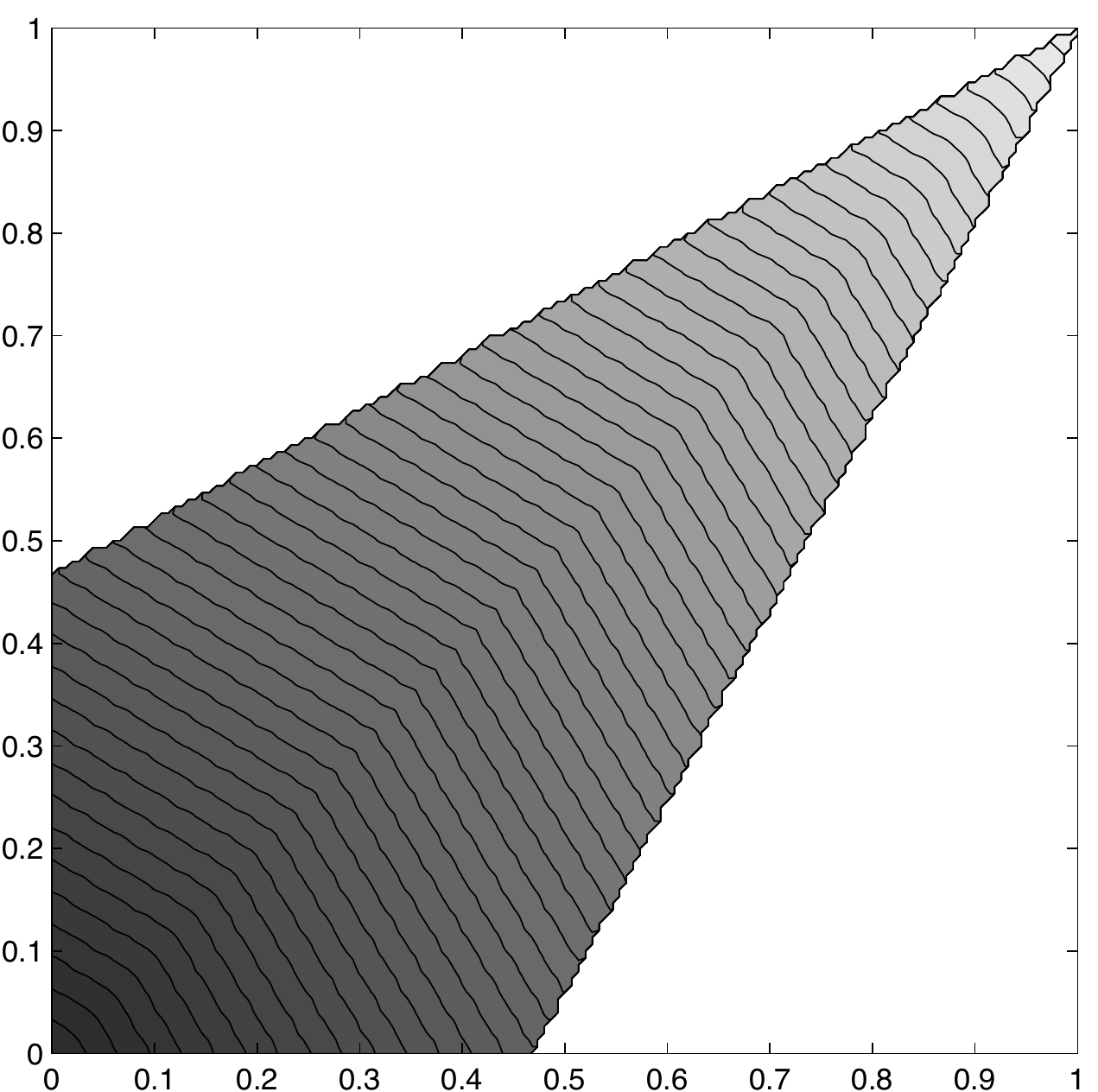} \vspace{-0.2cm} \\
&
\footnotesize $\mathcal{P} = 1$, $\mathcal{E}^*_N \approx 9\times 10^{-7} $ &
\footnotesize $\mathcal{P} = 0.96$, $\mathcal{E}^*_N \approx 2 \times 10^{-4}$ &
\footnotesize $\mathcal{P} = 0.72$, $\mathcal{E}^*_N = 0.051$ &
\footnotesize $\mathcal{P} = 0.47$, $\mathcal{E}^*_N = 0.127$ \\
&
\multicolumn{4}{c}{\includegraphics[scale=0.5]{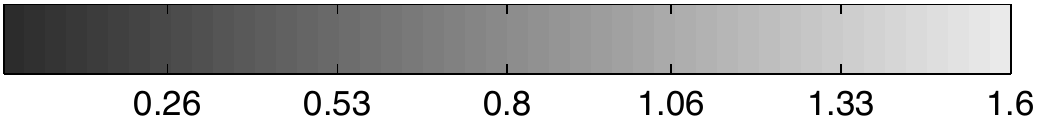}} \\
\begin{sideways}\textbf{\Large\hspace{1.5cm}AA*}\end{sideways}&
\includegraphics[scale=0.3]{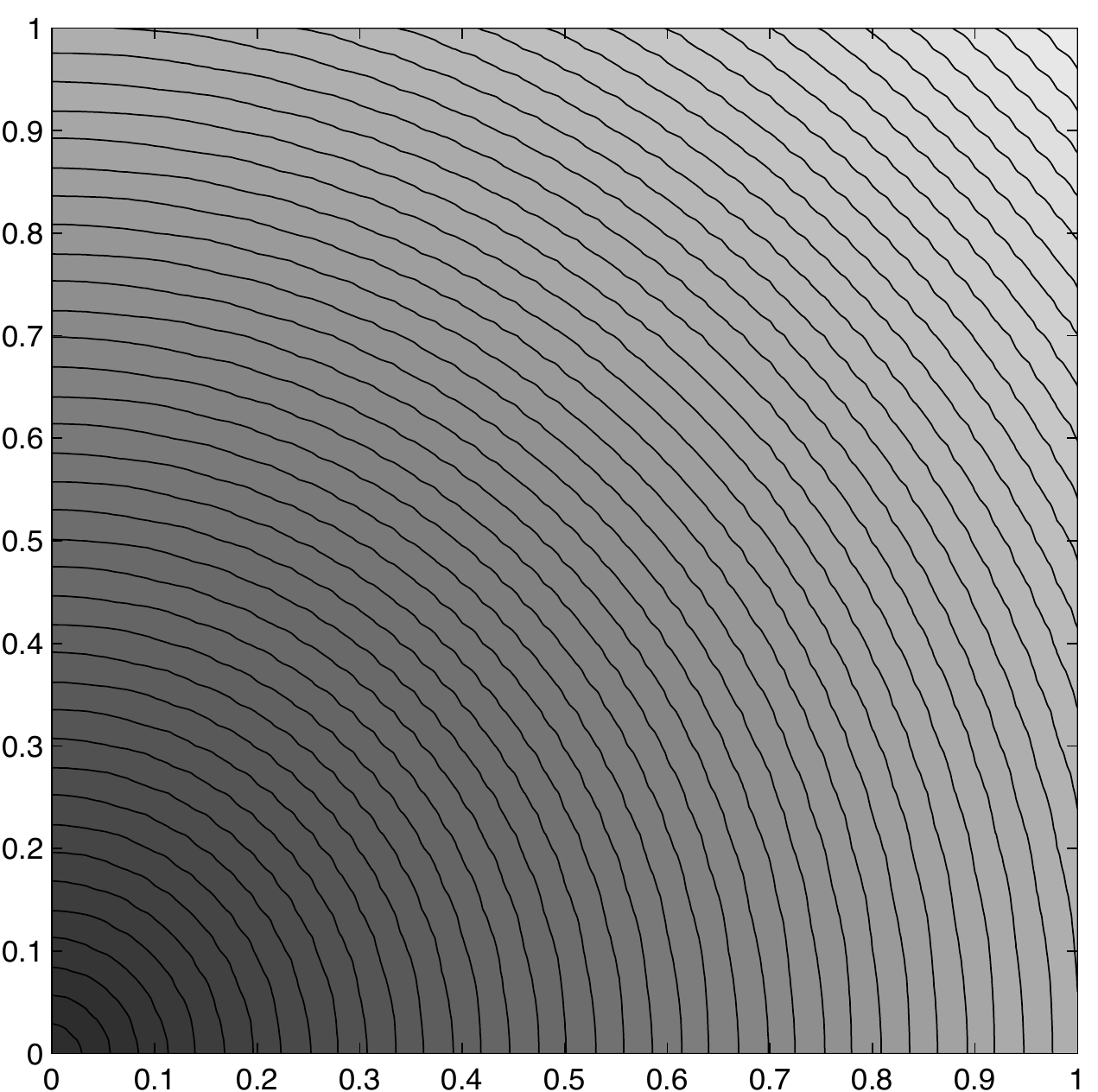} &
\includegraphics[scale=0.3]{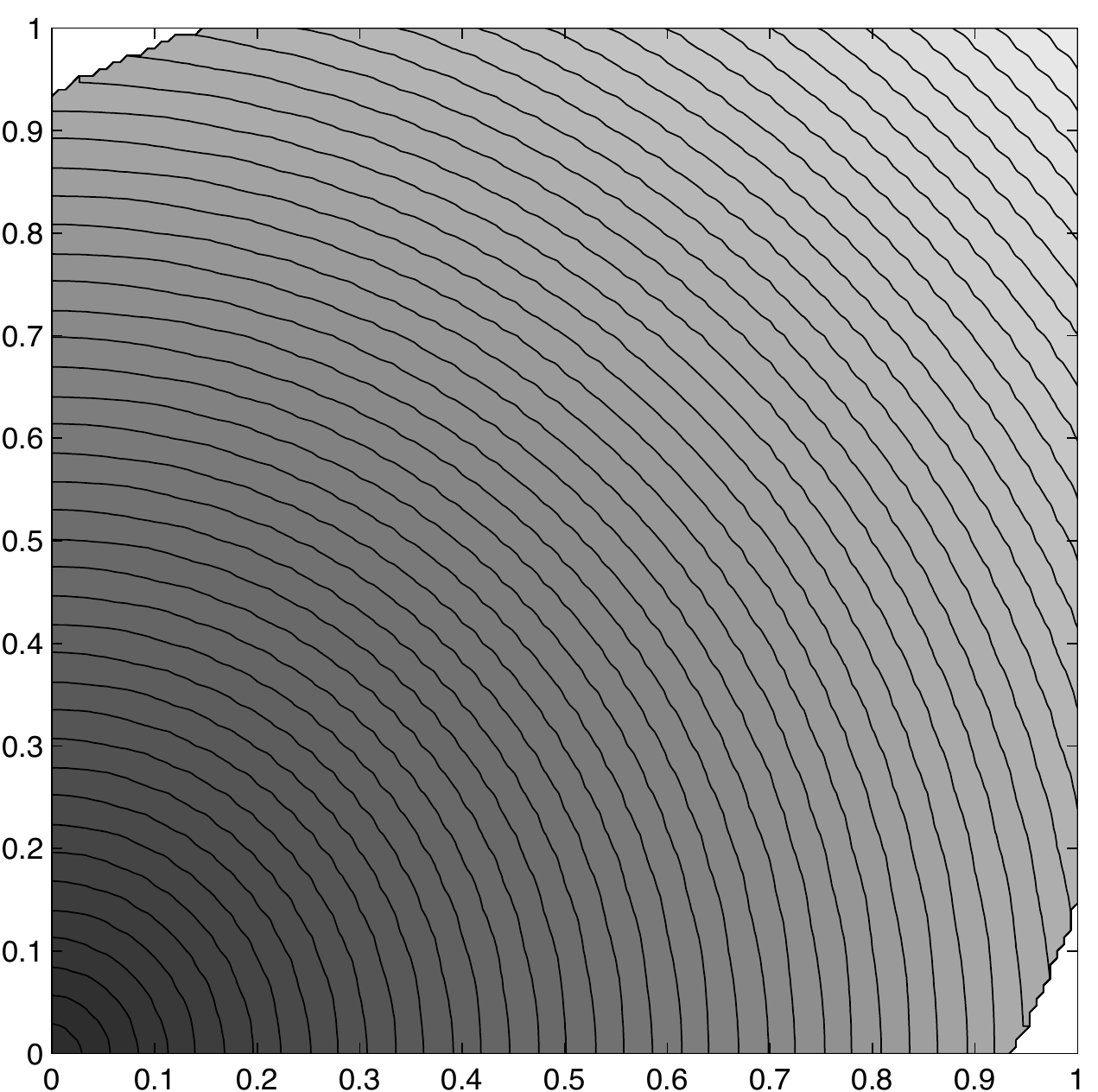} &
\includegraphics[scale=0.3]{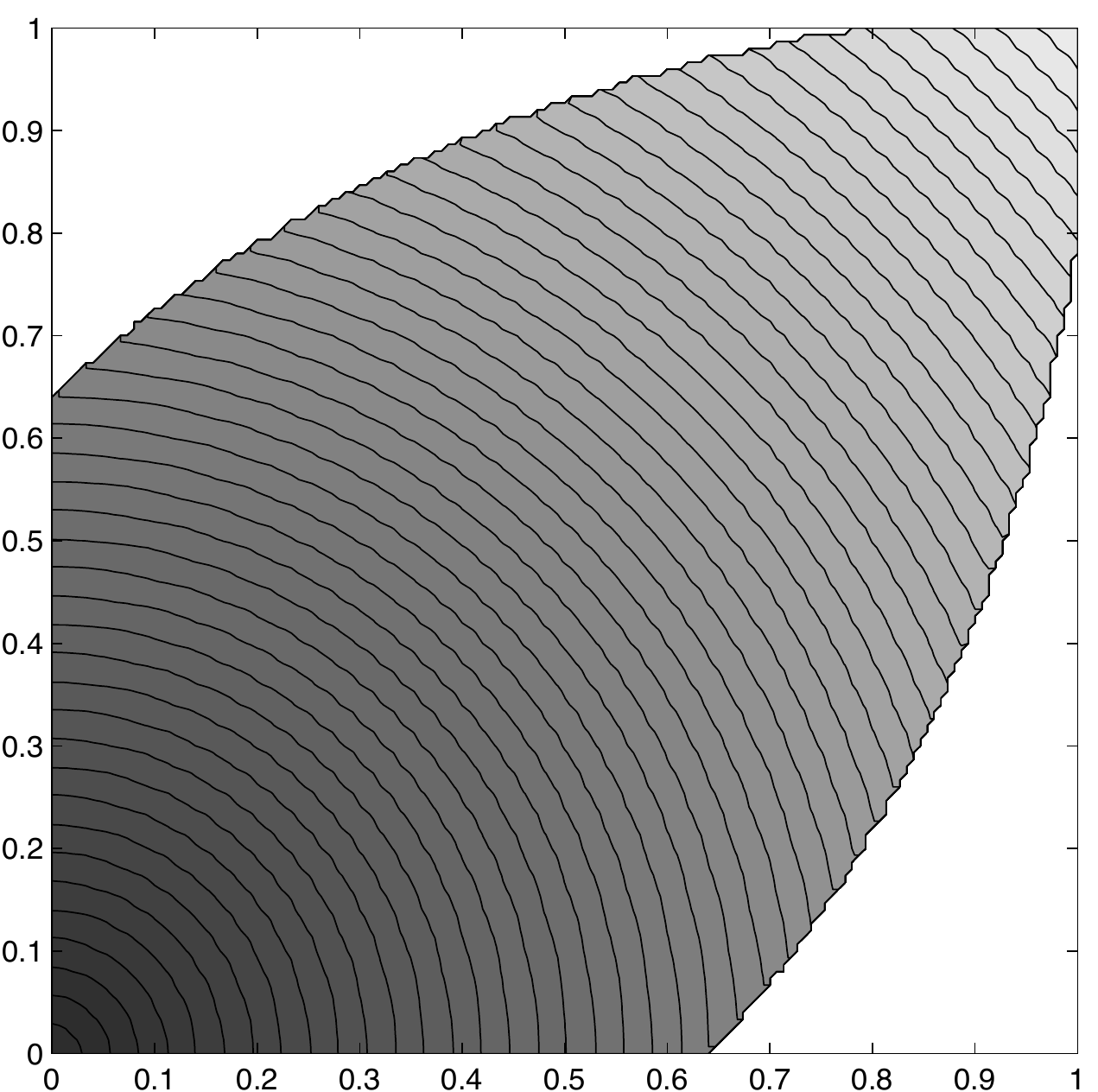} &
\includegraphics[scale=0.3]{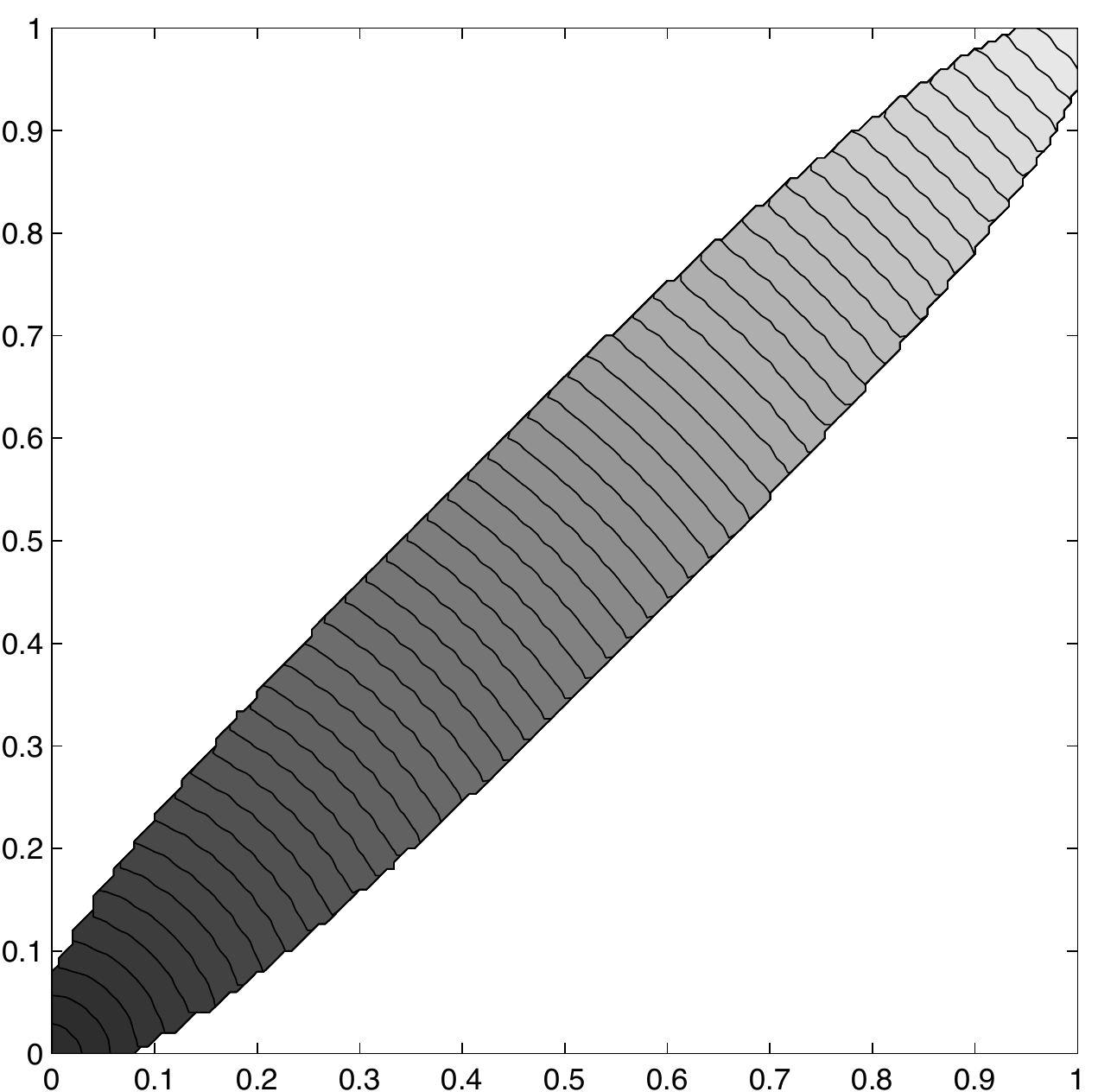} \vspace{-0.2cm} \\
&
\footnotesize $\mathcal{P} =1$, $\mathcal{E}^*_N = 0$ &
\footnotesize $\mathcal{P} =0.99$, $\mathcal{E}^*_N = 0$ &
\footnotesize $\mathcal{P} =0.79$, $\mathcal{E}^*_N = 0$ &
\footnotesize $\mathcal{P} =0.26$, $\mathcal{E}^*_N 4 \times \approx 10^{-7}$ \\
&
\multicolumn{4}{c}{\includegraphics[scale=0.5]{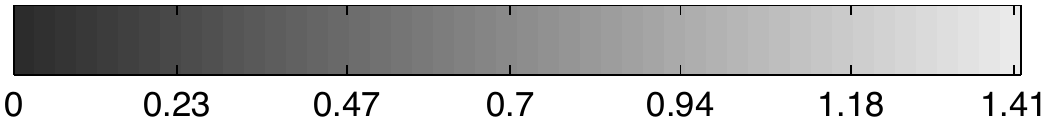}}
\end{tabular}
\end{adjustwidth}
\caption{\footnotesize The top row was produced with SA*-FMM and the error can be seen in two ways: (1) the deformation of the level sets and (2) the value at the source is $\approx 1.61$. The bottom row shows the results of AA*-FMM.
We hold $m = 351$ while $\lambda$ values increase from left to right.}
\label{fig:constCont}
\end{center}
\end{figure}
% COMMENT REMOVED
% COMMENT REMOVED

% COMMENT REMOVED
\figstart
\begin{center}
{\bf 2D constant speed: Error $ \ = \ {\log_{10}}(\astarErrOnly)$.} \\
% COMMENT REMOVED
% COMMENT REMOVED
% COMMENT REMOVED
% COMMENT REMOVED
\begin{tabular}{c c}
\em SA* & \em AA* \\
\includegraphics[scale=.5]{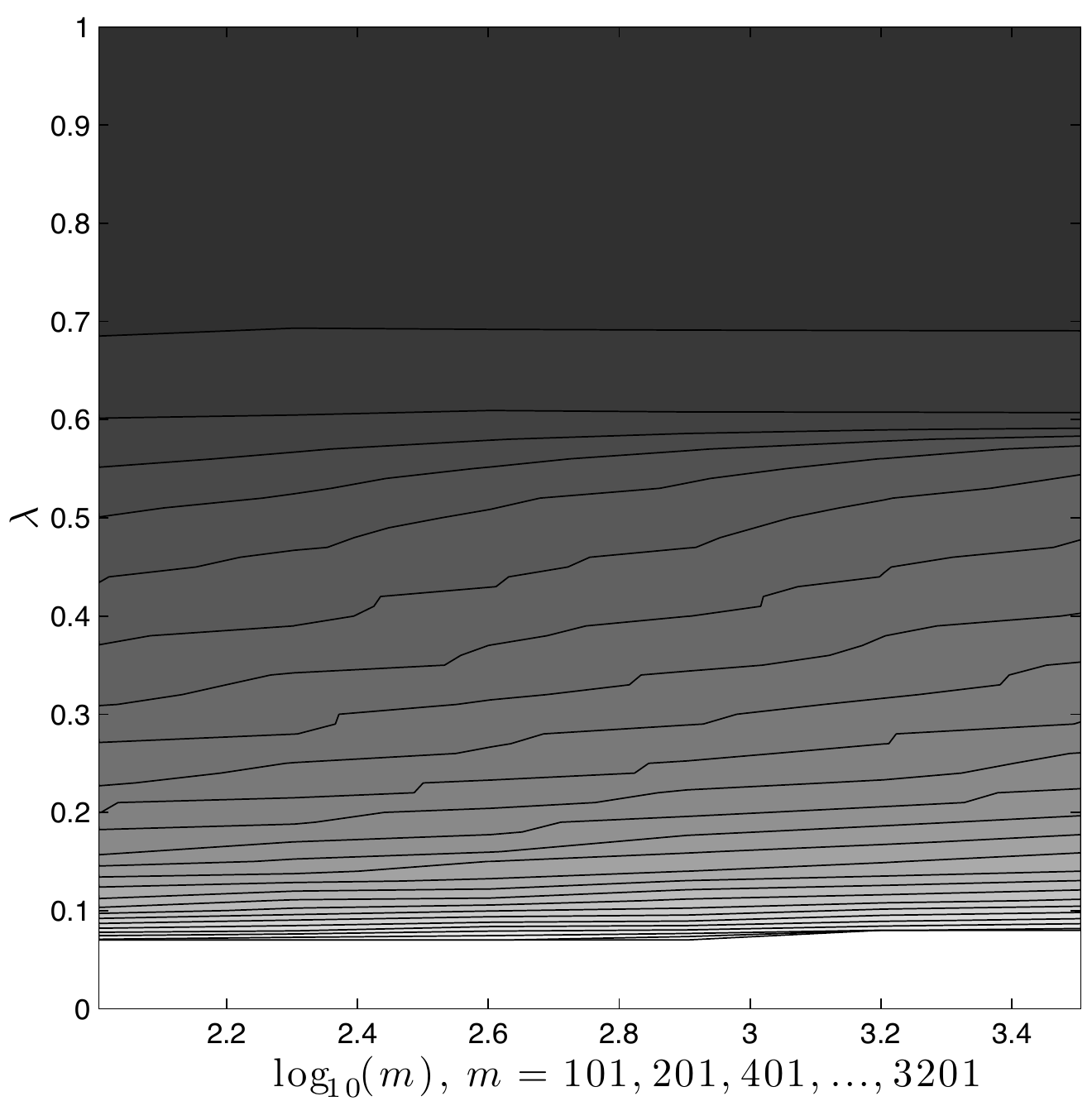} &
\includegraphics[scale=.5]{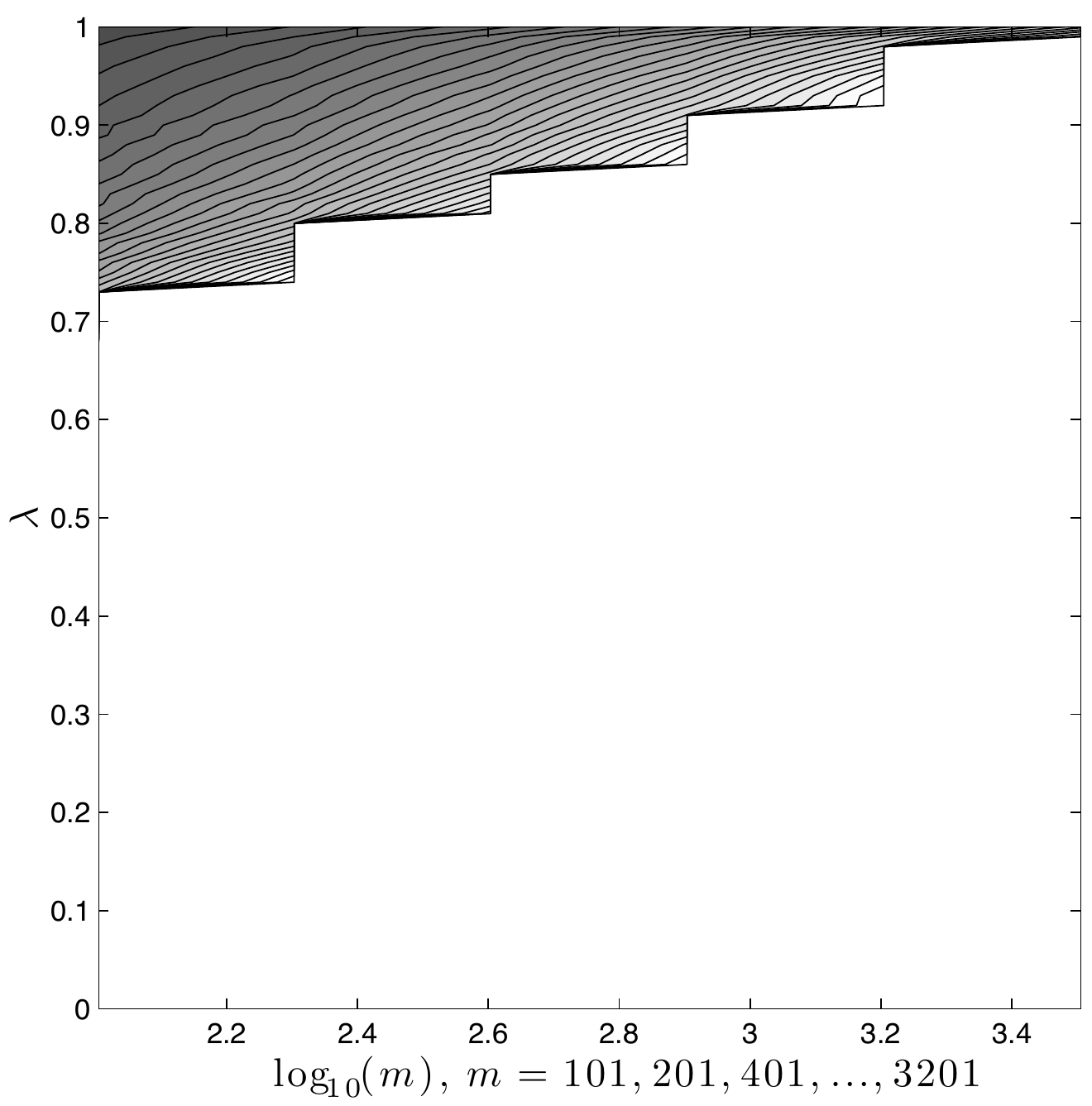} \\
\multicolumn{2}{c}{\includegraphics[scale=.7]{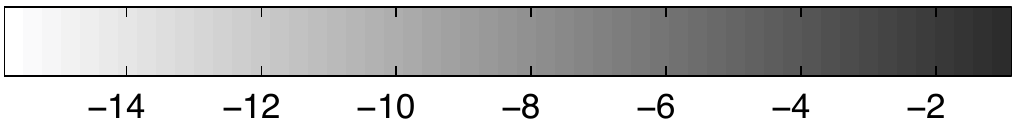}}
% COMMENT REMOVED
\end{tabular}
% COMMENT REMOVED
% COMMENT REMOVED
\vspace{-.35cm}
\caption{\footnotesize SA* versus AA* comparison based on $\astarErrOnly$ errors. The horizontal axis
% COMMENT REMOVED
shows the grid resolution $m$, and the vertical axis corresponds to the heuristic strength $\lambda$.
% COMMENT REMOVED
White corresponds to errors smaller than the machine $\varepsilon$.
% COMMENT REMOVED
% COMMENT REMOVED
% COMMENT REMOVED
% COMMENT REMOVED
}
\label{fig:constConv}
\end{center}
\end{figure}

}
% COMMENT REMOVED
\figstart
\begin{center}
% COMMENT REMOVED
{\bf 2D constant speed: Statistics.}
% COMMENT REMOVED
\fillmeup
\begin{adjustwidth}{-1.5cm}{}
\tabcolsep=0.05cm
\begin{tabular}{c c c}
A. {\em Time (sec)} &
B. {\em Fraction $\mathcal{P}$} &
C. {\em Error $\log_{10}(\astarErr)$} \\
\includegraphics[scale=\statsScale]{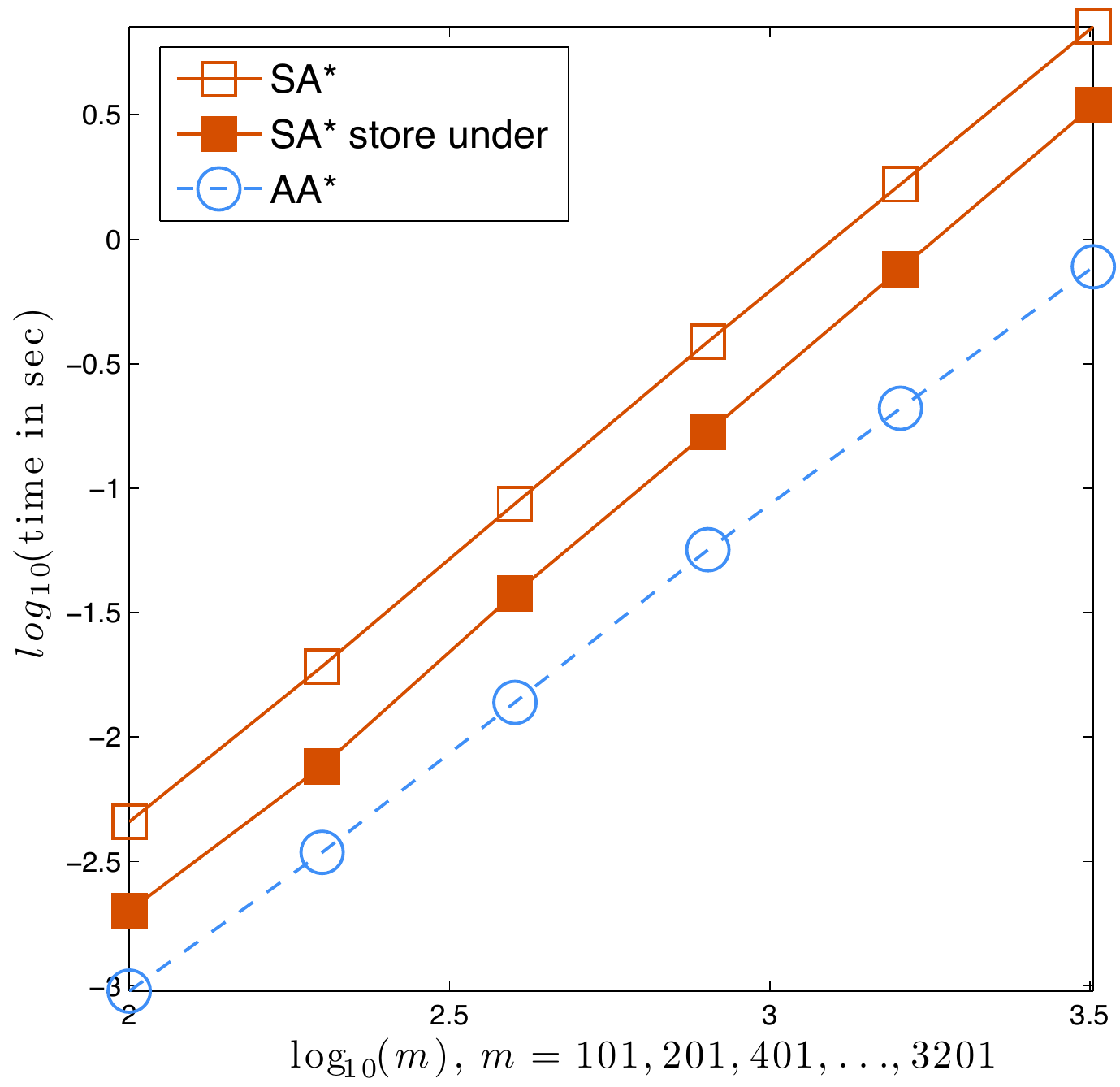} &
\includegraphics[scale=\statsScale]{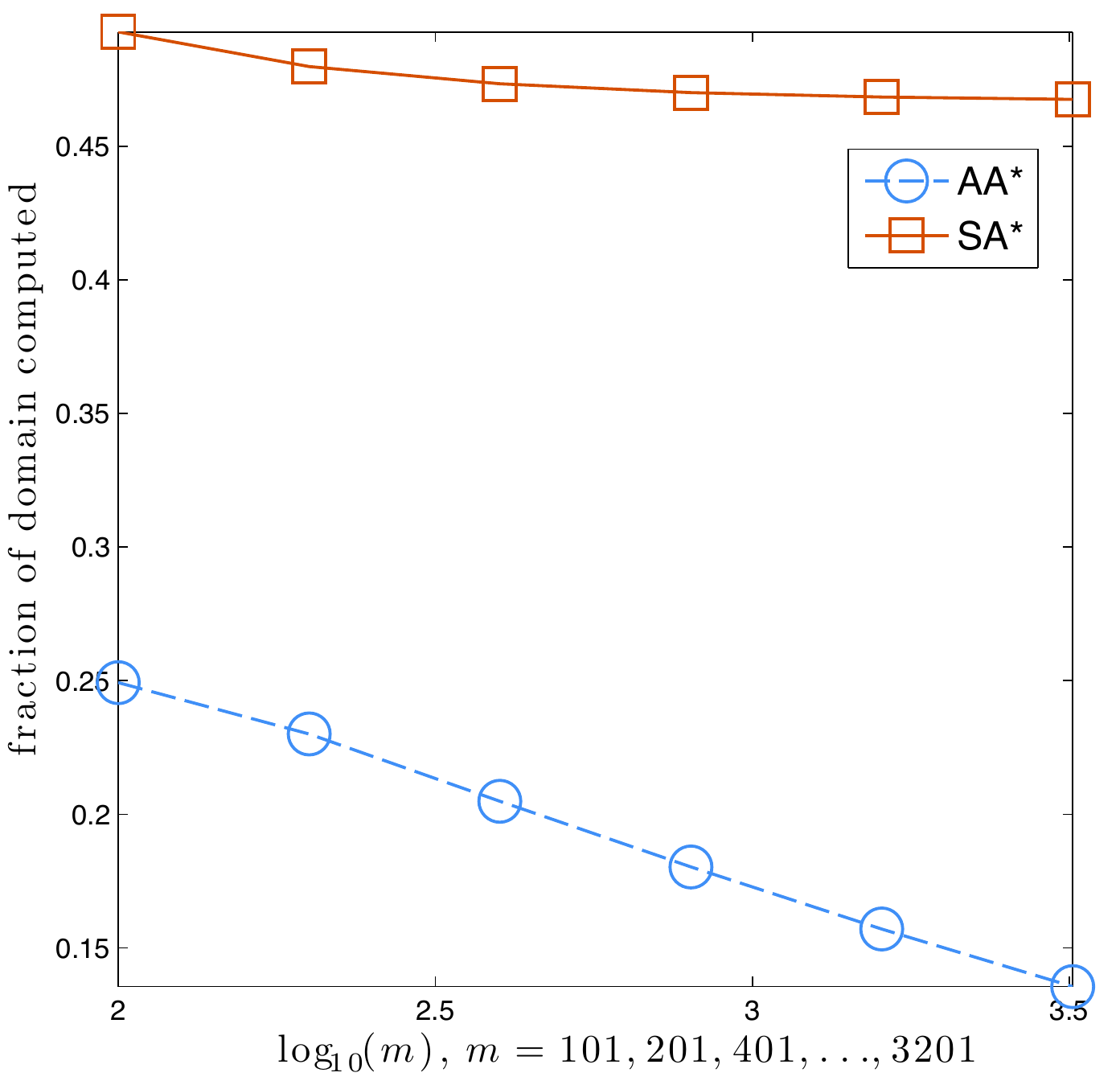} &
\includegraphics[scale=\statsScale]{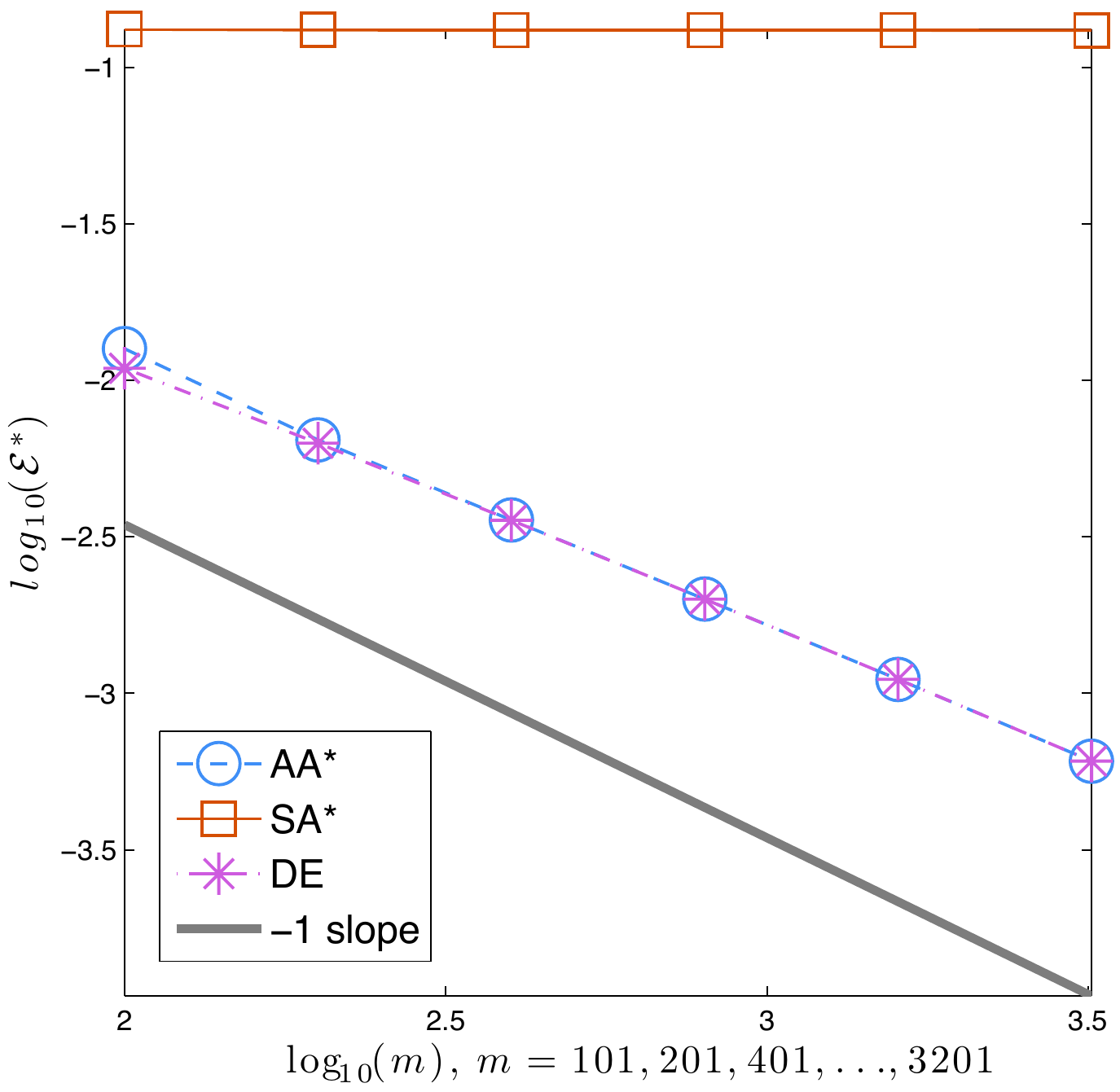}
\end{tabular}
\end{adjustwidth}
\vspace{-.25cm}
\caption{\footnotesize The CPU-time, the fraction $\mathcal{P}$ of the domain computed, and the error $\astarErr$ for both SA* and AA* using a constant speed function in 2D. The solid %orange 
square markers in the time plot indicate the time for a version of SA*-FMM that stores each $\varphi(\bx)$ after it is first computed. The underestimate function used is $\varphi^0$ and
The benchmarking is performed for $\lambda = 1$ (i.e., corresponding to the very top slice in Figure \ref{fig:constConv}).}
\label{fig:constNaive}
\end{center}
\end{figure}

% COMMENT REMOVED
\figstart
\begin{center}
\tabcolsep=0.05cm
{\bf 3D constant speed: Statistics.}
% COMMENT REMOVED
\begin{adjustwidth}{-1.5cm}{}
\begin{tabular}{c c c}
A. {\em Time (sec)} &
B. {\em Fraction $\mathcal{P}$} &
C. {\em Error $\log_{10}(\astarErr)$} \\
\includegraphics[scale=\statsScale]{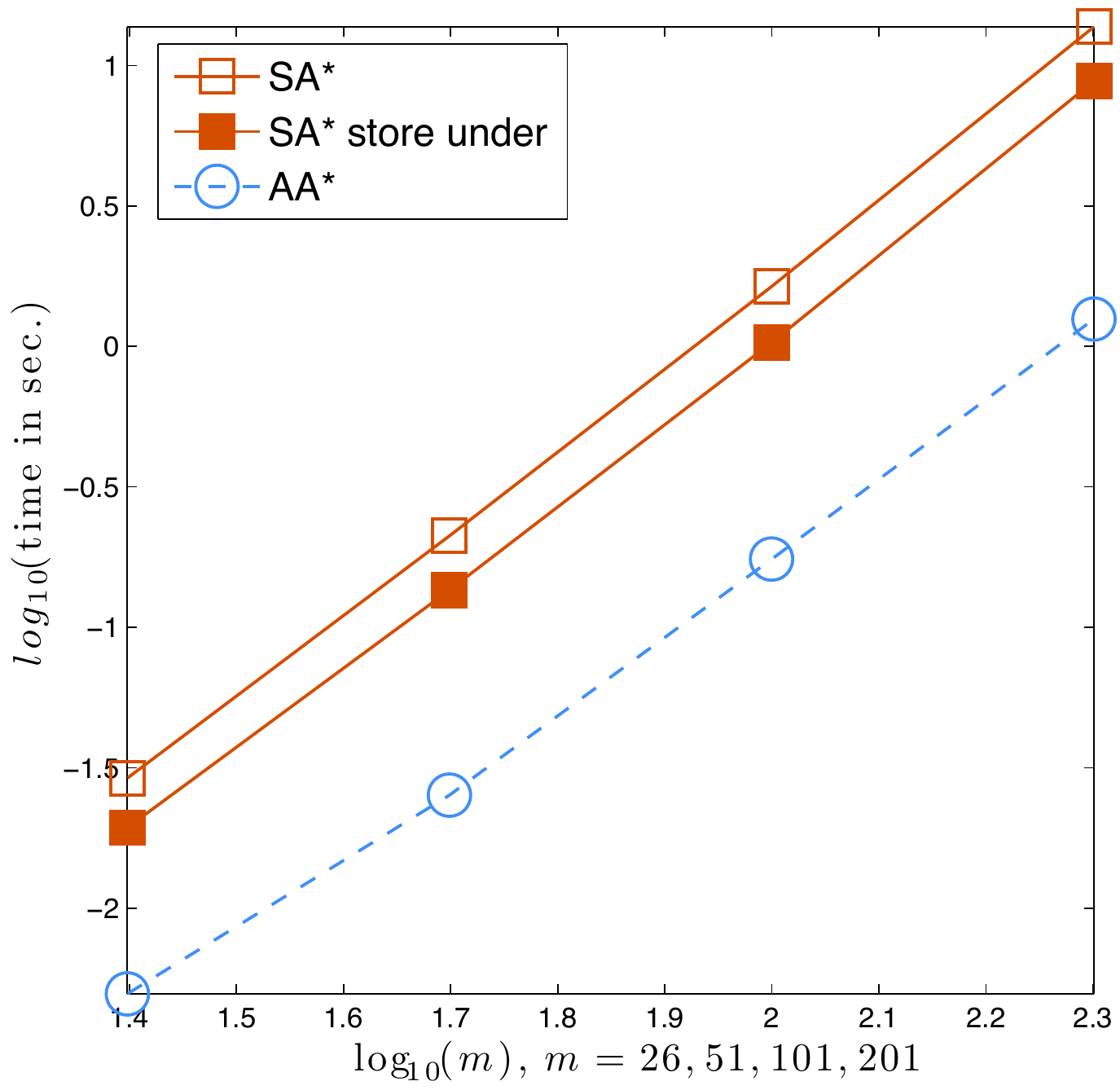} &
\includegraphics[scale=\statsScale]{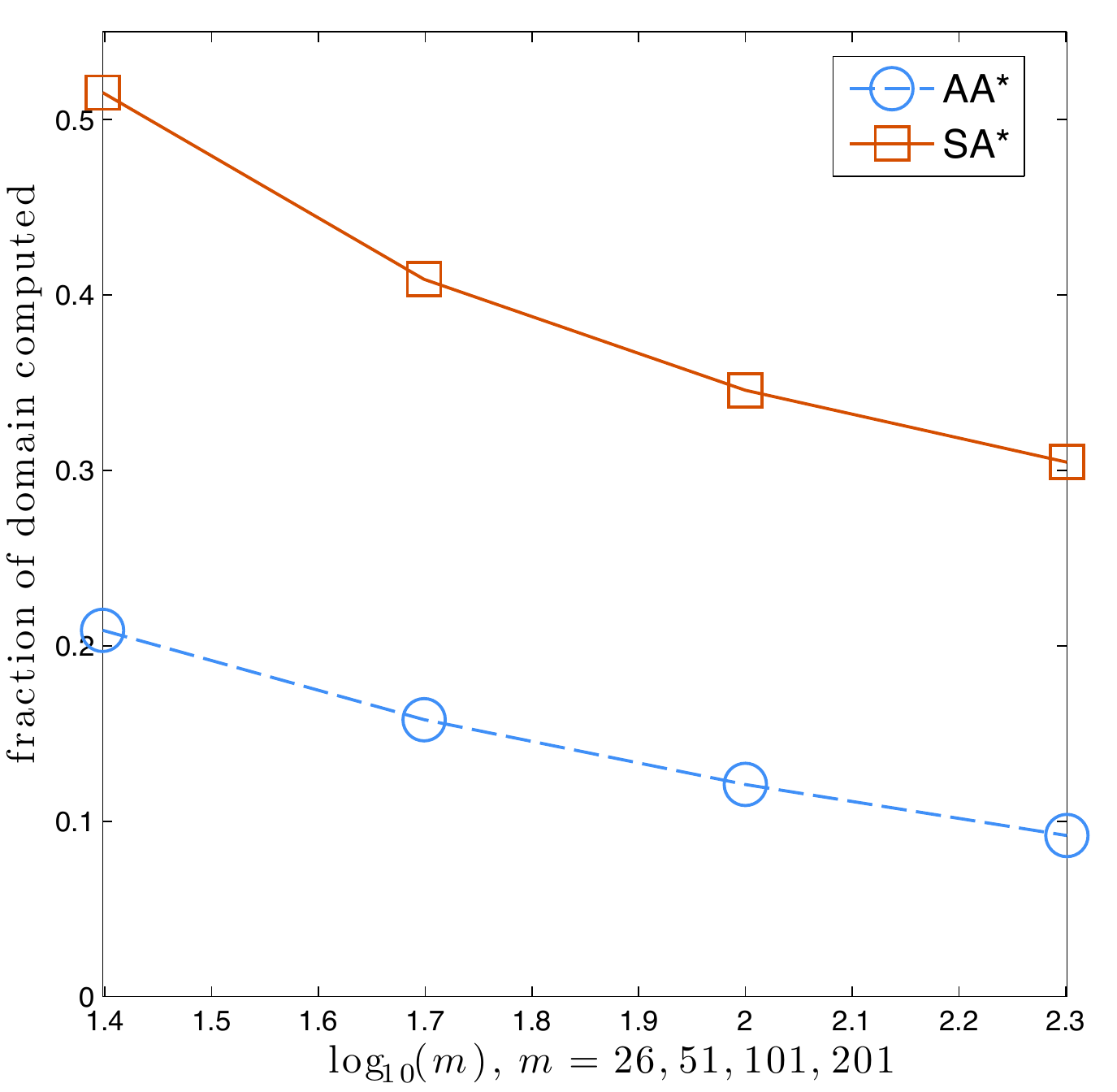} &
\includegraphics[scale=\statsScale]{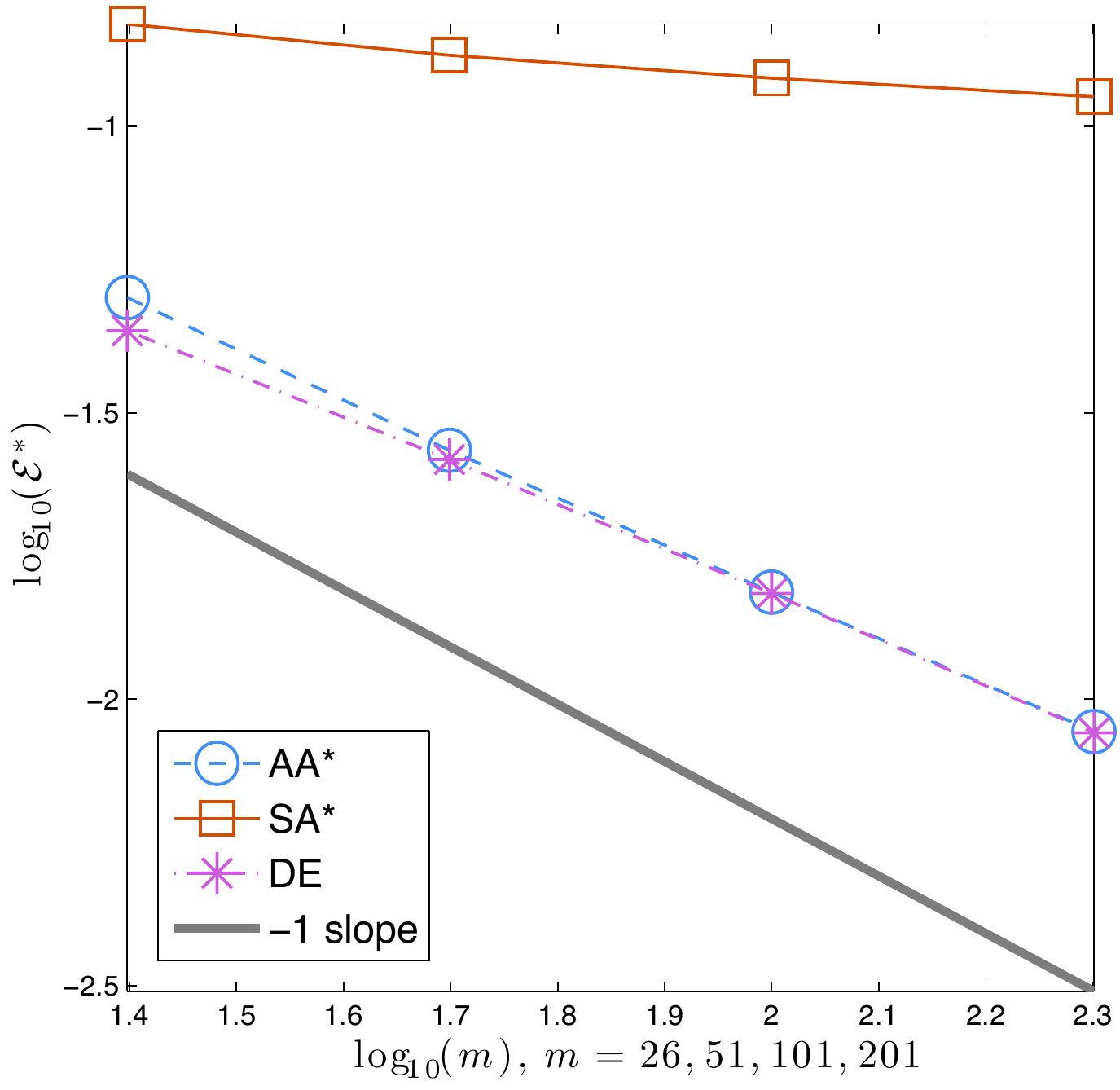}
\end{tabular}
\end{adjustwidth}
% COMMENT REMOVED
\vspace{-.25cm}
\caption{\footnotesize The same data as in Figure \ref{fig:constNaive},  but for 3D computations.
% COMMENT REMOVED
}
\label{fig:const3dNaive}
\end{center}
\end{figure}

% COMMENT REMOVED
% COMMENT REMOVED
% COMMENT REMOVED
% COMMENT REMOVED
% COMMENT REMOVED
\subsection{Oscillatory speed function in 2D and 3D}
\label{ss:sin_speed}
For the next 2D example, we set $\bs = (0.95, 0.7)$ and $\bt = (0.5, 0.5)$
and consider a highly oscillatory speed
\begin{equation}
f(x,y) \hspace{0.435cm} = \ 1 \ + \ 0.5\sin(20\pi x) \sin(20\pi y) \label{2D sin},
\end{equation}
resulting in frequent directional changes along most optimal paths.
We start by focusing on a scaled oracle heuristic
$\varphi = \bar{\varphi}_{\lambda}$ with AA*-FMM also relying on $\Upper = (1+ \epsilon_{tol}h^{\mu})v(\bs)$.
Figure \ref{fig:sinOracle} shows the level sets of numerical solutions obtained with $m=401$.
We note that the SA*-errors result in a significant distortion of the optimal trajectory (see the switch between $\lambda = 0.3$ and $\lambda = 0.7$).

\iftoggle{usecolor}{%

% COMMENT REMOVED
\figstart
\begin{center}
\begin{adjustwidth}{-1.25cm}{}
\tabcolsep=0.05cm
\begin{tabular}{c c c c c c c}% c}
&
\multicolumn{6}{c}{\bf 2D sinusoid speed: Solutions with A* using $\bar{\varphi}_{\lambda}$} \\
% COMMENT REMOVED
&
\small {\em$\mathit{\lambda = 0.00}$} &
\small {\em$\mathit{\lambda = 0.10}$} &
\small {\em$\mathit{\lambda = 0.30}$} &
\small {\em$\mathit{\lambda = 0.70}$} &
\small {\em$\mathit{\lambda = 0.90}$} &
\small {\em$\mathit{\lambda = 1.00}$} \\
\begin{sideways}\textbf{\Large\hspace{1cm}SA*}\end{sideways}&
\includegraphics[scale=\sinPicScale]{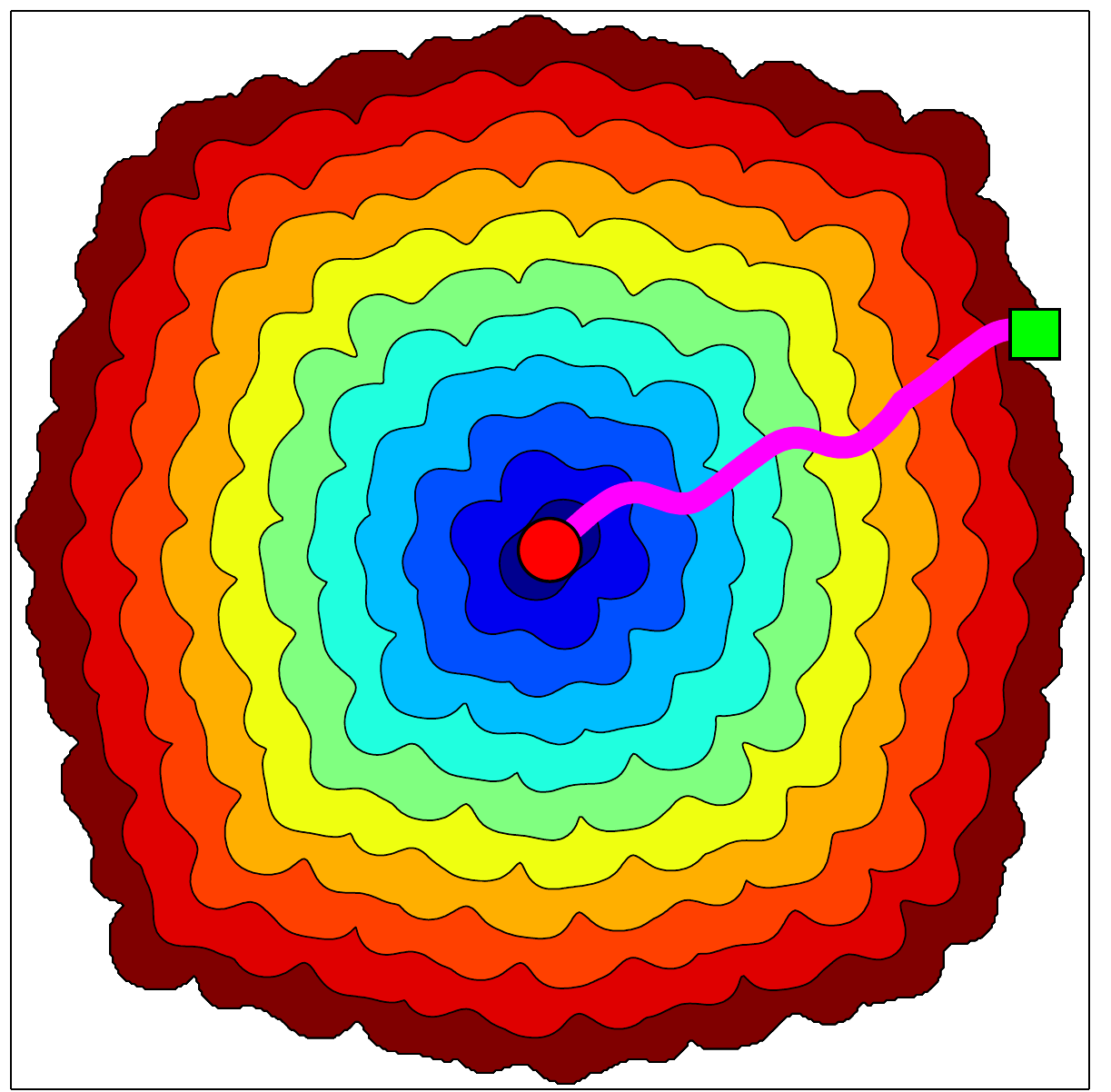} &
\includegraphics[scale=\sinPicScale]{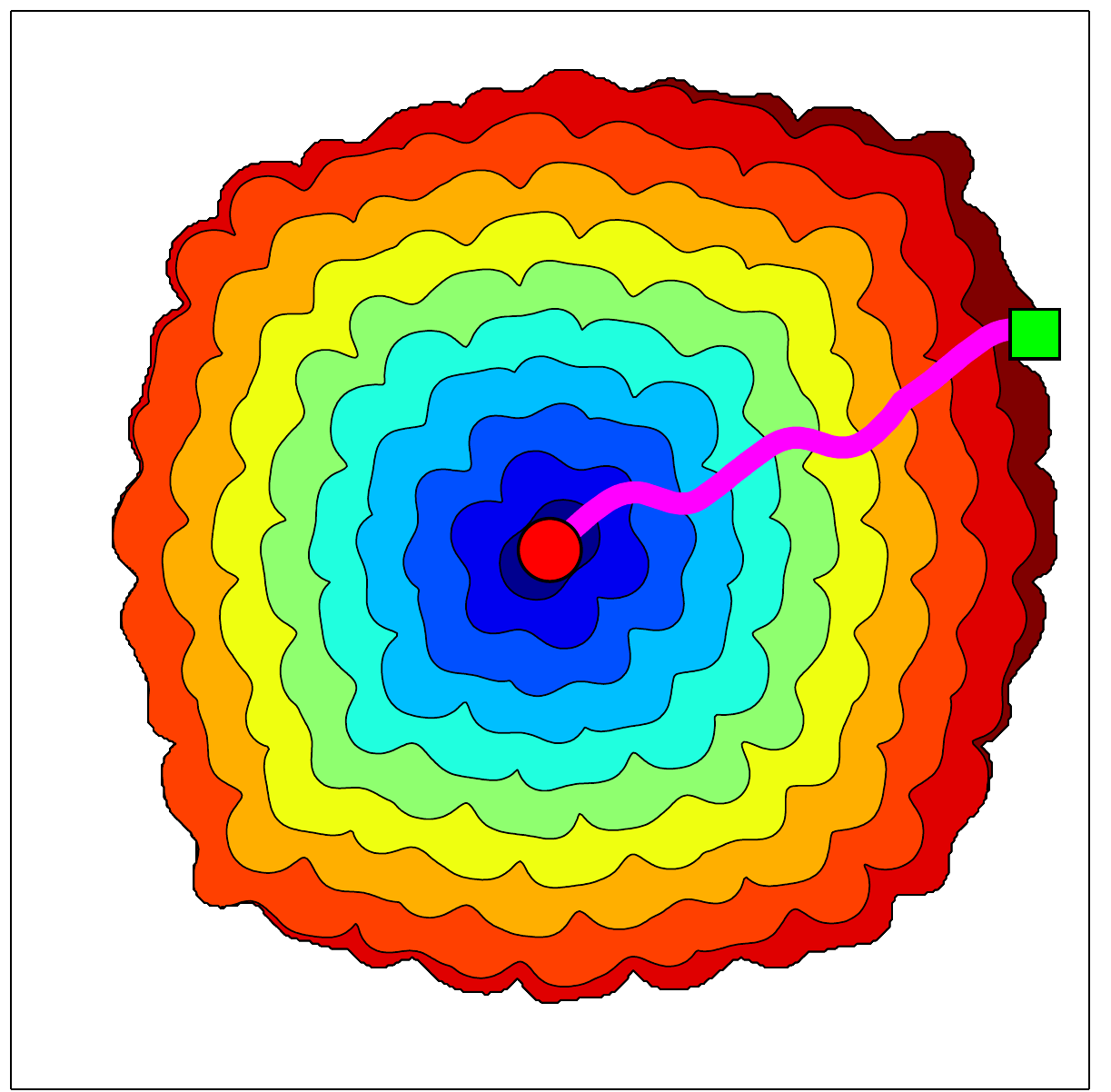} &
\includegraphics[scale=\sinPicScale]{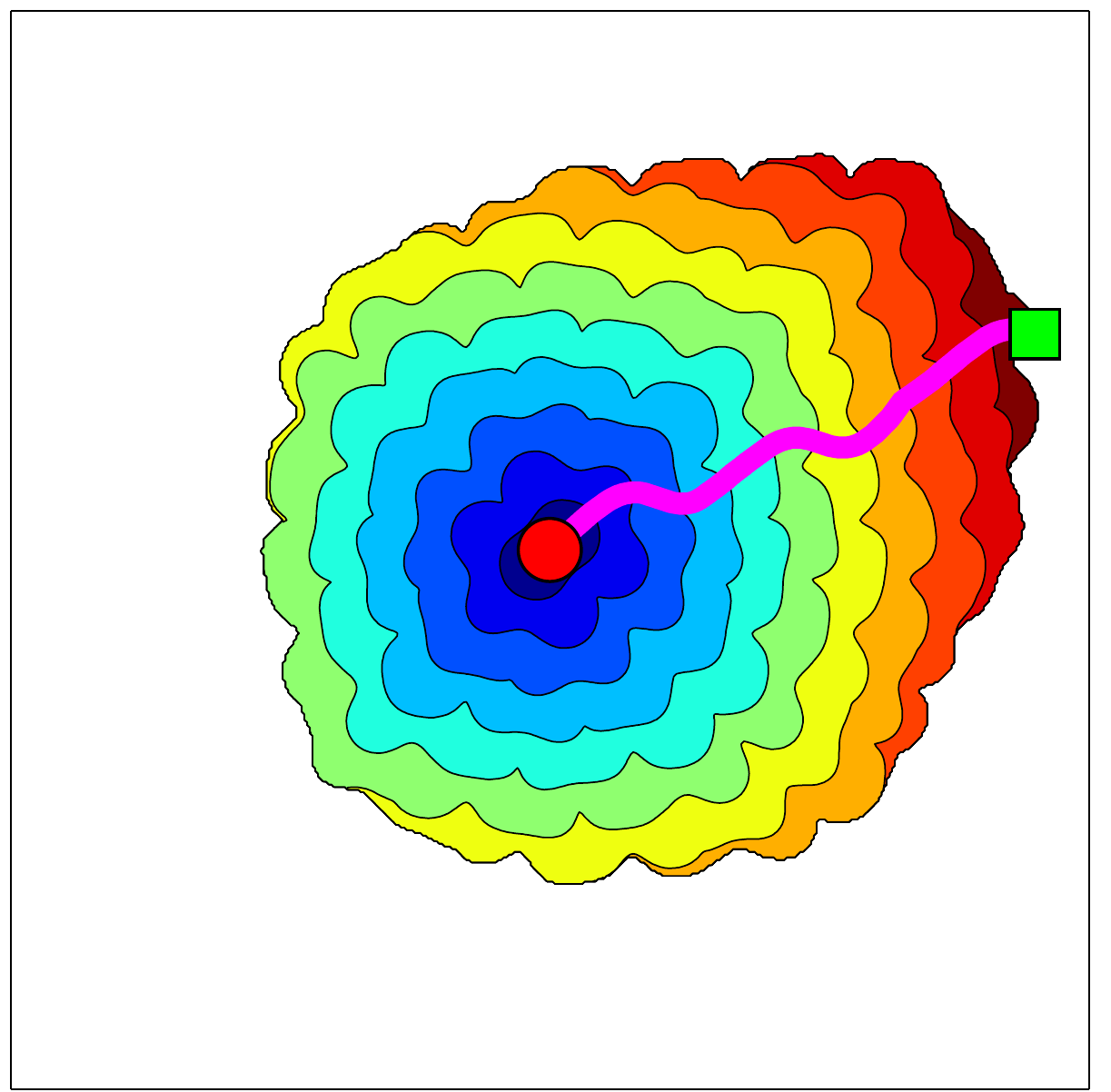} &
\includegraphics[scale=\sinPicScale]{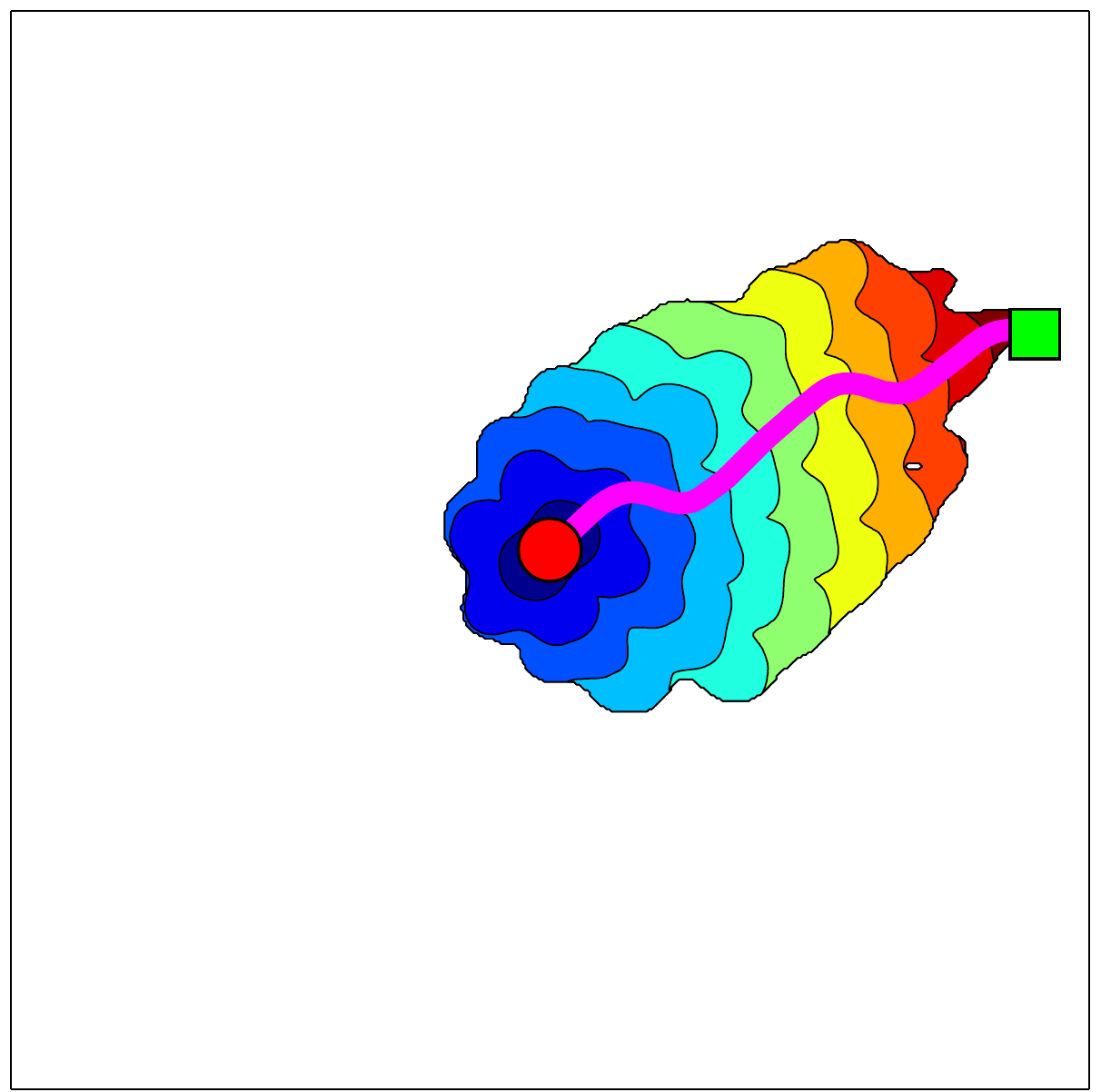} &
\includegraphics[scale=\sinPicScale]{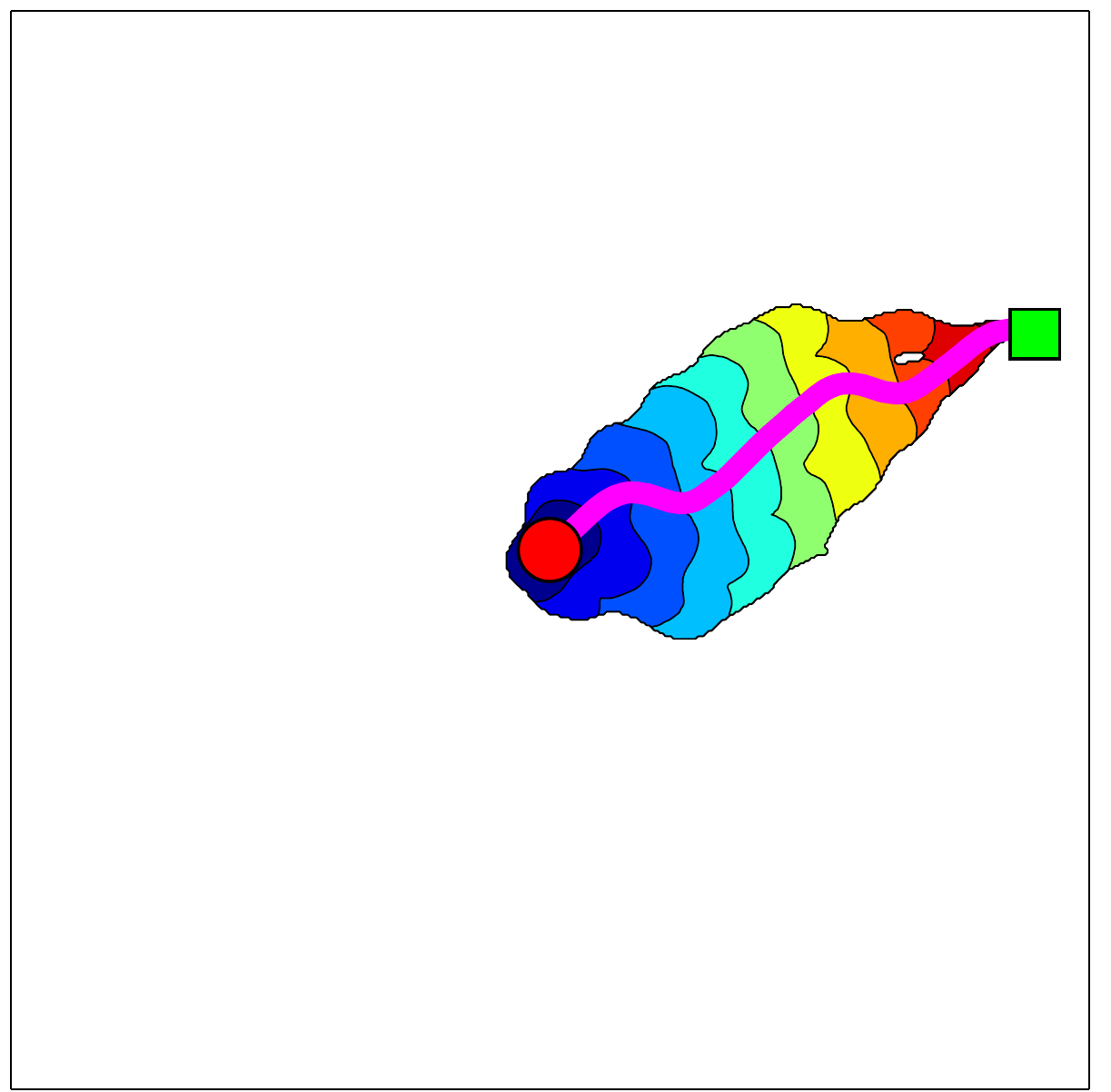} &
\includegraphics[scale=\sinPicScale]{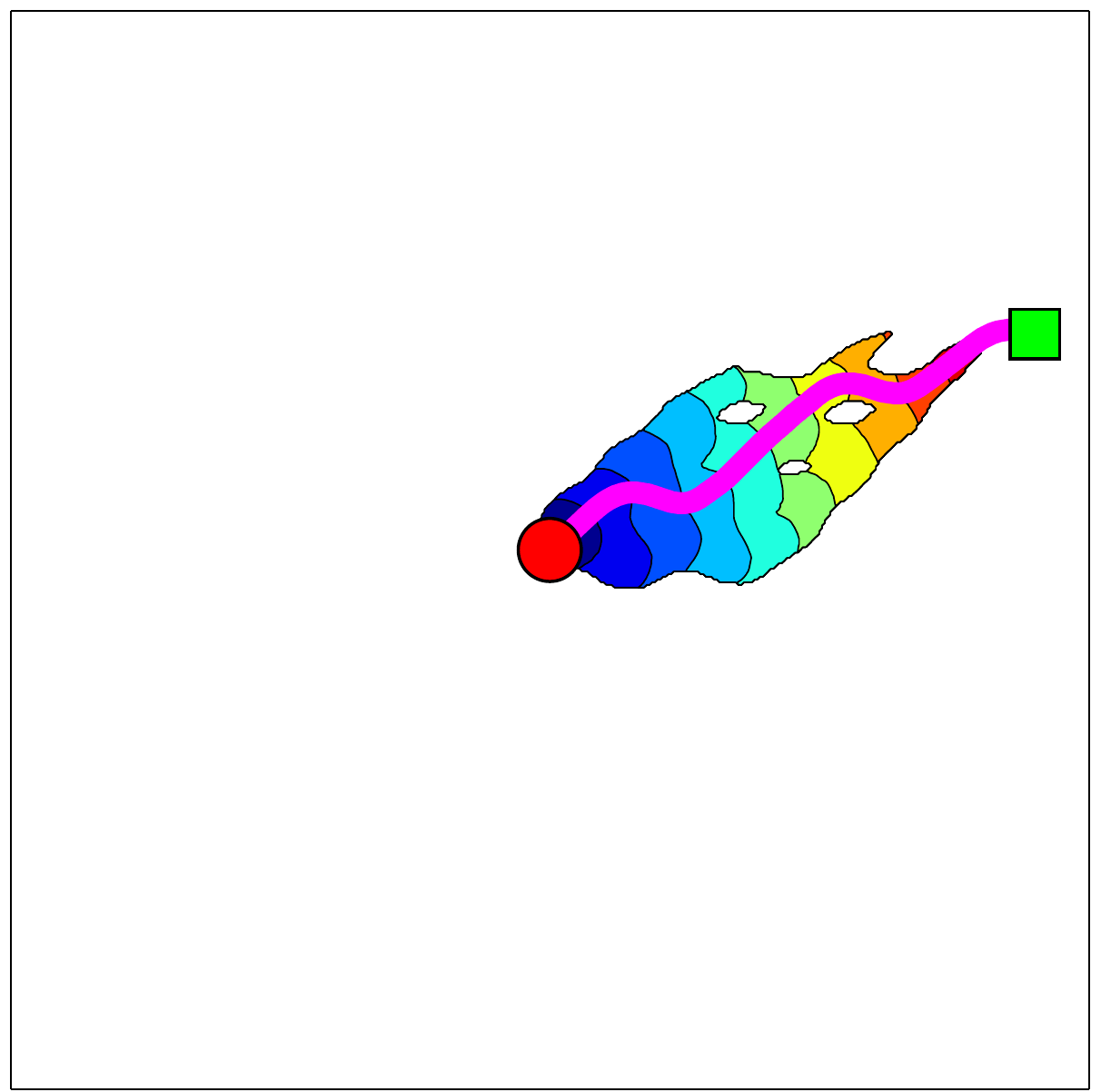} \\ [-1ex]
&
\footnotesize$\mathcal{P}=0.787$, $\astarErrOnly=0$ &
\footnotesize$\mathcal{P}=0.614$, $\astarErrOnly=10^{-6}$ &
\footnotesize$\mathcal{P}=0.385$, $\astarErrOnly=10^{-4}$ &
\footnotesize$\mathcal{P}=0.148$, $\astarErrOnly\approx 0.015$ &
\footnotesize$\mathcal{P}=0.079$, $\astarErrOnly=0.043$ &
\footnotesize$\mathcal{P}=0.050$, $\astarErrOnly=0.060$
\vspace{.1cm}
 \\

% COMMENT REMOVED
% COMMENT REMOVED
% COMMENT REMOVED
% COMMENT REMOVED
% COMMENT REMOVED
% COMMENT REMOVED
% COMMENT REMOVED
\begin{sideways}\textbf{\Large\hspace{1cm}AA*}\end{sideways}&
\includegraphics[scale=\sinPicScale]{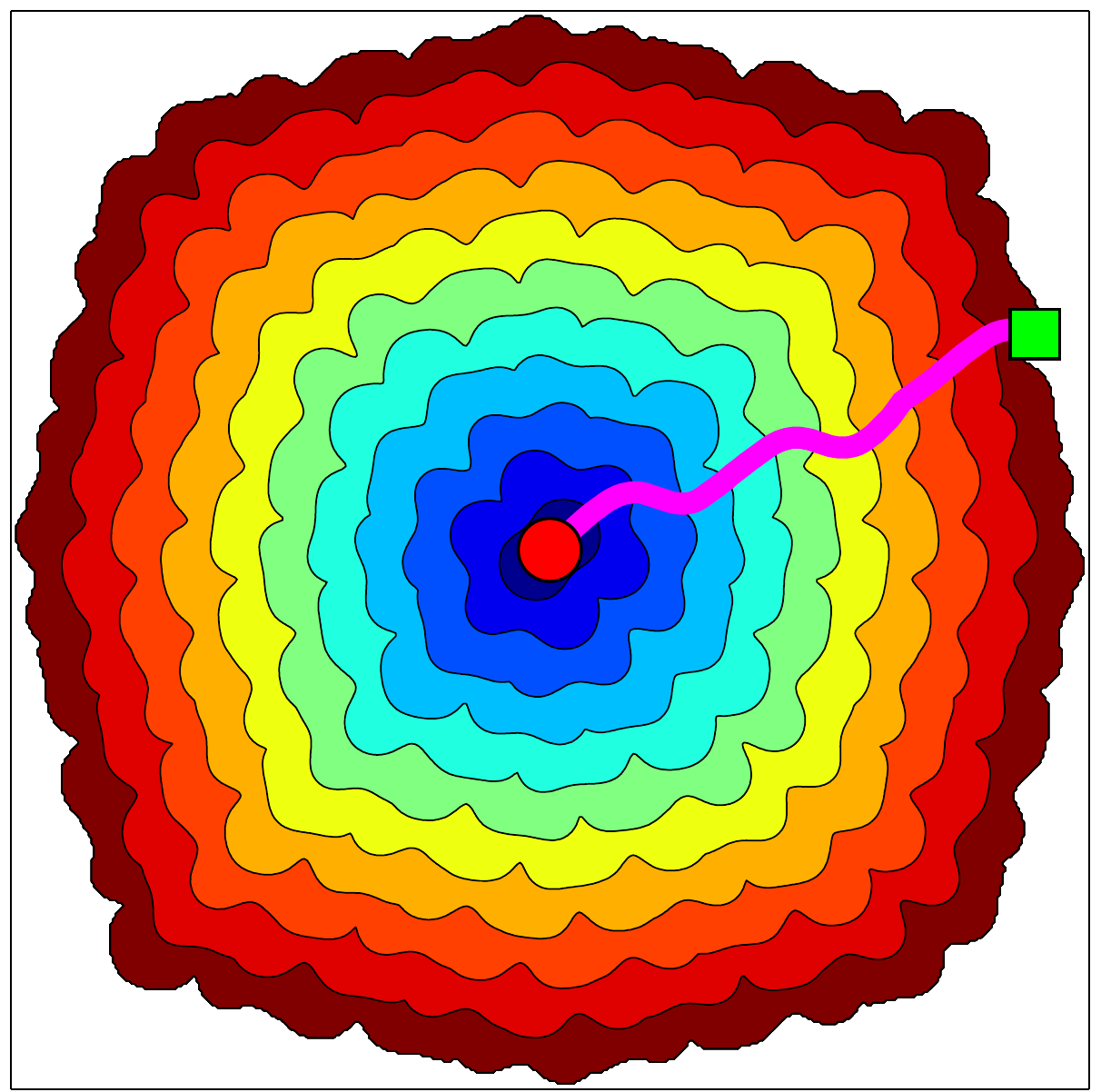} &
\includegraphics[scale=\sinPicScale]{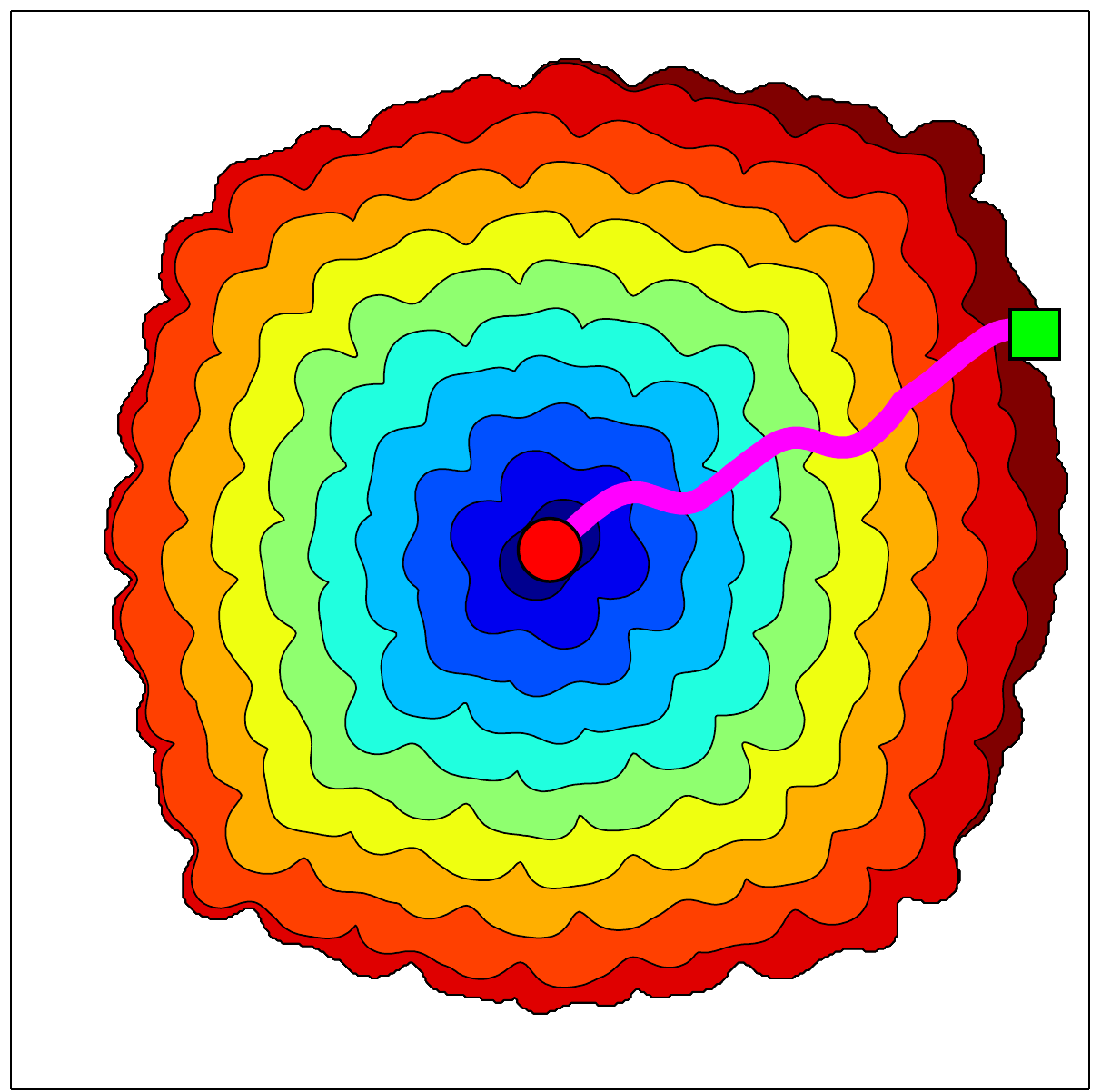} &
\includegraphics[scale=\sinPicScale]{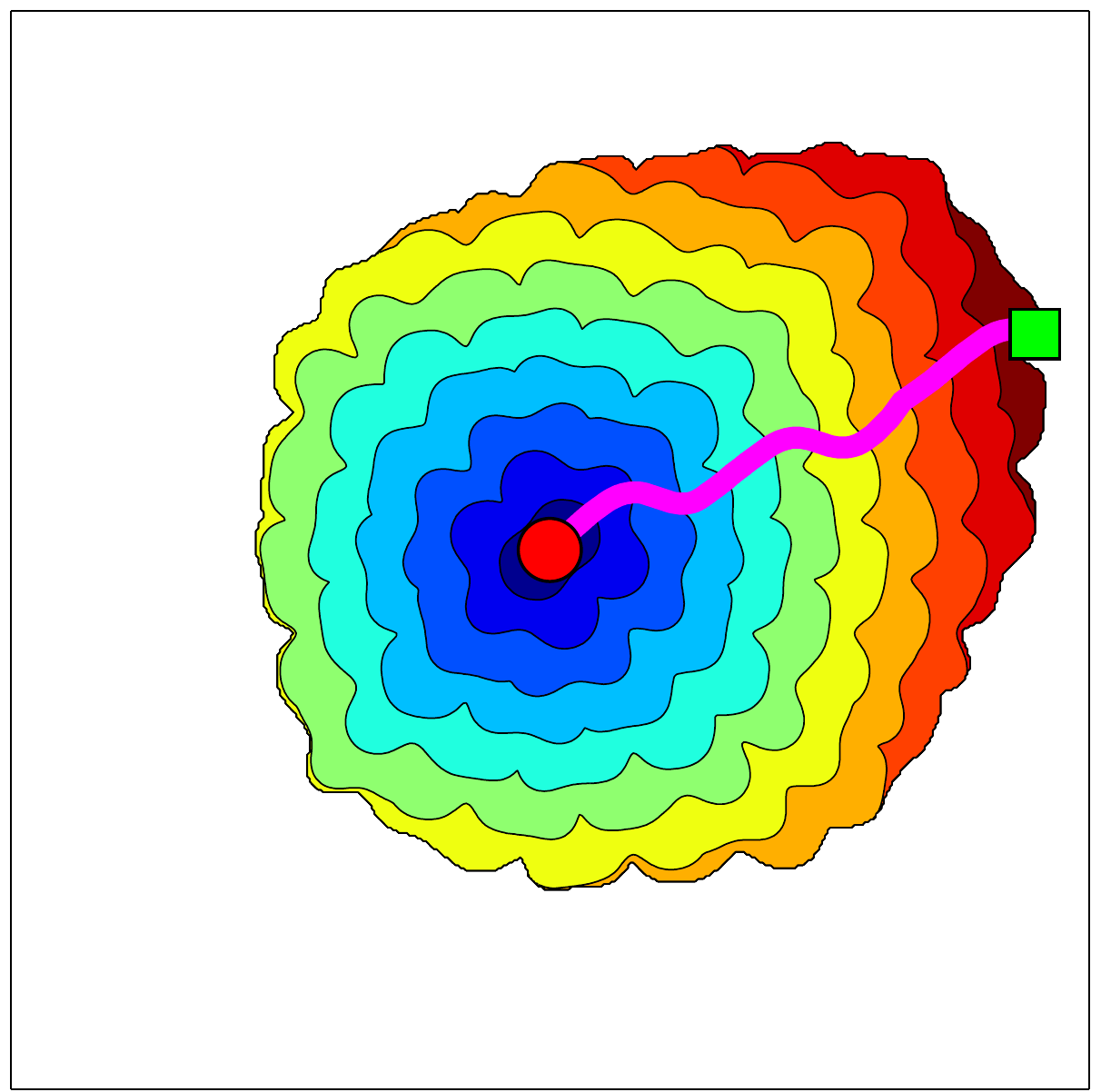} &
\includegraphics[scale=\sinPicScale]{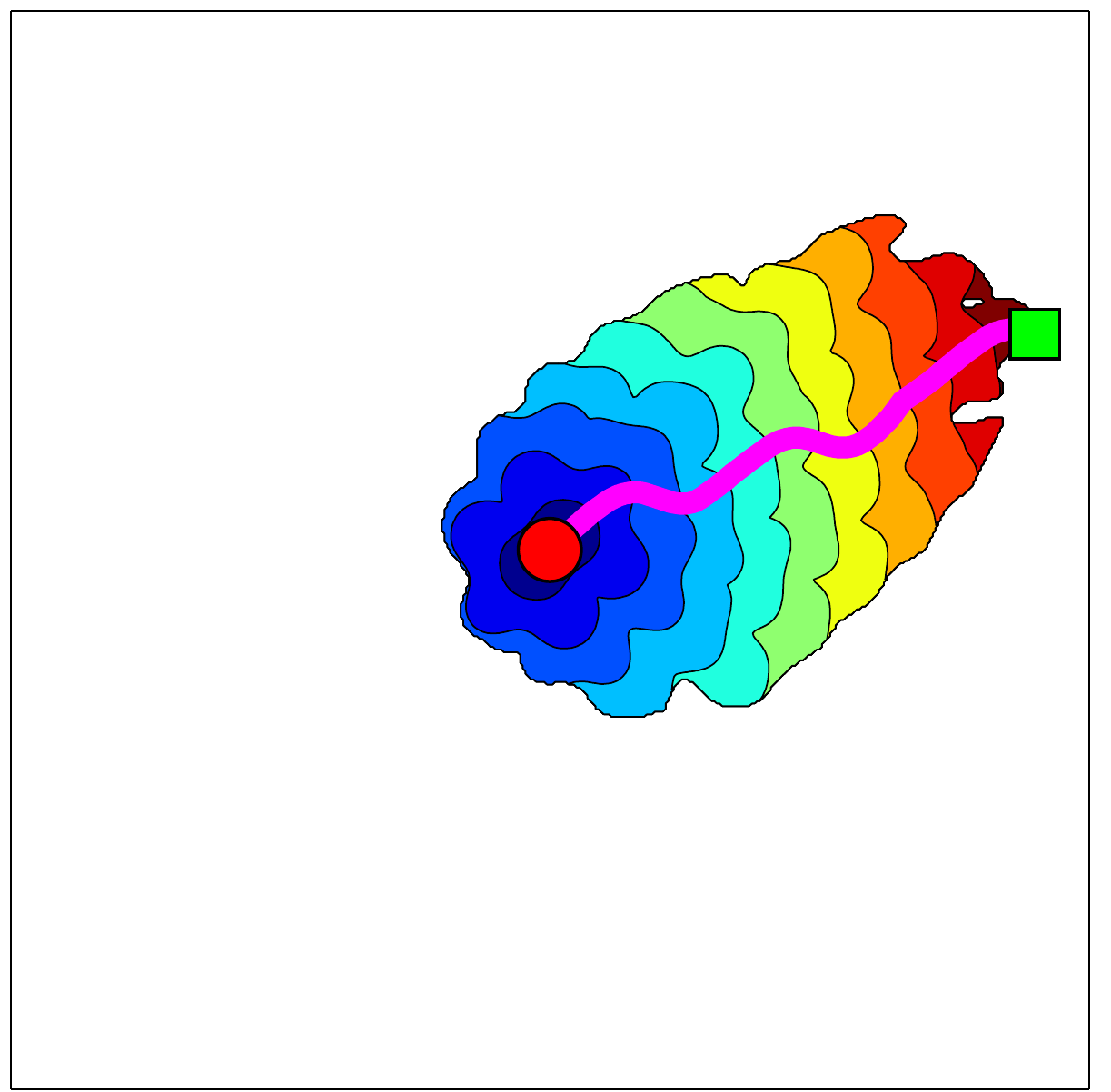} &
\includegraphics[scale=\sinPicScale]{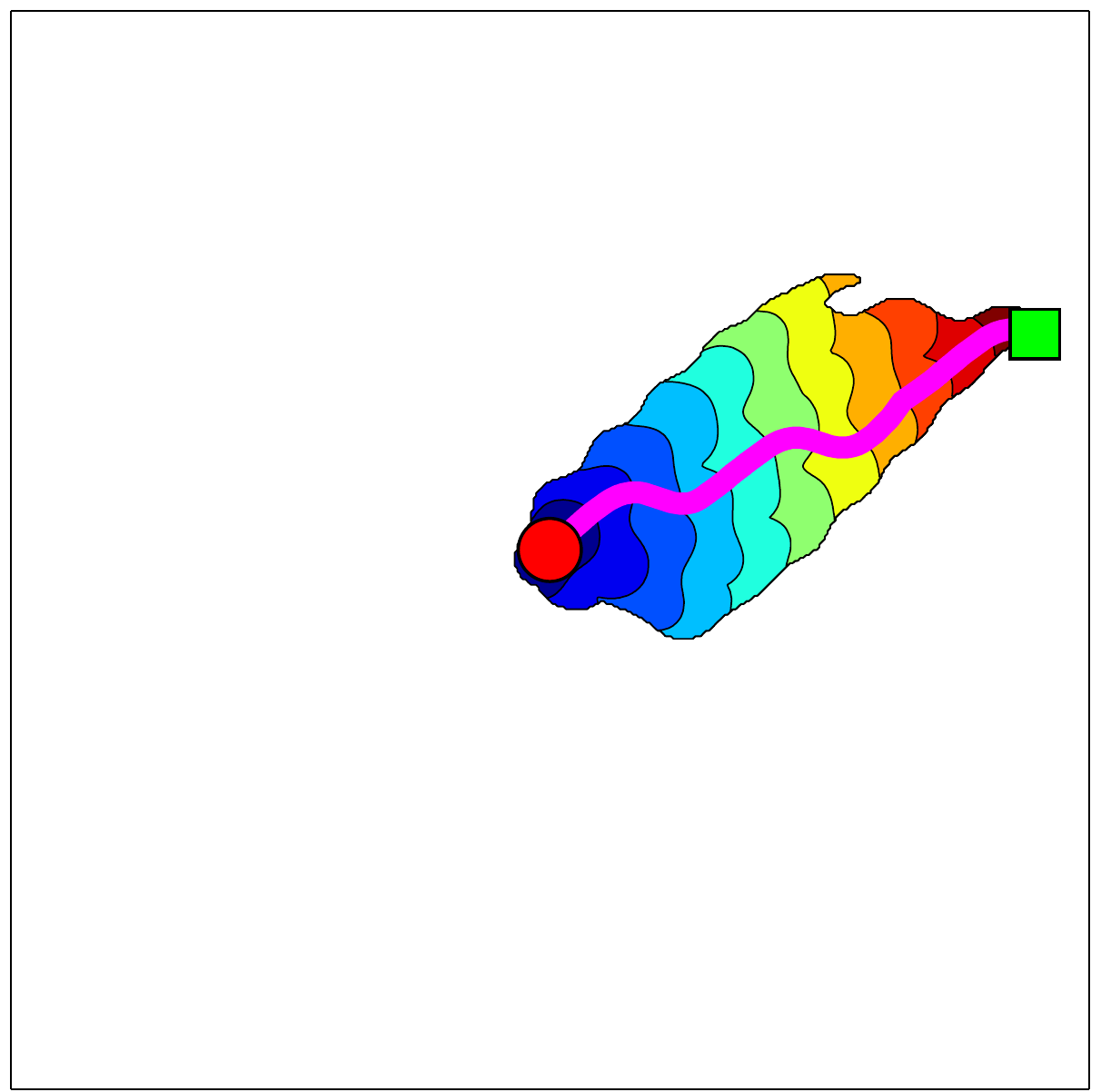} &
\includegraphics[scale=\sinPicScale]{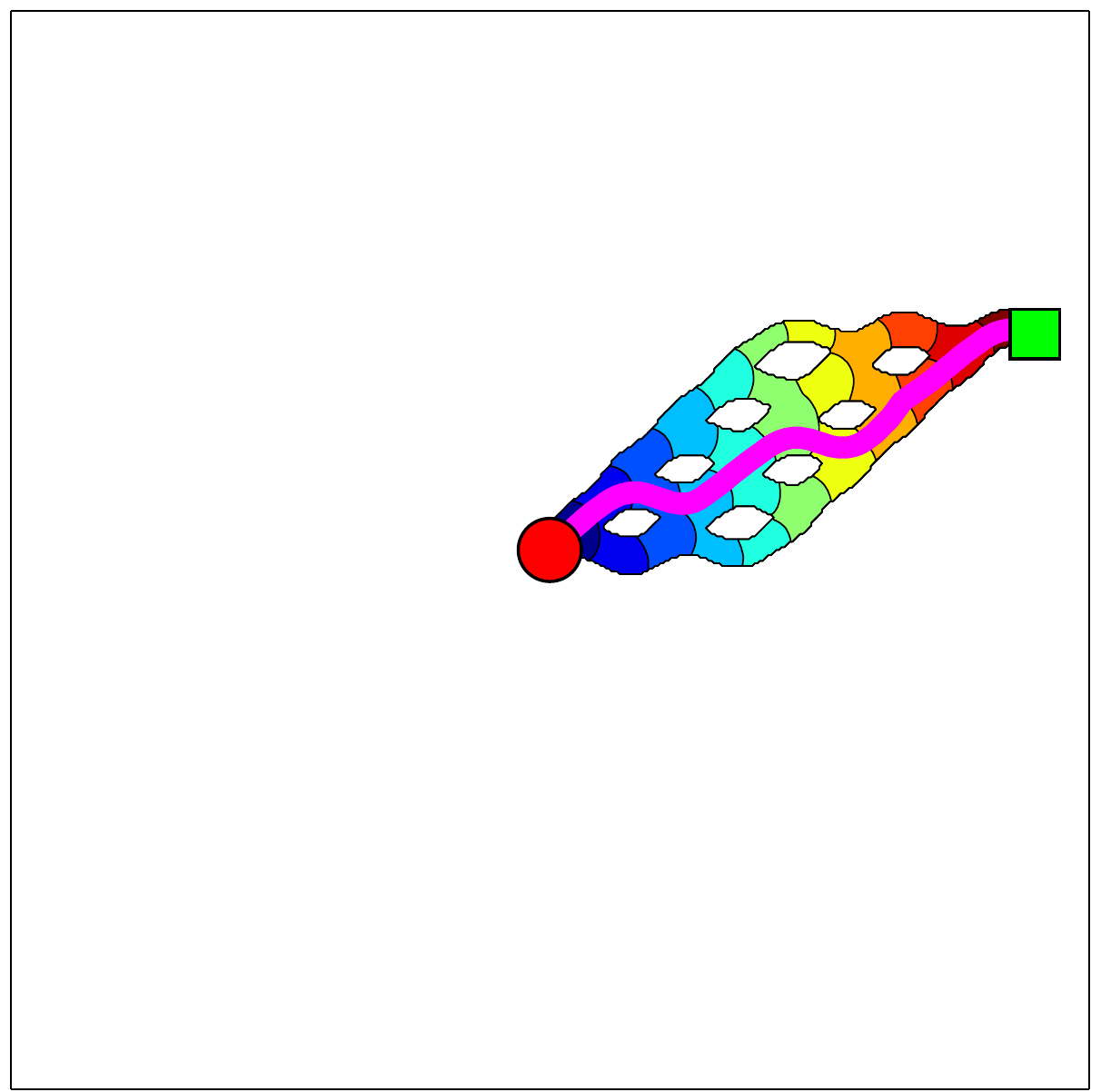} %&
% COMMENT REMOVED
 \\ [-1ex]
&
\footnotesize $\mathcal{P}=0.787$, $\astarErrOnly=0$ &
\footnotesize $\mathcal{P}=0.639$, $\astarErrOnly=0$ &
\footnotesize $\mathcal{P}=0.404$, $\astarErrOnly=0$ &
\footnotesize $\mathcal{P}=0.160$, $\astarErrOnly\approx 10^{-16}$ &
\footnotesize $\mathcal{P}=0.080$, $\astarErrOnly\approx10^{-10}$ &
\footnotesize $\mathcal{P}=0.048$, $\astarErrOnly\approx 10^{-5}$
\vspace{0.15cm} \\
&
\multicolumn{6}{c}{\includegraphics[scale=0.7]{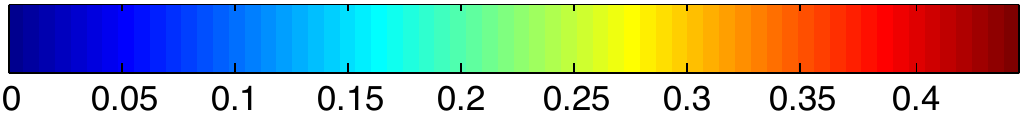}}
\end{tabular}
\end{adjustwidth}
\vspace{-.35cm}
\caption{\footnotesize Numerical results of FMM combined with SA* and AA*, showing the fraction of domain computed $\mathcal{P}$ and the relative error $\astarErrOnly$. Note the change in the ``optimal'' trajectory for SA* between $\lambda=0.3$ and $\lambda=0.70$. The solutions were produced using $m = 401$.}
\label{fig:sinOracle}
\end{center}
\end{figure}

}{%

% COMMENT REMOVED
\figstart
\begin{center}
\begin{adjustwidth}{-1.25cm}{}
\tabcolsep=0.05cm
\begin{tabular}{c c c c c c c}% c}
&
\multicolumn{6}{c}{\bf 2D sinusoid speed: Solutions with A* using $\bar{\varphi}_{\lambda}$} \\
% COMMENT REMOVED
&
\small {\em$\mathit{\lambda = 0.00}$} &
\small {\em$\mathit{\lambda = 0.10}$} &
\small {\em$\mathit{\lambda = 0.30}$} &
\small {\em$\mathit{\lambda = 0.70}$} &
\small {\em$\mathit{\lambda = 0.90}$} &
\small {\em$\mathit{\lambda = 1.00}$} \\
\begin{sideways}\textbf{\Large\hspace{1cm}SA*}\end{sideways}&
\includegraphics[scale=\sinPicScale]{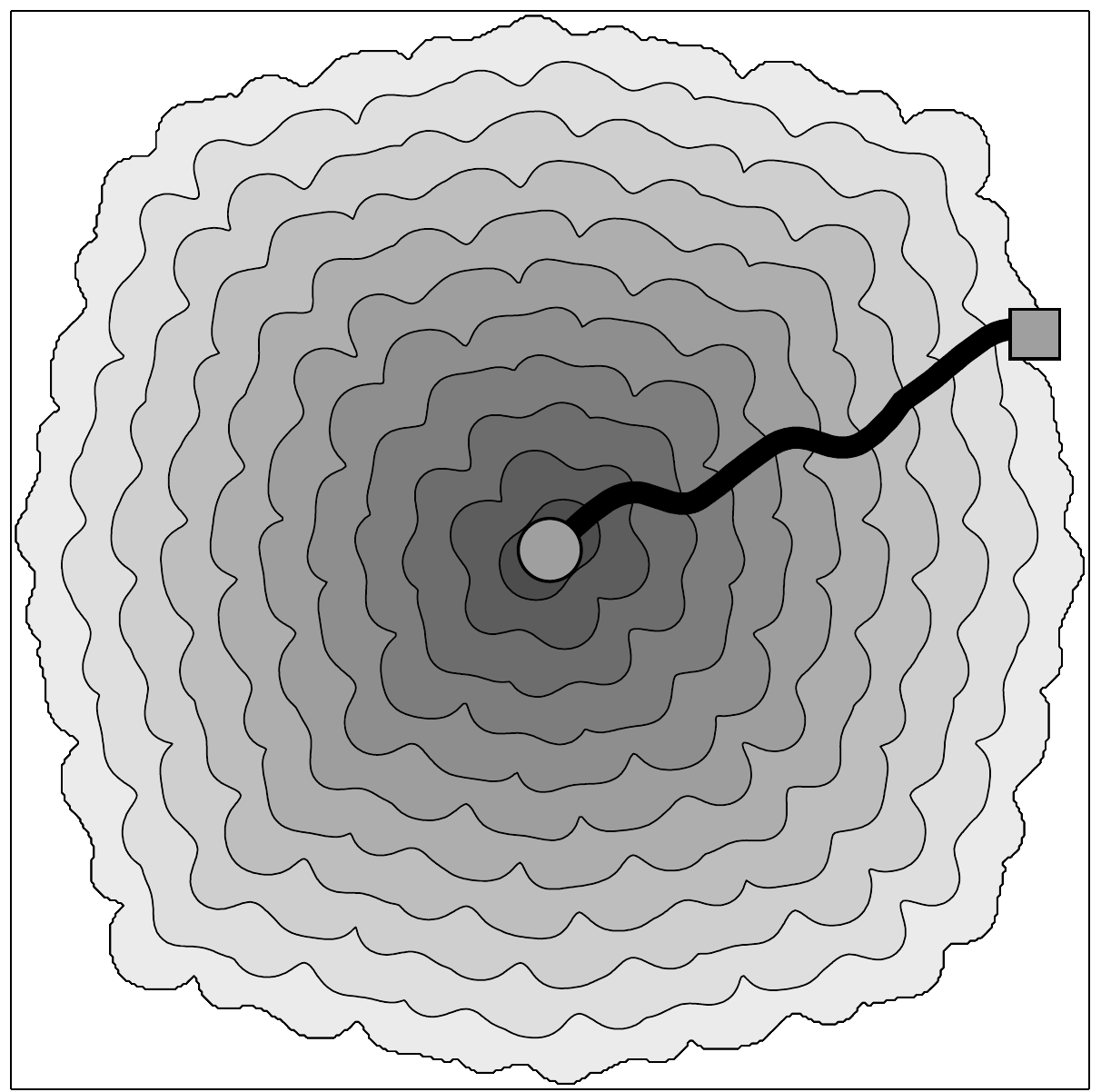} &
\includegraphics[scale=\sinPicScale]{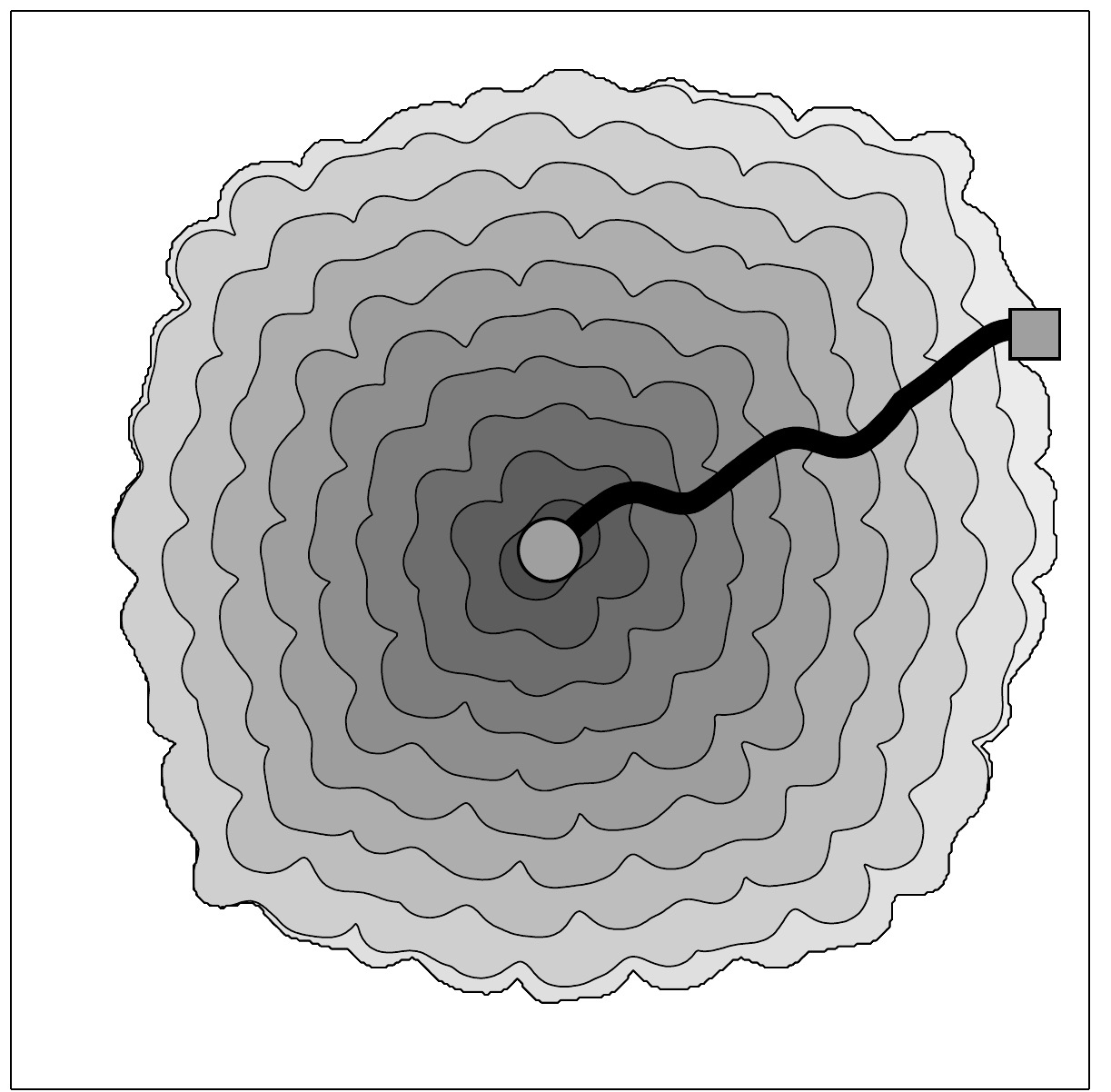} &
\includegraphics[scale=\sinPicScale]{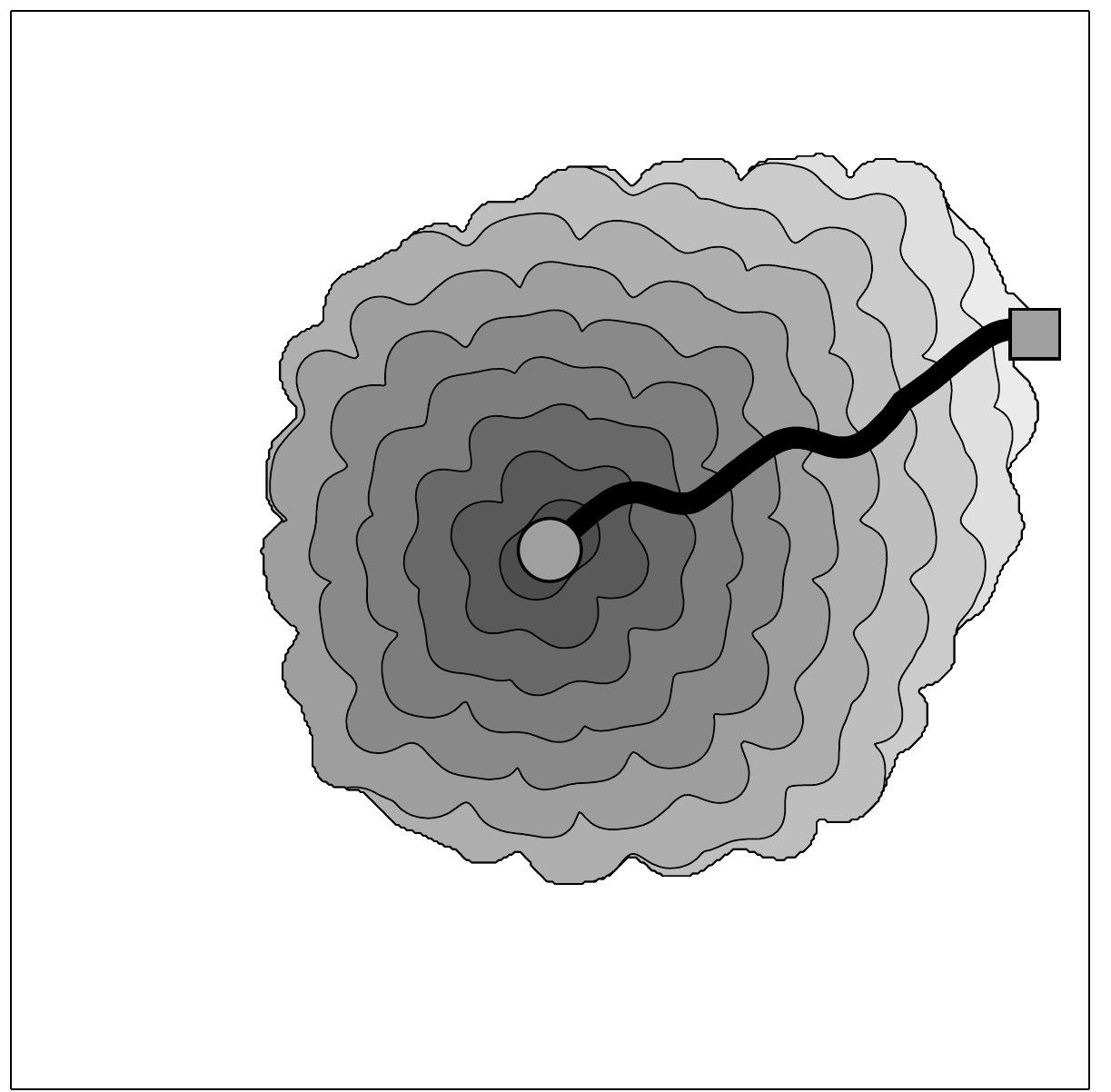} &
\includegraphics[scale=\sinPicScale]{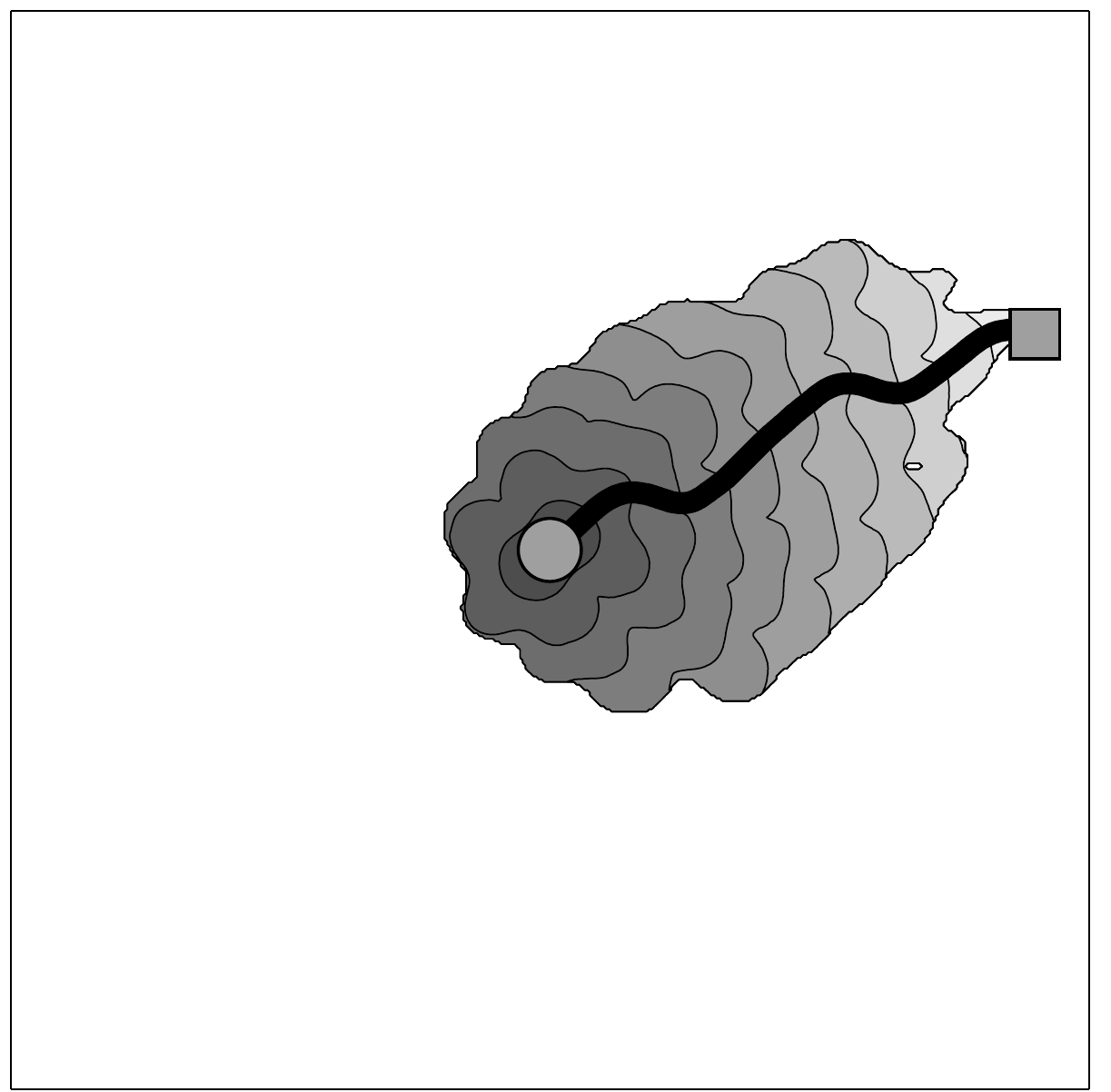} &
\includegraphics[scale=\sinPicScale]{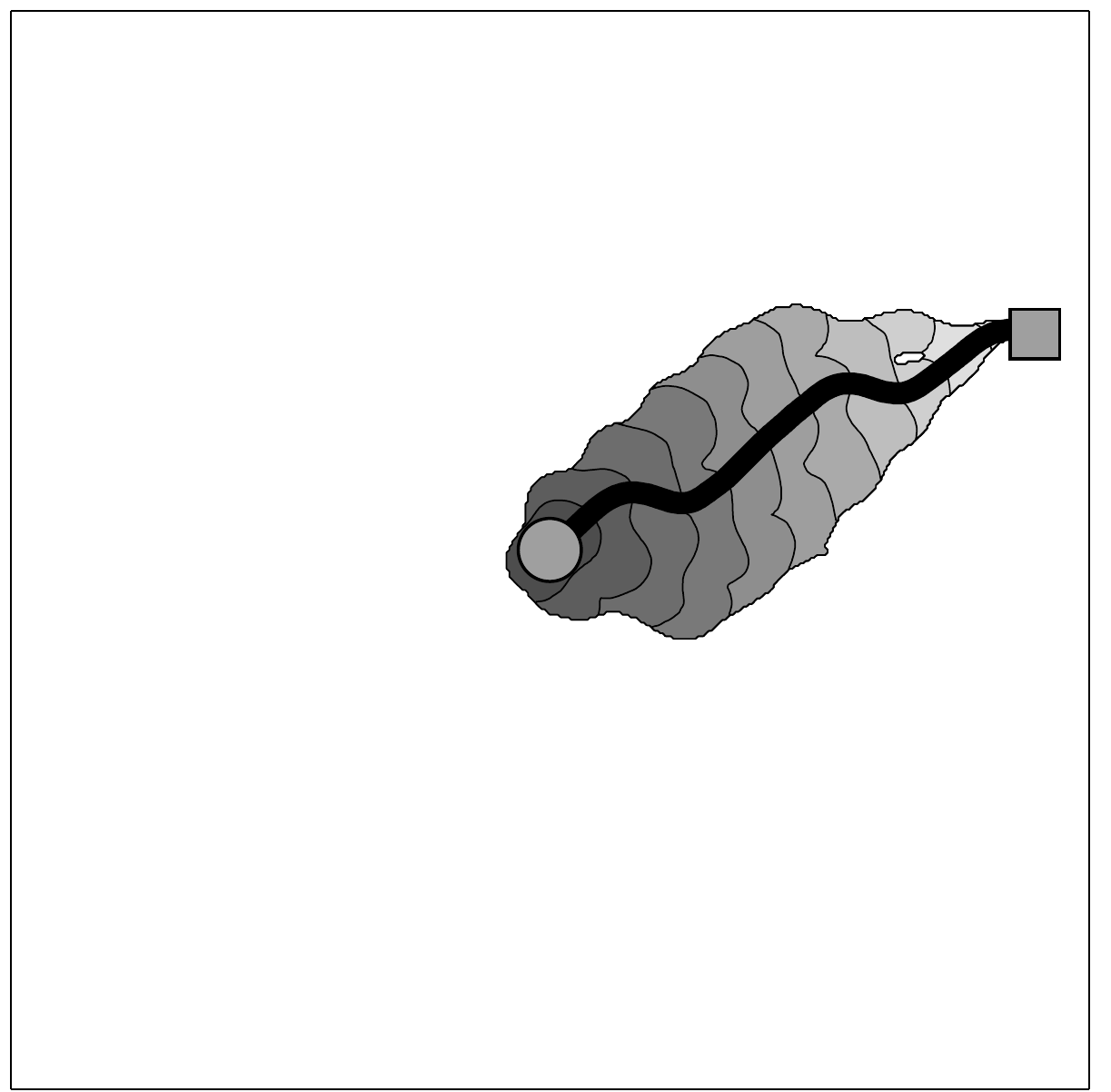} &
\includegraphics[scale=\sinPicScale]{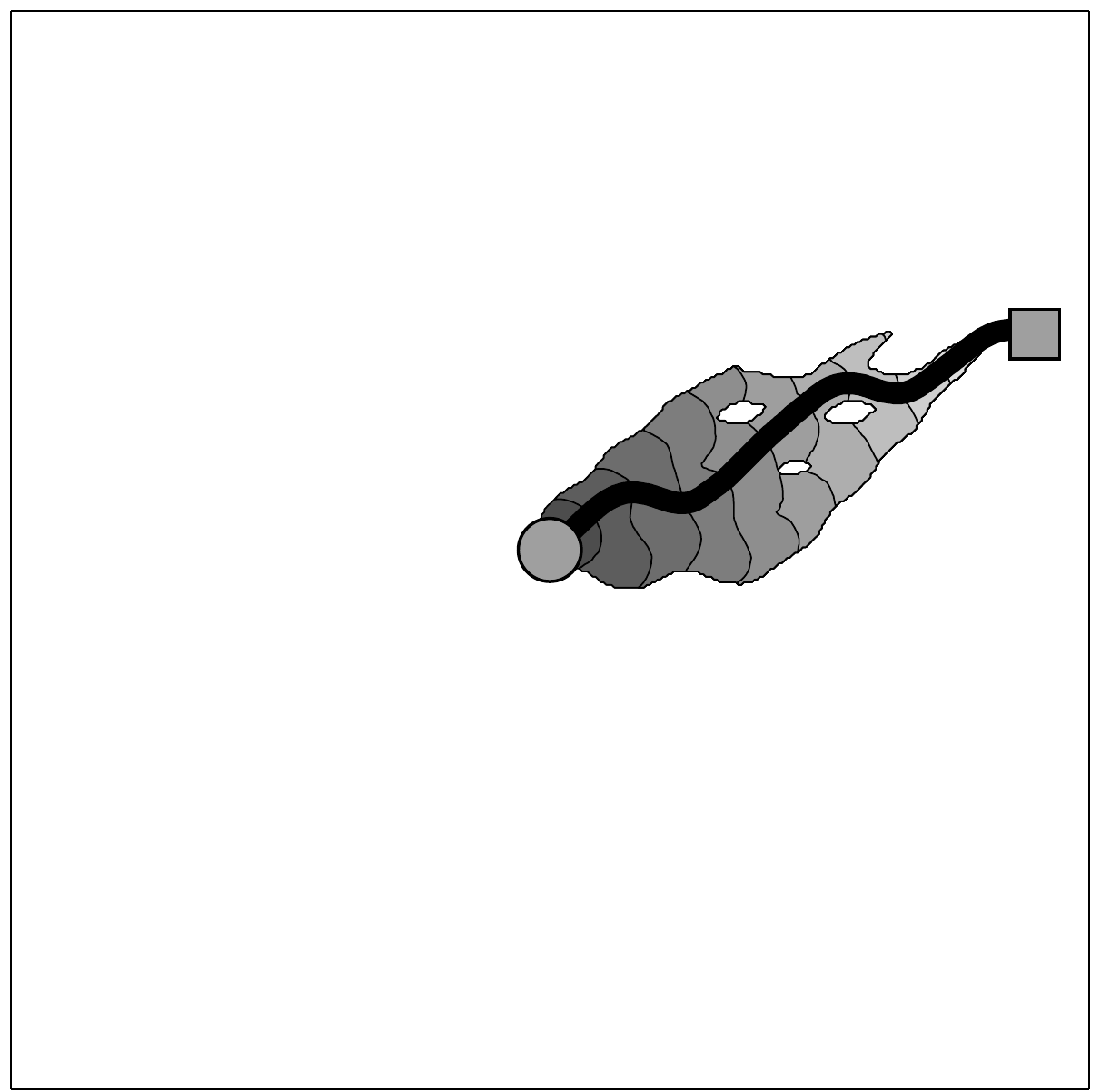} \\ [-1ex]
&
\footnotesize$\mathcal{P}=0.787$, $\astarErrOnly=0$ &
\footnotesize$\mathcal{P}=0.614$, $\astarErrOnly=10^{-6}$ &
\footnotesize$\mathcal{P}=0.385$, $\astarErrOnly=10^{-4}$ &
\footnotesize$\mathcal{P}=0.148$, $\astarErrOnly\approx 0.015$ &
\footnotesize$\mathcal{P}=0.079$, $\astarErrOnly=0.043$ &
\footnotesize$\mathcal{P}=0.050$, $\astarErrOnly=0.060$
\vspace{.1cm}
 \\

% COMMENT REMOVED
% COMMENT REMOVED
% COMMENT REMOVED
% COMMENT REMOVED
% COMMENT REMOVED
% COMMENT REMOVED
% COMMENT REMOVED
\begin{sideways}\textbf{\Large\hspace{1cm}AA*}\end{sideways}&
\includegraphics[scale=\sinPicScale]{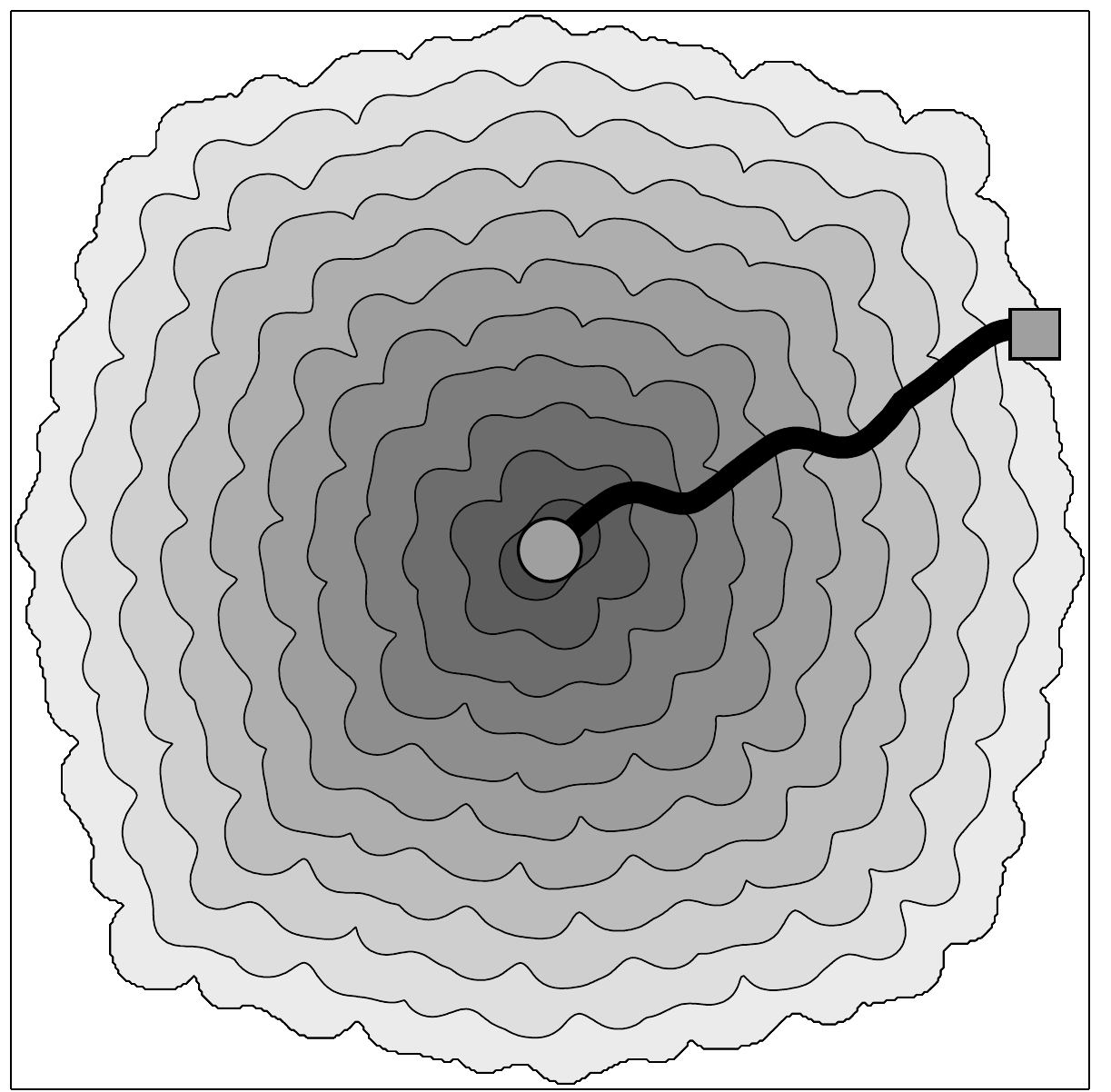} &
\includegraphics[scale=\sinPicScale]{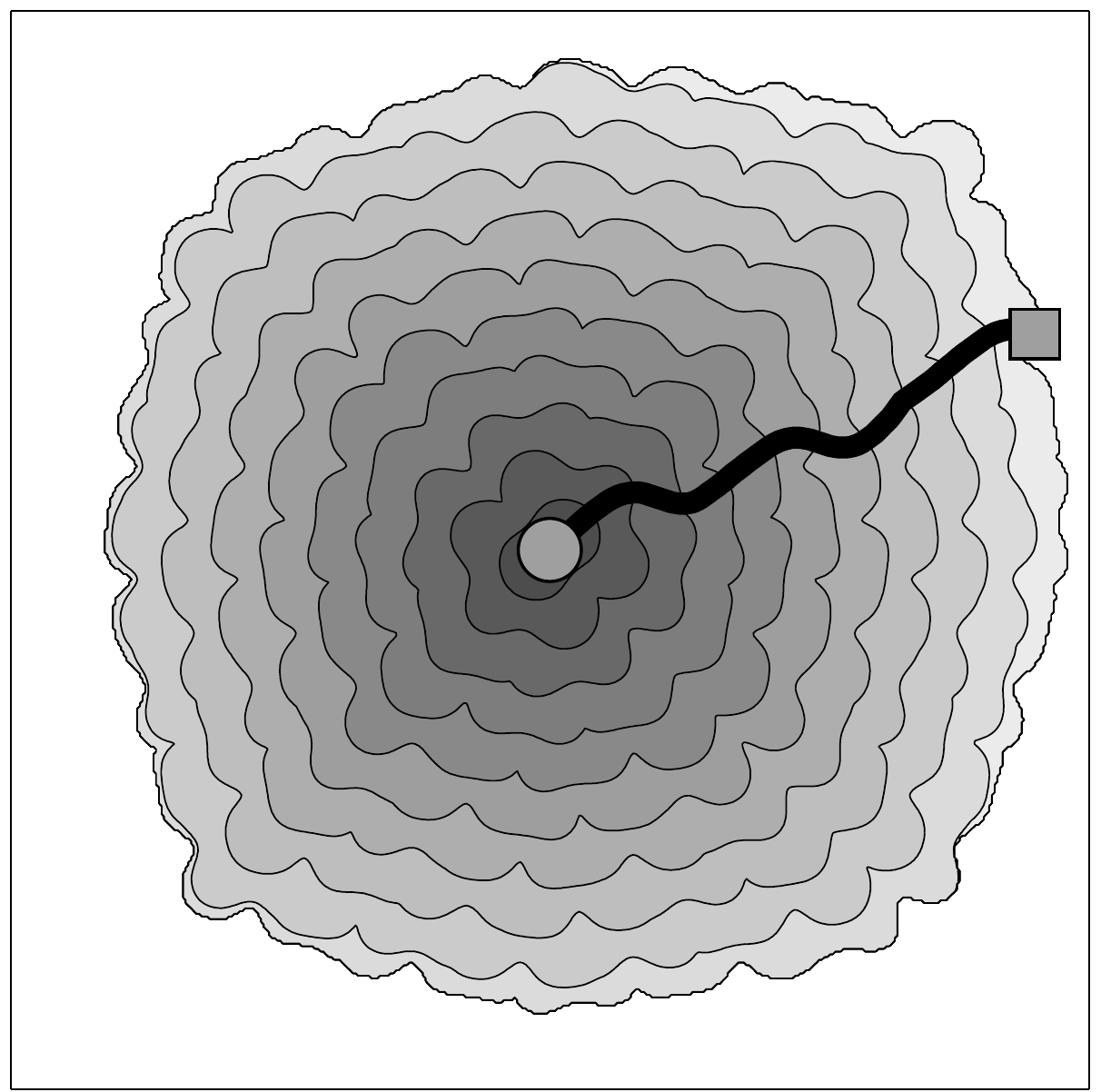} &
\includegraphics[scale=\sinPicScale]{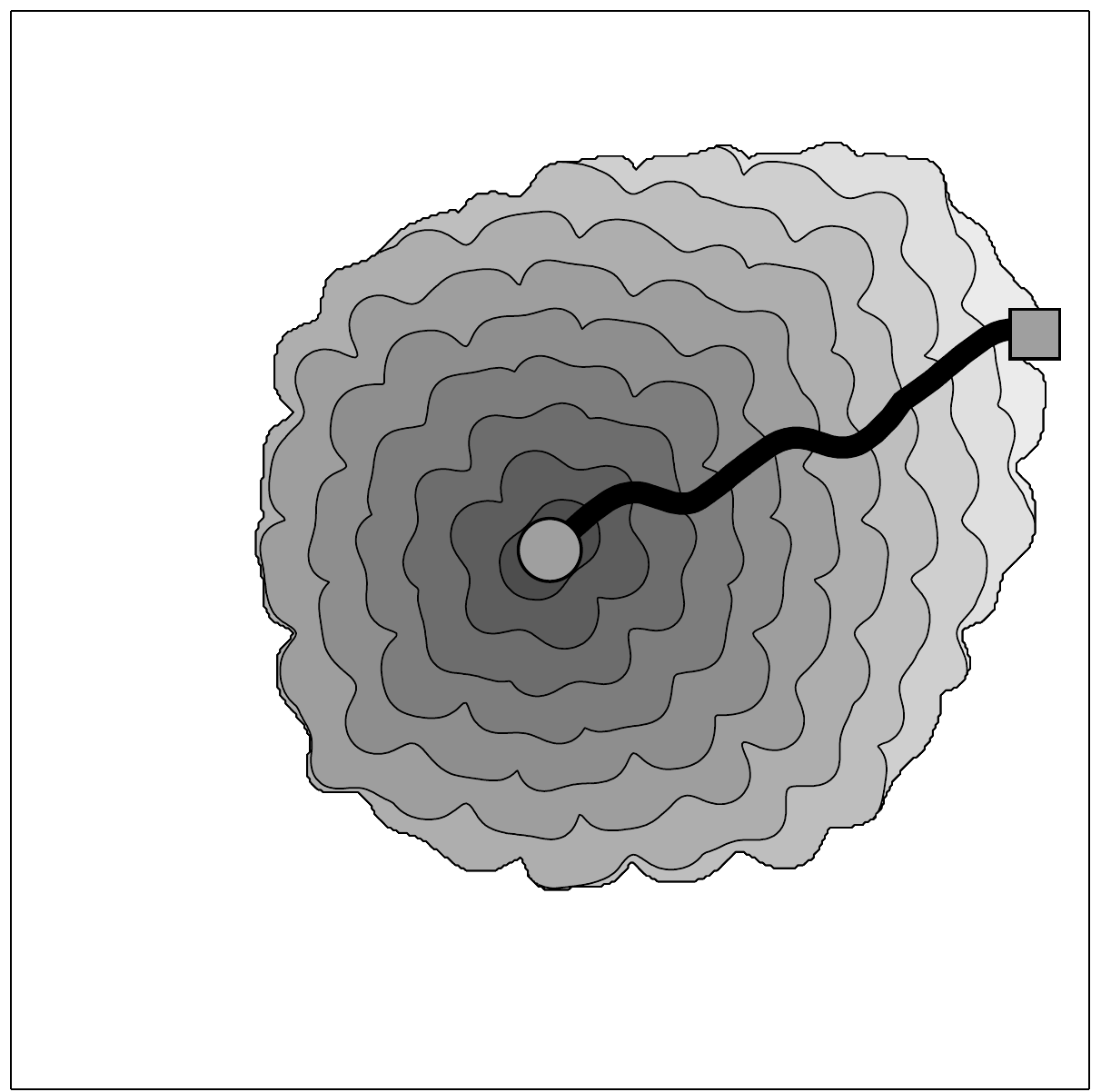} &
\includegraphics[scale=\sinPicScale]{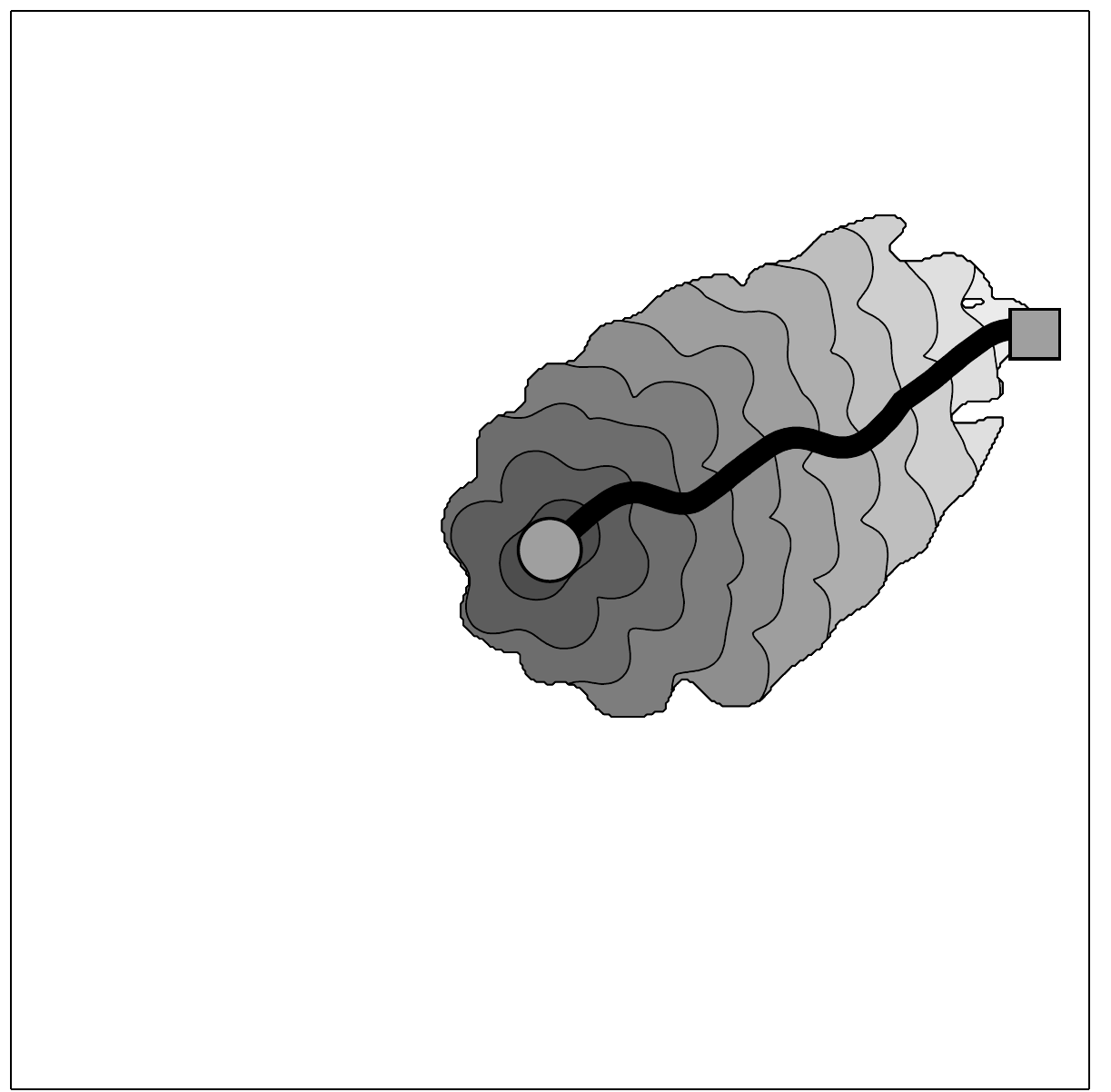} &
\includegraphics[scale=\sinPicScale]{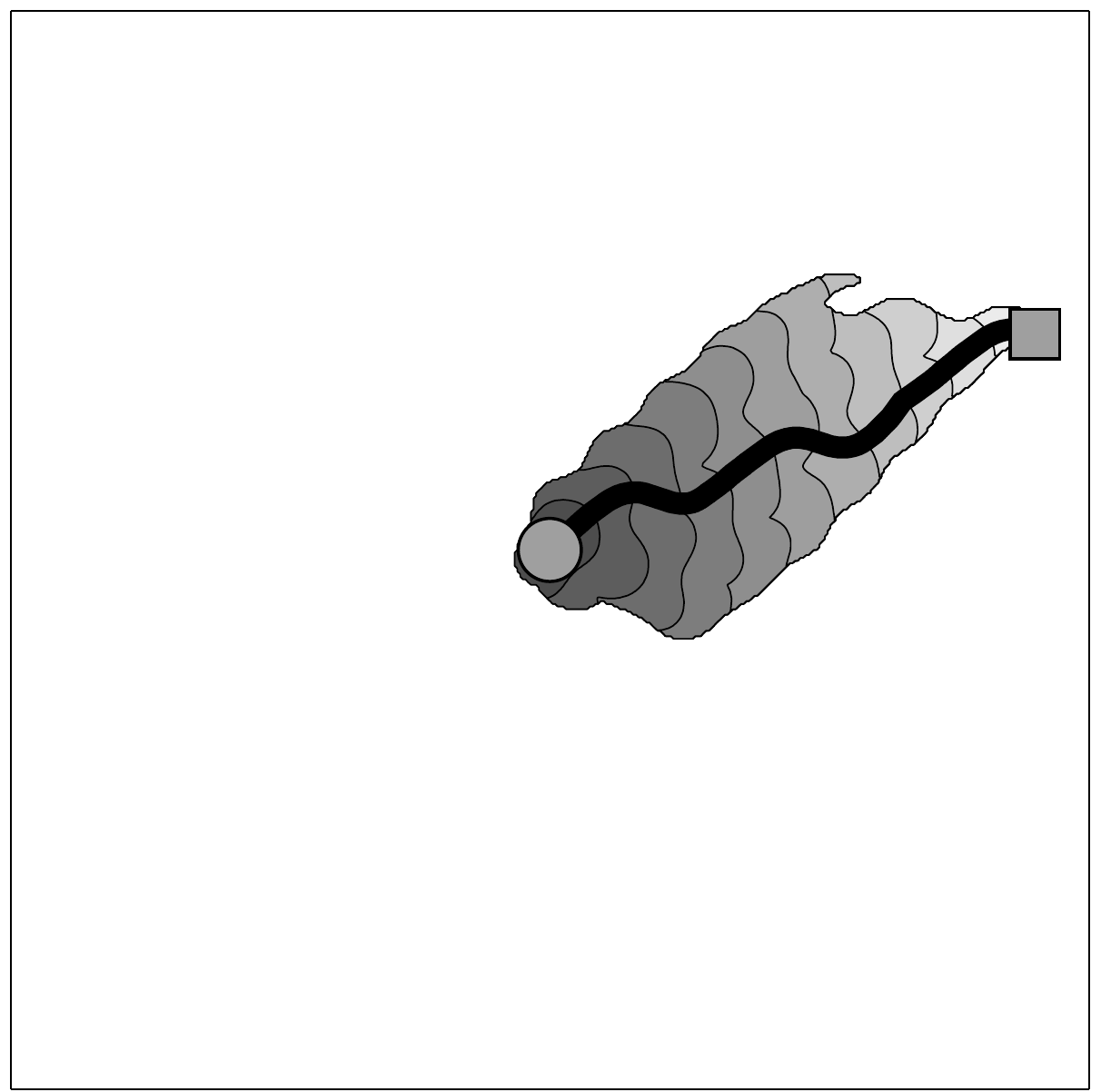} &
\includegraphics[scale=\sinPicScale]{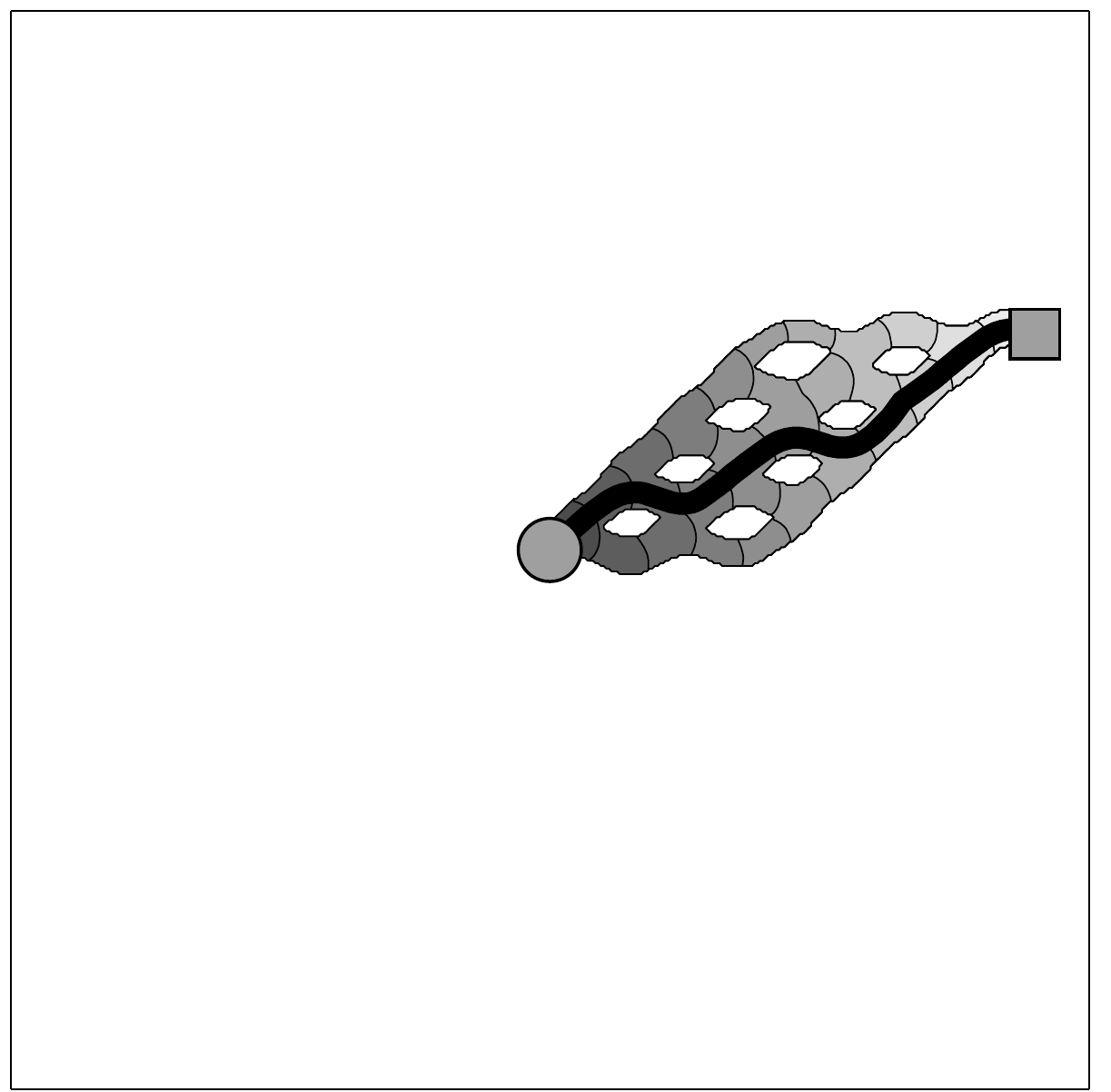} %&
% COMMENT REMOVED
 \\ [-1ex]
&
\footnotesize $\mathcal{P}=0.787$, $\astarErrOnly=0$ &
\footnotesize $\mathcal{P}=0.639$, $\astarErrOnly=0$ &
\footnotesize $\mathcal{P}=0.404$, $\astarErrOnly=0$ &
\footnotesize $\mathcal{P}=0.160$, $\astarErrOnly\approx 10^{-16}$ &
\footnotesize $\mathcal{P}=0.080$, $\astarErrOnly\approx10^{-10}$ &
\footnotesize $\mathcal{P}=0.048$, $\astarErrOnly\approx 10^{-5}$
\vspace{0.15cm} \\
&
\multicolumn{6}{c}{\includegraphics[scale=0.7]{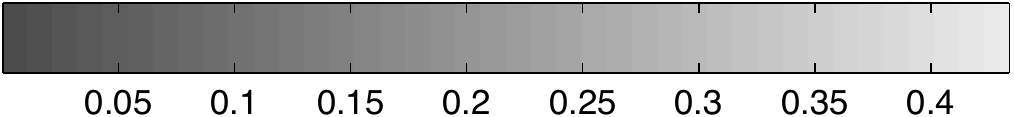}}
\end{tabular}
\end{adjustwidth}
\vspace{-.35cm}
\caption{\footnotesize Numerical results of FMM combined with SA* and AA*, showing the fraction of domain computed $\mathcal{P}$ and the relative error $\astarErrOnly$. Note the change in the ``optimal'' trajectory for SA* between $\lambda=0.3$ and $\lambda=0.70$. The solutions were produced using $m = 401$.}
\label{fig:sinOracle}
\end{center}
\end{figure}

}

\vspace{-0.2cm}
Figure \ref{fig:sinConv} compares the accuracy of these techniques for different $(m, \lambda)$ pairs.
Qualitatively the picture is largely the same as in Figure \ref{fig:constConv}, but with two non-trivial differences.
First, the `white block' in the lower-left corner of the SA* plot indicates the lack of additional errors with $m = 101$ and $\lambda \leq 0.15$.  Based on our computational experiments, this is an {\em extremely} rare situation --
the only example we could find, where the entire $G(\bs)$ is processed by SA*-FMM in the correct order despite the
fact that the heuristic $\varphi$ is inconsistent.  Second, we observe that the AA*-FMM-generated errors are not always monotone decreasing in $m$. E.g., the errors are present for $(m=401, \, \lambda = 0.75)$, but not for $(m=201, \, \lambda = 0.75)$, where the entire $G(\bs)$ is accepted.
% COMMENT REMOVED

\vspace{-0.1cm}
\iftoggle{usecolor}{%

% COMMENT REMOVED
\figstart
\begin{center}
{\bf 2D sinusoid speed: Error $ \ = \ {\log_{10}}(\astarErrOnly)$.} \\
% COMMENT REMOVED
% COMMENT REMOVED
% COMMENT REMOVED
% COMMENT REMOVED
\begin{tabular}{c c}
\em SA* & \em AA* \\
\includegraphics[scale=.5]{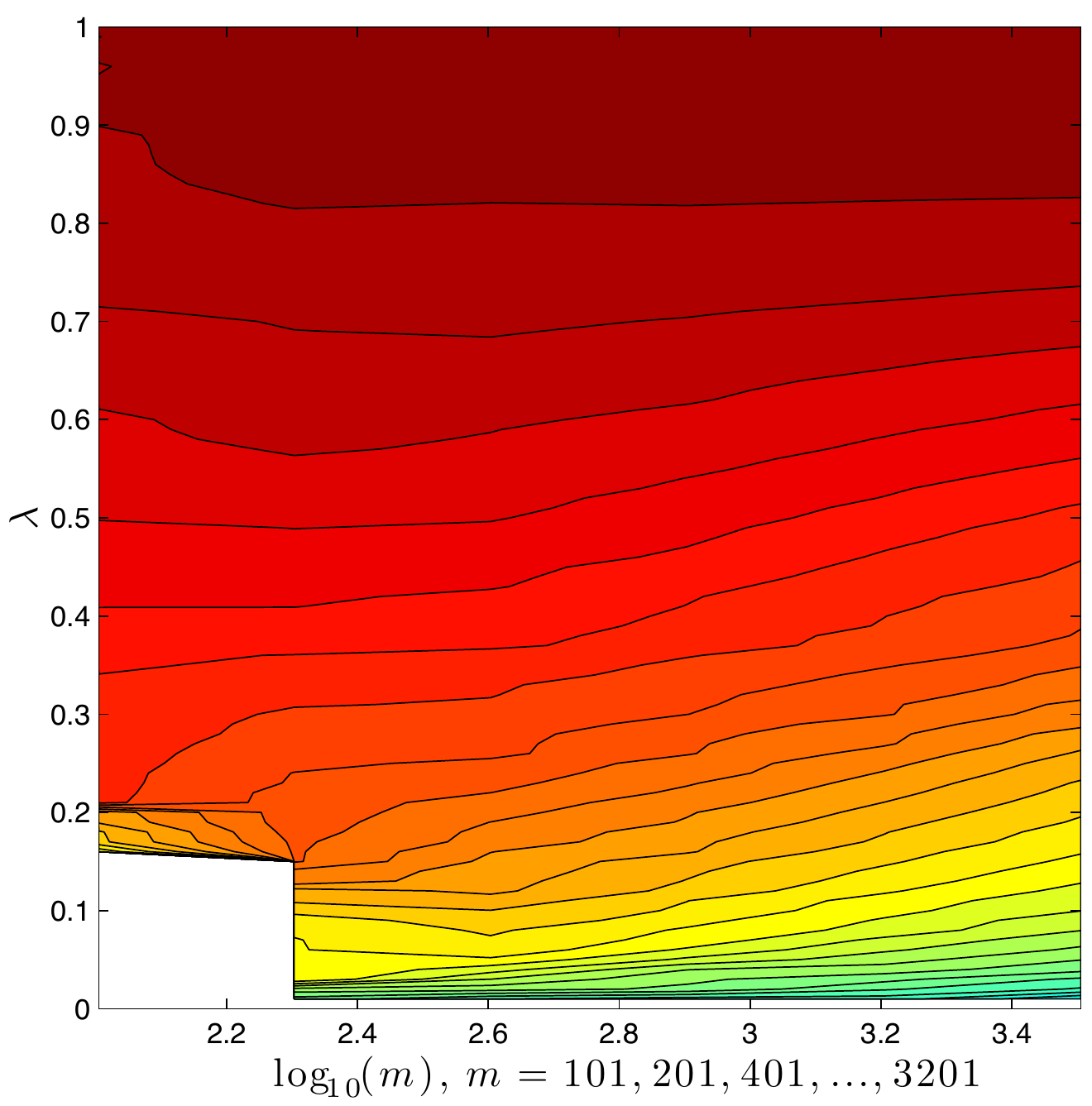} &
\includegraphics[scale=.5]{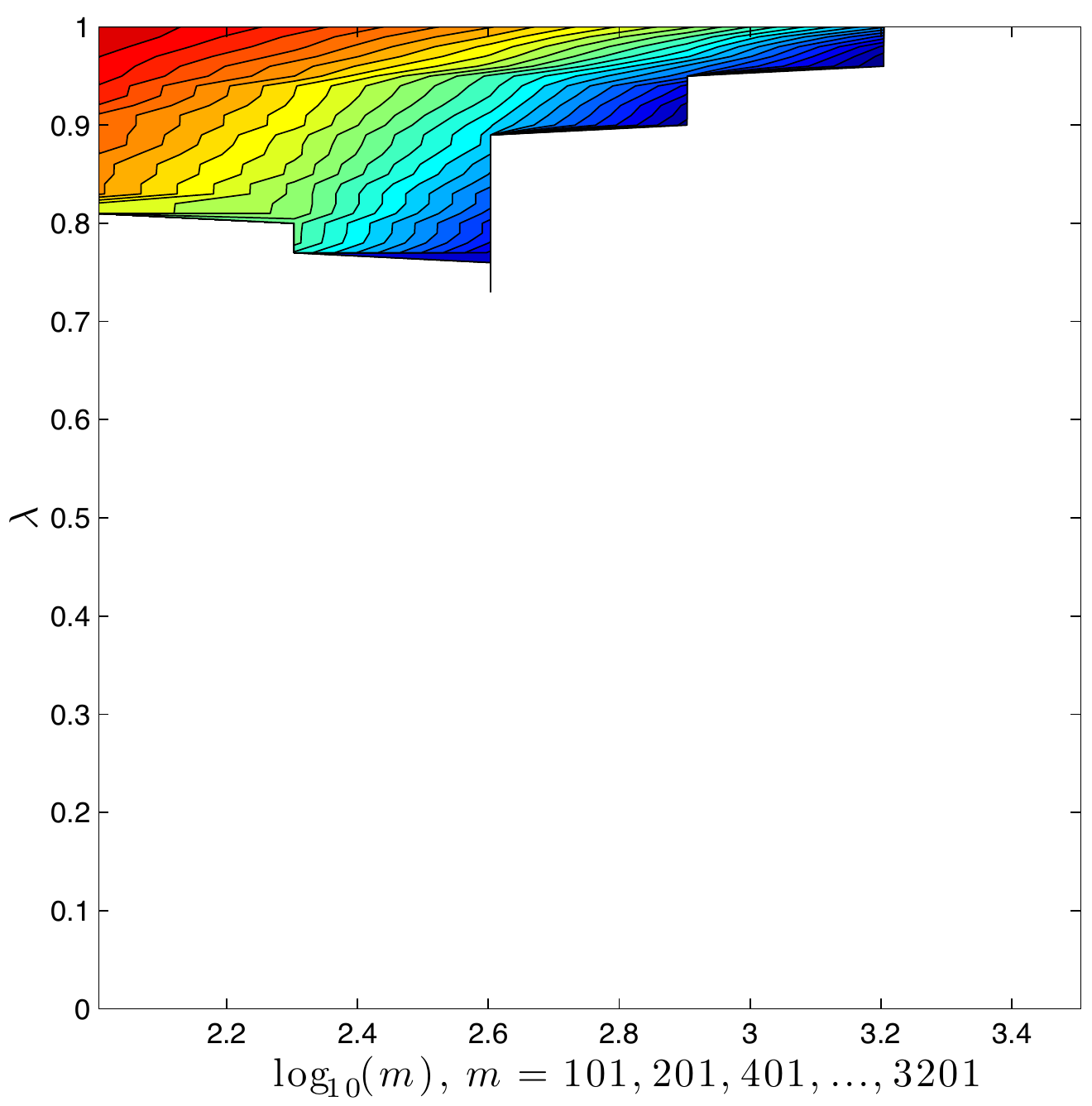} \\
\multicolumn{2}{c}{\includegraphics[scale=.7]{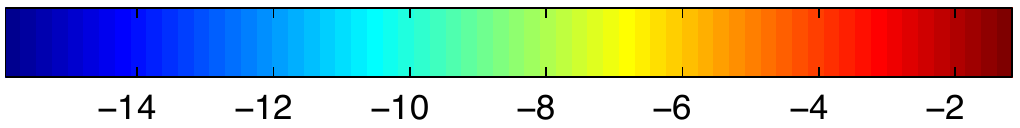}}
% COMMENT REMOVED
\end{tabular}
% COMMENT REMOVED
% COMMENT REMOVED
\vspace{-.35cm}
\caption{\footnotesize This plot shows the same results as Figure \ref{fig:constConv} except with the sinusoid speed \eqref{2D sin}.
}
% COMMENT REMOVED
\label{fig:sinConv}
\end{center}
\end{figure} 

}{%

% COMMENT REMOVED
\figstart
\begin{center}
{\bf 2D sinusoid speed: Error $ \ = \ {\log_{10}}(\astarErrOnly)$.} \\
% COMMENT REMOVED
% COMMENT REMOVED
% COMMENT REMOVED
% COMMENT REMOVED
\begin{tabular}{c c}
\em SA* & \em AA* \\
\includegraphics[scale=.5]{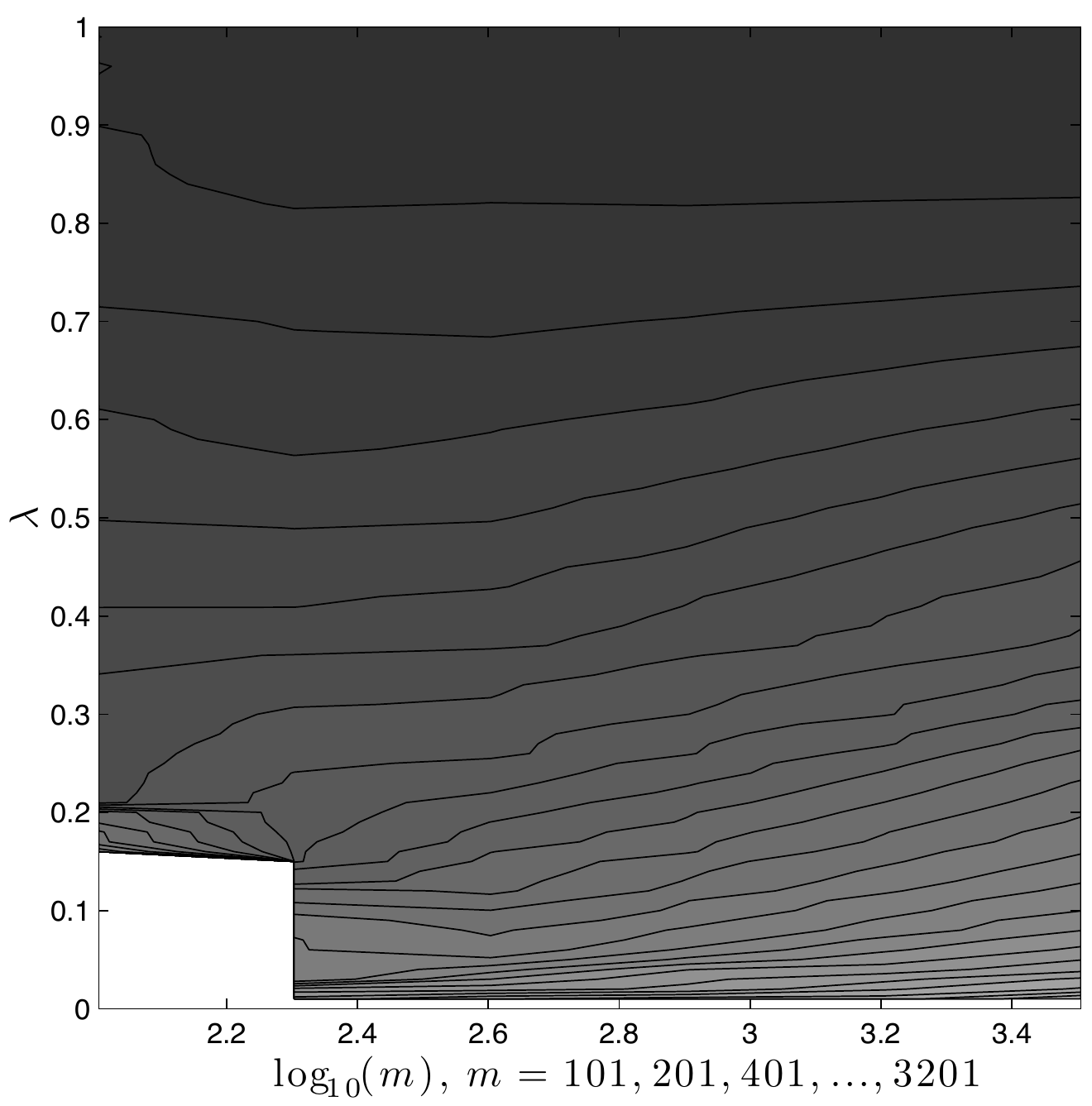} &
\includegraphics[scale=.5]{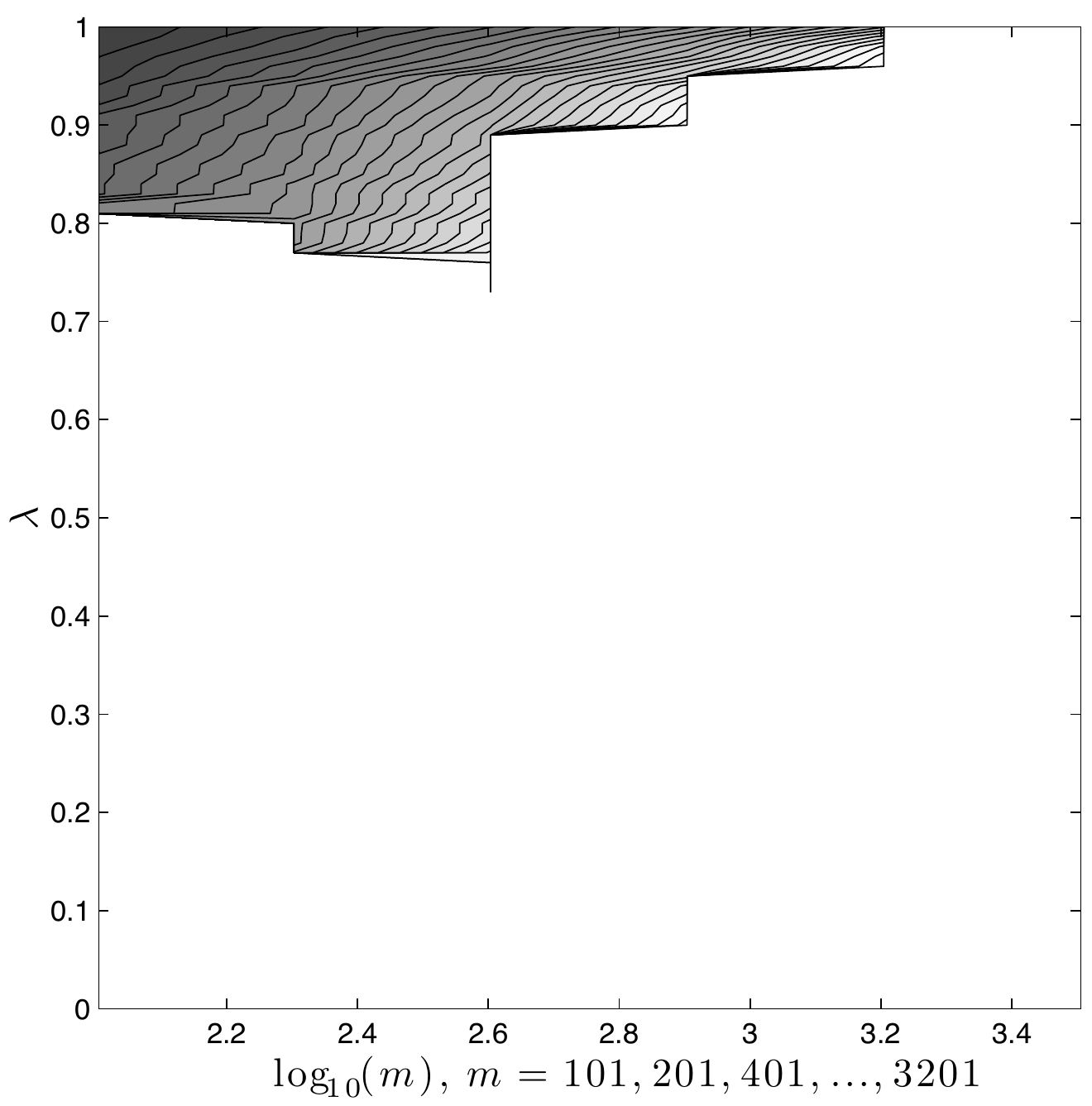} \\
\multicolumn{2}{c}{\includegraphics[scale=.7]{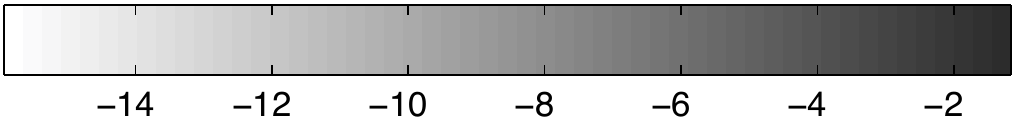}}
% COMMENT REMOVED
\end{tabular}
% COMMENT REMOVED
% COMMENT REMOVED
\vspace{-.35cm}
\caption{\footnotesize This plot shows the same results as Figure \ref{fig:constConv} except with the sinusoid speed \eqref{2D sin}.
}
% COMMENT REMOVED
\label{fig:sinConv}
\end{center}
\end{figure} 

}
% COMMENT REMOVED
% COMMENT REMOVED
% COMMENT REMOVED

Since the oracle heuristic is generally unavailable, we now consider the accuracy/efficiency tradeoffs using the na\"{i}ve heuristic $\varphi = \varphi^0$ and a realistically obtainable (but conservative) overestimate $\Upper = \Upper_2$.
Figure \ref{fig:sinNaive} shows that AA*-FMM yields comparable efficiency (despite accepting a larger part of the domain) while also ensuring $U^*(\bs)=U(\bs)$ since the entire $G(\bs)$ is accepted.
\begin{figure}[H]
\begin{center}
\tabcolsep=0.05cm
{\bf 2D sinusoid speed: Statistics.}
% COMMENT REMOVED
\begin{adjustwidth}{-1.5cm}{}
\begin{tabular}{c c c}
A. {\em Time (sec)} &
B. {\em Fraction $\mathcal{P}$} &
C. {\em Error $\log_{10}(\astarErr)$} \\
\includegraphics[scale=\statsScale]{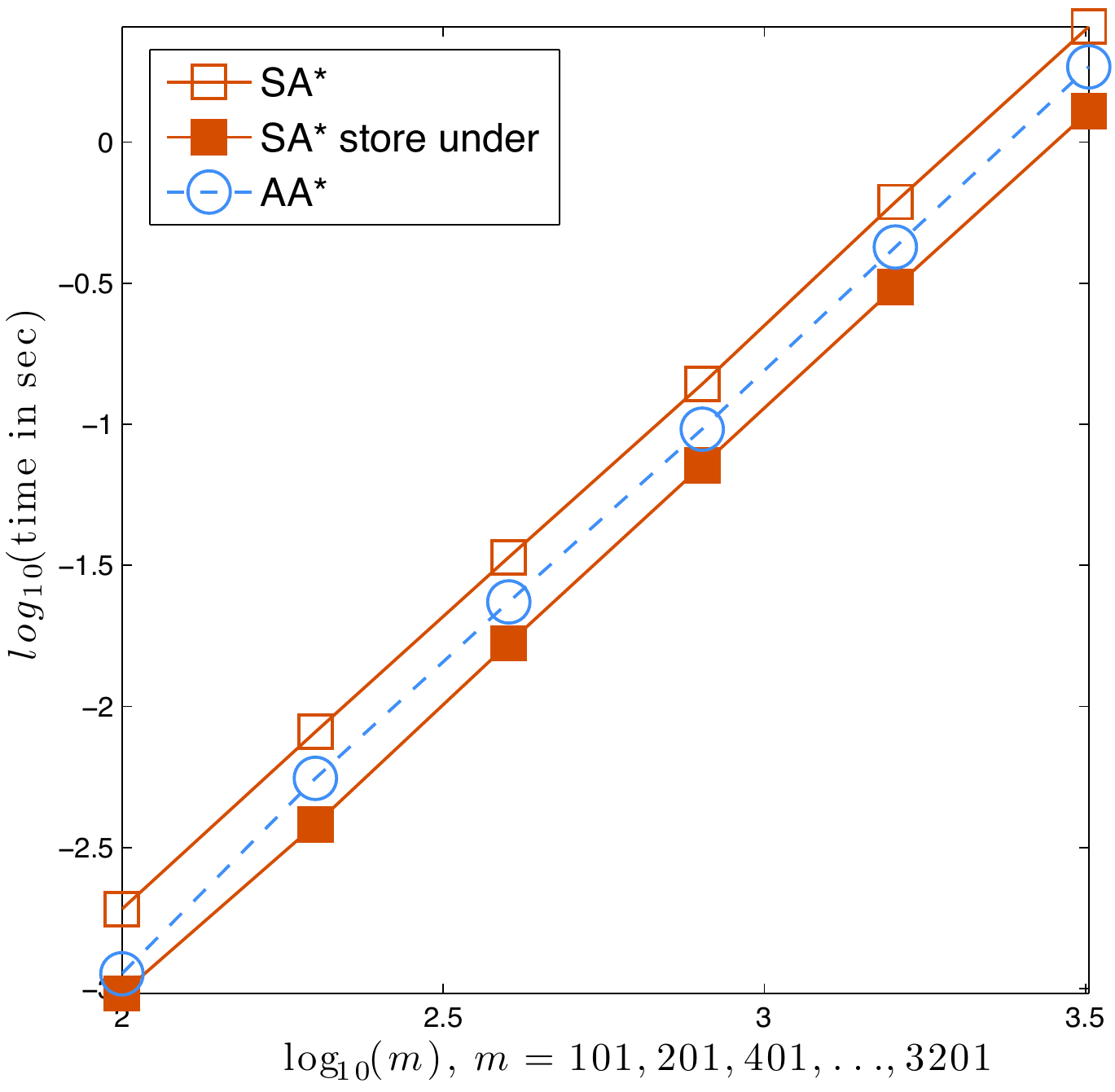} &
\includegraphics[scale=\statsScale]{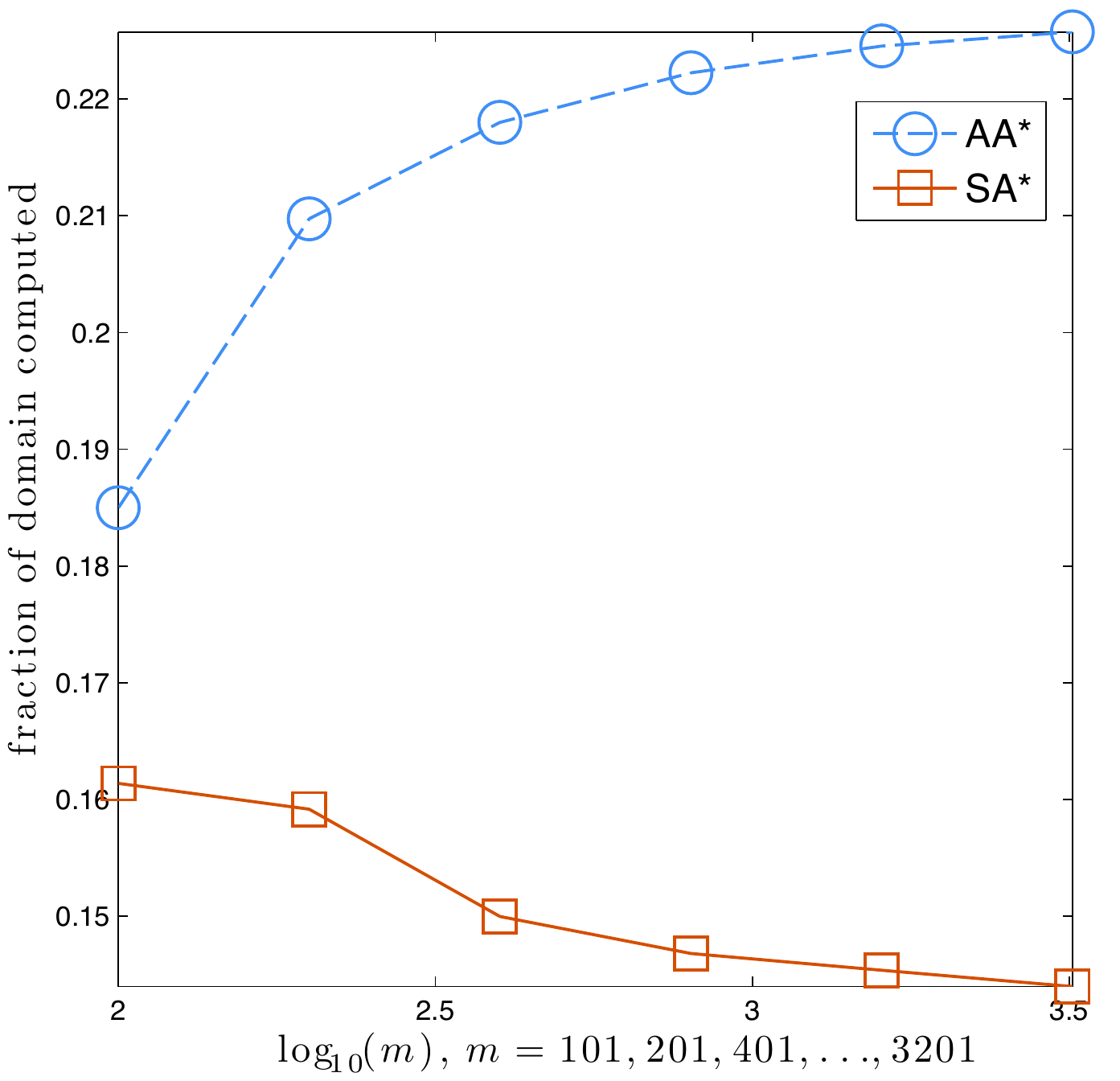} &
\includegraphics[scale=\statsScale]{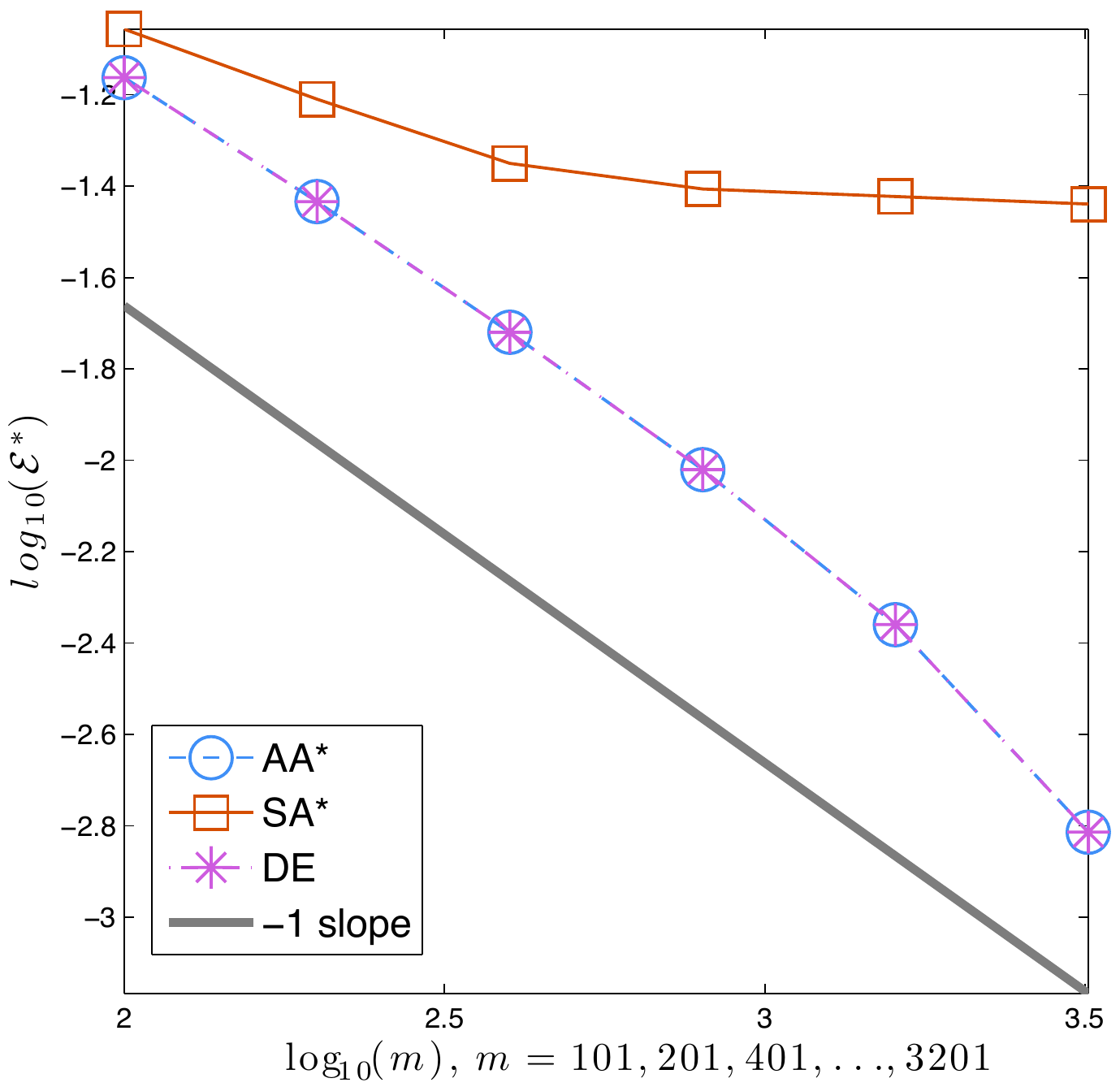}
\end{tabular}
\end{adjustwidth}
\caption{\footnotesize These results show the time (in seconds), fraction domain calculated, and the error $\astarErr$ for both SA* and AA* using a highly oscillatory sinusoid function in 2D. The na\"{i}ve heuristic was used, and for AA* $\Upper = \Upper_2$.}
\label{fig:sinNaive}
\end{center}
\end{figure} 

% COMMENT REMOVED
We also consider similar oscillatory examples in 3D with
\begin{equation}
f(x,y,z) \ = \ 1 \ + \ A \sin(10\pi x) \sin(10 \pi y) \sin (10\pi z) \label{3D sin},
\end{equation}
for two amplitudes  $A=0.1$ and $A=0.35$.
The source/target locations are $\bs = (0.72, 0.6, 0.8)$ and $\bt = (0.32, 0.4, 0.36)$.
Figure \ref{fig:sin3dNaive} shows the accuracy/efficiency data based on
realistic $\varphi = \varphi^0$ and $\Upper = \Upper_2$.  The errors due to AA* are negligible
compared to discretization errors, while the errors due to SA* are again quite noticeable
and decrease much slower as $h \to 0$.
% COMMENT REMOVED

% COMMENT REMOVED
\figstart
\begin{center}
\tabcolsep=0.05cm
{\bf 3D sinusoid speed: Statistics.}
% COMMENT REMOVED
\begin{adjustwidth}{-1.75cm}{}
\begin{tabular}{r c c c}
&
A.  {\em Time (sec)} &
B. {\em Fraction $\mathcal{P}$} &
C. {\em Error $\log_{10}(\astarErr)$} \\
\begin{sideways}\hspace{2.6cm}\large$\bm{A=0.1}$\end{sideways}&
\includegraphics[scale=\statsScale]{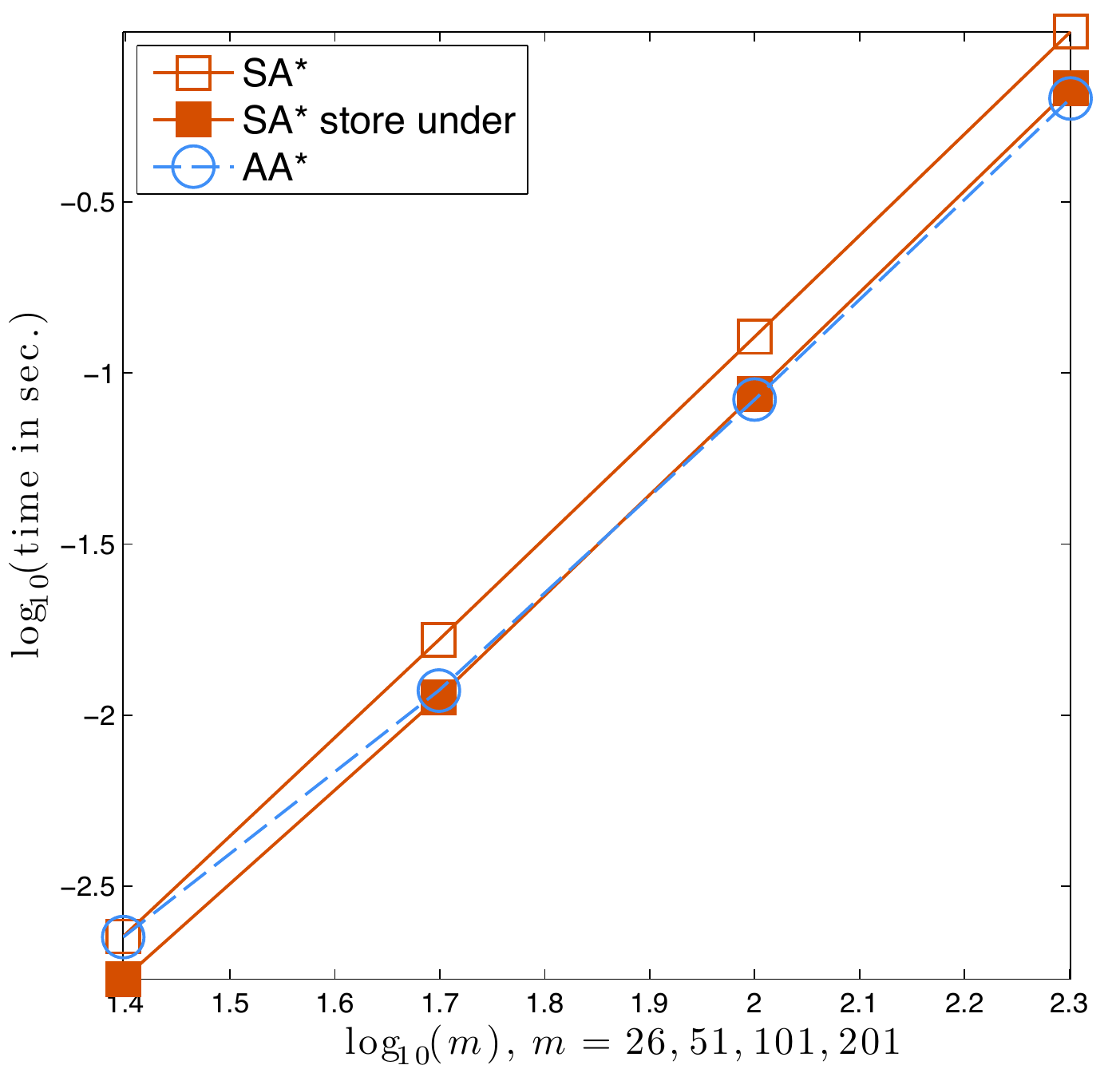} &
\includegraphics[scale=\statsScale]{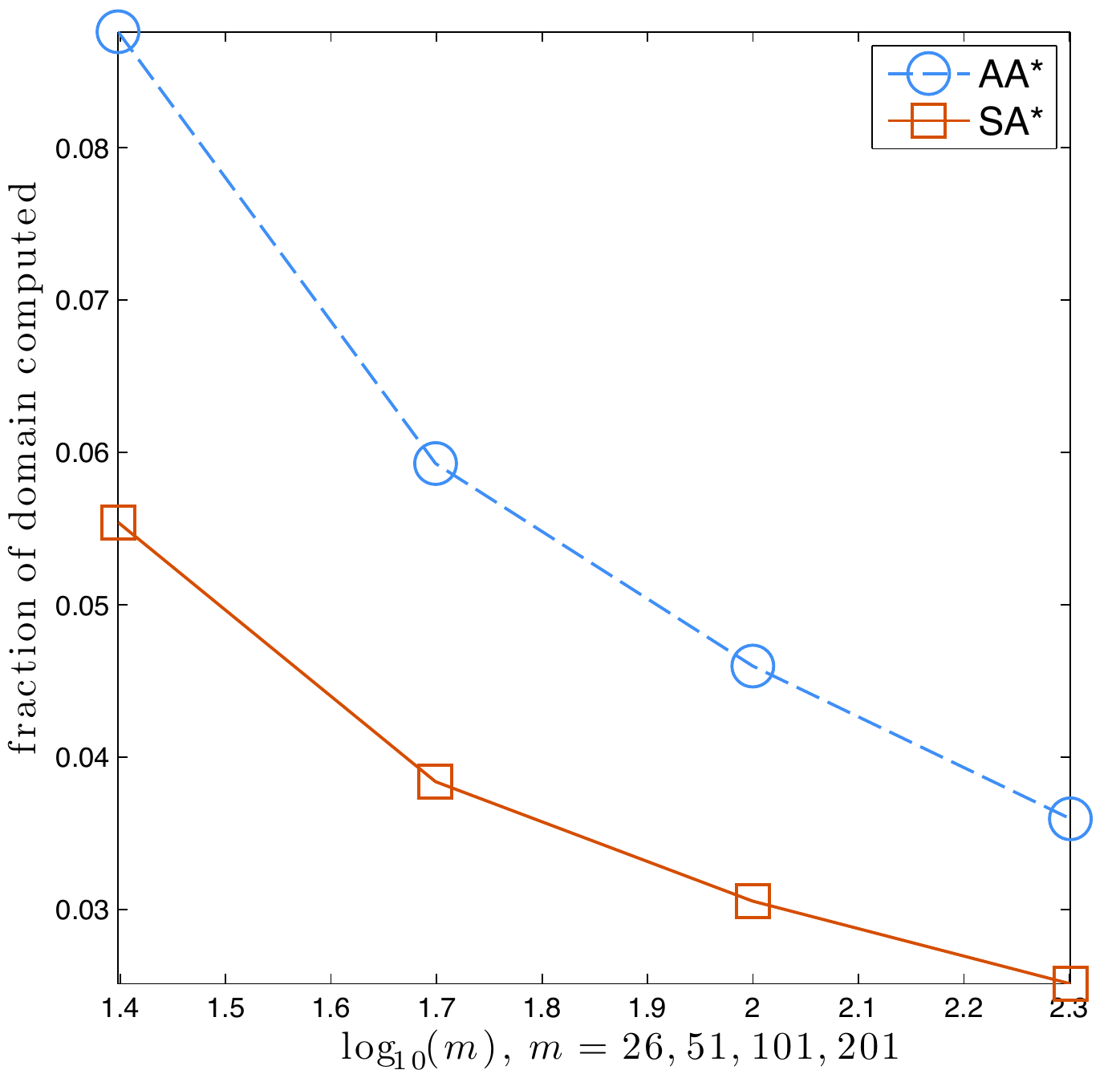} &
\includegraphics[scale=\statsScale]{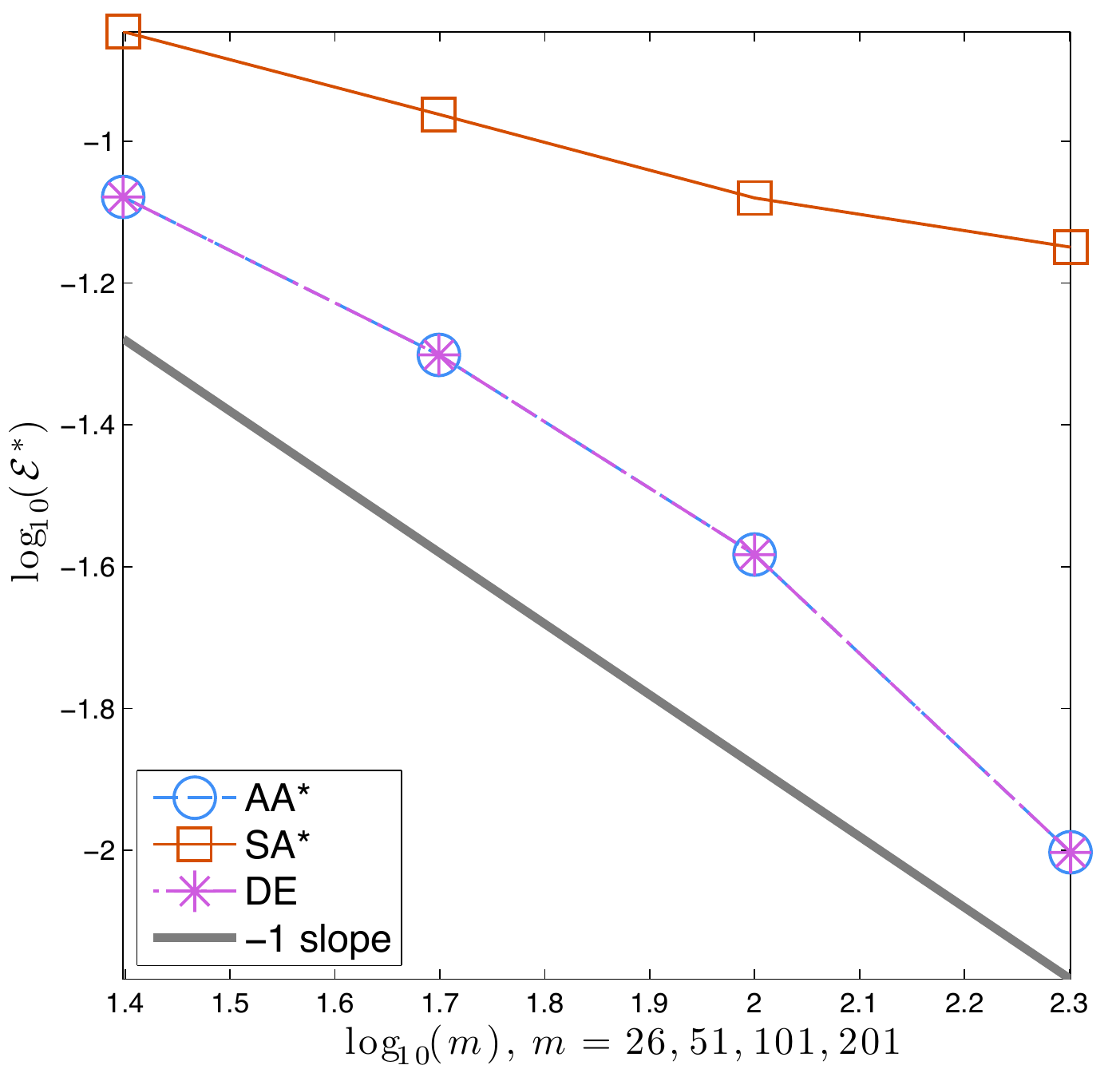} \vspace{-0.2cm} \\
\begin{sideways}\hspace{2.6cm}\large$\bm{A=0.35}$\end{sideways}&
\includegraphics[scale=\statsScale]{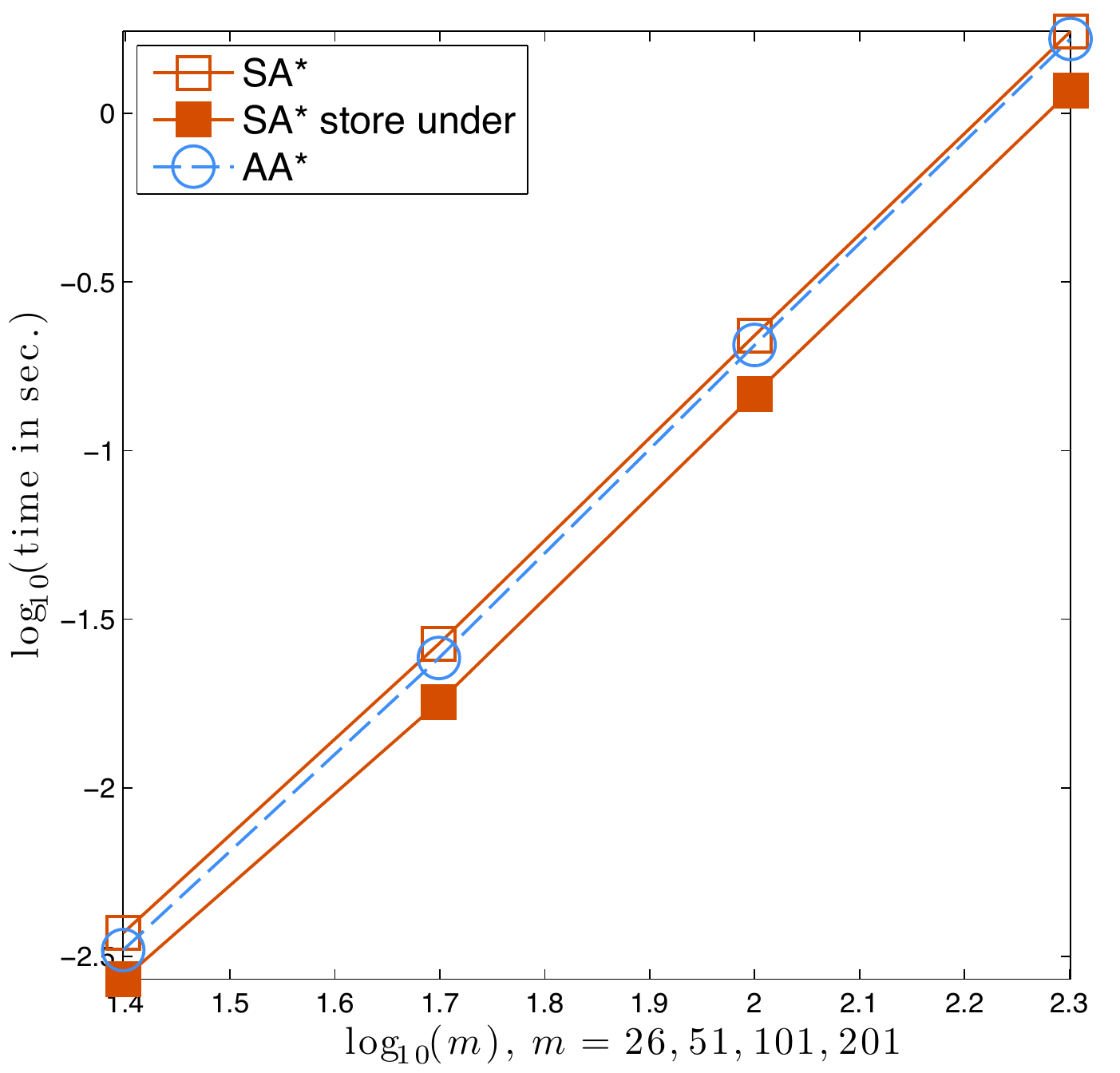} &
\includegraphics[scale=\statsScale]{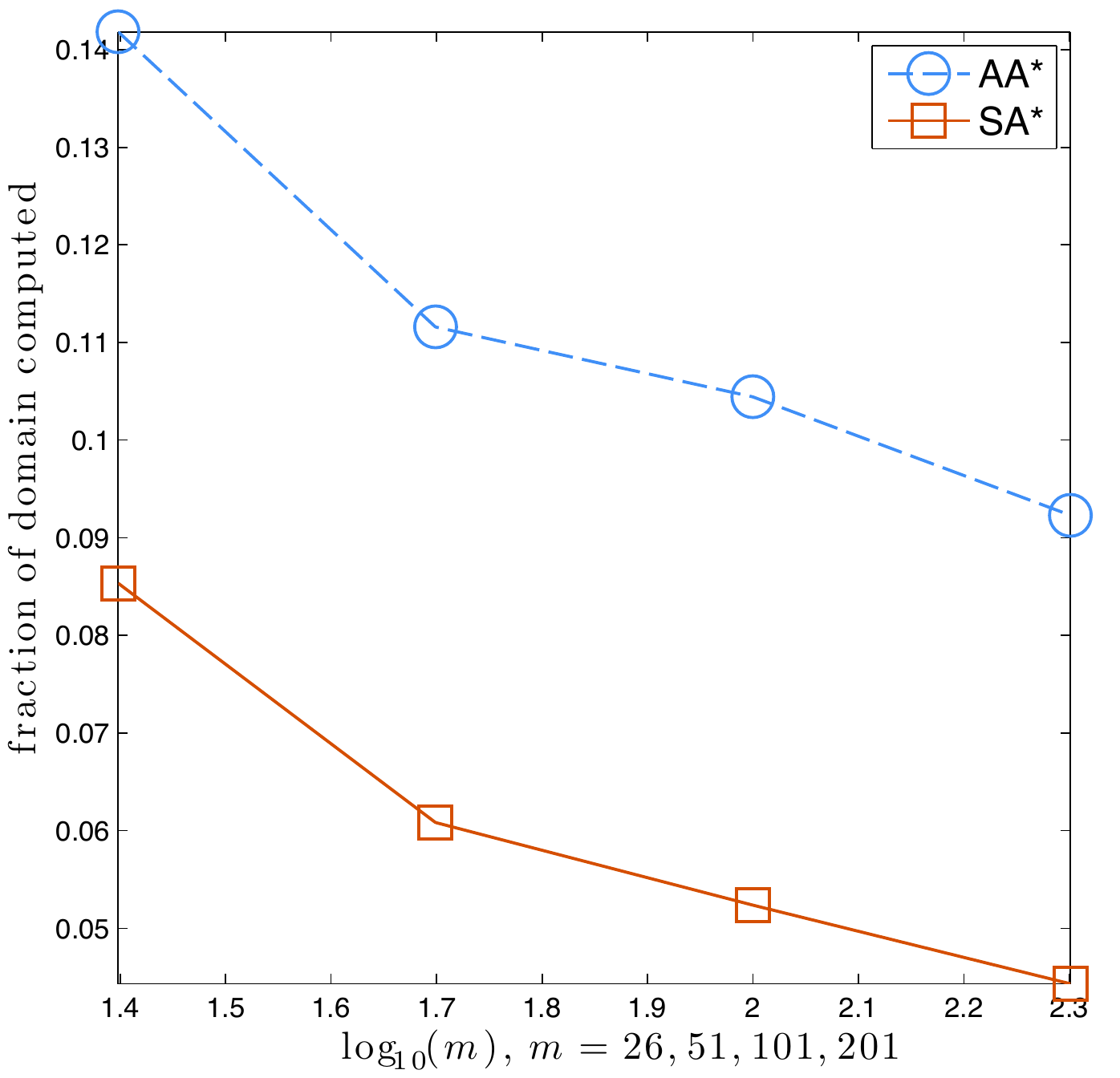} &
\includegraphics[scale=\statsScale]{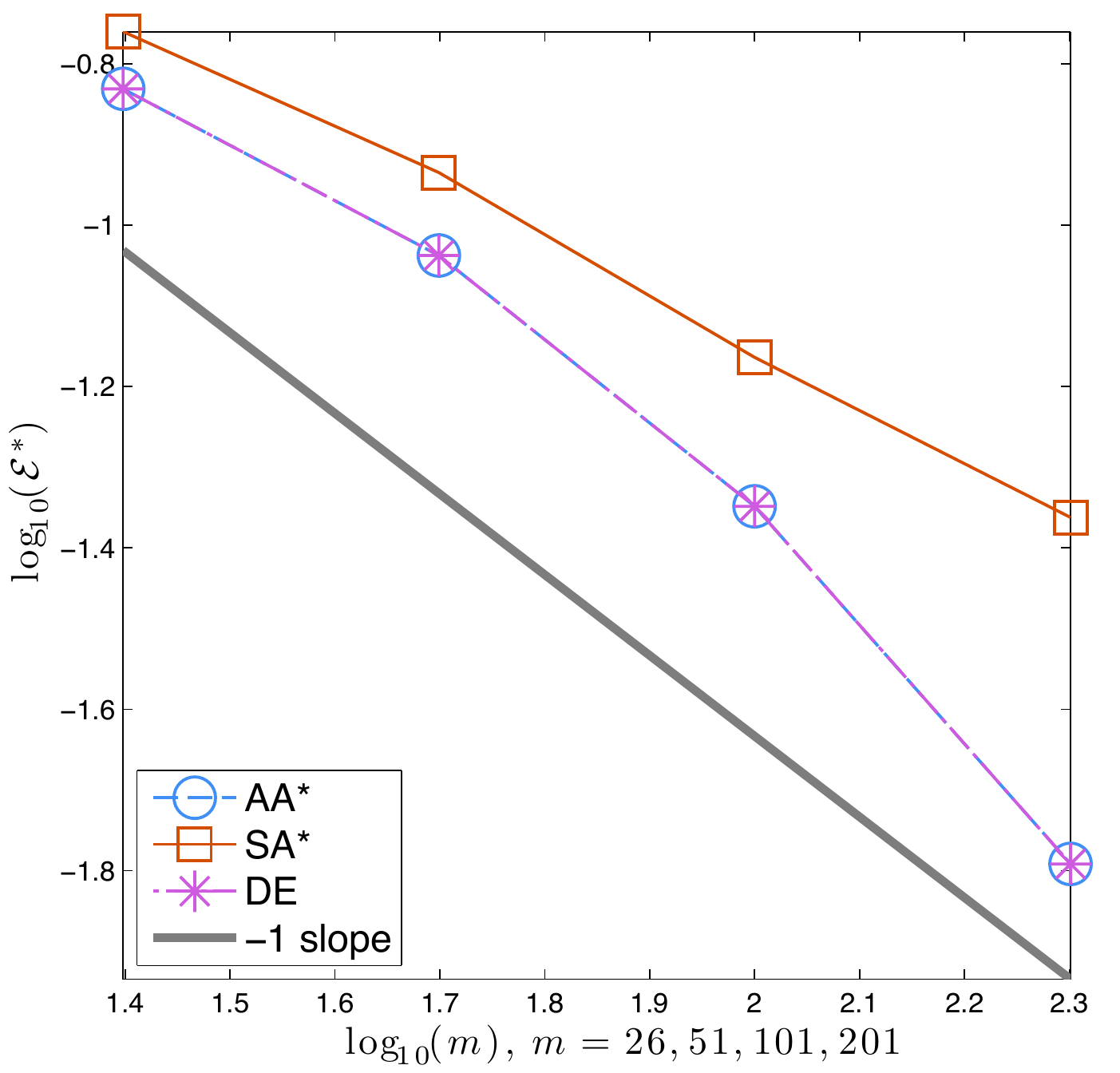} \vspace{-0.2cm}
\end{tabular}
\end{adjustwidth}
\caption{\footnotesize These results again show the average time (in seconds) of 10 trial runs, fraction domain calculated, and the error $\astarErr$ for both SA* and AA* using \eqref{3D sin}. The na\"{i}ve heuristic was used, and for AA* $\Upper = \Upper_2$. The top row corresponds to $A = 0.1$ and the bottom row shows the results when $A = 0.35$. When $A = 0.1$ the result might seem counterintuitive: AA* takes less CPU time even though it processes more of the domain. Careful profiling shows that for SA* the three-neighbor update fails more frequently and causes the algorithm to perform more two-sided updates. This makes an average node update in SA* more computationally expensive; hence the slower time.  %TODO: should this even be in the caption?
% COMMENT REMOVED
}
\label{fig:sin3dNaive}
\end{center}
\end{figure}

% COMMENT REMOVED
% COMMENT REMOVED
% COMMENT REMOVED

% COMMENT REMOVED

% COMMENT REMOVED
% COMMENT REMOVED
% COMMENT REMOVED
% COMMENT REMOVED
% COMMENT REMOVED
\subsection{Satellite image}
\label{ss:peyre_example}
The following path-planning example is borrowed from \cite{Peyre_coarse, Peyre_landmark, Peyre_Geodesic}.
% COMMENT REMOVED
\zachEdit{
The grayscale intensities of a satellite photograph (Figure \ref{fig:sat1}A) are imported into the range $[0,755]$ using \textsc{Matlab}'s \texttt{imread()} routine. For a given gridpoint $\bx$ assume that it falls into a pixel with grayscale value $i(\bx) \in [0,755]$. This then defines the speed $f:\bar{\Omega} \to [0.001, 1.001]$ via rescaling:
\[
f(\bx) \ \ = \ \ 0.001 \ + \ {i(\bx)} / {755}
\]
% COMMENT REMOVED
% COMMENT REMOVED
% COMMENT REMOVED
% COMMENT REMOVED
% COMMENT REMOVED
% COMMENT REMOVED
This is the same intensity/speed mapping used in \cite{Peyre_Geodesic}, but our experimental setup is slightly different:
}%
\begin{itemize}[leftmargin=6mm]\itemsep-2pt
\item Unlike Peyr\'{e} et al., we omit the pre-smoothing of the original $744 \times 744$ image and simply downsample it to $350 \times 350$.
\item Peyr\'{e} et al. use $\varphi=\varphi^C_{\lambda,R}$; they fix $\lambda = \frac{1}{2}$ and vary $R$.  Instead, we first use $\varphi=\varphi^0$ (Figure \ref{fig:sat1}B) and then switch to $\varphi=\bar{\varphi}_{\lambda}$ (Figure \ref{fig:sat_change}).  Unlike with $\varphi^C_{\lambda,R}$, the use of $\bar{\varphi}_{\lambda}$ directly illustrates the performance of A* as the quality of $\varphi$ improves.
\item We use slightly different source and target locations ($\bs = (337h, 161h)$ and $\bt = (16h, 188h)$; see Figure \ref{fig:sat1}C).  Our $\bs$ falls on the opposite side of a shockline compared to $\bs$ used in \cite{Peyre_Geodesic}.
\end{itemize}
% COMMENT REMOVED
Figure \ref{fig:sat1}B shows the level sets of the solution on the full domain, with $\partial L$ and $\partial C_2$ (using $\varphi^0$ and $\Psi_3$) shown in bold. (The set \ACC{} by SA* is approximately the same as AA*.)

% COMMENT REMOVED
\figstart
\begin{center}
% COMMENT REMOVED
\begin{adjustwidth}{-1.4cm}{}
\tabcolsep=0.01cm
\begin{tabular}{c c c}
A. {\em Original image} &
B. {\em PDE solution} &
C. {\em Zoom on PDE solution} \\
\includegraphics[scale=0.4]{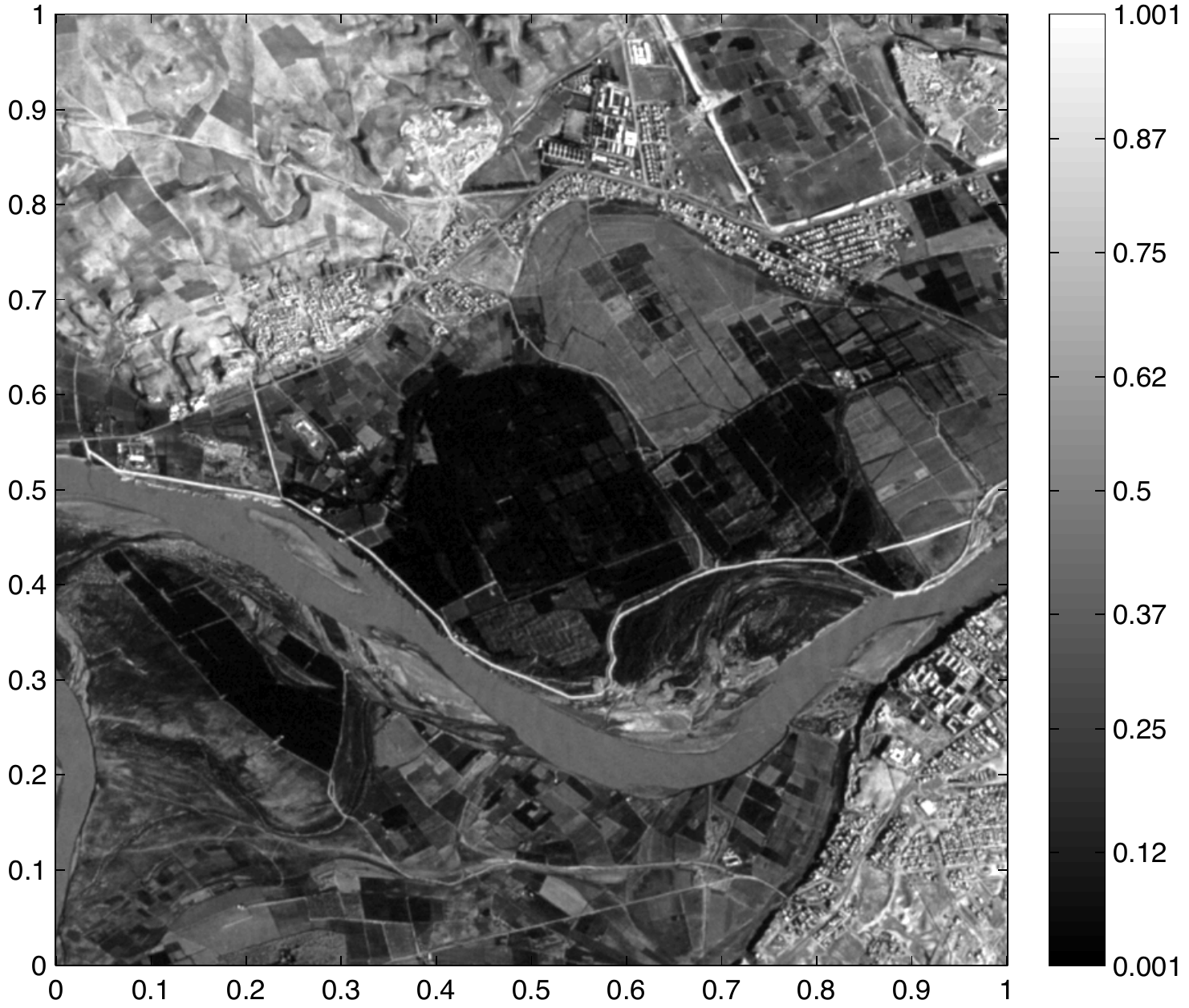} &
\iftoggle{usecolor}{%
\includegraphics[scale=0.4]{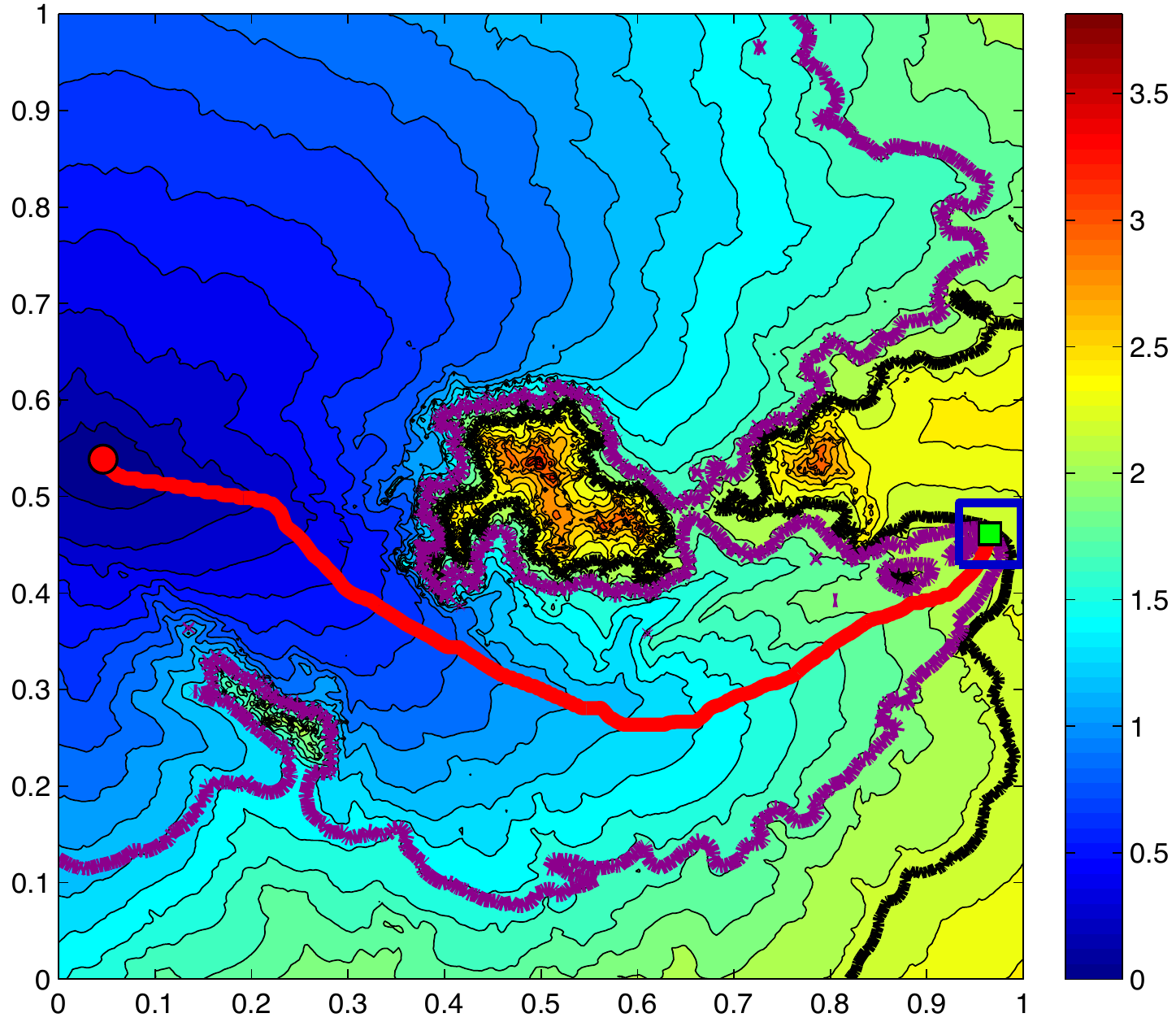} &
}{%
\includegraphics[scale=0.4]{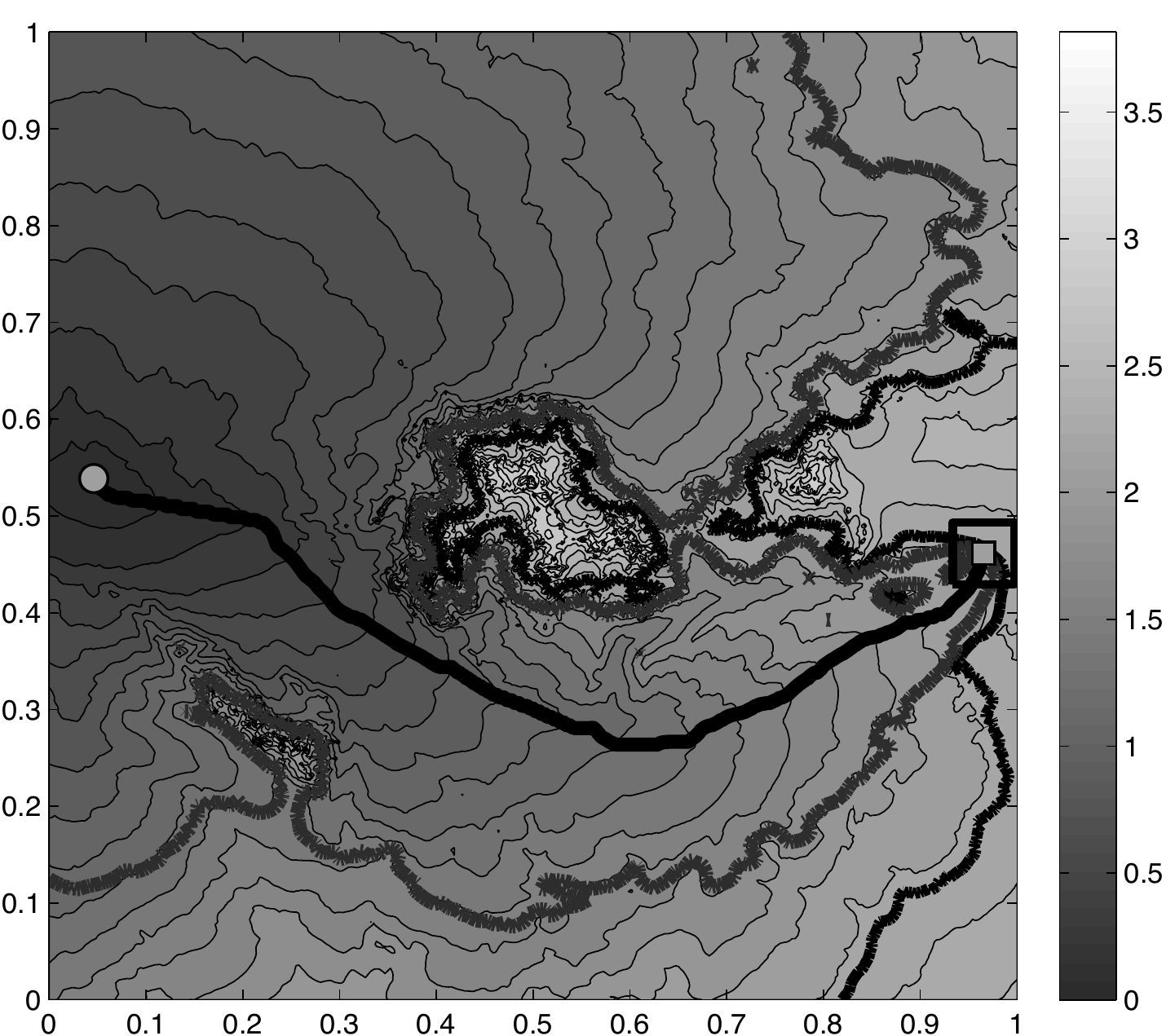} &
}
\includegraphics[scale=0.4]{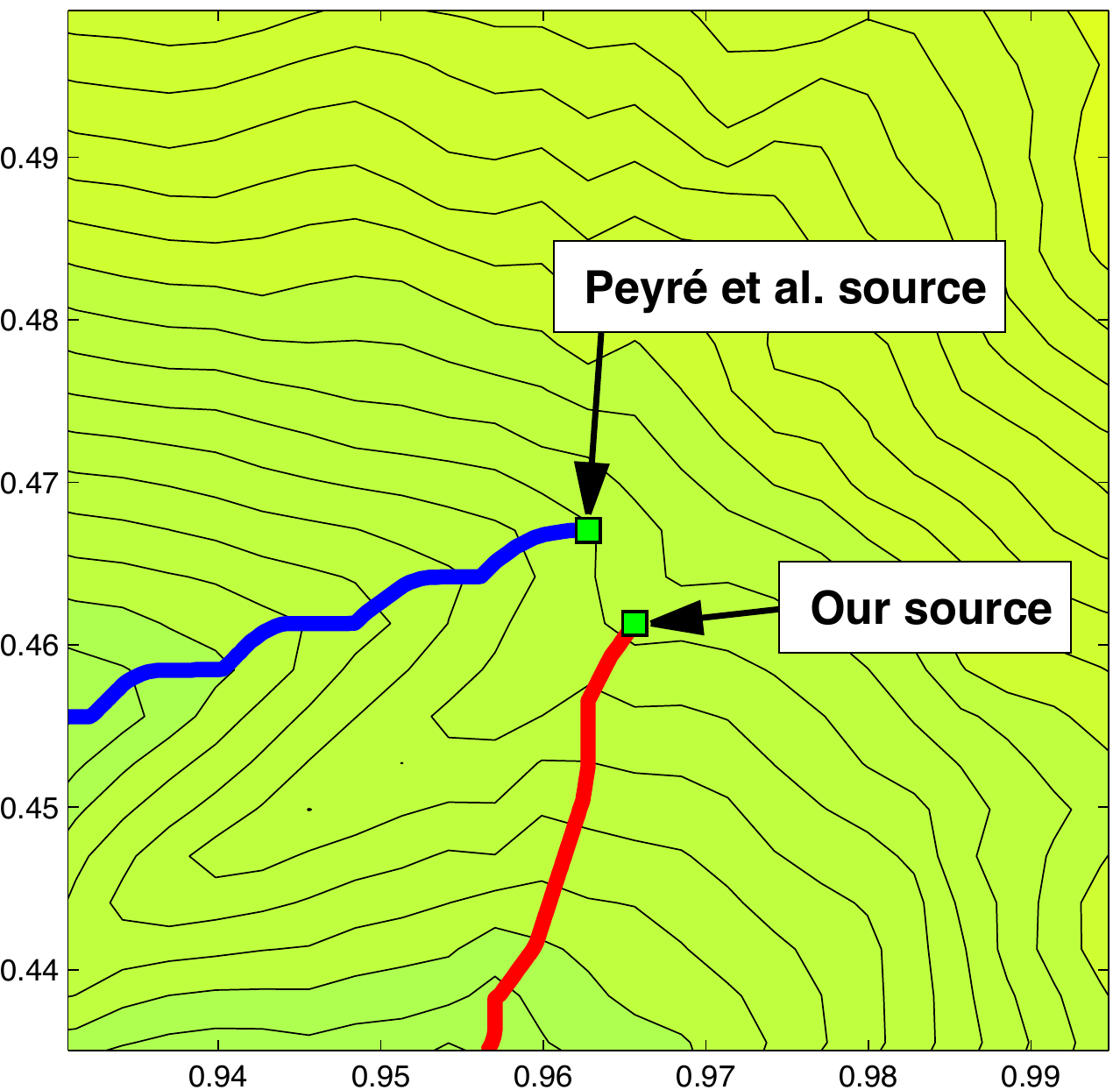} %\\
% COMMENT REMOVED
\end{tabular}
\end{adjustwidth}
\caption{\footnotesize A. The original satellite image %of size $744 \times 744$ pixels is 
mapped to a speed $f \in [0.001, 1.001]$. B. The solution to the PDE on a $350 \times 350$ grid with $\partial L$ and $\partial C_2$ (using $\varphi^0$ and $\Psi_3$) drawn in bold. %This is the \textbf{only} test that uses $\varphi^0$ in this subsection.
C. The upper marker is approximately the same $\bs$ as in \cite{Peyre_Geodesic}; the lower marker is the same $\bs$ used in B and Figure \ref{fig:sat_change}A.}
\label{fig:sat1}
\end{center}
\end{figure}

\alexEdit{This example illustrates the use of A*-techniques with a discontinuous speed function.
The rather limited computational savings in \ref{fig:sat1}B are clearly caused by the use of an ``overly optimistic''
$\varphi^0$.  However, the lack of accuracy of this naive underestimate is not caused by any discontinuities in $f$ -- instead it is simply a result of a large $F_2/F_1$, with $f$ values much closer to $F_1$ on most of the domain.}
\fillmeup
We now switch to ``oracle tests'' with $\varphi=\bar{\varphi}_{\lambda}$ and AA*-FMM relying on $\Psi = \parens{1 + \sqrt{h} / 8} V(\bt)$.  Using this heuristic, the domain restriction becomes much more effective for both SA* and AA*.
But since $\bs$ is close to a shockline, additional errors due to SA*-FMM are sufficiently large to change the optimal trajectory in several ways (see Figure \ref{fig:sat_change}A). In contrast, the errors from AA*-FMM are much smaller and the optimal trajectory remains the same for all $\lambda$.

% COMMENT REMOVED
\figstart
\begin{center}
% COMMENT REMOVED
\begin{adjustwidth}{-1.75cm}{}
\tabcolsep=0.0001pt
\begin{tabular}{c c c}
A. {\em Different trajectories} &
B. {\em Time} &
C. {\em Error $\astarErrOnly$} \\
\iftoggle{usecolor}{%
\includegraphics[scale=0.6]{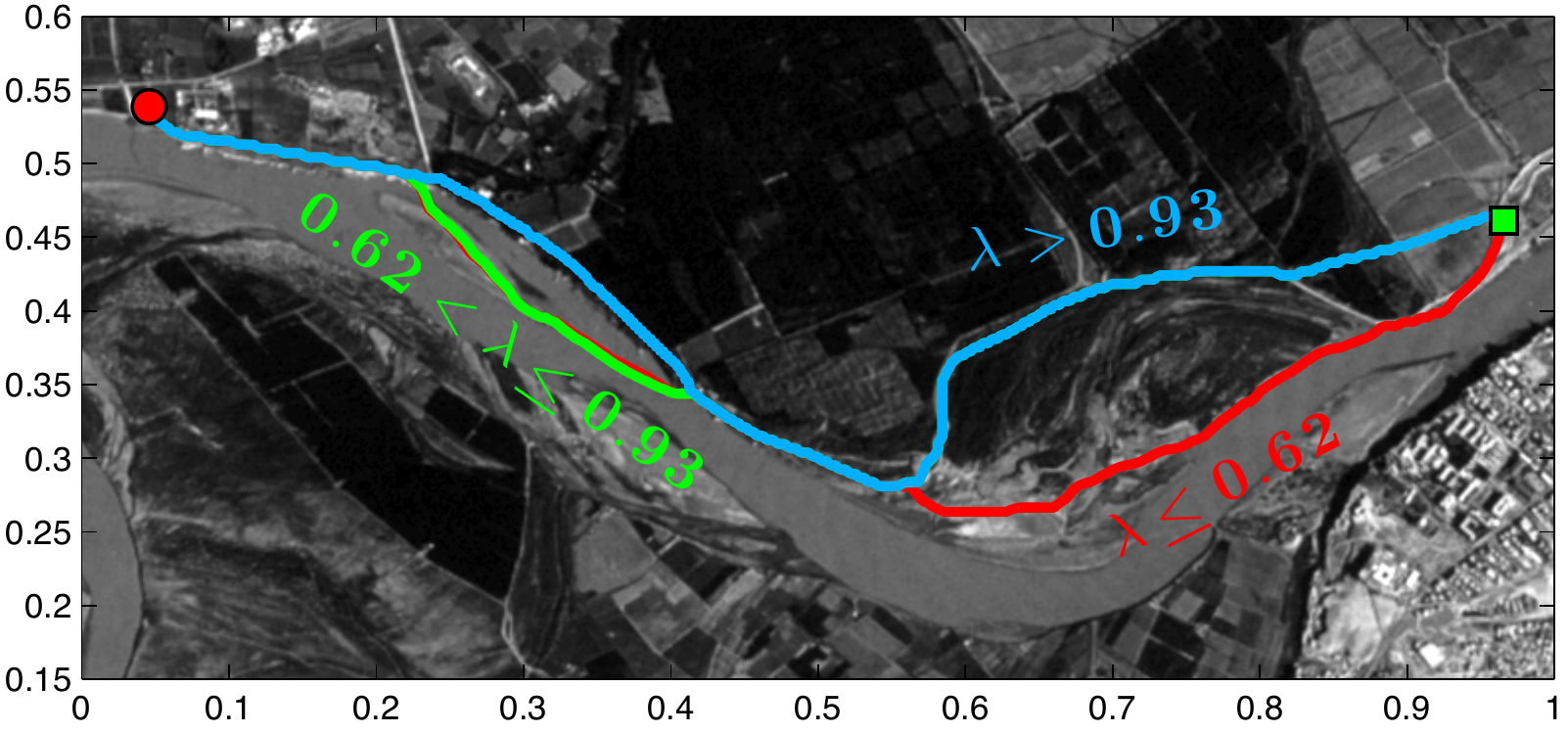} &
}{%
\includegraphics[scale=0.6]{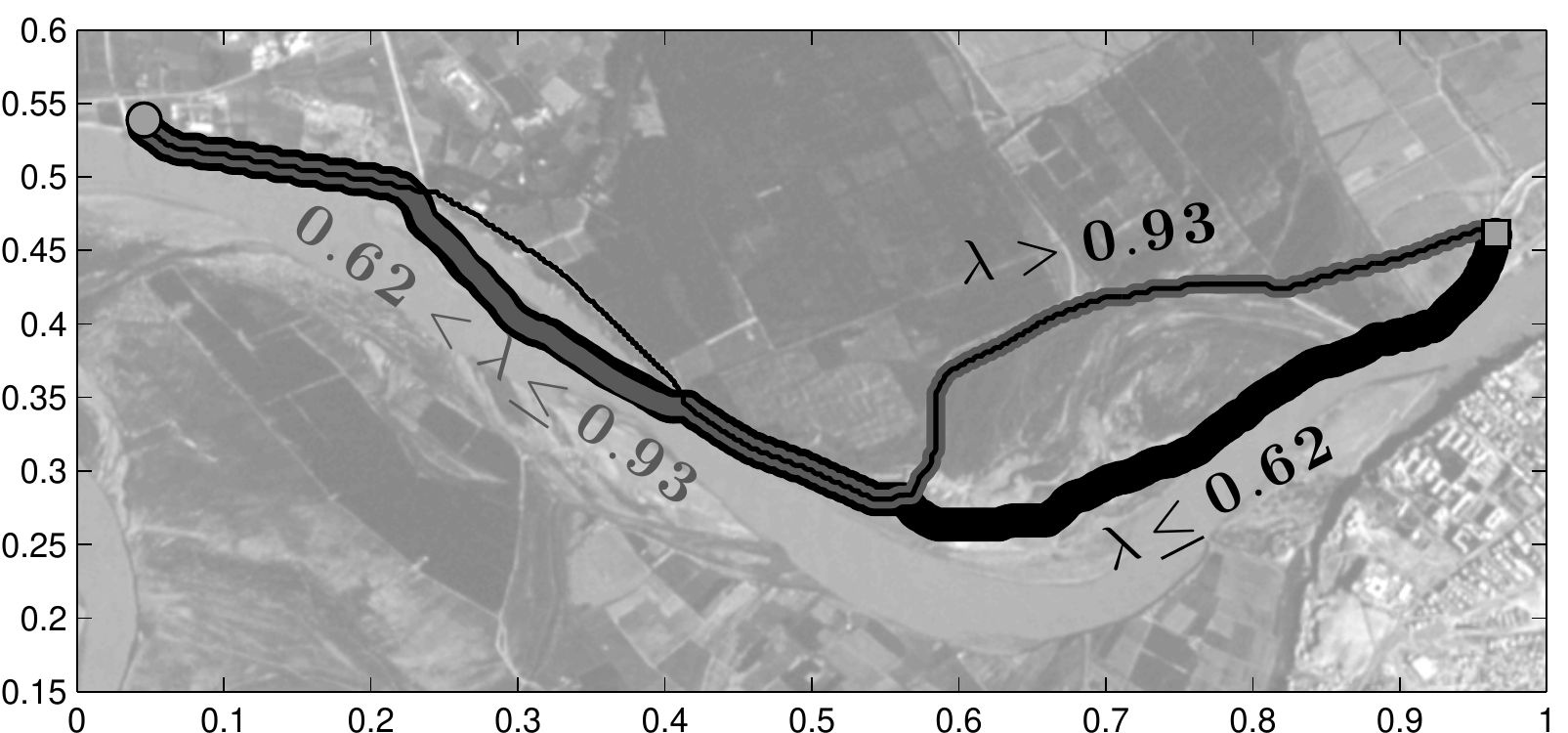} &
}
\includegraphics[scale=0.34]{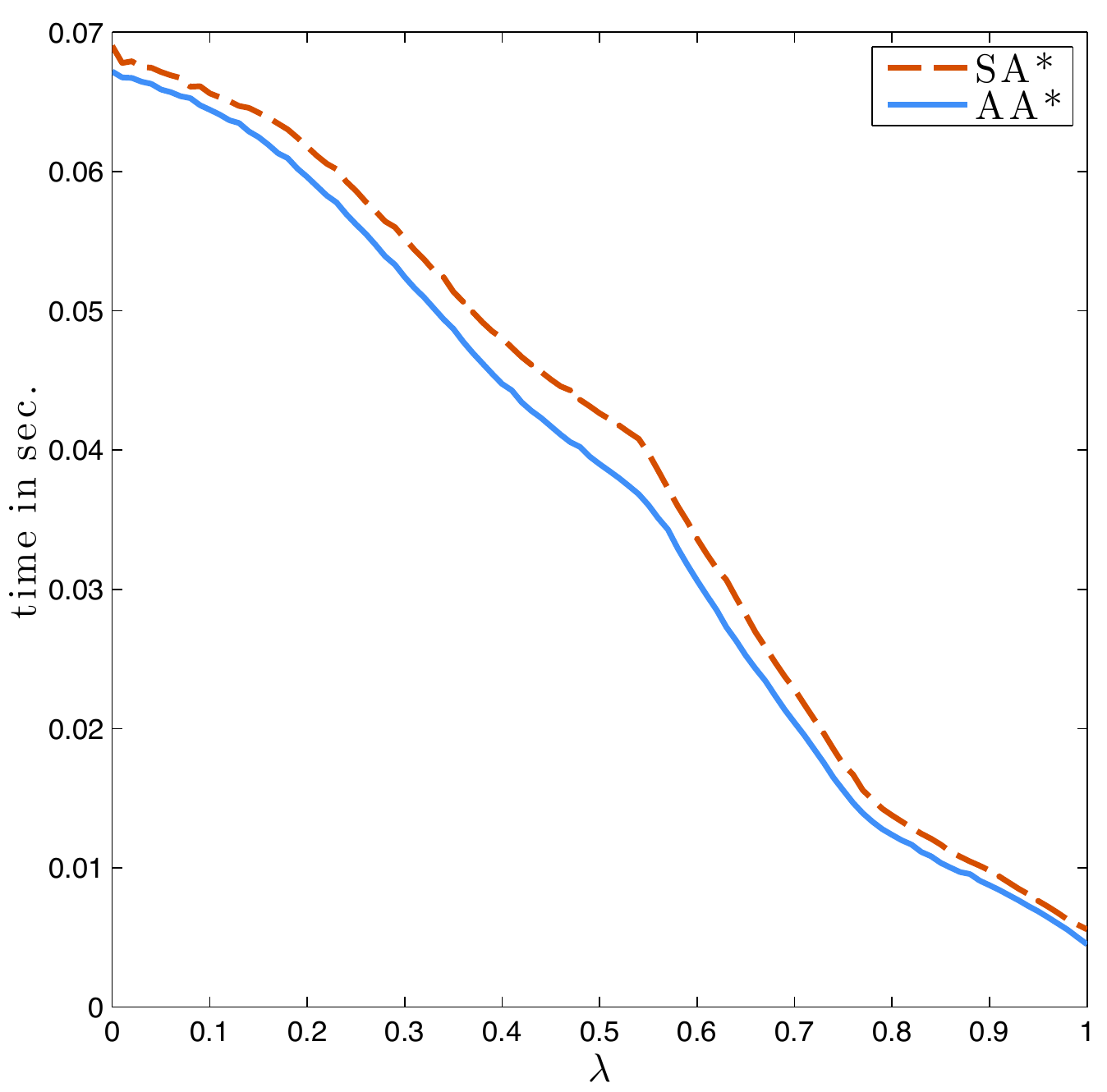} &
\includegraphics[scale=0.34]{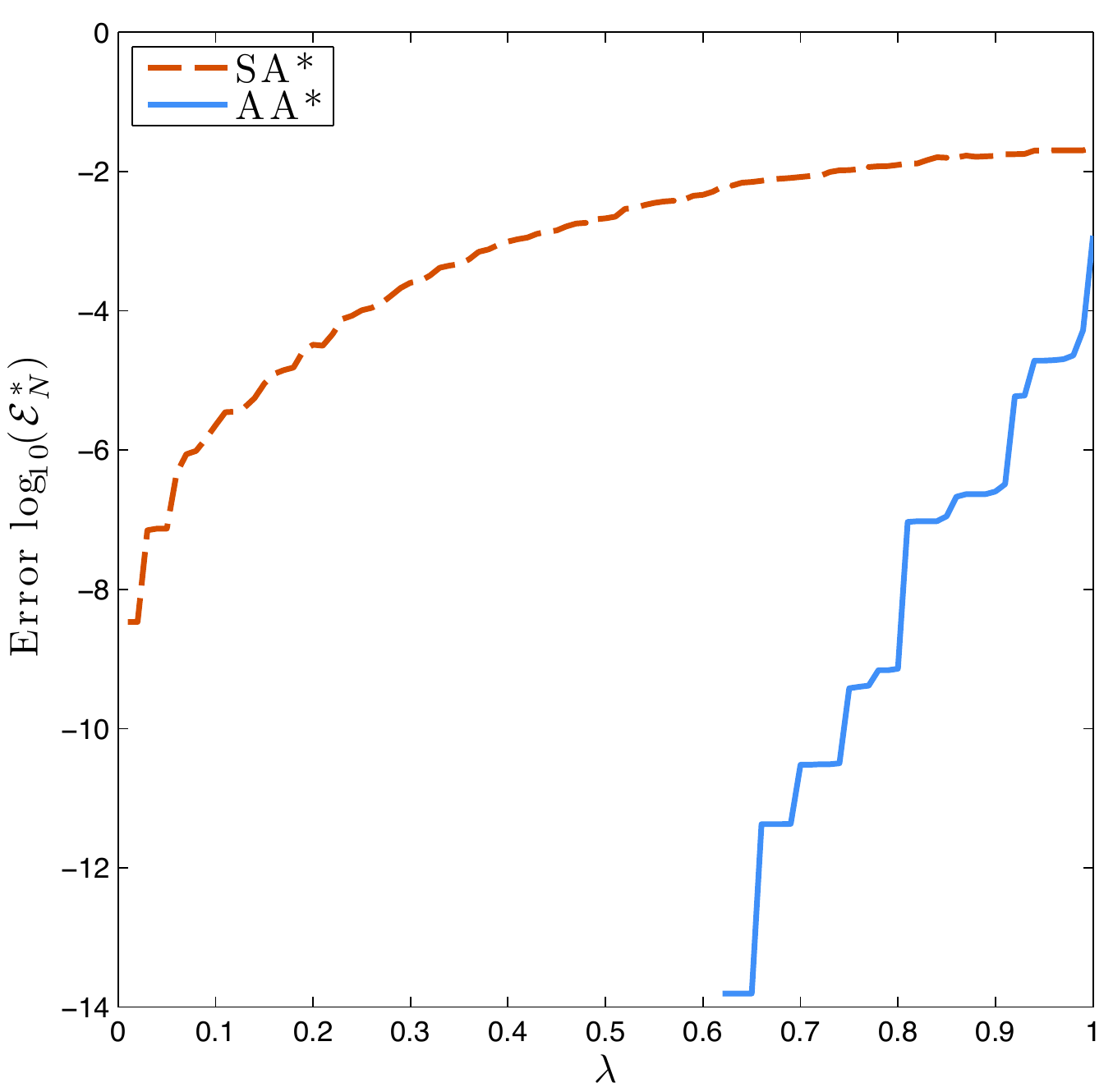}
\end{tabular}
\end{adjustwidth}
\caption{\footnotesize
% COMMENT REMOVED
\iftoggle{usecolor}{%
A. The $\lambda$-dependent ``optimal'' trajectories recovered by SA*-FMM (red, green, and blue curves).  AA*-FMM always recovers the truly optimal (red) trajectory.
When the trajectories overlap, the red red curve lies under the green, and the green curve lies under the blue.
% COMMENT REMOVED
B\&C. The time (in seconds) and $\astarErrOnly$ produced by SA* and AA* as $\lambda$ changes in $[0,1]$.
}{%
A. The $\lambda$-dependent ``optimal'' trajectories recovered by SA*-FMM: (1) the $\lambda \leq 0.62$ trajectory, (2) the $\lambda \in (0.62, 0.93]$ trajectory, and (3) the $\lambda > 0.93$ trajectory.  AA*-FMM always recovers the truly optimal trajectory (curve 1).
When the trajectories overlap, curve 1 lies under curve 2, which lies under curve 3.
% COMMENT REMOVED
B\&C. The time (in seconds) and $\astarErrOnly$ produced by SA* and AA* as $\lambda$ changes in $[0,1]$.
}%
}
% COMMENT REMOVED
\label{fig:sat_change}
\end{center}
\end{figure}

% COMMENT REMOVED
% COMMENT REMOVED
% COMMENT REMOVED
% COMMENT REMOVED
% COMMENT REMOVED
\subsection{Replanning in a dynamic environment}
\label{ss:observers}
Our final example illustrates several important points:
\begin{enumerate}[leftmargin=6mm]\itemsep-2pt
\item The use of special/custom underestimate $\varphi$ based on a related control problem.
\item The optimal trajectory from a related control problem is valuable as an initial guess for the Pontryagin Maximum Principle (PMP).
\item The PMP-computed trajectory  is not necessarily globally optimal, but can be used to produce an accurate $\Upper$.
\end{enumerate}
Here we will use a slightly more general setup where the task is to minimize the total \emph{cost} (instead of considering only the \emph{time}) to reach $\bt$. Given a running cost function $K:\Omega \to (0, +\infty)$ integrated along the trajectory and a speed $f_0$, the value function $u$ now satisfies a different Eikonal PDE given by
\begin{equation}\label{cost eikonal}
\abs{\nabla u (\bx)} f_0(\bx) \ = \ K(\bx).
\end{equation}
Our specific problem is to find the ``safest'' trajectory in an adversarial environment,
with $K$ higher on the parts of the domain more closely monitored by the adversary.
\fillmeup
If we assume no prior information on enemy locations and monitoring patterns, it is natural to select
$K \equiv 1$, which implies that the quickest trajectory is in fact the safest.
Consider the domain $\bar{\Omega} = [-0.05,0.85] \times [0, 0.9]$ with the speed and running-cost defined by
\[
f_0(x, y) \ \ = \ \ 1 \ + \ 0.99 \sin (4\pi x)\sin(4\pi y) \quad \mbox{ and } \quad K_0 \equiv 1.
\]
The solution $u$ to this \textsl{``no enemy observers''} problem is shown in Figure \ref{fig:priorFig}C.
Figure \ref{fig:priorFig}A shows the contours of $f_0$ with two locally optimal trajectories.
The `upper' solid trajectory is globally optimal and found by tracing the gradient of $u$; the `lower' locally optimal trajectory is computed using PMP.
\fillmeup
Our perception of the trajectory safety will change once we discover specific locations of enemy observers.
For example, if we know that there are two observers located at $\bx_1=(0.50, 0.77)$ and $\bx_2=(0.33, 0.45)$,
we might encode this new information in the cost function:
\[
K(\bx) \ \ = \ \ 1 \ + \ 2 \exp\parens{\frac{\abs{\bx - \bx_1}^2}{0.01}} \ + \ 8 \exp\parens{\frac{\abs{\bx - \bx_2}^2}{0.002}}.
\]
The solution to \eqref{cost eikonal} with speed $f_0$ and the above cost $K$ can be shown to satisfy \eqref{Eikonal intro} with $f = f_0 / K$.
% COMMENT REMOVED
The contours of this new modified speed $f$ can be seen in Figure \ref{fig:priorFig}B with two ``locally safest'' trajectories that can be viewed as perturbations of the locally time-optimal paths from Figure \ref{fig:priorFig}A.
Note that, because of the higher cost
% COMMENT REMOVED
around $\bx_1$, the `upper' locally optimal trajectory is no longer globally optimal.
\fillmeup
The full solution to the time-optimal problem becomes useful if we want to introduce A* techniques for all \textsl{``multiple enemy observers''} problems. Let $V_0(\bx)$ be the minimum time to reach $\bs$ using the speed $f_0$.  Suppose that $V_0$ is pre-computed by FMM and stored for the entire $X$.  Returning to the problem with known observers, we may take $\varphi = V_0$ since $f_0 \leq f \implies V_0 \leq V$. The overestimate $\Psi$ can be obtained
% COMMENT REMOVED
by integrating
% COMMENT REMOVED
$K/f_0$ along any feasible trajectory. If we use the globally time-optimal trajectory (the solid black curve in Figure \ref{fig:priorFig}A), this yields a good $\Psi_A \approx 0.6752$.  An even better overestimate $\Psi_B \approx 0.6447$ is obtained if we use the globally time-optimal trajectory as the initial guess for PMP, and then integrate $K/f_0$ along
the resulting ``locally safest'' trajectory (the `upper' black curve in Figure \ref{fig:priorFig}B).
\fillmeup
Figure \ref{fig:priorFig}D shows the solution level sets for the \textsl{``multiple enemy observers''} problem
together with boundaries of several computational sets.
% COMMENT REMOVED
The bold black curve is $\partial L$, showing the part of $\bar{\Omega}$ \ACC{} by FMM.  
% COMMENT REMOVED
% COMMENT REMOVED
The next (inward) bold curve is $\partial C_2$ with $\Psi = \Psi_B$ -- the boundary of a subset \ACC{} by AA*-FMM. 
The final bold curve is $\partial C_2$ with $\Psi = U(\bs)$ -- this approximates the boundary of a subset \ACC{} by SA*-FMM.  
% COMMENT REMOVED
(AA*-FMM would also restrict to the latter set, but only if we were lucky enough to start with $\Psi$ corresponding
to the `lower' %magenta 
% COMMENT REMOVED
curve in Figure \ref{fig:priorFig}B).
\fillmeup
Even though AA*-FMM computes the solution on a larger part of the domain, its computational efficiency is still comparable and the accuracy is superior to SA*-FMM. For example, with $m=201$ (using $100$ trial runs averaged for the time) we have
% COMMENT REMOVED
\begin{center}
\tabcolsep=0.52cm
\begin{tabular}{| c | c | c | c |}
\hline
\textbf{Method} &
\textbf{Time (seconds)} &
\textbf{Ratio $\mathbf{\mathcal{P}}$} &
\textbf{Error $\mathbf{\mathcal{E}_N^*}$} \\
\hline
\it FMM &
0.02665 &
0.82 &
0
 \\
\hline
\it SA*-FMM &
0.004950 &
0.130 &
0.02550 \\
\hline
\it AA*-FMM, $\mathit{\Psi = \Psi_B}$ &
0.0101 &
0.29 &
$1.77 \times 10^{-11}$ \\
\hline
\it AA*-FMM, $\mathit{\Psi = U(\bs)}$ &
0.00526 &
0.14 &
0.00150 \\
\hline
\end{tabular}
\end{center}

\iftoggle{usecolor}{%

% COMMENT REMOVED
\figstart
\begin{adjustwidth}{-1.5cm}{}
\begin{center}
\tabcolsep=2pt
\begin{tabular}{c c c c}
A. {\em Original speed $f_0$} &
B. {\em Modified speed $f = f_0 / C$} &
C. {\em Original solution $U_0$} &
D. {\em Modified solution $U$} \\
\includegraphics[scale=0.35]{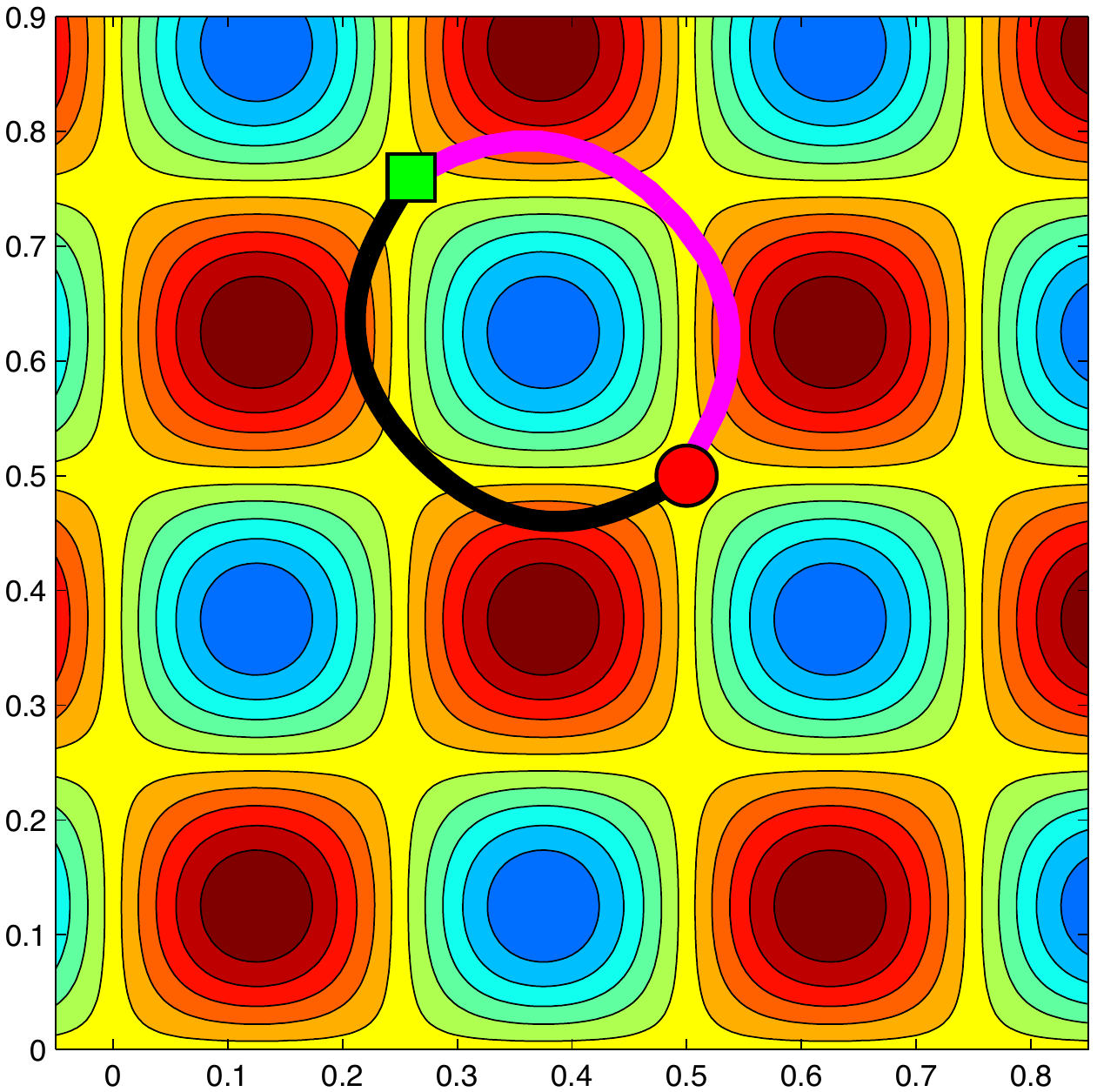} &
\includegraphics[scale=0.35]{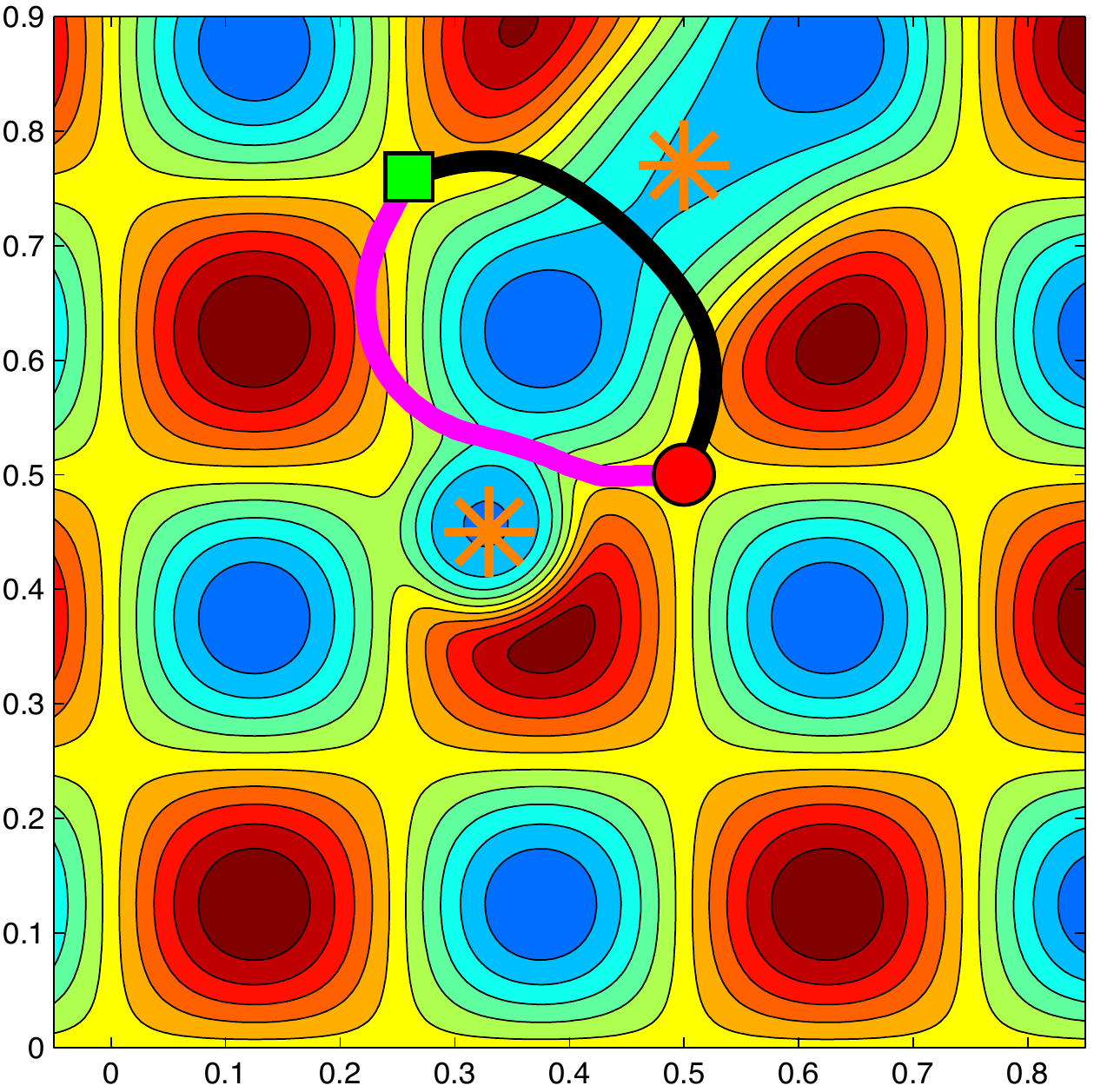} &
\includegraphics[scale=0.35]{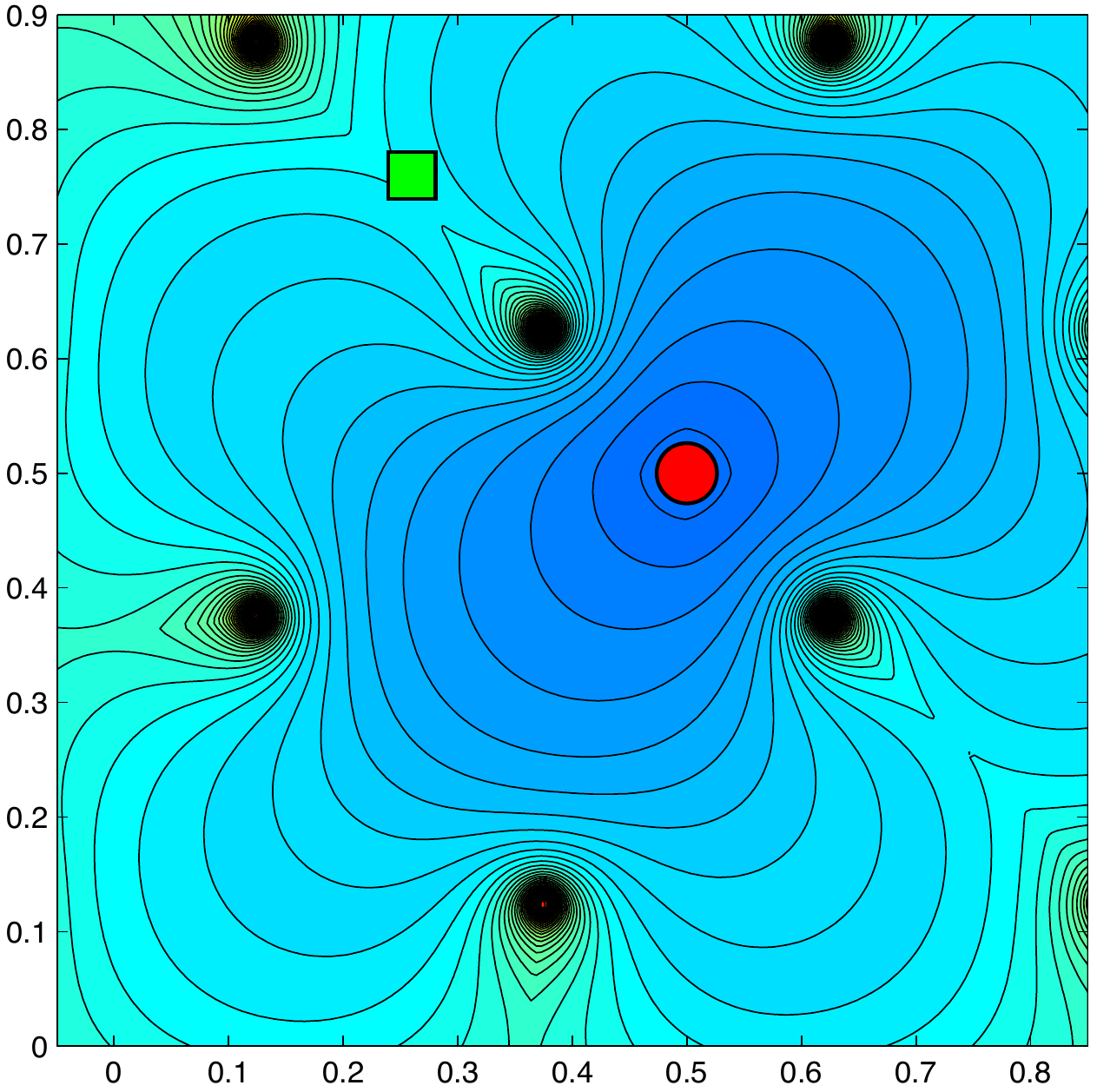} &
\includegraphics[scale=0.35]{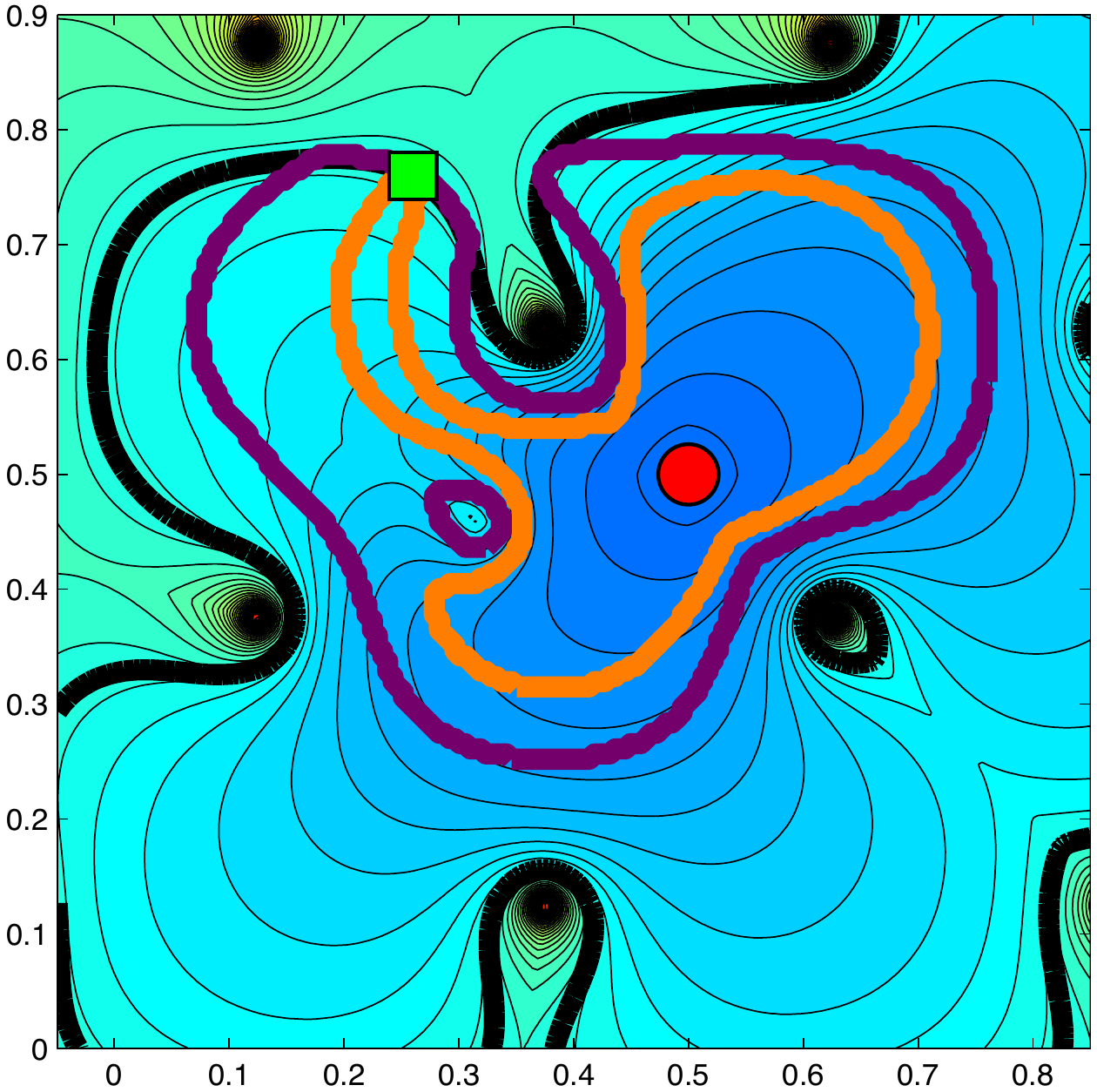} \\
% COMMENT REMOVED
\multicolumn{2}{c}{\includegraphics[scale=0.5]{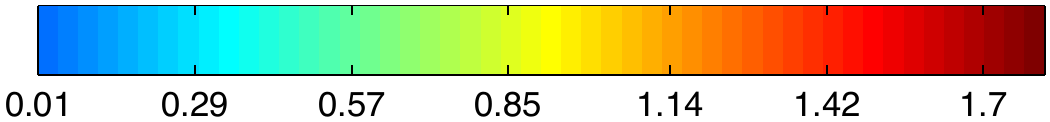} \ \ 
\includegraphics[scale=0.5]{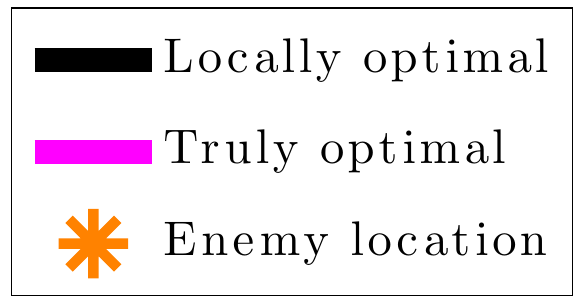}} &
\multicolumn{2}{c}{\includegraphics[scale=0.5]{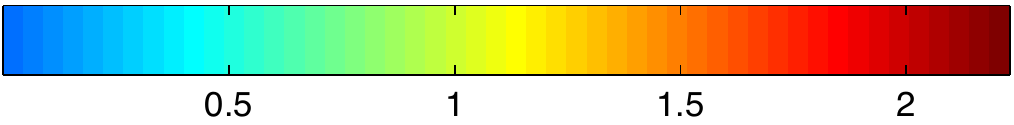} \ \ \includegraphics[scale=0.5]{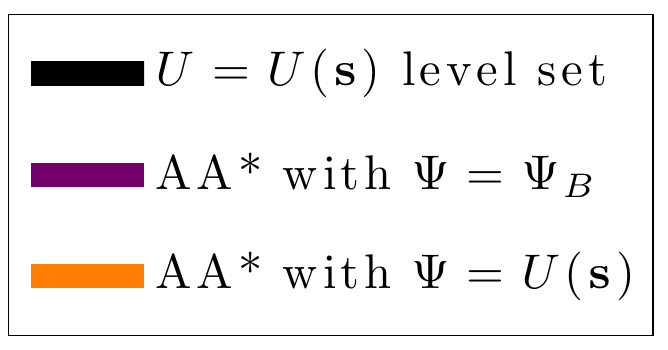}}
\end{tabular}
\end{center}
\end{adjustwidth}
\caption{\footnotesize \textbf{A.} Contours of the original speed function $f_0$. 
% COMMENT REMOVED
\textbf{B.} Contours of \textsl{``modified speed function''} $f = f_0 / K$ with the enemy locations shown by asterisks. 
% COMMENT REMOVED
\textbf{C.} Contours of the original solution to the problem with speed $f_0$ and constant running cost.  %\newline
\textbf{D.} Contours of the solution corresponding to the modified speed function $f = f_0 / K$ with $\partial L$ drawn in bold black. $\partial C_2$ is in dark purple using $\Upper = \Psi_B$, and in orange when using $\Upper = U(\bs)$. 
% COMMENT REMOVED
% COMMENT REMOVED
% COMMENT REMOVED
% COMMENT REMOVED
}
\label{fig:priorFig}
\end{figure}

}{%

% COMMENT REMOVED
\figstart
\begin{adjustwidth}{-1.5cm}{}
\begin{center}
\tabcolsep=2pt
\begin{tabular}{c c c c}
A. {\em Original speed $f_0$} &
B. {\em Modified speed $f = f_0 / C$} &
C. {\em Original solution $U_0$} &
D. {\em Modified solution $U$} \\
\includegraphics[scale=0.35]{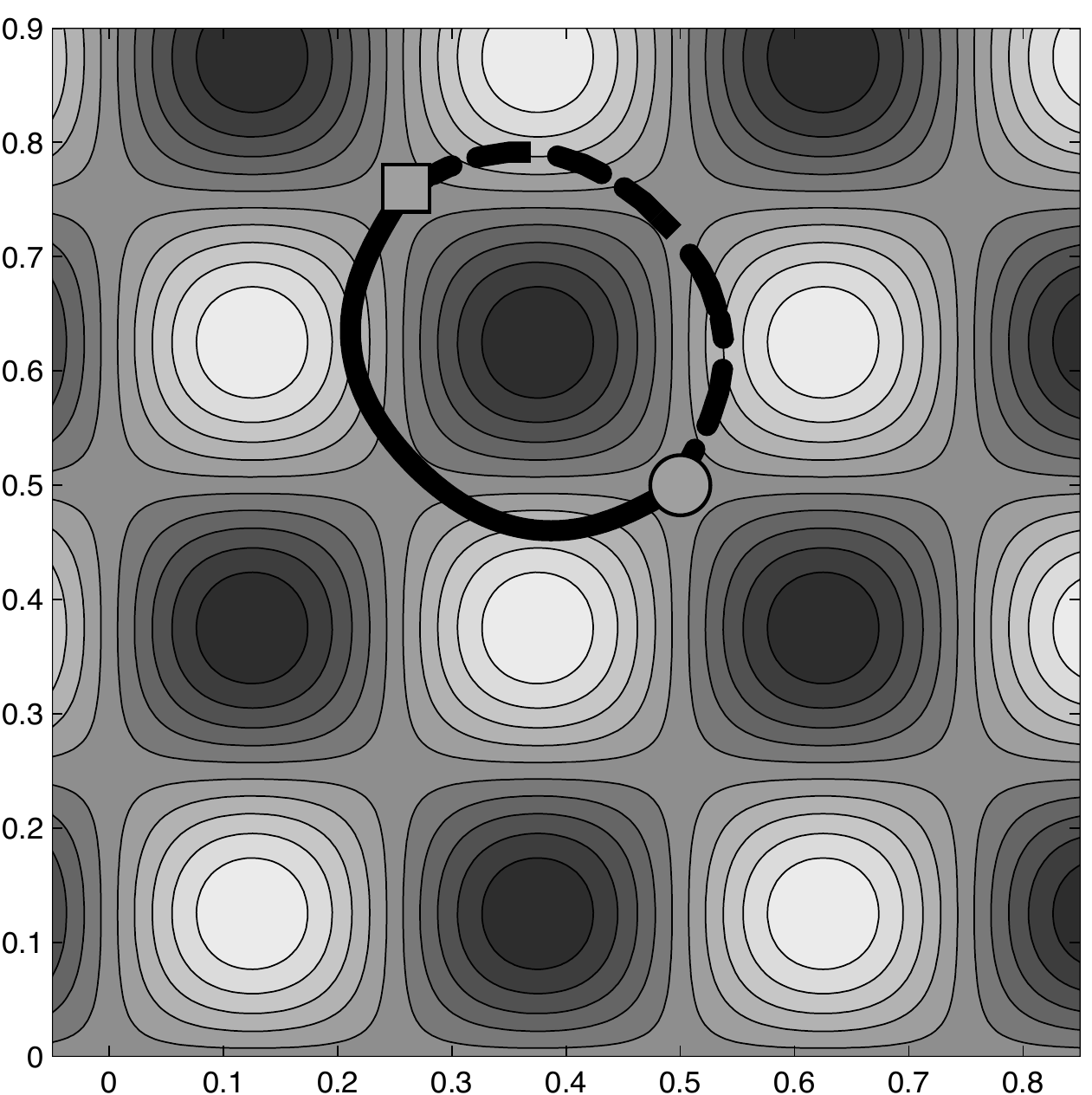} &
\includegraphics[scale=0.35]{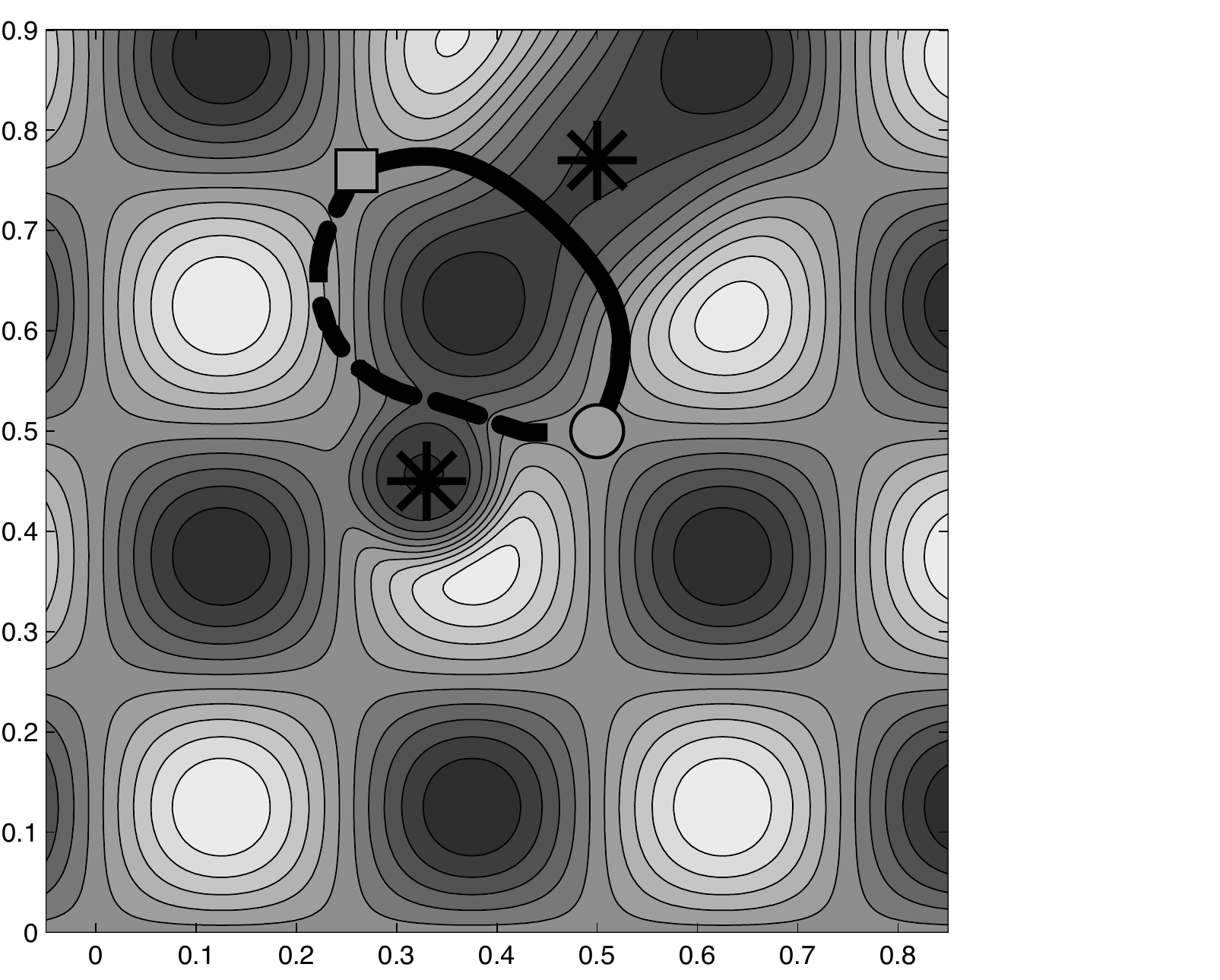} &
\includegraphics[scale=0.35]{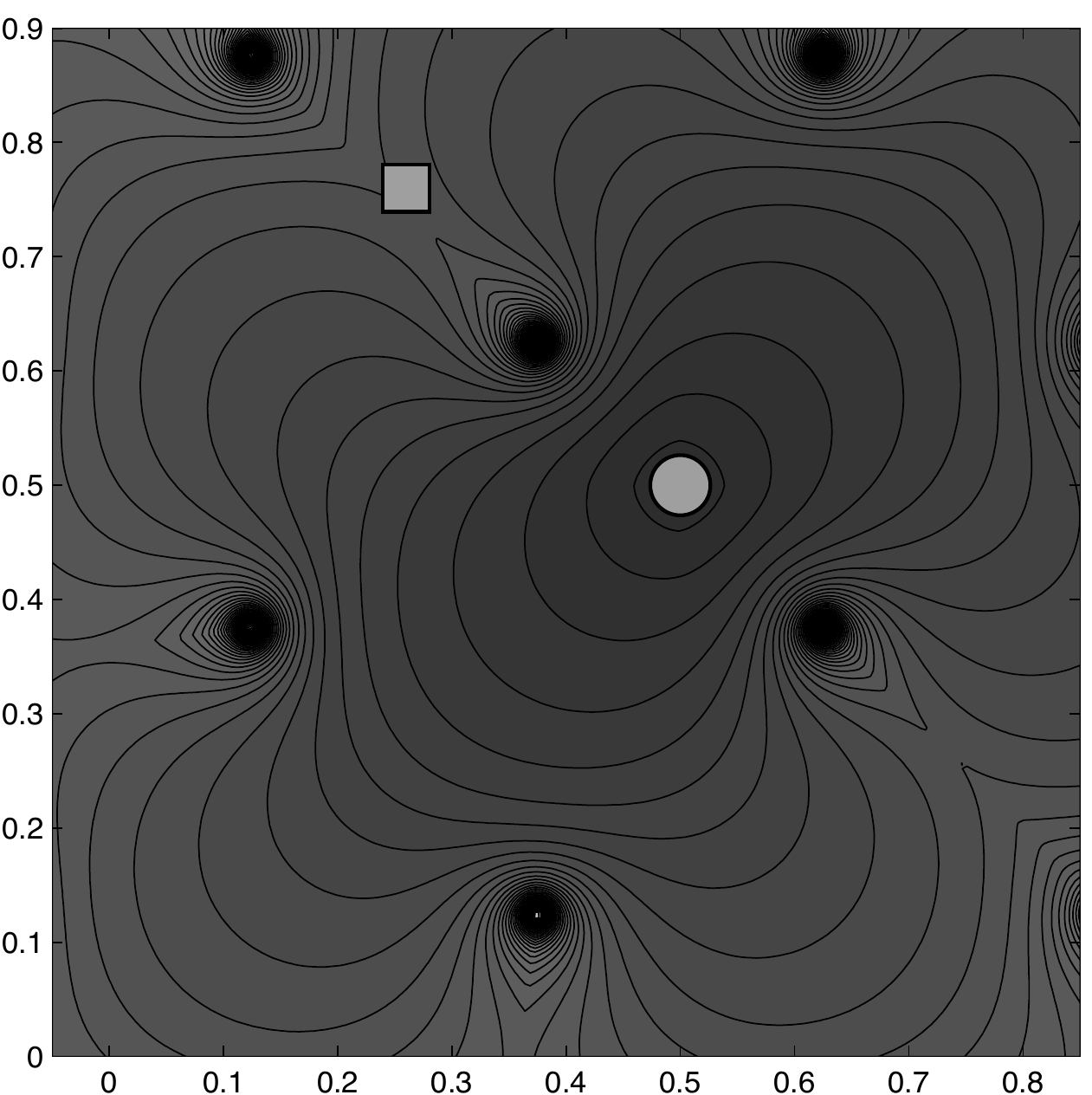} &
\includegraphics[scale=0.35]{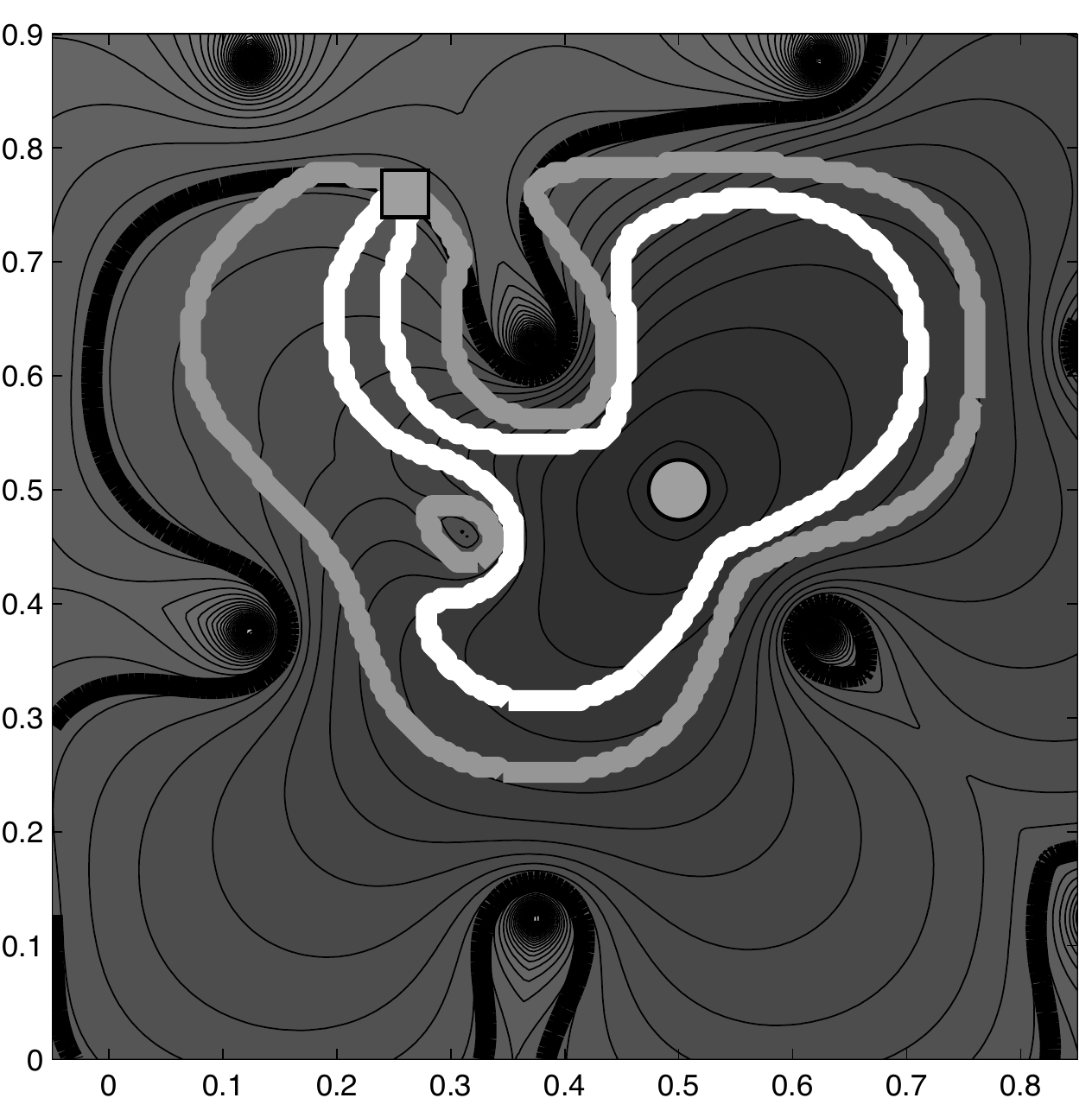} \\
% COMMENT REMOVED
\multicolumn{2}{c}{\includegraphics[scale=0.5]{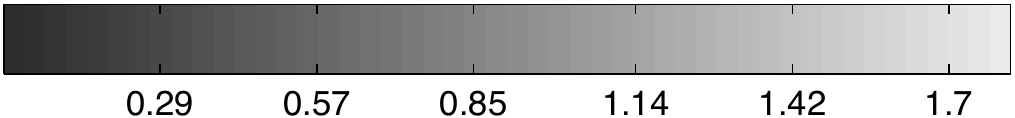} \ \ 
\includegraphics[scale=0.5]{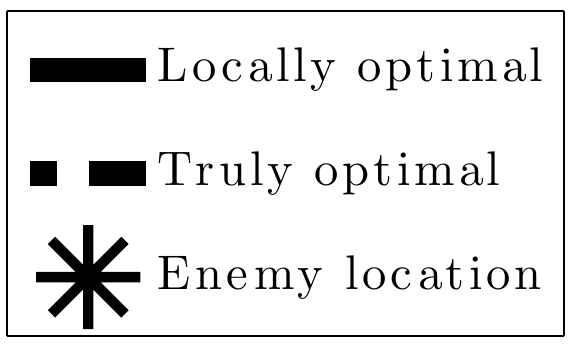}} &
\multicolumn{2}{c}{\includegraphics[scale=0.5]{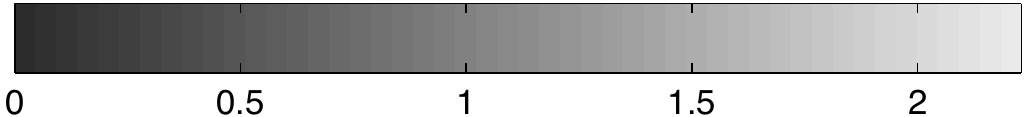} \ \ \includegraphics[scale=0.5]{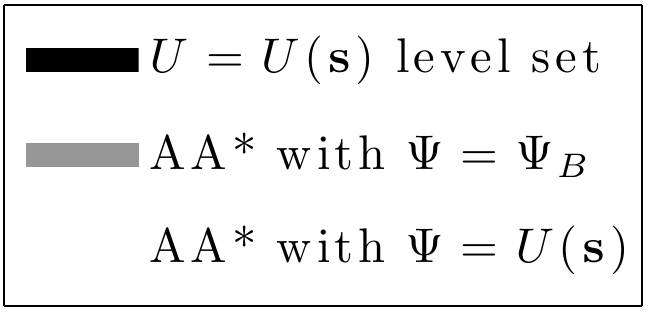}}
\end{tabular}
\end{center}
\end{adjustwidth}
\caption{\footnotesize \textbf{A.} Contours of the original speed function $f_0$. 
% COMMENT REMOVED
\textbf{B.} Contours of \textsl{``modified speed function''} $f = f_0 / K$ with the enemy locations shown by asterisks. 
% COMMENT REMOVED
\textbf{C.} Contours of the original solution to the problem with speed $f_0$ and constant running cost.  %\newline
\textbf{D.} Contours of the solution corresponding to the modified speed function $f = f_0 / K$. From out-to-in we show $\partial L$ (in bold black), $\partial C_2$ using $\Upper = \Psi_B$, and $\partial C_2$ using $\Upper = U(\bs)$. 
% COMMENT REMOVED
% COMMENT REMOVED
% COMMENT REMOVED
% COMMENT REMOVED
}
\label{fig:priorFig}
\end{figure}

}

% COMMENT REMOVED
\section{Conclusions}
\label{s:conclusions}

% COMMENT REMOVED

We have described a new A*-type modification of the Fast Marching Method solving Eikonal equations for a `single source/target' problem.
Unlike the prior methods for this problem, which were developed to mirror in the `standard A*' algorithm on graphs \cite{Hart_Astar},
our approach is based on a lesser known `alternative A*' \cite{Bertsekas_DPbook}.
These prior SA*-FMM methods \cite{FergusonStentz, Petres_thesis, Peyre_coarse, Peyre_landmark, Peyre_Geodesic, Yershov_1, Yershov_2}
either introduce additional errors that vanish slowly (if at all) under grid refinement, or must accept a much larger portion of the domain.
In contrast, our AA*-FMM is able to significantly restrict computations, with any additional errors quickly decreasing under grid refinement.
\fillmeup
One weakness of AA*-FMM is the reliance on an overestimate $\Upper$, especially when the feasibility of any $(\bs,\bt)$ trajectory is in question (e.g., in the presence of obstacles).
A good $\Upper$ can be also found from related control problems or based on Pontryagin Maximum Principle (PMP).  Here we mention two more approaches not tested in the current paper:
\begin{itemize}[leftmargin=6mm]\itemsep-2pt
\item
One can use $\Upper = \zeta U^C(\bs)$, where $\zeta > 1$ and $U^C$ is the solution found by FMM on a much coarser grid.
\item
One can also use the output of SA*-FMM on the same grid with an aggressive/inconsistent $\varphi$, setting $\Upper = U^*(\bs)$.
\end{itemize}
In the latter case, AA*-FMM should be viewed as a post-processing technique to improve the accuracy.  This might seem superfluous: after all, PMP could also be applied using the output of SA*-FMM as an initial guess. But as we show in Figures \ref{fig:sinOracle} and \ref{fig:sat_change}, the errors from SA*-FMM are likely to result in PMP converging to some other (locally, rather than globally) optimal trajectory.
\fillmeup
The effectiveness of the AA* domain restriction depends on the {\em quality} of
\zachEdit{$\varphi$ and}
$\Upper$.
If the initial $\Upper$ is overly conservative, it can also be improved dynamically using the Branch \& Bound techniques.
No benchmarking results for the latter approach were included here for the sake of brevity.
% COMMENT REMOVED
% COMMENT REMOVED
% COMMENT REMOVED
% COMMENT REMOVED
\fillmeup
We also list several desirable future extensions with significant impact on applications.
First, AA* can be used instead of SA* within D* and E* path replanners \cite{FergusonStentz, Phil}.
Second, the original AA* on graphs is applicable in both label-setting and label-correcting algorithms.
It should not be hard to incorporate the same idea into other non-iterative and fast iterative methods
for Hamilton-Jacobi PDEs.  Our preliminary results for the Locking Sweeping Method \cite{Bak} prove
the feasibility of this approach.
Third, since many gridpoints will never be used, allocating memory for the entire grid may be wasteful (particularly in high dimensions).
One approach, described in \cite{Peyre_coarse, Peyre_landmark, Peyre_Geodesic},
is to allocate gridpoints as needed and make use of a hash lookup table.
Our current implementation of AA*-FMM does not use this idea,
but we hope to explore it in the future.
Finally, we note that all of the A* techniques can be also trivially extended to problems with a single-source and multiple targets.
Similar underestimates can be also built for a moderately large set of sources $\{\bs_i\}$ (e.g., $\varphi = \min_i \varphi^0_i$).
\fillmeup
The error analysis in the Appendix relies on a conjecture, which so far has been only proven for a linear advection equation.  For the Eikonal case, we currently rely on experimental/numerical confirmation.  Nevertheless, we believe that a similar approach will be also useful in analyzing errors in more general domain restriction problems; e.g., for the errors due to an ``almost causal'' domain decomposition in \cite{Cacace_patchyFMM}.
\fillmeup
{\bf Aknowledgements.} The authors would like to thank Slav Kirov
% COMMENT REMOVED
for his contributions to the initial part of this project during the 2010-REU at Cornell University.
We would also like to thank Gabriel Peyr\'{e} for his input on the satellite image example used in Section \ref{ss:peyre_example}.

% COMMENT REMOVED
\section{Appendix: Why does it converge?}
\label{s:it_works}
If AA*-FMM is used with an inconsistent heuristic $\varphi$, the domain restriction usually affects the dependency graph (i.e., $G(\bs) \not\subset \hat{X}$), and the produced solution is larger than would result from running FMM on the full grid: $U^*(\bs) > U(\bs)$.
In this section we
% COMMENT REMOVED
analyze
why $(U^*(\bs) - U(\bs)) \to 0$ as $h \to 0$.
\fillmeup
We first note that the answer is simple if there exists an open set $\Omega_0 \subset \Omega$ such that
\begin{itemize}[leftmargin=6mm]\itemsep-2pt
\item the $(\bs,\bt)$-optimal trajectory lies in $\Omega_0$, and
\item and all gridpoints falling into $\Omega_0$ are accepted by AA*-FMM {\em regardless of $h$}.
\end{itemize}
In this case, an $\bar{\Omega}_0$-constrained viscosity solution will already yield the correct $u(\bs)$ in the limit.
In previous sections, we showed that such $\Omega_0$ often arises because $\Upper$ and/or $\varphi$ are
not tight.  But if the over/underestimates also improve in quality as $h \to 0$, then the AA*-FMM accepted region
shrinks under grid refinement, and a more careful argument is needed to explain the convergence.
\fillmeup
To address this, we compare
% COMMENT REMOVED
solutions produced by the original FMM
solving the same discretized system \eqref{Eikonal approx}
but on different grid subsets and with different boundary conditions.
For the rest of this section, we will not rely on the fact that $\hat{X}$ is defined through AA*-FMM.
As a benefit, our error analysis is also relevant for domain decomposition-based parallelizations of FMM;
e.g., see \cite{Cacace_patchyFMM}.\fillmeup
Consider a restriction of FMM computations to any
$\hat{X} \subset X$ containing both $\bs$ and $\bt$, and
define the ``restriction boundary'' set
$\Xi = \{ \bx \in X \backslash \hat{X} \, \mid \, N(\bx) \cap \hat{X} \neq \emptyset \}.$
For notational simplicity, we will assume that the $(\bs,\bt)$-optimal trajectory
is unique and that the upwind neighbors $(\bx^{}_{H}, \bx^{}_{V})$ are  uniquely defined for every gridpoint $\bx$.
\fillmeup
We will discuss the relationship between the following discretized solutions:
\begin{itemize}[leftmargin=6mm]\itemsep-2pt
\item
As before, $U$ denotes the solution on the entire $X$ with the boundary condition $U(\bt)=0$.
\item
$\hat{U}$ denotes the solution on $\hat{X}$ with the same boundary condition $\hat{U}(\bt)=0$.
We can also interpret it as a solution on $\hat{X} \bigcup \Xi$ with
$Q = \{\bt\} \bigcup \Xi$ and $q=+\infty$ on $\Xi$.  Recall that, if $\hat{X}$ is defined as the set of nodes \ACC{}
by AA*-FMM, then this method also produces the same solution (i.e., $U^*=\hat{U}$ on $\hat{X}$).
% COMMENT REMOVED
\item
$\bar{U}$ denotes the solution computed on $\hat{X} \bigcup \Xi$ with
$\bar{U}(\bt)=0$ and the more general boundary conditions $\bar{U}(\bx_i) = q_i$ specified$\Forall \bx_i \in \Xi$.
\end{itemize}
% COMMENT REMOVED
% COMMENT REMOVED
% COMMENT REMOVED
% COMMENT REMOVED
% COMMENT REMOVED
% COMMENT REMOVED
\begin{observ}
\label{obs:restrict}
The following properties are easy to verify based on the causality of \eqref{Eikonal approx}:
\begin{enumerate}[leftmargin=6mm]\itemsep-1pt
\item
$
q_i = U_i, \Forall \bx_i \in \Xi
\qquad \Longrightarrow \qquad
\bar{U}_j = U_j, \Forall \bx_j \in \hat{X};
$
\item
$
q_i \geq U_i, \Forall \bx_i \in \Xi
\qquad \Longrightarrow \qquad
\bar{U}_j \geq U_j, \Forall \bx_j \in \hat{X};
$
\item
$
\hat{U}_j \geq \bar{U}_j, \Forall \bx_j \in \hat{X};
$
\item
Suppose $C$ is a constant such that $C \geq \max_{\bx_j \in \hat{X}} \hat{U}_j$. Then\\
$
q_i \geq C, \Forall \bx_i \in \Xi
\qquad \Longrightarrow \qquad
\bar{U}_j = \hat{U}_j, \Forall \bx_j \in \hat{X}.
$
\item
Suppose $D(\bx)$ is the arclength of the shortest grid-aligned path within $\hat{X}$ from $\bx$ to $\bt$.
Then $C = \max_{\bx_j \in \hat{X}} D(\bx) / F_1 \, \geq \, \max_{\bx_j \in \hat{X}} \hat{U}_j.$
\end{enumerate}
\end{observ}
For any specific $\bx_i \in X$, if we define $\hat{X} = X \backslash \{\bx_i\}$
and choose $q_i > U_i$ this might result in $\hat{U}(\bs) > U(\bs)$.
This ``add-one-gridpoint-to-$Q$'' procedure motivates our definition of {\em sensitivity coefficients}:
$$
\alpha_i = \alpha(\bx_i) = \frac{\partial U(\bs)}{\partial U_i}
\text{ or, more rigorously, }
\alpha_i  = \frac{\partial \hat{U}(\bs)}{\partial q_i}
\text{ computed on $\hat{X} = X \backslash \{\bx_i\}$ with $q_i = U_i$.}
$$
Due to the monotonicity of \eqref{Eikonal approx}, $\alpha_i \geq 0$
and it is strictly positive if and only if $\bx_i \in G(\bs)$.
% COMMENT REMOVED
% COMMENT REMOVED
% COMMENT REMOVED
\begin{lemma}
\label{lm:alpha_relevance}
The net effect of a domain restriction can be bounded from above using $\alpha$'s even for a {\em general} set $\hat{X}$:
\begin{enumerate}[leftmargin=6mm]\itemsep-2pt
\item
If $q(\bx) \geq U(\bx), \Forall \bx \in \Xi,$ then
$\displaystyle
\bar{U}(\bs) - U(\bs) \; \leq  \; \sum\limits_{\bx \in \Xi} \alpha(\bx) \left(q(\bx)-U(\bx) \right).
$
\item
If $C \geq \hat{U}(\bx), \Forall \bx \in \hat{X}$, then
$\displaystyle
\hat{U}(\bs) - U(\bs) \; \leq  \; C \sum\limits_{\bx \in \Xi} \alpha(\bx).
$
\end{enumerate}
\end{lemma}
\begin{proof}
% COMMENT REMOVED
The upwind finite difference discretization \eqref{Eikonal discrete formula} is equivalent to
a semi-Lagrangian discretization:
\begin{equation}
\label{eq:semiL}
U(\bx_{ij}) \ \ = \ \  \min_{\beta \in [0, 1]}
\left\{ \frac{\abs{\beta \bx^{}_{H} + (1-\beta) \bx^{}_{V} - \bx_{ij}}}{f(\bx)} \ + \
\beta U(\bx^{}_{H}) \ + \ (1-\beta) U(\bx^{}_{V}) \right\}.
\end{equation}
Despite the very different Eulerian perspective and notation, \eqref{Eikonal discrete formula} can be actually derived from
Kuhn-Tucker optimality conditions for \eqref{eq:semiL}; see \cite{Tsitsiklis, SethVladOUMtheory, Vlad_MSSP}.
% COMMENT REMOVED
Moreover,
% COMMENT REMOVED
the latter can be also viewed as the dynamic programming equation for
a {\em Stochastic Shortest Path Problem} on the grid $X$; see \cite{Vlad_MSSP} for a detailed discussion.
In this interpretation, the transition from $\bx_{ij}$ to the neighboring node (either $\bx^{}_H$ or $\bx^{}_V$)
happens probabilistically, with respective probabilities $\beta$ and $(1-\beta)$, and
$\frac{| ( \beta \bx^{}_{H} + (1-\beta) \bx^{}_{V} - \bx_{ij} |}{f(\bx)}$ is the cost we incur for choosing
this probability distribution.  The process continues until we
reach
% COMMENT REMOVED
% COMMENT REMOVED
% COMMENT REMOVED
$\bt$,
and the goal is to select
$\beta_*: X \to [0,1]$ that minimizes the expected cumulative cost up to that termination.
We note that, for $\bx_{ij} = \bs$, we have $\alpha(\bx^{}_{V})  = (1-\beta_*(\bs))$, $\alpha(\bx^{}_{H})  = \beta_*(\bs)$
and $\alpha$ values on the rest of $G(\bs)$ can be similarly computed using \eqref{eq:semiL} recursively; see \cite{Chacon_thesis}.
Moreover, if we start from $\bs$ and use the optimal ``stochastic routing policy'' $\beta_*(\cdot)$,
then $\alpha(\bx)$ can be naturally interpreted as a probability of %reaching
passing through $\bx$ before arriving at $\bt$.
\fillmeup
Suppose now we use $\beta_*(\cdot)$, but on a $\hat{X}$-restricted problem, starting from $\bs$ and terminating the process
(+ paying the additional cost of $q(\bx)$) if we transition into any $\bx \in \Xi$ before reaching $\bt$.  Denote by $\tilde{U}$ the expected total cost of using this policy and by $\tilde{\alpha}(\bx)$ the probability of reaching $\bx$ before termination.
We first note that  $\tilde{\alpha}(\bx) \leq \alpha(\bx), \Forall \bx \in \hat{X} \cup \Xi$ since some stochastic paths previously leading through $\bx$ are now removed due to an earlier entry to $\Xi$.  Secondly,  $\tilde{U} \geq \bar{U}$, since
the latter is found by optimizing over all possible $\beta: \hat{X} \to [0,1]$, including the restriction of $\beta_*(\cdot)$.
Thus,
$$
\bar{U}(\bs) - U(\bs) \; \leq \; \tilde{U}(\bs) - U(\bs) \; = \;
\sum\limits_{\bx \in \Xi} \tilde{\alpha}(\bx) \left( q(\bx) - U(\bx) \right)
\; \leq \; \sum\limits_{\bx \in \Xi} \alpha(\bx) \left( q(\bx) - U(\bx) \right),
$$
which completes the proof of part 1.  To prove part 2, select $q(\bx) = C, \Forall \bx \in \Xi.$
Since the exit-penalty $C$ is prohibitively high, the stochastic path starting from $\bs \in \hat{X}$ and using the optimal routing policy will avoid $\Xi$ with probability 1.  Thus, $\bar{U}(\bs) = \hat{U}(\bs)$ (see the last part of Observation \ref{obs:restrict}), and using the above result
$$
\hat{U}(\bs) - U(\bs) \; \; \leq \; \sum\limits_{\bx \in \Xi} \alpha(\bx) \left( C - U(\bx) \right)
\; \leq \; C \sum\limits_{\bx \in \Xi} \alpha(\bx).
\vspace{-0.4cm}
$$
\end{proof}

Let $d(\bx)$ be the distance from $\bx$ to the characteristic passing through $\bs$ (i.e., the $(\bs,\bt)$-optimal trajectory).
\begin{conjecture}
\label{conj}
% COMMENT REMOVED
There exists a constant $\rho > 0$ such that, for small enough $h$, $\alpha(\bx) \leq e^{-\rho [d(\bx)]^2 / h}$.
\end{conjecture}
As of right now, we only have a rigorous proof of this statement for an upwind discretization
of a constant-coefficient advection PDE \cite[Chapter 4]{Chacon_thesis}.  The same proof also
covers the Eikonal equation when all characteristics are parallel, but this clearly does not
hold for the case $Q=\{\bt\}$.
Still, the numerical evidence (see Figure \ref{fig:alphas}) indicates that this exponential decay is also present
in the current context as well.

% COMMENT REMOVED
\figstart
% COMMENT REMOVED
\begin{center}
\tabcolsep=20pt
\begin{tabular}{c c}
A &
B\\
\iftoggle{usecolor}{%
\includegraphics[scale=0.44]{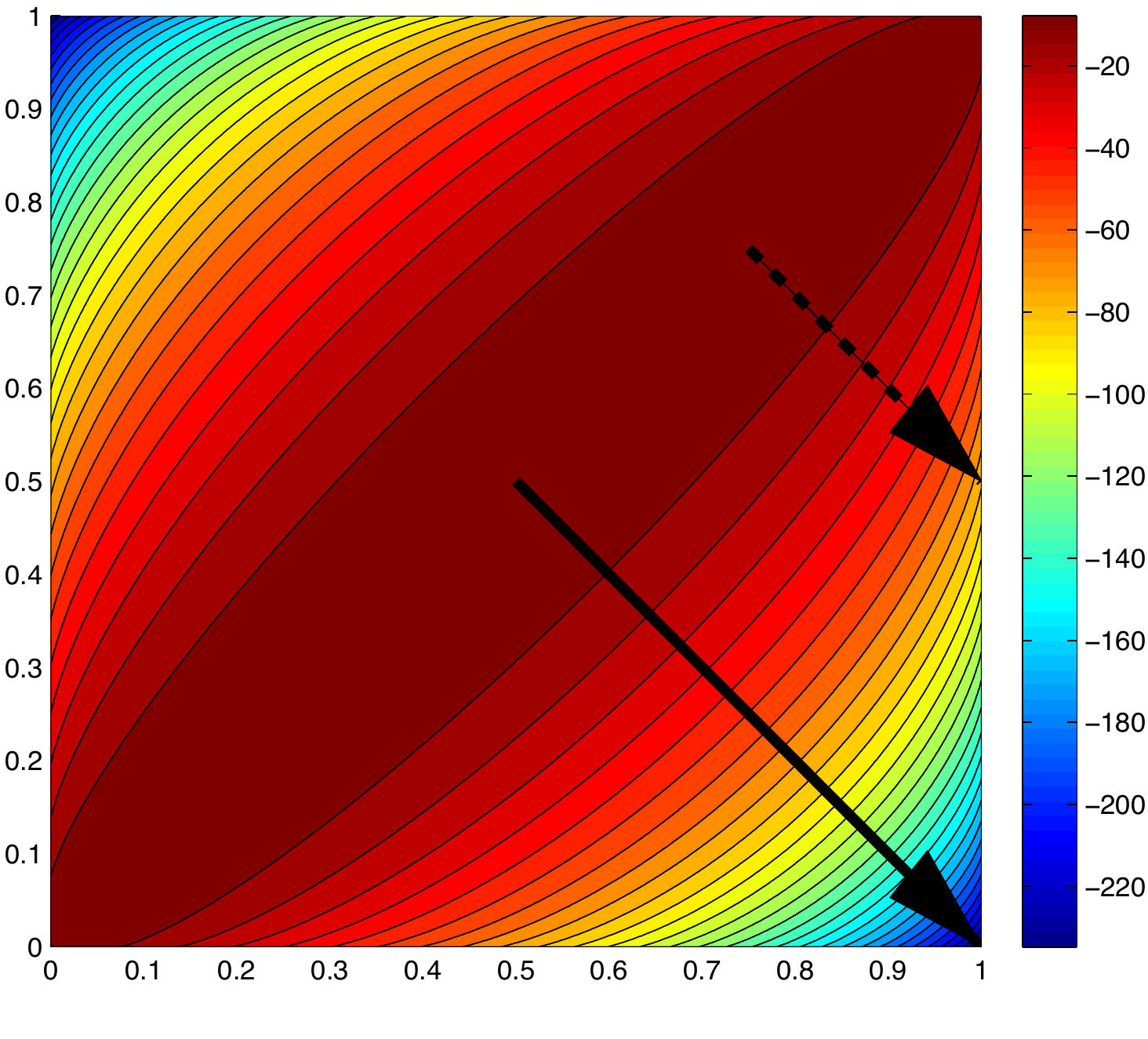} &
}{%
\includegraphics[scale=0.44]{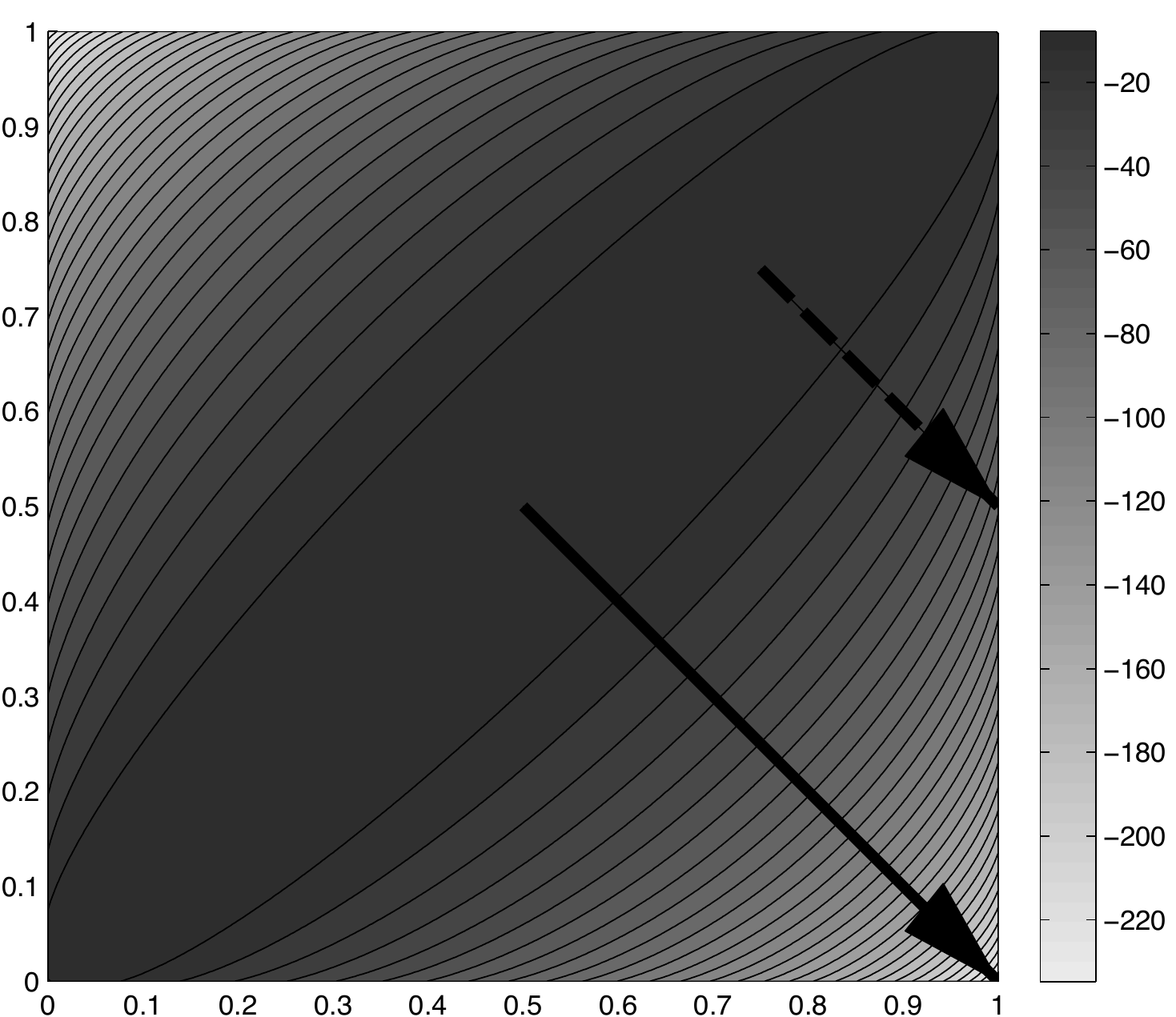} &
}
% COMMENT REMOVED
\includegraphics[scale=0.44]{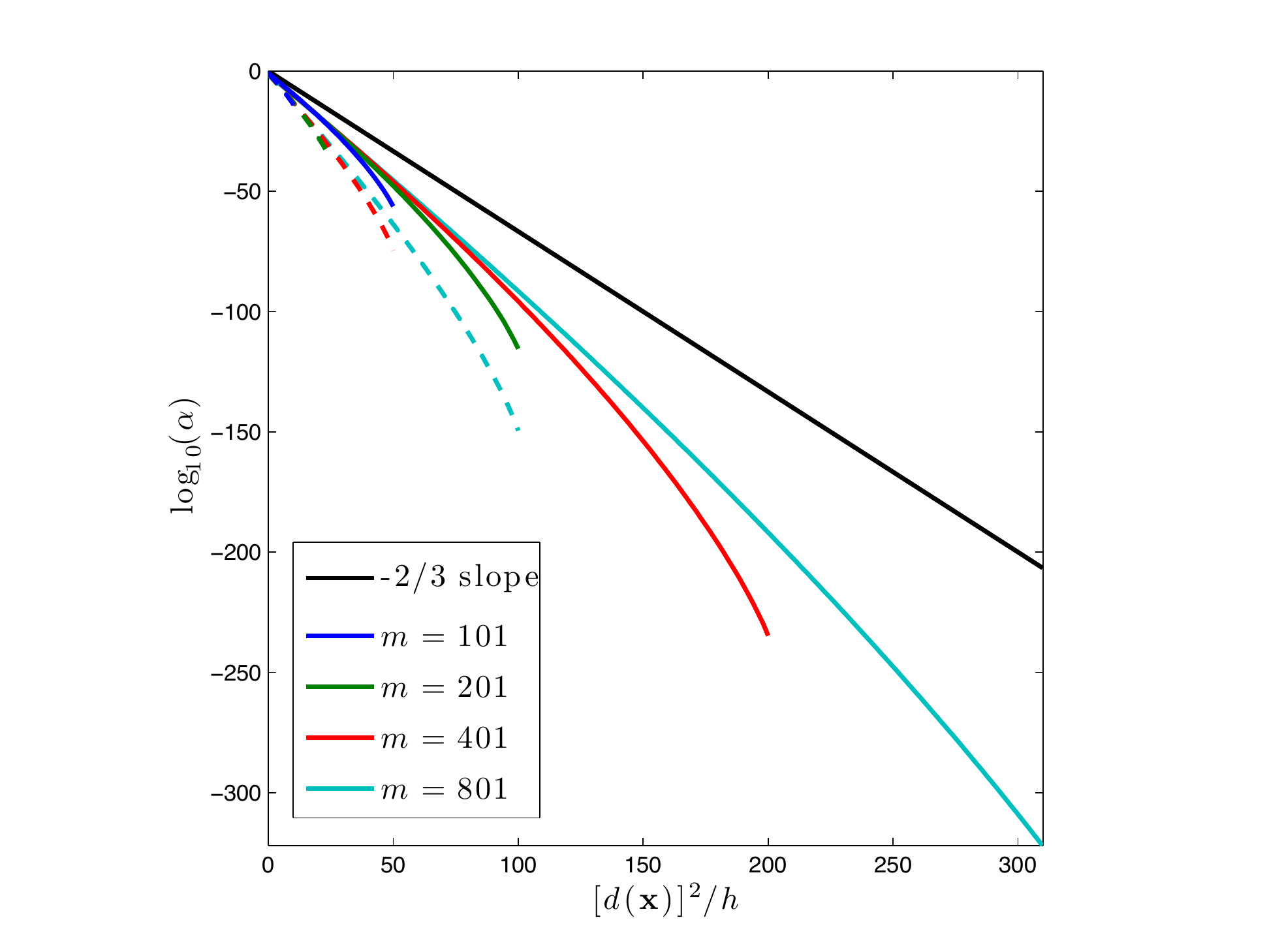}
\end{tabular}
\caption{\footnotesize Alpha values decaying away from the characteristic.  
Subfigure \textbf{A:} shows the level sets of $\log_{10}(\alpha)$ for the constant 
speed example considered in \S \ref{ss:example_constant}.  
The solid and dashed arrows are perpendicular to the $(\bs,\bt)$-optimal trajectory.
Subfigure \textbf{B} shows the rate of decay of $\log_{10}(\alpha)$ along each of these 
arrows computed for several different grid resolutions.
}
\label{fig:alphas}
\end{center}
% COMMENT REMOVED
\end{figure}

\begin{thm}
Let $\{X^h\}$ be a family of Cartesian grids on $\Omega$ with gridsize $h=1/(m-1)$ such that both $\bs$ and $\bt$ are gridpoints for all $m$.
Define $\hat{X}^h = \{ \bx \in X^h \mid d(\bx) < r$ \}, where $r = O(h^{\mu}),$ for some $\mu \in [0,\frac{1}{2})$.
Let $U^h$ and $\hat{U}^h$ be numerical solutions of the system \eqref{Eikonal approx} on $X^h$ and $\hat{X}^h$ respectively.
If Conjecture \ref{conj} holds, then $\left(\hat{U}^h(\bs) - U^h(\bs) \right) \to 0$ as $h \to 0$.
\end{thm}
\begin{proof}
% COMMENT REMOVED
We note that the $n$-volume of the optimal-trajectory-centered $r$-cylinder approaches zero, though the total number of gridpoints in $\hat{X}^h$ grows as $h \to 0$.
For convenience, we also define
% COMMENT REMOVED
\alexEdit{$k = r^2/h = O(h^{2\mu-1}),$ which tends to $+\infty$ as $h \to 0$.}
% COMMENT REMOVED
If $S$ is the path length of the $(\bs,\bt)$-optimal trajectory,
then the
% COMMENT REMOVED
number of gridpoints in $\Xi^h$ is $O(\frac{S r^{n-2}}{h^{n-1}}) = O( k^\nu )$, where
$\nu = \frac{(n-1) - \mu (n-2)}{1-2\mu} > 0$.
Considering the shortest grid-aligned and $\hat{X}^h$-constrained path from any $\bx \in \hat{X}^h$ to $\bt$,
it is easy to show that $D^h = (S+r) \sqrt{n}$ is the upper bound for that path's length.
Thus, $C^h = D^h / F_1$ is an upper bound for $\max_{\bx \in \hat{X}^h} \hat{U}^h(\bx)$.
If Conjecture \ref{conj} holds, then asymptotically
$\alpha^h(\bx) \leq e^{-\rho r^2 / h} = e^{-\rho k},$ for all $\bx \in \Xi_h$.
By Lemma \ref{lm:alpha_relevance},  $\left(\hat{U}^h(\bs) - U^h(\bs) \right)$ is bound from above by
$\left[C^h \sum\limits_{\bx \in \Xi^h} \alpha^h(\bx) \right] = O(k^\nu  e^{-\rho k}),$
which converges to $0$ under grid refinement.
\end{proof}

% COMMENT REMOVED


\begin{thebibliography}{99}

{\small

% COMMENT REMOVED
\bibitem{Bak} S. Bak, J. McLaughlin, \& D. Renzi, {\it Some Improvements for the Fast Sweeping Method}, SIAM J. Sci. Comput., Vol. 32, No. 5, pp. 2853-2874, 2010.

% COMMENT REMOVED
\bibitem{Bardi_ViscBook} M. Bardi and I. Capuzzo-Dolcetta, \textsl{Optimal Control and Viscosity Solutions of Hamilton Jacobi-Bellman Equations}, Birkh\"{a}user, 1997.

\bibitem{BarlesSouganidis}
G. Barles and P. E. Souganidis,
{\it Convergence of approximation schemes for fully nonlinear
second order equations}, Asymptot. Anal., 4:271-283, 1991.

% COMMENT REMOVED
\bibitem{Bellman_Opt} R.E. Bellman, {\it Dynamic Programming}, Princeton University Press, Princeton, NJ, 1957.

% COMMENT REMOVED
% COMMENT REMOVED

% COMMENT REMOVED
\bibitem{Bertsekas_DPbook} D.P. Bertsekas, {\it Dynamic Programming and Optimal Control}, 2nd Edition, Volumes I and II, Athena Scientific, Boston, MA, 2001.

% COMMENT REMOVED
% COMMENT REMOVED

% COMMENT REMOVED
% COMMENT REMOVED

% COMMENT REMOVED
% COMMENT REMOVED

\bibitem{Cacace_patchyFMM}
Cacace, S., Cristiani, E., Falcone, M., Picarelli, A.
{\it A patchy Dynamic Programming scheme for a class of Hamilton-Jacobi-Bellman equations},
SIAM J. Sci. Comp. Vol. 34, No. 5, pp. A2625–A2649, 2012.

\bibitem{Chacon_thesis} A. Chacon,
{\it Eikonal Equations: new two-scale algorithms and error analysis}, Ph.D. Thesis, Cornell University, 2013.

% COMMENT REMOVED
\bibitem{Adam} A. Chacon \& A. Vladimirsky, {\it Fast two-scale methods for Eikonal equations}, SIAM J. on Scientific Computing, 34/2, 2012.

% COMMENT REMOVED
% COMMENT REMOVED

% COMMENT REMOVED
\bibitem{Crandall_ViscOriginal} M.G. Crandall, P.-L. Lions, \textsl{Viscosity solutions of Hamilton-Jacobi equations}, Transactions of the American Mathematical Society, 277 (1), pp. 1--42,1983.

% COMMENT REMOVED
% COMMENT REMOVED

% COMMENT REMOVED
\bibitem{Dijkstra} E.W. Dijkstra, {\it A Note on Two Problems in Connexion with Graphs}, Numerische Mathematik, Vol. 1, pp. 269-271, 1959.

% COMMENT REMOVED
% COMMENT REMOVED

% COMMENT REMOVED
\bibitem{FergusonStentz} D. Ferguson \& A. Stentz, {\it Field D*: An interpolation-based path planner and replanner}, Proceedings of International Symposium on Robotics Research (ISRR), 2005.

% COMMENT REMOVED
% COMMENT REMOVED

% COMMENT REMOVED
% COMMENT REMOVED

% COMMENT REMOVED
\bibitem{Goldberg_Landmarks} A.V. Goldberg \& C. Harrelson, {\it Computing the Shortest Path: A* Search Meets Graph Theory}, Technical Report, Microsoft Research, 2004.

% COMMENT REMOVED
% COMMENT REMOVED

% COMMENT REMOVED
\bibitem{Hart_Astar} P.E. Hart, N.J. Nilsson, \& B. Raphael, {\it A Formal Basis for the Heuristic Determination of Minimum Cost Paths}, IEEE Transactions of Systems Science and Cybernetics, Vol. SSC-4, No. 2, pp. 100-107, 1968.

% COMMENT REMOVED
% COMMENT REMOVED

\bibitem{KimmelSethian_tri}
Kimmel, R. \& Sethian, J.A.,
{\it Fast Marching Methods on Triangulated Domains},
Proc. Nat. Acad. Sci., 95, pp. 8341-8435, 1998.


% COMMENT REMOVED
% COMMENT REMOVED

% COMMENT REMOVED
% COMMENT REMOVED

% COMMENT REMOVED
% COMMENT REMOVED

% COMMENT REMOVED
\bibitem{Petres_thesis} C. P\^{e}tr\`{e}s, \textit{Trajectory Planning for Autonomous Underwater Vehicles}, Heriot-Watt University, PhD Dissertation, 2007.

% COMMENT REMOVED
% COMMENT REMOVED

% COMMENT REMOVED
\bibitem{Peyre_coarse} G. Peyr\'{e} \& L.D. Cohen, \textit{Heuristically Driven Front Propagation for Geodesic Paths Extraction}, Proc. of VLSM '05 (N. Paragios, O. D. Faugeras, T. Chan, C. Schn\"{o}rr, eds.), Springer, vol. 3752, pp. 173-185, 2005.

% COMMENT REMOVED
\bibitem{Peyre_landmark} G. Peyr\'{e} \& L.D. Cohen, \textit{Landmark-Based Geodesic Computation for Heuristically Driven Path Planning}, Proc. of CVPR '06, IEEE Computer Society, pp. 2229-2236, 2006

% COMMENT REMOVED
\bibitem{Peyre_Geodesic} G. Peyr\'{e} \& L.D. Cohen, {\it Heuristically Driven Front Propogation for Fast Geodesic Path Extraction}, International Journal for Computational Vision and Biometrics Vol. 1, No. 1, pp. 55-67, 2008.

% COMMENT REMOVED
\bibitem{Phil} R. Philippsen, {\it A Light Formulation of the E* Interpolated Path Replanner}, Technical report, Autonomous Systems Lab, \'{E}cole Polytechnique F\'{e}d\'{e}rale de Lausanne, 2006.

% COMMENT REMOVED
\bibitem{bidir dijkstra} I. Pohl, {\it Bi-directional Search}, Machine Intelligence, vol. 6, eds. Meltzer and Michie, Edinburgh University Press, pp. 127-140, 1971.

% COMMENT REMOVED
% COMMENT REMOVED
% COMMENT REMOVED

% COMMENT REMOVED
\bibitem{Pontryagin_original} L. S. Pontryagin, V. Boltyanskii, R. V. Gamkrelidze, \& E. F. Mishenko, \textsl{The Mathematical Theory of Optimal Processes}, Wiley, 1962.

% COMMENT REMOVED
\bibitem{RouyTour} E. Rouy \& A. Tourin, {\it A Viscosity Solutions Approach to Shape-From-Shading}, SIAM J. Num. Anal., 29, 3, pp. 867-884, 1992.

% COMMENT REMOVED
\bibitem{Seth_FMM} J.A. Sethian, {\it A Fast Marching Level Set Method for Monotonically Advancing Fronts}, Proc. Nat. Acad. Sci., 93, 4, pp. 1591--1595, February 1996.

% COMMENT REMOVED
\bibitem{SethBook} J.A. Sethian, {\it Level Set Methods and Fast Marching Methods: Evolving Interfaces in Computational Geometry, Fluid Mechanics, Computer Vision and Materials Sciences}, Cambridge University Press, 1996.

% COMMENT REMOVED
\bibitem{SethVlad_trimesh} J.A. Sethian \& A. Vladimirsky, {\it Fast Methods for the Eikonal and Related Hamilton--Jacobi Equations on Unstructured Meshes}, Proc. Nat. Acad. Sci., 97, 11 (2000), pp. 5699--5703.

% COMMENT REMOVED
% COMMENT REMOVED

% COMMENT REMOVED
\bibitem{SethVladOUMtheory} J.A. Sethian \& A. Vladimirsky, {\it Ordered Upwind Methods for Static Hamilton-Jacobi Equations: Theory \& Algorithms}, SIAM J. on Numerical Analysis 41, 1, pp. 325-363, 2003.

% COMMENT REMOVED
\bibitem{Tsitsiklis} J.N. Tsitsiklis {\it Efficient Algorithms for Globally Optimal Trajectories}, IEEE Tran. Automatic Control, 40, pp. 1528--1538, 1995.

% COMMENT REMOVED
% COMMENT REMOVED

\bibitem{Vlad_MSSP} A. Vladimirsky,
{\it Label-setting methods for Multimode Stochastic Shortest Path
problems on graphs}, Mathematics of Operations Research 33(4), pp. 821-838, 2008.

% COMMENT REMOVED
\bibitem{Yershov_1} D.S. Yershov, S.M. LaValle, \textit{Simplicial Dijkstra and A* Algorithms for Optimal Feedback Planning}, in Proceedings IEEE/RSJ International Conference on Intelligent Robots and Systems (IROS), 2011.

% COMMENT REMOVED
\bibitem{Yershov_2} D.S. Yershov, S.M. LaValle, \textit{Simplicial Dijkstra and A* Algorithms: From Graphs to Continuous Spaces}, Advanced Robotics, Vol. 26, no. 17, pp. 2065-2085, 2012

% COMMENT REMOVED
\bibitem{Zhao_FSM} H. Zhao, {\it A Fast Sweeping Method for Eikonal Equations}, Mathematics of Computation, Vol. 74, Num. 250, pp. 603-627, 2004.

}

\end{thebibliography}
\end{document}